\documentclass{amsart}
\usepackage{amsmath,amssymb}
\usepackage{graphicx,tikz}
\newtheorem{theorem}{Theorem}[section]

\newtheorem{corollary}[theorem]{Corollary}
\newtheorem{conjecture}[theorem]{Conjecture}
\newtheorem{lemma}[theorem]{Lemma}
\newtheorem*{principle}{Heuristic principle}

\theoremstyle{definition}
\newtheorem{definition}[theorem]{Definition}

\theoremstyle{remark}
\newtheorem{remark}[theorem]{Remark}

\newcommand{\al}{\alpha}

\newcommand{\de}{\delta}
\newcommand{\ep}{\varepsilon}

\newcommand{\ga}{\gamma}

\newcommand{\la}{\lambda}

\newcommand{\si}{\sigma}

\newcommand{\vp}{\varphi}

\newcommand\vka\varkappa
\newcommand{\De}{\Delta}

\newcommand{\Si}{\Sigma}
\newcommand{\Om}{\Omega}

\def\NN{\mathbb{N}}
\def\RR{\mathbb{R}}

\def\ZZ{\mathbb{Z}}

\newcommand{\pd}{\partial}
\newcommand{\id}{{\rm id}}

\newcommand\lan\langle
\newcommand{\sign}{\operatorname{sign}}

\newcommand{\inte}{\operatorname{int}}

\DeclareMathOperator\Vol{Vol}

\DeclareMathOperator\dist{dist}

\renewcommand\leq\leqslant
\renewcommand\geq\geqslant
\newlength{\intwidth}

\numberwithin{equation}{section}

\newcommand{\triple}[1]{{\left\vert\kern-0.25ex\left\vert\kern-0.25ex\left\vert #1
        \right\vert\kern-0.25ex\right\vert\kern-0.25ex\right\vert}}

\newcounter{bigthm}

\newtheorem{bigtheorem}[bigthm]{Theorem} 

\usepackage{etoolbox}\makeatletter\newif\ifdavid@number\preto\equation{\david@numberfalse}\preto\endequation{\ifdavid@number\else\notag\fi}\patchcmd\label@in@display{\@empty}{\@empty\david@numbertrue}{}{}\makeatother

\theoremstyle{definition}

\theoremstyle{remark}

\usepackage{hyperref}

\def\RR{\mathbb{R}}

\DeclareMathOperator\MW{MW}
\newcommand\dMW{d_{\MW}}

\newcommand\muRMW{\mu_{\mathrm{RMW}}}
\newtheorem*{conjecture*}{Conjecture}
\def\ZZ{\mathbb{Z}}
\def\NN{\mathbb{N}}
\usepackage{graphicx,tikz}

\usepackage{multirow}

\DeclareMathOperator\LCM{LCM} 

\usetikzlibrary{calc}
\DeclareMathOperator\GCD{GCD} 

\newcommand\restr[2]{{%
		\left.\kern-\nulldelimiterspace 
		#1 
		\vphantom{\big|} 
		\right|_{#2} }}

\usepackage{amsmath}
\usepackage{array}
\usepackage{multirow}

\newcommand{\I}{{\mathrm I}}

\usepackage{tikz-3dplot}

\newcommand{\LM}{LM}

\usepackage{appendix}
\newcommand\dz{d_{0}}
\newcommand\dk{d_{k}}

\newcommand\PsiR{\Psi_{\mathrm{RMW}}}
\usepackage{amstext} 
\usepackage{array}   

\begin{document}

\title[Local limits of high energy eigenfunctions on integrable billiards]{Local limits of high energy eigenfunctions\\ on integrable billiards}
 \author{Alberto Enciso}
 \address{Instituto de Ciencias Matem\'aticas, Consejo Superior de
   Investigaciones Cient\'\i ficas, 28049 Madrid, Spain}
 \email{aenciso@icmat.es}

 \author{Alba Garc\'\i a-Ruiz}
\address{Instituto de Ciencias Matem\'aticas, Consejo Superior de
   Investigaciones Cient\'\i ficas, 28049 Madrid, Spain}
 \email{alba.garcia@icmat.es}

 \author{Daniel Peralta-Salas}
 \address{Instituto de Ciencias Matem\'aticas, Consejo Superior de
   Investigaciones Cient\'\i ficas, 28049 Madrid, Spain}
 \email{dperalta@icmat.es}

%
%

\begin{abstract}
Berry's random wave conjecture posits that high energy eigenfunctions of cha\-otic systems resemble random monochromatic waves at the Planck scale. One important consequence is that, at the Planck scale around ``many'' points in the manifold, any solution to the Helmholtz equation $\Delta\varphi+\varphi =0$ can be approximated by high energy eigenfunctions. This property, sometimes called {\em inverse localization}\/, has useful applications to the study of the nodal sets of eigenfunctions. Alas, the only manifold for which the local limits of a sequence of high energy eigenfunctions are rigorously known to be given by random waves is the flat torus $(\RR/\ZZ)^2$, which is certainly not chaotic. 

Our objective in this paper is to study the validity of this ``inverse localization'' property in the class of integrable billiards, exploiting the fact that integrable polygonal billiards are classified and that Birkhoff conjectured that ellipses are the only smooth integrable billiards. Our main results show that, while there are infinitely many integrable polygons exhibiting good inverse localization properties, for ``most'' integrable polygons and ellipses, this property fails dramatically. We thus conclude that, in a generic integrable billiard, the local limits of Dirichlet and Neumann eigenfunctions do not match random waves, as one might expect in view of Berry's conjecture. Extensions to higher dimensions and nearly integrable polygons are discussed too.
\end{abstract}
\maketitle
\tableofcontents

\section{Introduction}\label{Intro}

Berry's random wave conjecture~\cite{Be77} asserts that, when localized around a point at the Planck scale, high energy eigenfunctions on a chaotic billiard (or on any other manifold whose geodesic flow is ergodic and mixing) should typically behave like random monochromatic waves. This is a very challenging open problem in quantum chaos, and striking little is known about it after almost fifty years. More precisely, while no counterexamples are known, the only manifold where these local limits have been computed for a sequence of eigenfunctions (and in fact shown to correspond to random waves) is the standard flat 2-torus $\mathbb T^2:=(\RR/\ZZ)^2$, which is hardly chaotic. The proof relies on Bourgain's derandomization technique~\cite{Bourgain}.

To obtain some intuition, recall that a key property of random waves is that they approximate any fixed solution to Helmholtz equation $\Delta u + u=0$ on any ball with positive probability, and that this probability does not depend on the point the ball is centered at. Therefore, a logical consequence of Berry's conjecture (which can indeed be rigorously derived once the conjecture is precisely formulated) is that high-energy eigenfunctions of a chaotic system should replicate the behavior of any such~$u$ over the natural length scale. Determining whether local limits of eigenfunctions on a model chaotic system, such as the stadium billiard or a compact hyperbolic surface,  satisfy this property would be extremely interesting; however, achieving this appears to be beyond current reach.

Instead, our objective is to analyze whether this approximation property (which we often refer to as {\em inverse localization}\/) holds for the simpler class of {\em integrable}\/ billiards. Our motivation is twofold. On the one hand, no domains where this inverse localization property holds are known, and this property is useful because it ensures the existence of, for instance, eigenfunctions with highly complicated nodal sets. On the other hand, proving that integrable billiards do not have this property is also interesting, as it ensures that local limits are not given by random waves. Despite the example of~$\mathbb T^2$ mentioned above, this is what one would expect to happen in a typical integrable system.

A significant advantage of integrable billiards is that they are essentially classified: within the class of polygons, there are all rectangles and three specific triangles, while within the class of smooth domains, it is famously conjectured that ellipses are the only examples. Our main result is that, on the one hand, ``most'' of these billiards (specifically, all irrational rectangles and a full measure set of ellipses) do not exhibit sequences of eigenfunctions whose local limits are given by random waves. On the other hand, all rational rectangles and all integrable triangles possess useful inverse localization properties. Some of these results extend to higher dimensions and to the wider class of ``almost integrable'' polygons.

\subsection{Berry's random wave conjecture}

Let us start by introducing some notation. Let us consider a bounded planar domain $\Om\subset\RR^2$ whose boundary $\partial\Om$ is a connected, piecewise smooth curve. We then consider the eigenvalue problem
\begin{equation}
	\De u + \la u=0\quad \text{in }\Om\,,
\end{equation}	
supplemented with the boundary condition
\begin{equation}\label{BCsi}
	\cos\si \, \pd_\nu u +\sin \si \, u =0 \quad \text{on }\pd\Om
\end{equation}
for some fixed constant $\si$. We typically consider either Neumann ($\si=0$) or Dirichlet ($\si=\frac\pi2$) boundary conditions. We will denote by
\begin{equation}
	0\leq \la_1<\la_2\leq \la_3\leq \cdots
\end{equation}
the spectrum of the Laplacian on~$\Omega$ with the boundary condition~\eqref{BCsi}, and fix an associated orthonormal basis of eigenfunctions $\{ u_n\}_{n=1}^\infty$. 

The asymptotic properties of high-energy eigenfunctions are well known to be related to the billiard defined by the domain~$\Om$. This billiard is the dynamical system which corresponds to a point particle undergoing rectilinear motion inside~$\Omega$, with specular reflections at the boundary~$\pd\Omega$, as described e.g.\ in~\cite{ChernovMarkarian2006}.

A paradigmatic result showing the impact of billiard trajectory structure on eigenfunctions\footnote{As is well known, this result holds on general Riemannian manifold with boundary, under mild assumptions.} is the Quantum Ergodicity Theorem~\cite{sh,ze-QUE,CdV,GL,ZZ}. Informally, it states that if the billiard flow, defined on a full-measure subset of the cosphere bundle \( S^*\Om \), is ergodic, then a full-density subsequence of eigenfunctions  becomes asymptotically equidistributed. Specifically, the  measure associated with the squared modulus of the eigenfunctions, $u_{n_j}^2\, dz$, converges weakly to the Lebesgue measure on \(\Om\). Moreover, after a microlocal lift, the eigenfunctions exhibit equidistribution over the cosphere bundle, as one would expect in view of the classical dynamics.

If the classical system is not only ergodic but also mixing, so it exhibits chaotic dynamics, it remains a mostly open problem to determine how this stronger form of unpredictability influences the behavior of eigenfunctions. In fact, there are at least three major conjectures in quantum chaos exploring specific aspects of this problem. First, the Quantum Unique Ergodicity (QUE) conjecture, proposed by Rudnick and Sarnak in 1994~\cite{RudnickSarnak1994}, asserts that if the system exhibits chaotic dynamics, then all eigenfunctions, not just a full-density subsequence, become equidistributed as $\lambda_n\to\infty$. While QUE has been proven for certain arithmetic hyperbolic surfaces (e.g.\ \cite{lindenstrauss2006Invariant,silberman2007Arithmetic, holowinsky2010mass}), counterexamples exist in some ergodic billiards~\cite{Hassel} and in quantum cat maps~\cite{debievre2003scarred}.

Another conjecture, due to Berry--Tabor (1977) and  Bohigas--Giannoni--Schmidt (1984), asserts~\cite{BerryTabor1977,BohigasGiannoniSchmit1984}  that eigenvalue spacings in chaotic systems should typically exhibit Wigner--Dyson statistics, while in integrable systems they typically follow a Poisson distribution. Recent results of Rudnick~\cite{Rudnick2022} suggest (yet fail to prove) that the Berry--Tabor conjecture could hold for surfaces of large genus and typical hyperbolic metrics.

A third conjecture, which serves as an important motivation for this paper, is the Berry's 1977 random wave conjecture~\cite{Be77}. This asserts that if the system is chaotic, then the eigenfunctions~$u_n$ with $n\gg1$ should typically resemble random monochromatic waves at the natural length scale $O(\lambda_n^{-1/2})$, known as the {\em Planck scale}\/. Despite numerous numerical explorations, there are essentially no rigorous partial results on this conjecture, and even rigorously formulating it is not immediate.

One should observe that it is of course natural to obtain a monochromatic wave when one picks a sequence of high energy eigenfunctions and zooms in to analyze it at the Planck scale. Basically, this is because the rescaled eigenfunction
\begin{equation}\label{E.unz0}
	u_{n,z^0}(x):=u_n(z^0+\la_n^{-1/2}x)\,,
\end{equation}
with $x$ in some ball~$B$ and $\la_n\gg1$, for some fixed $z^0\in\Om$, satisfies the Helmholtz equation
\begin{equation} \label{E-Helmholtz}
\Delta\vp+\vp=0
\end{equation}
on~$B\subset\RR^2$. An analogous argument obviously works for eigenfunctions on general Riemannian manifolds, modulo inessential lower order errors. Of course, this does not mean that {\em any}\/ solution to~\eqref{E-Helmholtz} (which is often called a {\em monochromatic wave}\/) can be recovered this way, that is, as a sort of ``local limit'' of high energy eigenfunctions, up to a suitably small error. What is true in great generality is that this can be done using {\em approximate}\/ high energy eigenfunctions. The proof of this fact, which is standard, can be found in~\cite[Appendix~A]{ILtori}.

In contrast to these general facts, Berry's conjecture that local limits of eigenfunctions behave as random monochromatic waves provides a powerfully compelling picture of the expected typical structure of high energy eigenfunctions at the Planck scale, when~$\Om$ is chaotic. In particular, since an important qualitative property of random monochromatic waves is that the support of its probability measure  is the space of all monochromatic waves (see e.g.~\cite[Lemma A.4]{NazarovSodin2016}), it stands to reason that, given a sequence $u_{n_j}$ of orthonormal eigenfunctions on a billiard~$\Om$ (or on any other Riemannian manifold), one would expect that {either} the local limits of these eigenfunctions can reproduce {\em any}\/ monochromatic wave up to an arbitrarily small error, or there is no way that a sequence can asymptotically behave like random monochromatic waves over the Planck scale. 

Our objective in this paper is to explore whether a typical {\em integrable}\/ billiard has the above property, which we often call {\em inverse localization}\/ for short. Knowing that a sequence of high energy eigenfunctions has this inverse localization property is interesting by itself, particularly in high dimensions, because one can use this property to ensure the existence of eigenfunctions with highly nontrivial nodal and critical sets. However, our main motivation is to show that high energy eigenfunctions on typical integrable billiards do {\em not}\/ satisfy this property and so, do not behave like random monochromatic waves. 

It should be emphasized that the fact that Berry's random wave conjecture implies the inverse localization property is a natural consequence. This can be in fact proven once Berry's conjecture is formulated rigorously.

A neat rigorous formulation of Berry's conjecture on chaotic billiards, which we recall in Appendix~\ref{AppendixA} for the benefit of the reader, was presented recently by Ingremeau~\cite{Ingremeau}. Around the same time, Abert, Bergeron and Le Masson~\cite{ABLM} developed an essentially equivalent formulation on symmetric spaces of negative curvature. The link between these two formulations has been clarified
in \cite{Yo}. Essentially, the idea is to associate to each subsequence of eigenfunctions $\{u_{n_j}\}_{j=1}^\infty$ a set of Borel measures $\si^\Om(\{u_{n_j}\}_{j=1}^\infty)$ on the space of monochromatic waves
\begin{equation}
	\MW:=\{\vp\in C^\infty(\RR^2,\RR): \De \vp + \vp =0\}
\end{equation}
endowed with the distance
\begin{equation}
	\dMW(\vp_1,\vp_2):=\inf\left\{\ep>0: \sup_{|z|<1/\ep}|\vp_1(z)-\vp_2(z)|<\ep  \right\}\,,
\end{equation}
which makes it a Polish space. Specifically, $\si_\Om(\{u_{n_j}\}_{j=1}^\infty)$ is the set of accumulation points of the {\em local measures} $\{m^\Om_{u_{n_j}}\}$ defined by the eigenfunctions~$\{u_{n_j}\}$. Heuristically, the local measure $m^\Om_{u_{n}}$ quantifies how often the eigenfunction~$u_n$ will resemble each monochromatic wave $\vp\in\MW$ in a neighborhood of a point~$x$ as this point varies over~$\Om$.  One should also notice that the random monochromatic wave naturally defines a Gaussian probability measure on the space~$\MW$, which we denote by $\muRMW$. Berry's conjecture can thus be precisely formulated as follows:

\begin{conjecture*}[Berry's random wave conjecture] Let $\Om\subset\RR^2$ be a chaotic billiard (in the sense of~\cite[Section 2]{ChernovMarkarian2006}, so in particular the billiard flow is ergodic and it has positive topological entropy). Then, there exists a subsequence of positive integers $\{n_j\}_{j=1}^\infty$ with full density such that
	\begin{equation}\label{E.siQUE}
		\si^\Om(\{u_{n_j}\}_{j=1}^\infty)=\{\muRMW\}\,.
	\end{equation}
\end{conjecture*}

With these definitions in hand, the principle that Berry's conjecture implies a strong type of inverse localization can then be easily established. In fact, our fairly elementary Theorem~\ref{ILBerry} ensures the validity of the following\footnote{By elliptic regularity, the choice of norm is not important; we could have used any fixed H\"older or Sobolev norm instead of~$C^0(B)$. One can also replace~$B$ by any other compact set. Also, the same results hold in higher dimension; we state everything in two dimensions for concreteness.}:

\begin{principle}
	[If Berry's conjecture holds, local limits can typically reproduce any monochromatic wave.] Let $u_{n_j}$ be a sequence of eigenfunctions on~$\Om$ satisfying~\eqref{E.siQUE} and fix a ball $B\subset\RR^d$. Then, for any $\vp\in \MW$, any measurable $\Om'\subset\Om$ and any $\ep>0$, it can be shown that there exists some integer~$j$ and a positive measure subset $O_{\ep,\vp}\subset\Om'$ such that the rescaled eigenfunctions~\eqref{E.unz0} satisfy
	\begin{equation}\label{E.impliedbyBerry}
	\|u_{n_j,z^0}-\vp\|_{C^0(B)}<\ep
\end{equation}
for all $z^0\in O_{\ep,\vp}$.
\end{principle}

In any case, Berry's random wave conjecture has proven to be extremely challenging, and not a single chaotic system has been rigorously shown to satisfy it, in any reasonable sense. The only space in which a sequence of eigenfunctions has been shown to converge~\cite{Ingremeau} to the random wave model is the standard 2-dimensional flat torus $(\RR/\ZZ)^2$, where the flow is neither ergodic nor chaotic. As a matter of fact, the proof that there exists a sequence  of eigenfunctions satisfying~\eqref{E.siQUE} hinges on the derandomization trick introduced by Bourgain~\cite{Bourgain}, which in turn exploits the high multiplicity of certain eigenvalues, a feature characteristic of integrable systems. 

Despite this, one would expect that~\eqref{E.siQUE} does not hold for typical  integrable billiards. This is the question we aim to explore in this paper, utilizing the strong inverse localization property~\eqref{E.impliedbyBerry}.

Before moving on to presenting our main domains, it is worth recalling that the list of spaces where a (often weak) form of inverse localization is known to hold is strikingly short. It essentially consists of Laplace eigenfunctions on $d$-dimensional flat tori satisfying certain arithmetic conditions \cite{ILtori}, on the $d$-dimensional round sphere~\cite{EPT21} and its quotients by finite isometry groups (lens spaces), and the eigenfunctions of the quantum harmonic oscillator~\cite{JEMS}. Note that these systems do not exhibit ergodic nor chaotic behavior. In a way, the case of flat tori \cite{ILtori} turns out to be the most interesting, as one can show not only the existence of certain tori where an inverse localization property holds, but that generic flat tori do not possess this property.

\subsection{Integrable polygonal domains}

\subsubsection{Classification of integrable planar polygons}

To study the validity of inverse localization on integrable billiards, let us first consider the class of integrable {\em polygons}\/. It is known~\cite{GutkinII} that the only integrable planar polygons are all rectangles and three triangles, which we simply call {\em integrable triangles}\/:

\begin{definition}\label{IntPol}
	Modulo isometries and dilations, the {\em integrable polygons}\/ in $\RR^2$ are the following (see Figure~\ref{polygons}):
		\begin{itemize}
			\item The {\em equilateral triangle}\/ $\mathcal{T}_{\mathrm{equi}}$,  with angles $\frac{\pi}{3},\frac{\pi}{3},\frac{\pi}{3}$.
			\item The {\em isosceles right triangle}\/ $\mathcal{T}_{\mathrm{iso}}$, with angles $\frac{\pi}{2},\frac{\pi}{4},\frac{\pi}{4}$.
			\item The {\em hemiequilateral triangle}\/ $\mathcal{T}_{\mathrm{hemi}}$, with angles $\frac{\pi}{2},\frac{\pi}{3},\frac{\pi}{6}$.
		\item All the {\em rectangles}\/ $\mathcal{Q}_l:=(0,l)\times(0,1)$, with sides $1$ and $l>0$. The rectangle~$\mathcal Q_l$ is said to be {\em rational}\/ if $l^2\in\mathbb{Q}$ and {\em irrational}\/ otherwise.
		\end{itemize}
\end{definition}
\begin{remark}
	In the case of rectangles, we will also consider $d$-dimensional rectangles, i.e., domains defined in terms of $l=(l_1,\dots,l_{d-1})\in\RR_+^{d-1}$ as
	\begin{equation}
		\mathcal{Q}_l:=(0,l_1)\times\cdots \times (0,l_{d-1})\times (0,1)\subset \RR^d\,.
	\end{equation}
	These are integrable billiards in~$\RR^d$. In this setting, we say that $\mathcal{Q}_l$ is {\em  rational} if $l_j^2\in\mathbb{Q}$ for all $1\leq j\leq d-1$, and  {\em irrational} otherwise.  
\end{remark}

\begin{figure}\renewcommand\thefigure{1}
	\includegraphics[width=12.5cm]{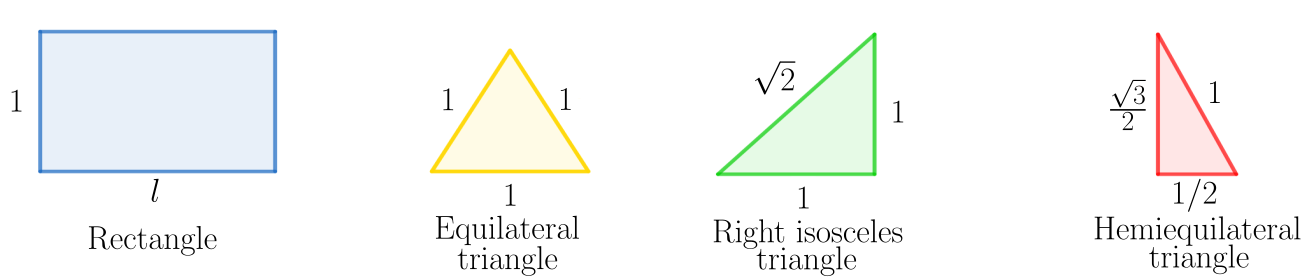}
	\caption{The integrable polygonal billiards.}\label{polygons}
\end{figure} 

For the benefit of the reader, let us recall that the classical Katok-Zemljakov construction \cite{Construct} associates any simple polygon $\mathcal{P}$ with a surface $S=S(\mathcal{P})$. If $\mathcal{P}$ is a  polygon with rational angles (i.e., the angles between its sides are all rational multiples of~$\pi$) the surface is compact. Otherwise, $S\left(\mathcal{P}\right)$ leads to a non-compact surface homeomorphic to the Loch Ness monster \cite[Theorem 1]{Valdez}.	 For  polygons with rational angles and $p$ sides, it is known (see e.g. \cite{Gutkin}) that the genus of $S(\mathcal P)$ satisfies 
\begin{equation}\label{eqbil}
	g(S(\mathcal P))=1+\frac{N}{2} \sum_{i=1}^p \frac{m_i-1}{n_i},
\end{equation}
where all the angles of $\mathcal{P}$ are denoted by $\pi m_i/n_i$, and where $N$ is the least common multiple of the integers $n_i$.

The integrable polygonal billiards~\cite{GutkinII} are those whose associated  surfaces are flat tori. Their billiard flow has no singularities, and therefore they do define an integrable Hamiltonian system \cite{GutkinII,Gutkin}.  Since Equation \eqref{eqbil} implies that $S(\mathcal P)$ is a torus if and only if $\mathcal P$ has all its angles of the form $\frac{\pi}{k}$, this condition leads to the list of polygons presented in the above definition. It is not difficult to check that integrable polygons are the only polygons that tile the plane under reflections in their sides.

By generalizing this idea, one arrives at another very interesting, and less restrictive, class of polygons~\cite{GutkinII}: {almost integrable polygons}. A polygon $\mathcal{P}\subset\mathbb{R}^{2}$ is called {\em almost integrable}\/ if the group $G_\mathcal{P}$ generated by the reflection in the sides of $\mathcal{P}$ is a discrete subgroup of $O(\mathbb{R}^2)$. It is known that the only infinite discrete groups of motions of the plane generated by reflections are the ones that come from reflections in the sides of the integrable polygons. Therefore, the group $G_\mathcal{P}$ of an almost integrable polygon must be isomorphic to one of those four. 

Let us denote by $L_{\mathcal{P}}$, the lattice, i.e., the pattern of lines, obtained by tiling the plane with copies of an integrable polygon $\mathcal{P}$. Thus $L_{Q_{l}}$ is a rectangular lattice, $L_{\mathcal{T}_{\mathrm{equi}}}$ is the lattice of equilateral triangles, $L_{\mathcal{T}_{\mathrm{iso}}}$ is usually called the Tetrakis square tiling, and  $L_{\mathcal{T}_{\mathrm{hemi}}}$ is the $3-6$ Kisrhombille tiling. Since $G_\mathcal{P}$ must be isomorphic to one of the groups mentioned above, the almost integrable polygon $\mathcal{P}$ must be drawn on the corresponding lattice\footnote{Recall that it is said that a polygon $\mathcal{P}$ is drawn on the lattice $L$ if the vertices of $\mathcal{P}$ belong to the set of lattice vertices, and the sides of $\mathcal{P}$ belong to the lines of the lattice.}. This is illustrated in Figure~\ref{F.ai}. Taking all this into account, we can use the following definition to list all almost integrable polygonal billiards:

\begin{definition}\label{ai}
	We say that $\mathcal{P}\subset\RR^2$ is an {\em almost integrable}\/ polygon if it is drawn on one of the following four lattices: $L_{\mathcal{Q}_{l}}$, $L_{\mathcal{T}_{\mathrm{equi}}}$, $L_{\mathcal{T}_{\mathrm{iso}}}$ or $L_{\mathcal{T}_{\mathrm{hemi}}}$.
	If $\mathcal{P}$ is drawn in the lattice $L_{\mathcal{Q}_{l}}$ for an irrational rectangle  we say that it is an {\em irrational}\/ almost integrable polygon, and otherwise we say it is a {\em rational}\/ one.
\end{definition}

Almost integrable polygons are studied in \cite{GutkinII}. There it is proved that the geodesic flow of an almost integrable billiard, which is a Hamiltonian
system with two degrees of freedom, has an additional integral of motion that ``almost commutes'' with the Hamiltonian. Thus, almost integrable billiards are not integrable, but they are still ``very far'' from being chaotic.

\begin{figure}\label{F.ai}\renewcommand\thefigure{2}
	\includegraphics[width=10.5cm]{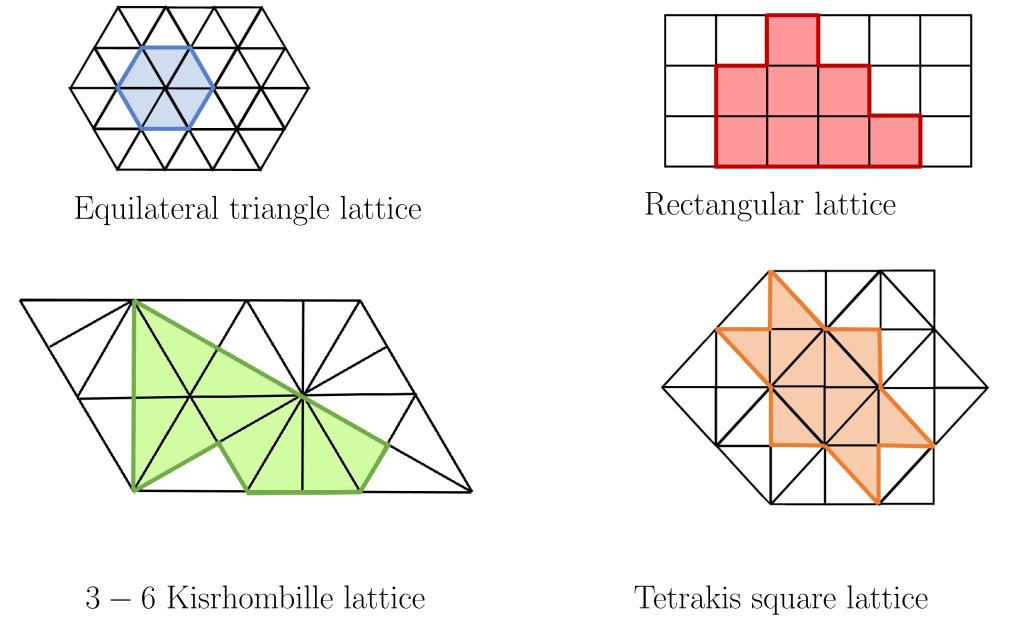}
	\caption{Examples of some almost integrable billiards drawn in lattices.}
\end{figure}

\subsubsection{Main results for integrable polygonal billiards}

We shall next state our main theorem for integrable billiards. This theorem consists of a ``positive statement'' and a ``negative statement''. The positive part ensures that there are some integrable polygonal billiards with good inverse localization properties: there exists a sequence of eigenfunctions whose local limits can approximate any given monochromatic wave. Furthermore, one can approximate around most points in the billiard, and even around any point if the monochromatic wave has certain discrete symmetry. The negative part of the theorem ensures that, nevertheless, in most integrable billiards, local limits of eigenfunctions cannot approximate a generic monochromatic wave; in fact, inverse localization fails in such a dramatic way that local limits of eigenfunctions cannot be given by random waves as in Berry's conjecture. What happens is that the local limits of eigenfunctions do not exhibit enough complexity to reproduce typical monochromatic waves.

To make the latter idea precise, we need to introduce some notation.  First, we introduce the set of monochromatic waves with sharp pointwise decay at infinity
\begin{equation}\label{def.MWs}
	\MW^{\mathrm{s}}:=\left\{f\in \MW : |z|^{1/2}f(z)\in L^\infty(\RR^2)\right\},
\end{equation}
 endowed with the norm \begin{equation}\label{norm}
	\left\|f\right\|:=\sup_{z\in \mathbb{R}^2}\left(1+\left|z\right|\right)^{\frac{1}{2}} \left| f(z)\right|,
\end{equation}this is a Banach space. On~$\RR^d$, one just has to replace the exponent $\frac12$ by $\frac{d-1}2$, as in the classical theorem of Herglotz~\cite[Theorem 7.1.28]{Hor}.

\begin{definition}\label{NOIL}
	For a fixed ball $B\subset\mathbb{R}^2$ and a bounded domain $\Omega$ with boundary conditions $\mathrm{BC}$ (typically, Dirichlet or Neumann), we say that the \emph{inverse localization property fails strongly}\/ if the set \begin{equation}
		\mathcal{N}_{\Omega,\mathrm{BC}}=\left\{\varphi\in \MW^{\mathrm{s}}: \inf_{z^0\in\mathcal{P}}\inf_{\lambda\in \Lambda_\Omega^{\mathrm{BC}}}\inf_{u_\lambda\in\mathcal{V}_{\Omega,\mathrm{BC}}^\lambda}\left\|\varphi-u_\lambda\left(z^0+\frac{\cdot}{\sqrt{\lambda}}\right)\right\|_{C^0\left(B\right)}>0\right\},
	\end{equation} is dense and open in $\MW^{\mathrm{s}}$ with the topology given by \eqref{norm}. Here $\Lambda_\Omega^{\mathrm{BC}}$ is the set of all eigenvalues in $\Omega$ with $\mathrm{BC}$ boundary conditions, and $\mathcal{V}_{\Omega,\mathrm{BC}}^\lambda$ is the eigenspace associated to $\lambda$ in the same eigenvalue problem.\footnote{By elliptic regularity, we could replace $C^0(B)$ by any other H\"older norm.}
\end{definition}

Our main result for integrable polygons can then be stated as follows:

\begin{bigtheorem}\label{BT.polygons}
	Consider some ball~$B\subset\RR^2$ and let $u_n$ be an orthonormal sequence of Dirichlet or Neumann eigenfunctions on an integrable polygon  $\mathcal P\subset\RR^2$, with eigenvalues~$\la_n$. The following statements hold:
	\begin{itemize}
		\item \textbf{Inverse localization for some integrable polygons:} Suppose that $\mathcal P$ is an integrable triangle or a rational rectangle. Fix any $\ep>0$ and any monochromatic wave $\vp\in\MW$. 

		Then there exists a subsequence $u_{n_j}$ and a sequence of open sets $O_{j,\vp}\subset\Om$ satisfying~\eqref{E.siQUE} for all $j\geq1$ and all $z^0\in O_{j,\vp}$. Moreover, $|\mathcal P\backslash O_{j,\vp}|\to0$ as $j\to\infty$. 
		
		If $\vp$ is symmetric under certain finite group of reflections, depending on~$\mathcal P$ and described in Table \eqref{TableA} of  Appendix \ref{AppendixC}, one can take $O_{j,\vp}:=\Om$.\smallskip
		
		\item \textbf{Failure of inverse localization for generic integrable polygons:} Suppose that $\mathcal P$ is an irrational rectangle. Then the inverse localization property fails strongly. In particular, there does not exist a subsequence of eigenfunctions converging to random waves in the sense of~\eqref{E.siQUE}.
	\end{itemize}
	
	\end{bigtheorem}
	This theorem deserves several remarks.
	
	\begin{remark}\label{R.symmetry}
	It is indeed necessary that $\vp$ satisfies certain $\mathcal P$-depending symmetry conditions if one wants the positive localization result to hold around any point of~$\Om$.  Roughly speaking, this is because the symmetry of the domain and the boundary conditions impose some restrictions on the localized eigenfunctions and, consequently, on the monochromatic waves that we can approximate. For example, in the square $\mathcal{Q}_1$ with Dirichlet conditions, any eigenfunction localized around the central point $z^0:=\left(\frac{1}{2},\frac{1}{2}\right)$ possesses odd or even symmetry with respect to~$z^0$, thereby restricting the kind of monochromatic waves that one can approximate. See Remark \ref{RemarkSymm} for a more detailed discussion. 
	\end{remark}

\begin{remark}\label{R.highdim}
	In the case of rectangles, the result can be readily extended to higher dimensions. Specifically, fix some $B\subset\RR^d$ and consider the rectangle $\mathcal Q_l$ as above, with $l\in\RR_+^{d-1}$. The positive part of the theorem holds whenever $\mathcal Q_l$ is rational, while the negative part of the theorem holds for all $l$ in a full measure subset of~$\RR_+^{d-1}$.
\end{remark}

\begin{remark}
	An interesting application of the positive part of this result is that, thanks to our knowledge of level sets of monochromatic waves~\cite{Adv13,Yaiza}, when inverse localization holds, one can use this property to study nodal sets and critical points of eigenfunctions. More precisely, one can ensure the existence of eigenfunctions with  nodal sets of complex topologies (of interest in dimensions higher than two), complicated nodal nesting and critical points located close to a fixed point of the domain. We will make this assertion precise in Theorem \ref{nsandcp} and Theorem \ref{Yaiz}.
	
\end{remark}
Let us provide some indications about the proof of this result. In the positive part of the theorem, there are two essential ingredients:  an approximation result for solutions of the Helmholtz equation by a linear combination of translations of Bessel functions, and the special form of certain sequence of eigenfunctions  in these domains. To be more precise, we use the fact that these are the only polygonal domains having a complete set of trigonometric eigenfunctions.  For the negative part of the theorem, the idea is to study some jet space associated to localized eigenfunctions as a subspace of the one associated to all Helmholtz solutions, in a uniform way that does not depend on the eigenfunction. We manage to see that the former jet space has empty interior in the latter, and the open mapping theorem allows us to pass from this property of jets to a similar restriction for the local limits of eigenfunctions.

This set of ideas can be applied to other spectral problems. Concerning negative results, another setting in which one can try to explore whether Berry's conjecture fails generically corresponds to fixing the integrable billiard and considering general Robin boundary conditions, which depend on the parameter~$\si$. For concreteness, let us take the square $\mathcal{Q}_1$ or the equilateral triangle $\mathcal{T}_{\mathrm{equi}}$, which do possess good inverse localization properties by the first part of Theorem~\ref{BT.polygons}. Interestingly, the next result shows that this is not what happens for typical boundary conditions, as one would expect in the context of Berry's random wave conjecture:

\begin{bigtheorem}\label{T.noILrobin}
	Let us fix some ball $B\subset\RR^2$ and consider $\mathcal{P}:=\mathcal{Q}_1\subset\RR^2$ or $\mathcal{P}:=\mathcal{T}_{\mathrm{equi}}\subset\RR^2$. We impose  Robin boundary conditions with parameter~$\si$, as in~\eqref{BCsi}. For any small enough $\si>0$, the inverse localization property fails strongly.
\end{bigtheorem}

In the framework of positive results, it turns out that the ideas behind the proof can be extended, with minor modifications, to the more general class of rational almost integrable polygons (see Definition \ref{ai}). A slightly weaker version of the first part of Theorem \ref{BT.polygons} holds in this case, and this is interesting because this is enough for applications to nodal sets and critical points.

\begin{bigtheorem}\label{T.LATICE}
	Let $\mathcal P\subset\mathbb{R}^{2}$ be a rational almost integrable polygon and fix some ball $B\subset\RR^2$. Then there exists an integer~$J$, only depending on~$\mathcal P$, such that, given any monochromatic wave $\vp\in\MW$, there exist $J$ sequences of orthonormal eigenfunctions $\{u_{n_j}^k\}_{j=1}^\infty$, with $1\leq k\leq J$, and a sequence of open sets $O_{j,\vp}\subset\mathcal P$ with the following properties:
	\begin{enumerate}
		\item For each $z^0\in O_{j,\varphi}$ there exists some $1\leq k\leq J$ such that $\|u^k_{n_j}-\vp\|_{C^0(B)}<\ep$ for all $j\geq1$.
		\item $|\mathcal P\backslash O_{j,\vp}|\to0$ as $j\to\infty$.
	\end{enumerate}
	Furthermore, if~$\vp$ has the reflection symmetry presented in Table \eqref{TableB}, there exists an orthonormal sequence $u_{n_j}$ such that $\|u^k_{n_j}-\vp\|_{C^0(B)}<\ep$ for all $j\geq1$ and for any point $z^0 \in\mathcal P \backslash L_{\mathcal{P}}$ that does not lie on the lattice $L_{\mathcal P}$.
	\end{bigtheorem}

\subsection{Balls and ellipses}

Let us now consider the case of smooth integrable billiards on the plane. These are not classified, but a well known conjecture, attributed to Birkhoff \cite{Birkhoff} and published by Poritsky \cite{Poritsky}, states that if a Birkoff billiard is (locally) integrable, then the domain is elliptic. More precisely, the conjecture says as follows: 

\begin{conjecture*}[Birkhoff's conjecture]
	Consider a domain whose boundary $\Omega\subset\RR^2$  is a $C^2$-smooth convex bounded curve in the plane. If a neighborhood of the boundary $\partial \Om$  is foliated by caustics, $\pd\Omega$ is an ellipse.
\end{conjecture*}

Although the conjecture remains open, much work has been done recently, as summarized in the surveys~\cite{Bialy,Kaloshin} (see also \cite{Avila,Bialy93,Ros,Innami,Huang,KaloshinAnnals,Kaloshin24,Mather}). Of course, it is classical that elliptic billiards are integrable.

Motivated by this conjecture, we are interested in the local limits of eigenfunctions of ellipses. As spectral properties are unchanged if one applies
an isometry or a homothety, we will consider ellipses of the form \begin{equation}
	\mathcal{E}_{b}=\left\{(z_1,z_2)\in\mathbb{R}^{2},\ \left(z_1\right)^2+\left(\frac{z_2}{b}\right)^2\leq 1\right\}
\end{equation} 
with $1\geq b>0$. Note that $\mathcal{E}_{b}$ is the ellipse with foci at $(\pm c,0)$ for $c=\sqrt{1-b^2}>0$ and minor and mayor axes given by $b\leq1$ and $1$, respectively. Note that the disk is precisely $\mathcal B:=\mathcal{E}_{1}$. 

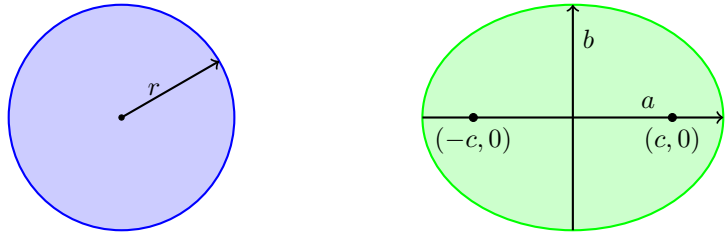
\begin{figure}\renewcommand\thefigure{1.3}
\begin{tikzpicture}
	\def\a{1.5} 
	\def\b{2} 
	\def\c{sqrt(\b*\b-\a*\a)} 
	
	\filldraw[blue!20, thick, draw=blue] (-3,0) circle (1.5cm); 
	\draw[fill=black] (-3,0) circle (1pt); 

	\draw[thick, ->] (-3,0) -- (-1.701,0.75) node[midway, left] {$r$}; 
	
	\filldraw[green!20, thick, draw=green] (3,0) ellipse (\b cm and \a cm); 
	
	\draw[fill=black] ({3+\c},0) circle (1.5pt); 
	\draw[fill=black] ({3-\c},0) circle (1.5pt); 
	\node[below] at ({3+\c},0) {$(c,0)$}; 
	\node[below] at ({3-\c},0) {$(-c,0)$}; 
	
	\draw[thick, ->] (3,-\a) -- (3,\a) node[above right, xshift=0cm, yshift=-0.7cm] {$b$}; 
	\draw[thick, ->] (3-\b,0) -- (3+\b,0) node[above, xshift=-1cm, yshift=0cm] {$a$}; 

\end{tikzpicture}
	\caption{Circular and elliptic domains}
\end{figure}

Our main result for smooth integrable billiards is that, for most values of the eccentricity (or, equivalently, of the parameter~$b$), elliptical domains do not have good inverse localization properties, so in particular, local limits of eigenfunctions are not given by random waves:

\begin{bigtheorem} \label{T.noILellipses}
	Let us fix any ball $B\subset\RR^2$ and Dirichlet or Neumann boundary conditions on the elliptic billiard $\mathcal E_b$. For all $b$ in a full measure subset of $\left]0,1\right]$, and in particular for the disk~$\mathcal E_1$, the inverse localization property fails strongly. 
	\end{bigtheorem}
	
\begin{remark}\label{R.balldimd}
	The proof shows that the inverse localization property also fails strongly on $d$-dimensional balls, which are integrable billiards in~$\RR^d$. \end{remark}	
	
\subsection{Structure of the paper}
In Chapter 2, we gather the explicit formulas for eigenvalues and eigenfunctions in all the different contexts considered and, in the case of polygonal domains, we write them in a unified formula. In Chapter 3, the positive part of Theorem \ref{BT.polygons} is proved, except the special case where $\varphi$ is symmetric.  In Chapter 4, this special case of the positive part of Theorem \ref{BT.polygons} is proved. The negative part of Theorem \ref{BT.polygons}, together with Theorems \ref{T.noILrobin} and \ref{T.noILellipses} are proven in Chapter 5. In the last one of these, it is necessary to differentiate two cases, that are considered in different sections: the circular domain and other elliptical domains. Chapter 6 contains the proof of Theorem \ref{T.LATICE} and Chapter 7 includes the applications of inverse localization to the study of nodal sets and critical points of eigenfunctions. In Appendix A, the relation between Berry's conjecture and inverse localization is explored, and moreover the local weak limit formulation of Berry's conjecture is proved to be true in a square domain both with Dirichlet and Neumann boundary conditions. As a consequence of both, a slightly stronger version of inverse localization is found to be true in this case. We collect this result in Theorem \ref{BILSquare}. All the needed definitions and lemmas are also recalled in Appendix A. To conclude, Appendix B contains tables with the required symmetry conditions for Theorems \ref{BT.polygons} and \ref{T.LATICE}.
\section{Eigenfunctions of integrable billiards}\label{formulas}

Throughout this article, we employ various properties of the eigenvalues and eigenfunctions of the Laplacian in different contexts, and it is essential to know some of their explicit expressions. In this chapter, we collect the known expressions for the eigenfunctions in each context and present them in a unified form to facilitate the proofs in later chapters. The models we study here are all integrable polygonal billiards in the plane, the $d$-dimensional rectangles, ellipses, and $d$-dimensional balls.

In what follows, for any bounded domain $\Omega$, we denote by $\Lambda_{\Omega}^{\mathrm{BC}}$ the set of all eigenvalues of  $\Omega$ with boundary conditions $\mathrm{BC} \in \{\mathrm{D},\mathrm{N},\mathrm{P},\mathrm{R}_\Sigma\}$ (representing Dirichlet, Neumann, periodic, and Robin conditions with parameter $\Sigma$, respectively). We denote by $\mathcal{V}_{\Omega,\mathrm{BC}}^\lambda$ the eigenspace associated with the eigenvalue $\lambda$, i.e., the set of all eigenfunctions on $\Omega$ corresponding to $\lambda$ with boundary condition $\mathrm{BC}$.

\subsection{Integrable polygonal billiards}
We recall that the polygons here considered are: the rectangles $\mathcal{Q}_l$, the equilateral triangle $\mathcal{T}_{\mathrm{equi}}$ (with angles $\frac{\pi}{3},\frac{\pi}{3},\frac{\pi}{3}$), the right isosceles triangle $\mathcal{T}_\mathrm{iso}$ (whose angles are $\frac{\pi}{2},\frac{\pi}{4},\frac{\pi}{4}$), and the hemiequilateral triangle $\mathcal{T}_\mathrm{hemi}$ (given by the angles $\frac{\pi}{2},\frac{\pi}{3},\frac{\pi}{6}$). In each one of them, say $\mathcal{P}$, we study eigenvalues $	\lambda_n^\mathcal{P}$ and eigenfunctions $u_\lambda^\mathcal{P}$ for both the Dirichlet problem, \begin{equation}
	\left\{	\begin{aligned}
		\Delta &u_\lambda^{\mathcal{P},\mathrm{D}}+\lambda u_\lambda^{\mathcal{P},\mathrm{D}}=0, \ \ \text{in } \mathcal{P}\\
		&u_\lambda^{\mathcal{P},\mathrm{D}}=0,\ \ \text{on } \partial \mathcal{P},
	\end{aligned}\right.
\end{equation} and the Neumann problem,
\begin{equation}
	\left\{	\begin{aligned}
		\Delta &u_\lambda^{\mathcal{P},\mathrm{N}}+\lambda u_\lambda^{\mathcal{P},\mathrm{N}}=0, \ \ \text{in } \mathcal{P}\\
		&\partial_\nu u_\lambda^{\mathcal{P},\mathrm{N}}=0,\ \ \text{on } \partial \mathcal{P},
	\end{aligned}\right. 
\end{equation}
where $\pd_\nu u$ is the derivative in the direction of the outward pointing normal to $\partial\mathcal{P}$.

In this section, we gather the explicit formulas for eigenvalues and eigenfunctions in all integrable polygons and write them in a unifying formula: 
\begin{equation}
	u_\lambda^{\mathcal{P},\mathrm{BC}}=\sum_{N\in\mathcal{N}_\lambda^{\mathcal{P},\mathrm{BC}}}c_Nu_N^{\mathcal{P},\mathrm{BC}}.
\end{equation}
Here, $\mathrm{BC}$ corresponds to Dirichlet ($\mathrm{D}$) or Neumann ($\mathrm{N}$) conditions, $\mathcal{N}_\lambda^{\mathcal{P},\mathrm{BC}}$ is a set of $d-$tuples ($d$ is the dimension of the ambient space) of integer numbers satisfying some extra conditions (that will depend on $\mathcal{P}$ and on the boundary conditions), $c_N$ are real constants and $u_N^{\mathcal{P},\mathrm{BC}}$ are eigenfunctions that belong to an orthonormal basis of $L^2(\mathcal{P})$.

In the easiest cases of the unit square $\mathcal{Q}_1=(0,1)\times (0,1)$ and the equilateral triangle $\mathcal{T}_{\mathrm{equi}}$, we will also consider the Robin eigenvalue problem: \begin{equation}\left\{\ \begin{aligned}
		&\Delta u_\lambda^{\mathcal{P},\mathrm{R}_\sigma}+\lambda u_\lambda^{\mathcal{P},\mathrm{R}_\sigma}=0 \text{ in }\mathcal{Q}_1\text{ or }\mathcal{T}_{\mathrm{equi}}\\
		&	\cos\si \, \pd_\nu u_\lambda^{\mathcal{P},\mathrm{R}_\sigma} +\sin \si \, u_\lambda^{\mathcal{P},\mathrm{R}_\sigma} =0 \quad \text{on }\partial \mathcal{Q}_1\text{ or }\partial\mathcal{T}_{\mathrm{equi}},
	\end{aligned}\right.\end{equation}
where $\sigma\in\mathbb{R}$ is a fixed constant. In all what follows, we will utilize better the parameter $\Sigma:=\frac{\sin(\sigma)}{\cos(\sigma)}\in\mathbb{R}\cup\left\{\infty\right\}$, where $\Sigma=\infty$ corresponds to the Dirichlet condition and $\Sigma=0$ corresponds to the Neumann condition.

For the ease of notation, whenever we work in dimension $d=2$, the vector $N\in\mathbb{Z}^2$ will be $N=(m,n)$.
\subsubsection{Rectangles}\label{rect}
For the case of rectangles we consider the class of orthogonal parallelepipeds in $\RR^d$ given by \begin{equation}
	\mathcal{Q}_l=(0,l_1)\times(0,l_2)\times\ldots\times(0,l_{d-1})\times(0,1),\ \ l_i>0 \ \ \forall 1\leq i\leq d-1.
\end{equation}We set $l_d=1$ for convenience. 
We introduce for $N=\left(N_1,\ldots,N_d\right)\in\mathbb{R}^d$ the quadratic form associated to $\mathcal{Q}_l$ given by \begin{equation}
	\mathbb{Q}_l(N)=\pi^2\left[\left(\frac{N_1}{l_1}\right)^2+\left(\frac{N_2}{l_2}\right)^2+\ldots+\left(\frac{N_{d-1}}{l_{d-1}}\right)^2+\left(N_d\right)^2\right].
\end{equation}
We first pay attention to the Dirichlet problem, where a basis of solutions is given by the following
\begin{equation}
	\lambda=\lambda_N^{\mathcal{Q}_l}=\mathbb{Q}_l(N)=\pi^2\left[\left(\frac{N_1}{l_1}\right)^2+\left(\frac{N_2}{l_2}\right)^2+\ldots+\left(\frac{N_{d-1}}{l_{d-1}}\right)^2+\left(N_d\right)^2\right],
\end{equation}with $N_j\in \mathbb{N},\ 1\leq j\leq d$, and
\begin{equation}
	u_{N}^{\mathcal{Q}_l,\mathrm{D}}(z)=u_{N_1}^{1,\mathrm{D}}(z_1)\cdot\ldots\cdot u_{N_d}^{d,\mathrm{D}}(z_d),\ \text{ with }u_{N_j}^{j,\mathrm{D}}(z_j)=\sin\left(\frac{\pi N_j z_j}{l_j}\right).
\end{equation}
On the other hand, for the Neumann problem the eigenvalues are given by the same 
\begin{equation}
	\lambda=\lambda_N^{\mathcal{Q}_l}=\mathbb{Q}_l\left(N\right)=\pi^2\left[\left(\frac{N_1}{l_1}\right)^2+\left(\frac{N_2}{l_2}\right)^2+\ldots+\left(\frac{N_{d-1}}{l_{d-1}}\right)^2+\left(N_d\right)^2\right],\end{equation}where now $N_j\in \mathbb{N}\cup \{0\},\ \forall1\leq j\leq d,$ and an orthonormal basis of eigenfunctions is given by the following
\begin{equation}
	u_{N}^{\mathcal{Q}_l,\mathrm{N}}(z)=u_{N_1}^{1,\mathrm{N}}(z_1)\cdot\ldots\cdot u_{N_d}^{d,\mathrm{N}}(z_d),\ \text{ with }u_{N_j}^{j,\mathrm{N}}(z_j)=\cos\left(\frac{\pi N_j z_j}{l_j}\right).
\end{equation}

Finally, for periodic boundary conditions (i.e., for any $1\leq j\leq d$ and any $z\in\mathbb{R}^d$, it is true that $u\left(z_1,\ldots,z_j,\ldots,z_d\right)=u\left(z_1,\ldots,z_j+l_j,\ldots,z_d\right)$), we will consider a basis of complex valued eigenfunctions and impose some extra conditions on the coefficients. In this context, eigenvalues are given by the same $\lambda=\lambda_N^{\mathcal{Q}_l}=\mathbb{Q}_l(N)$ where $N_j\in\mathbb{Z}$ for any $1\leq j\leq d$, and an orthonormal basis of complex valued  eigenfunctions is given by\begin{equation}
	u_{N}^{\mathcal{Q}_l,\mathrm{P}}(z)=u_{N_1}^{1,\mathrm{P}}(z_1)\cdot\ldots\cdot u_{N_d}^{d,\mathrm{P}}(z_d),\ \text{ with }u_{N_j}^{j,\mathrm{P}}(z_j)=\exp\left(\frac{i\pi N_j z_j}{l_j}\right).
\end{equation}

With these definitions, and using that these functions are basis of $L^2\left(\mathcal{Q}_l\right)$ with the respective boundary conditions, we can say
\begin{equation}
	\begin{aligned}
		&\mathcal{V}_{\mathcal{Q}_l,\mathrm{BC}}^\lambda=\left\{	\sum_{\substack{N\in\mathcal{N}^{\mathrm{BC}}\\ \mathbb{Q}_l(N)=\lambda}}c_Nu_{N}^{\mathcal{Q}_l,\mathrm{BC}},\ 	 c_N\in\mathbb{R}\right\},
	\end{aligned}
\end{equation} where we are taking $\mathcal{N}^D=\mathbb{N}^d$ and $\mathcal{N}^N=\left(\mathbb{N}\cup\{0\}\right)^d$. For the case of periodic boundary conditions, we consider something similar\begin{equation}
	\begin{aligned}
		&\mathcal{V}_{\mathcal{Q}_l,\mathrm{P}}^\lambda=\left\{	\sum_{\substack{N\in\mathcal{N}^{\mathrm{P}}\\ \mathbb{Q}_l(N)=\lambda}}c_Nu_{N}^{\mathcal{Q}_l,\mathrm{P}},\ 	 c_N\in\mathbb{C},\ c_{-N}=\overline{c_N}\right\},
	\end{aligned}
\end{equation} where now $\mathcal{N}^P=\mathbb{Z}^d$ and $\overline{c_N}$ represents the complex conjugate of $c_N$. Then standard trigonometric formulas allow us to notice that for any $\lambda\in\Lambda_{\mathcal{Q}_l}^{\mathrm{D}}$ and any $\lambda\in\Lambda_{\mathcal{Q}_l}^{\mathrm{N}}$, \begin{equation}\label{torus}
\mathcal{V}_{\mathcal{P},\mathrm{D}}^\lambda\subset\mathcal{V}_{\mathcal{P},\mathrm{P}}^\lambda \text{ and }\mathcal{V}_{\mathcal{P},\mathrm{N}}^\lambda\subset\mathcal{V}_{\mathcal{P},\mathrm{P}}^\lambda.
\end{equation}

As stated in Chapter \ref{Intro}, we distinguish the rational rectangles, this means,  those $d-$dimensional rectangles with rational square of the sides:  \begin{equation}
	\frac{p_i}{q_i}=l_i^2\in\mathbb{Q}, \ \ 1\leq i\leq d-1.
\end{equation}
In this case, let $p=\LCM(p_1,p_2,\ldots,p_{d-1})$, the least common multiple, and we define the following quadratic form: 
\begin{equation}		\mathbb{Q}^{\mathrm{int}}_l(N)=\left[q_{1}\frac{p}{p_{1}}(N_1)^2+q_{2}\frac{p}{p_2}\left(N_2\right)^2+\ldots+p\left(N_d\right)^2\right],
\end{equation}which is a positive definite quadratic form  with integer coefficients that satisfy \begin{equation}
	\lambda_{N}^{\mathcal{Q}_l}=\mathbb{Q}_l(N)=\frac{\pi^2}{p}\mathbb{Q}^{\mathrm{int}}_l(N).
\end{equation}We also define the following sets:
\begin{equation}
	\begin{aligned}
		&\mathcal{N}_\lambda^{\mathcal{Q}_l}=\left\{N\in\mathbb{Z}^d,\ \mathbb{Q}^{\mathrm{int}}_l(N)=\lambda\frac{p}{\pi^2}\right\}\text{ and }\\
		&\mathcal{N}_\lambda^{\mathcal{Q}_l,\mathrm{D}}=\left\{N\in\mathbb{N}^d,\ \mathbb{Q}^{\mathrm{int}}_l(N)=\lambda\frac{p}{\pi^2}\right\},\\
		&\mathcal{N}_\lambda^{\mathcal{Q}_l,\mathrm{N}}=\left\{N\in\left(\mathbb{N}\cup\{0\}\right)^d,\ \mathbb{Q}^{\mathrm{int}}_l(N)=\lambda\frac{p}{\pi^2}\right\}.\\
	\end{aligned}
\end{equation}
With this notation, eigenfunctions can be written as
\begin{equation}
u_{\lambda}^{\mathcal{Q}_l,\mathrm{BC}}(z)=\sum_{N\in\mathcal{N}_\lambda^{\mathcal{Q}_l,\mathrm{BC}}}c_Nu_{N}^{\mathcal{Q}_l,\mathrm{BC}}(z),
\end{equation}where $c_N$ are real constants for the Dirichlet and Neumann cases and complex constants satisfying $\overline{c_{N}}=c_{-N}$ in the periodic case; and $\mathrm{BC}$ represents again the boundary conditions.

We consider also the opposite case: whenever some $l_i^2$ cannot be written as $l_i^2= \frac{p_i}{q_i}\in\mathbb{Q}$. An efficient way of using this hypothesis is by noticing that this means that there is some integer $m\in\left[2,d\right]$ for which one can write \begin{equation}\label{irrationalsquare}
	\lambda=\mathbb{Q}_l\left(N\right)=\sum_{r=1}^m\beta_r\mathbb{Q}_r\left(N\right),
\end{equation}where the real numbers $\beta_1,\ldots,\beta_m$ are independent over the integers and the diagonal quadratic forms $\mathbb{Q}_r$ have integer coefficients, i.e. \begin{equation}
	\mathbb{Q}_r\left(N\right)=\sum_{i\in I_r}b_{i,r}N_i^2,
\end{equation}with $b_{i,r}\in\mathbb{Z}$ and $I_1,\ldots,I_m$ is a partition of $\left\{1,2,\ldots,d\right\}$. The following lemma is just a particular case of \cite[Lemma 4.2]{ILtori}.\begin{lemma}\label{independence}
	For $\mathbb{Q}_l$ as stated, any $N,M\in\mathbb{Z}^d$ satisfy $\mathbb{Q}_l(N)=\mathbb{Q}_l(M)$ if and only if $\mathbb{Q}_r(N)=\mathbb{Q}_r(M)$ for all $1\leq r\leq m$.
\end{lemma}This allows us to write any eigenvalue, in this case, as\begin{equation}
	\lambda=\sum_{r=1}^m\beta_r\lambda_r,\end{equation}with $\lambda_r\in\mathbb{N}\cup\{0\}$.

For the case of $\mathcal{Q}_1=(0,1)\times(0,1)$ with the Robin eigenvalue problem, we will denote $\Sigma:=\frac{\sin(\sigma)}{\cos(\sigma)}=\tan(\sigma)$ for  those $\sigma\neq \frac{\pi}{2}+\pi\times\mathbb{Z}$. In this work, we restrict ourselves to the case in which $\Sigma\geq0$. This boundary condition arises e.g. in the study of heat conduction (see, for example, \cite[Chapter 1]{heatconduction}).

This problem is separable (see e.g. \cite{RudWig}), with an orthonormal basis of eigenfunctions of the form \begin{equation}
	u_{mn}^{\Sigma}(z_1,z_2)=u_m(z_1)\cdot u_n(z_2).\end{equation}
Here, $u_n$ are eigenfunctions of the Laplacian on the unit interval that satisfy a one-dimensional Robin boundary condition. Precisely, the frequencies $k_n$ are the unique solutions of the equation \begin{equation}
	\tan(k_n)=\frac{2\Sigma k_n}{k_n^2-\Sigma^2},
\end{equation}in the range $n\pi<k_n<(n+1)\pi$, for $n\geq 0$. The eigenfunction corresponding to a frequency $k_n$ is given by the formula \begin{equation}
	u_n(x)=k_n\cos\left(k_nx\right)+\Sigma\sin(k_nx). 
\end{equation}
Doing these for the two different variables, we get that the Robin eigenvalues on the unit square are
\begin{equation}
	\lambda_{mn}^\Sigma=k_m^2+k_n^2,
\end{equation}with a basis of orthogonal eigenfunctions given by\begin{equation}
	\begin{aligned}
	   & u_{mn}^{\Sigma}(z)=u_m(z_1) u_n(z_2)=\\&\left(k_m\cos\left(k_mz_1\right)+\Sigma\sin(k_mz_1)\right) \left(k_n\cos\left(k_nz_2\right)+\Sigma\sin(k_nz_2)\right).
	\end{aligned}
\end{equation}
\subsubsection{The isosceles right triangle}
The isosceles right triangle is \begin{equation}
\mathcal{T}_\mathrm{iso}=\left\{(z_1,z_2):\ z_1\in(0,1), z_2\in(0,z_1)\right\}\subset \mathcal{Q}_1=(0,1)\times(0,1). 
\end{equation}  For the Dirichlet problem, the spectrum of eigenvalues is given by
\begin{equation}
	\lambda=	\lambda_{mn}^{\mathcal{T}_\mathrm{iso}}=\pi^2(m^2+n^2 )=\pi^2\mu,\text{ for } m, n \in \mathbb{N}\text{ and }m>n
\end{equation} and the corresponding eigenfunctions are
\begin{equation}
	\overline{u}_{mn}^{\mathcal{T}_\mathrm{iso},\mathrm{D}}(z_1,z_2)=\sin (\pi m z_1) \sin (\pi n z_2)-\sin (\pi n z_1) \sin (\pi m z_2),
\end{equation}
that form a complete orthogonal basis. 
Therefore, any eigenfunction $u_\lambda$ associated to the eigenvalue $\lambda$ can be written as
\begin{equation}
	\begin{aligned}
		&u_\lambda^{\mathcal{T}_\mathrm{iso},\mathrm{D}}(z_1,z_2)=\sum_{\substack{m >n>0 \\ n^2+m^2=\mu}}\tilde{c}_{mn}	\overline{u}_{mn}^{\mathcal{T}_\mathrm{iso},\mathrm{D}}(z_1, z_2)\\&=\sum_{\substack{m >n>0 \\ n^2+m^2=\mu}}\tilde{c}_{mn}(\sin (\pi m z_1) \sin (\pi n z_2)-\sin (\pi n z_1) \sin (\pi m z_2)),
	\end{aligned}
\end{equation}for real constants $\tilde{c}_{mn}$.

Let us introduce the sets \begin{equation}
	\mathcal{N}_\lambda^{\mathcal{T}_\mathrm{iso}}:=\left\{N=(m,n)\in\mathbb{Z}^2, m^2+n^2=\frac{\lambda}{\pi^2}\right\}
\end{equation} and \begin{equation}
	\mathcal{N}_\lambda^{\mathcal{T}_\mathrm{iso},\mathrm{D}}:=\left\{N=(m,n)\in\mathbb{N}^2, m^2+n^2=\frac{\lambda}{\pi^2}\right\}.
\end{equation}

This allows us to write, \begin{equation}
	u_\lambda^{\mathcal{T}_\mathrm{iso},\mathrm{D}}(z)=\sum_{(m,n)\in	\mathcal{N}_\lambda^{\mathcal{T}_\mathrm{iso},\mathrm{D}}}c_{mn}\sin (\pi m z_1) \sin (\pi n z_2)=\sum_{(m,n)\in	\mathcal{N}_\lambda^{\mathcal{T}_\mathrm{iso},\mathrm{D}}}c_{mn}u_{mn}^{\mathcal{T}_\mathrm{iso},\mathrm{D}}(z),
\end{equation}where the new coefficients are defined as \begin{equation}\label{condD}
	c_{mn}=\left\{\begin{aligned}
		&\tilde{c}_{mn}\text{ if }m>n,
		\\& 0\text{ if }m=n,
		\\& -\tilde{c}_{nm}\text{ if }n>m;
	\end{aligned}\right.
\end{equation}
and we are setting, for convenience,
\begin{equation}
	u_{mn}^{\mathcal{T}_\mathrm{iso},\mathrm{D}}(z_1,z_2)=\sin (\pi m z_1) \sin (\pi n z_2).
\end{equation}
For the Neumann problem, the spectrum is \begin{equation}
	\lambda=	\lambda_{mn}^{\mathcal{T}_\mathrm{iso}}=\pi^2(m^2+n^2 )=\pi^2\mu,\text{ for } m, n \in \mathbb{N}\cup\left\{0\right\}\text{ and }m\geq n,
\end{equation} a complete orthogonal basis of eigenfunction is given by
\begin{equation}
	\overline{u}_{mn}^{\mathcal{T}_\mathrm{iso},\mathrm{N}}(z_1, z_2)=\cos (\pi m z_1) \cos (\pi n z_2)+\cos (\pi n z_1) \cos (\pi m z_2),
\end{equation}and we can also define the set
\begin{equation}
	\mathcal{N}_\lambda^{\mathcal{T}_\mathrm{iso},\mathrm{N}}:=\left\{N=(m,n)\in\left(\mathbb{N}\cup\{0\}\right)^2, m^2+n^2=\frac{\lambda}{\pi^2}\right\}.
\end{equation}

Again, any eigenfunction can be written as 
\begin{equation}
	\begin{aligned}
		&u_\lambda^{\mathcal{T}_\mathrm{iso},\mathrm{N}}(z_1,z_2)=\sum_{\substack{m \geq n\geq 0 \\ n^2+m^2=\mu}}\tilde{c}_{mn}	\overline{u}_{mn}^{\mathcal{T}_\mathrm{iso},\mathrm{N}}(z_1, z_2)\\&=\sum_{\substack{m \geq n\geq 0 \\ n^2+m^2=\mu}}\tilde{c}_{mn}(\cos (\pi m z_1) \cos (\pi n z_2)+\cos (\pi n z_1) \cos (\pi m z_2))
		\\&	=\sum_{(m,n)\in\mathcal{N}_\lambda^{\mathcal{T}_\mathrm{iso},\mathrm{N}}}c_{mn}\cos (\pi m z_1) \cos (\pi n z_2)=\sum_{(m,n)\in\mathcal{N}_\lambda^{\mathcal{T}_\mathrm{iso},\mathrm{N}}}c_{mn}u_{mn}^{\mathcal{T}_\mathrm{iso},\mathrm{N}}(z),
	\end{aligned}
\end{equation}for real $\tilde{c}_{mn}$ with the relation \begin{equation}\label{condN}
	c_{mn}=\left\{\begin{aligned}
		&\tilde{c}_{mn}\text{ if }m>n,
		\\& \frac{\tilde{c}_{mm}}{2}\text{ if }m=n,
		\\& \tilde{c}_{nm}\text{ if }n>m;
	\end{aligned}\right.
\end{equation}and for 
\begin{equation}
	u_{mn}^{\mathcal{T}_\mathrm{iso},\mathrm{N}}(z_1,z_2)=\cos (\pi m z_1) \cos (\pi n z_2).
\end{equation}

\subsubsection{The equilateral triangle}\label{equi}
The equilateral triangle of side length $1$ in standard position is the set of points \begin{equation}
	\mathcal{T}_\mathrm{equi}=	\left\{(z_1,z_2):z_1\in (0,1), z_2\in\left(0,\min\left\{\sqrt{3}z_1,\sqrt{3}(1-z_1)\right\}\right)\right\}.
\end{equation}We can define the triangular coordinates $(u,v,w)$ of a point $(z_1,z_2)$ by
\begin{equation}
	\begin{aligned}	&u=\frac{1}{2\sqrt{3}}-z_2,
		\\&v=\frac{\sqrt{3}}{2}\left(z_1-\frac{1}{2}\right)+\frac{1}{2}\left(z_2-\frac{1}{2\sqrt{3}}\right),
		\\&w=\frac{\sqrt{3}}{2}\left(\frac{1}{2}-z_1\right)+\frac{1}{2}\left(z_2-\frac{1}{2\sqrt{3}}\right).
	\end{aligned}
\end{equation}
These coordinates may be described as the distances of the triangle center to the projections of the point onto the altitudes, measured positively toward a side and negatively toward a vertex.

Following \cite{trigI}, we can introduce new orthogonal coordinates $(\xi,\eta)$ in which we are able to make separation of variables: \begin{equation}
	\xi=u=\frac{1}{2\sqrt{3}}-z_2,\ \eta=v-w=\sqrt{3}\left(z_1-\frac{1}{2}\right).
\end{equation}

In this new variables one can construct two types of solutions to the Dirichlet problem: those which are symmetric with respect to the line $v=w$ and those which are anti symmetric. Let us call $r=\frac{1}{2\sqrt{3}}$ the in-radius of the triangle. A basis of eigenfunctions for the Dirichlet problem is then given by 
\begin{equation}
	\begin{aligned}
		&	T_s^{\alpha,\beta}(\xi,\eta)=\sin\left(\frac{\pi \gamma}{3r}\left(\xi+2r\right)\right)\cos\left(\frac{\pi(\alpha-\beta)}{9r}\eta\right)+\\&\sin\left(\frac{\pi \alpha}{3r}\left(\xi+2r\right)\right)\cos\left(\frac{\pi(\beta-\gamma)}{9r}\eta\right)+\sin\left(\frac{\pi \beta}{3r}\left(\xi+2r\right)\right)\cos\left(\frac{\pi(\gamma-\alpha)}{9r}\eta\right),
		\\&	T_a^{\alpha,\beta}(\xi,\eta)=\sin\left(\frac{\pi \gamma}{3r}\left(\xi+2r\right)\right)\sin\left(\frac{\pi(\alpha-\beta)}{9r}\eta\right)+\\&\sin\left(\frac{\pi \alpha}{3r}\left(\xi+2r\right)\right)\sin\left(\frac{\pi(\beta-\gamma)}{9r}\eta\right)+\sin\left(\frac{\pi \beta}{3r}\left(\xi+2r\right)\right)\sin\left(\frac{\pi(\gamma-\alpha)}{9r}\eta\right),
	\end{aligned}
\end{equation}with $\gamma,\alpha,\beta\in\mathbb{Z}$ satisfying $\gamma+\alpha+\beta=0$ and $\beta\geq \alpha>0$ in the symmetric case and $\beta>\alpha>0$ in the antisymmetric case. The associated eigenvalue is\begin{equation}
	\lambda=\frac{4\pi^2}{27r^2}(\alpha^2+\beta^2+\alpha\beta)=\frac{4\pi^2}{27r^2}\mu=\frac{16\pi^2}{9}\mu.
\end{equation}Any eigenfunction associated to such an eigenvalue can be written as the following linear combination\begin{equation}
	u_\lambda^{\mathcal{T}_\mathrm{equi},\mathrm{D}}(\xi,\eta)=\sum_{\substack{\beta \geq \alpha>0 \\ \beta^2+\alpha^2+\alpha\beta=\mu}}\overline{c}_s^{\alpha,\beta}T_s^{\alpha,\beta}(\xi,\eta)+\overline{a}_a^{\alpha,\beta}T_a^{\alpha,\beta}(\xi,\eta),\ \ \overline{c}_s^{\alpha,\beta},\overline{c}_a^{\alpha,\beta}\in\mathbb{R}.
\end{equation}

We introduce now the new variables $\overline{z}_1=z_1-\frac{1}{2}$ and $\overline{z}_2=\frac{\sqrt{3}}{2}-z_2$ (which essentially means that we rotate and translate the triangle until the new position showed in picture \ref{trigprueba}, so the spectral properties remain unchanged) and rewrite the functions as
		
		\begin{figure}[!ht]\renewcommand\thefigure{2.1}
			\centering
			\begin{tikzpicture}
				\coordinate (A) at (0,3);
				\coordinate (B) at (-2.6,-1.5);
				\coordinate (C) at (2.6,-1.5);
				
				\coordinate (O) at (0,0);
				
				\coordinate (P) at (-0.5,1.2);
				
				\coordinate (M1) at ($ (A)!0.5!(B) $);
				\coordinate (M2) at ($ (B)!0.5!(C) $);
				\coordinate (M3) at ($ (A)!0.5!(C) $);
				
				\coordinate (a) at (0,1.2);
				\coordinate (b) at (-0.88,0.51);
				\coordinate (c) at (0.143,0.085);
				\draw (a) circle (1.3pt);
				\draw (b) circle (1.3pt);
				\draw (c) circle (1.3pt);
				
				\draw[thick] (A) -- (B) -- (C) -- cycle;
				
				\draw[dashed] (O) -- (M1);
				\draw[dashed] (O) -- (M2);
				\draw[dashed] (O) -- (M3);
				
				\draw[dashed] (O) -- (A);
				\draw[dashed] (O) -- (B);
				\draw[dashed] (O) -- (C);
				
				\draw[dashed] (A) -- (B);
				\draw[dashed] (B) -- (C);
				\draw[dashed] (C) -- (A);
				
				 \draw [line width=0.1mm] (P) -- (a);
				 \draw [line width=0.1mm] (P) -- (b);
				 \draw [line width=0.1mm] (P) -- (c);
				
				\node[above] at (A) {(-2r, r, r)};
				\node[left] at (B) {(r, -2r, r)};
				\node[right] at (C) {(r, r, -2r)};

				\filldraw (P) circle (2pt);
				\node[above right] at (P) {$P$};
				
				\node[right,yshift=-30pt,xshift=-50pt] at ($(O)!0.5!(C)$) {$u = r$};
				\node[left,xshift=10pt, yshift=70pt] at ($(O)!0.5!(B)$) {$w = r$};
				\node[above,xshift=40pt] at ($(O)!0.5!(A)$) {$v = r$};

			\end{tikzpicture}
			\caption{Triangular Coordinate System}
		\end{figure}
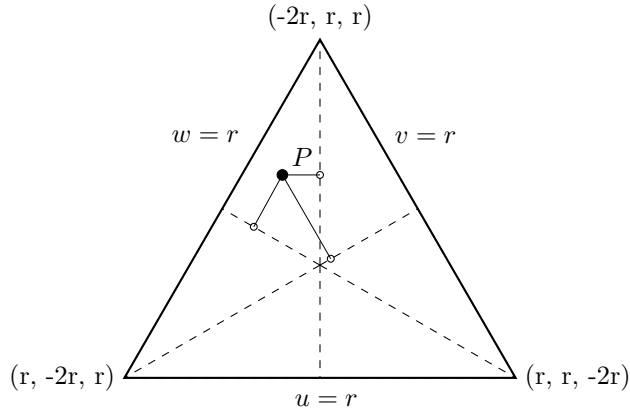
	
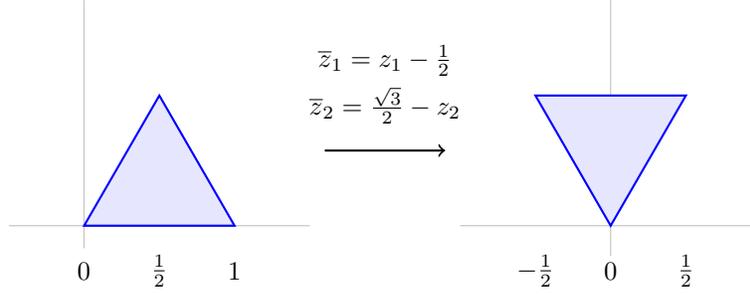
\begin{figure}	\renewcommand\thefigure{2.2}\begin{tikzpicture}	[scale=2]
		\draw[gray!50] (-0.5,0) -- (1.5,0);
		\draw[gray!50] (0,-0.15) -- (0,1.5);
		\node [scale=1] at (0, -0.3) {$0$};
		\node [scale=1] at (0.5, -0.3) {$\frac{1}{2}$};
		\node [scale=1] at (1, -0.3) {$1$};
		
		\fill[blue!10] (0,0) -- (0.5,0.866) -- (1,0) -- cycle;
		\draw[blue, thick] (0,0) -- (0.5,0.866) -- (1,0) -- cycle;

  \node [scale=1] at (2,1.1) {$\overline{z}_1=z_1-\frac{1}{2}$};
  \node [scale=1] at (2,0.8) {$\overline{z}_2=\frac{\sqrt{3}}{2}-z_2$};
		\draw[thick, ->] (1.6,0.5) -- (2.4,0.5);
		\draw[gray!50] (2.5,0) -- (4.5,0);
		\draw[gray!50] (3.5,-0.2) -- (3.5,1.5);
		\node [scale=1] at (3, -0.3) {$-\frac{1}{2}$};
		\node [scale=1] at (3.5, -0.3) {$0$};
		\node [scale=1] at (4, -0.3) {$\frac{1}{2}$};
		
		\fill[blue!10] (3,0.866) -- (3.5,0) -- (4,0.866) -- cycle;
		\draw[blue, thick] (3,0.866) -- (3.5,0) -- (4,0.866) -- cycle;

	\end{tikzpicture}
	\caption{Change of coordinates in the equilateral triangle given by $\overline{z}_1=z_1-\frac{1}{2}$ and $\overline{z}_2=\frac{\sqrt{3}}{2}-z_2$ }\label{trigprueba}	
\end{figure}\begin{equation}
	\begin{aligned}
		T_s^{\alpha,\beta}(\overline{z}_1,\overline{z}_2)=&-\sin\left(\frac{2\sqrt{3}\pi (\alpha+\beta)}{3}\overline{z}_2\right)\cos\left(\frac{2\pi(\alpha-\beta)}{3}\overline{z}_1\right)\\&+\sin\left(\frac{2\sqrt{3}\pi \alpha}{3}\overline{z}_2\right)\cos\left(\frac{2\pi(2\beta+\alpha)}{3}\overline{z}_1\right)\\&+\sin\left(\frac{2\sqrt{3}\pi \beta}{3}\overline{z}_2\right)\cos\left(\frac{2\pi(2\alpha+\beta)}{3}\overline{z}_1\right),
		\\T_a^{\alpha,\beta}(\overline{z}_1,\overline{z_2})=&-\sin\left(\frac{2\sqrt{3}\pi (\alpha+\beta)}{3}\overline{z}_2\right)\sin\left(\frac{2\pi(\alpha-\beta)}{3}\overline{z}_1\right)\\&+\sin\left(\frac{2\sqrt{3}\pi \alpha}{3}\overline{z}_2\right)\sin\left(\frac{2\pi(2\beta+\alpha)}{3}\overline{z}_1\right)\\&-\sin\left(\frac{2\sqrt{3}\pi \beta}{3}\overline{z}_2\right)\sin\left(\frac{2\pi(2\alpha+\beta)}{3}\overline{z}_1\right).
	\end{aligned}
\end{equation}

In each term we apply one of the following changes of variables:\begin{equation}
	\begin{aligned}
		&n=\beta+\alpha,\ m=\beta-\alpha,\ n>m\geq 0,\text{ for the first one,}
		\\&n=\alpha,\ m=2\beta+\alpha,\ m\geq 3n>0,\text{ for the second one and}
		\\&n=\beta,\ m=2\alpha+\beta,\ 3n\geq m>n>0\text{ for the last one,}
		\\&\text{ all satisfying }3n^2+m^2=4\mu.
	\end{aligned}
\end{equation}

We define also the sets \begin{equation}
	\mathcal{N}_\lambda^{\mathcal{T}_\mathrm{equi}}=\left\{N=(m,n)\in\mathbb{Z}^2, 3n^2+m^2=4\mu=\frac{9\lambda}{\pi^24}\right\}
\end{equation} and 
\begin{equation}
	\mathcal{N}_\lambda^{\mathcal{T}_\mathrm{equi},\mathrm{D}}=\left\{N=(m,n)\in\left(\mathbb{N}\cup\{0\}\right)\times\mathbb{N}, 3n^2+m^2=4\mu=\frac{9\lambda}{\pi^24}\right\}
\end{equation}

If we impose the condition $\mu\neq 3n^2$, to avoid the case $m=3n$, and the condition $\mu\neq n^2$ because it is not possible with the coefficients so defined, a generic eigenfunction can then be written as\begin{equation}
	\begin{aligned}
	    &u_\lambda^{\mathcal{T}_\mathrm{equi},\mathrm{D}}\left(\overline{z}_1,\overline{z}_2\right)=\\&\sum_{(m,n)\in	\mathcal{N}_\lambda^{\mathcal{T}_\mathrm{equi},\mathrm{D}}}\sin\left(\frac{2\pi\sqrt{3}n}{3}\overline{z}_2\right)\left[c_s^{m,n}\cos\left(\frac{2\pi m}{3}\overline{z}_1\right)+c_a^{m,n}\sin\left(\frac{2\pi m}{3}\overline{z}_1\right)\right],
	\end{aligned}
\end{equation}where
\begin{equation}
	c_{s,\mathrm{D}}^{m,n}=\left\{	\begin{aligned}
		&-\overline{c}_s^{\frac{n-m}{2},\frac{n+m}{2}}\ \text{if}\ m<n,\\&\overline{c}_s^{n,\frac{m-n}{2}}\ \text{if}\ n<m<3n\\&\overline{c}_s^{\frac{m-n}{2},\mathrm{N}}\ \text{if}\ m>3n.
	\end{aligned}\right.,\ \
	\tilde{c}_{a,\mathrm{D}}^{m,n}=\left\{	\begin{aligned}
			&\overline{c}_a^{\frac{n-m}{2},\frac{n+m}{2}}\ \text{if}\ m<n,\\&\overline{c}_a^{n,\frac{m-n}{2}}\ \text{if}\ n<m<3n\\&-\overline{c}_a^{\frac{m-n}{2},\mathrm{N}}\ \text{if}\ m>3n.
	\end{aligned}\right.
\end{equation}

A similar approach can be followed for the Neumann problem (see \cite{trigII}) to conclude that any eigenfunction can be written as \begin{equation}
	\begin{aligned}
	    &u_\lambda^{\mathcal{T}_\mathrm{equi},\mathrm{N}}\left(\overline{z}_1,\overline{z}_2\right)=\\&\sum_{(m,n)\in \mathcal{N}_\lambda^{\mathcal{T}_\mathrm{equi},\mathrm{N}}}\cos\left(\frac{2\pi\sqrt{3}n}{3}\overline{z}_2\right)\left[c_s^{m,n}\cos\left(\frac{2\pi m}{3}\overline{z}_1\right)+c_a^{m,n}\sin\left(\frac{2\pi m}{3}\overline{z}_1\right)\right],
	\end{aligned}
\end{equation}where
\begin{equation}
	c_{s,\mathrm{N}}^{m,n}=\left\{	\begin{aligned}
		&\overline{c}_s^{\frac{n-m}{2},\frac{n+m}{2}}\ \text{if}\ m<n,\\&\overline{c}_s^{n,\frac{m-n}{2}}\ \text{if}\ n<m<3n\\&\overline{c}_s^{\frac{m-n}{2},\mathrm{N}}\ \text{if}\ m>3n.
	\end{aligned}\right.,\ \
	\tilde{c}_{a,\mathrm{N}}^{m,n}=\left\{	\begin{aligned}
		&\overline{c}_a^{\frac{n-m}{2},\frac{n+m}{2}}\ \text{if}\ m<n,\\&-\overline{c}_a^{n,\frac{m-n}{2}}\ \text{if}\ n<m<3n\\&\overline{c}_a^{\frac{m-n}{2},\mathrm{N}}\ \text{if}\ m>3n,
	\end{aligned}\right.
\end{equation}
and 
\begin{equation}
	\mathcal{N}_\lambda^{\mathcal{T}_\mathrm{equi},\mathrm{N}}=\left\{N=(m,n)\in\left(\mathbb{N}\cup\{0\}\right)^2, 3n^2+m^2=4\mu=\frac{9\lambda}{\pi^24}\right\}.
\end{equation}
For the ease of notation, let us denote \begin{equation}
	\begin{aligned}
		&u^{\mathcal{T}_\mathrm{equi},\mathrm{D}}_{mn,S}(\overline{z}_1,\overline{z}_2)=\sin\left(\frac{2\pi\sqrt{3}n}{3}\overline{z}_2\right)\cos\left(\frac{2\pi m}{3}\overline{z}_1\right),\\
		&u^{\mathcal{T}_\mathrm{equi},\mathrm{D}}_{mn,A}(\overline{z}_1,\overline{z}_2)=\sin\left(\frac{2\pi\sqrt{3}n}{3}\overline{z}_2\right)\sin\left(\frac{2\pi m}{3}\overline{z}_1\right),\\
		&u^{\mathcal{T}_\mathrm{equi},\mathrm{N}}_{mn,S}(\overline{z}_1,\overline{z}_2)=\cos\left(\frac{2\pi\sqrt{3}n}{3}\overline{z}_2\right)\cos\left(\frac{2\pi m}{3}\overline{z}_1\right),\\
		&u^{\mathcal{T}_\mathrm{equi},\mathrm{N}}_{mn,A}(\overline{z}_1,\overline{z}_2)=\cos\left(\frac{2\pi\sqrt{3}n}{3}\overline{z}_2\right)\sin\left(\frac{2\pi m}{3}\overline{z}_1\right),\\
	\end{aligned}
\end{equation}So that we have, 
\begin{equation}
	u_\lambda^{\mathcal{T}_\mathrm{equi},\mathrm{BC}}\left(\overline{z}_1,\overline{z}_2\right)=\sum_{(m,n)\in	\mathcal{N}_\lambda^{\mathcal{T}_\mathrm{equi},\mathrm{BC}}}c_{s,\mathrm{BC}}^{m,n}u^{\mathcal{T}_\mathrm{equi},\mathrm{BC}}_{mn,S}(\overline{z}_1,\overline{z}_2)+c_{a,\mathrm{BC}}^{m,n}u^{\mathcal{T}_\mathrm{equi},\mathrm{BC}}_{mn,A}(\overline{z}_1,\overline{z}_2),
\end{equation}
where $\mathrm{BC}$ denotes either $\mathrm{D}$ or $\mathrm{N}$ and $c_{a,\mathrm{BC}}^{m,n}$ and $c_{s,\mathrm{BC}}^{m,n}$ are real constants.

For this polygon, we also consider the Robin eigenvalue problem under the same hypothesis as before: $\Sigma=\frac{\sin(\sigma)}{\cos(\sigma)}=\tan(\sigma)>0$ and real. We follow an approach similar to the one of previous sections, as in \cite{trigIII}. An orthonormal basis of eigenfunctions for this problem is given by a symmetric part 

\begin{equation}
	\begin{aligned}
		T_s^{m,n}(z) & =\cos \left(\frac{2\sqrt{3}\pi \zeta}{3}z_2-\delta_1\right) \cdot \cos \left(\frac{2\pi(\mu-\nu)}{3}z_1\right) \\
		& +\cos \left(\frac{2\pi\sqrt{3} \mu}{3 }z_2-\delta_2\right) \cdot \cos \left(\frac{2\pi(\nu-\zeta)}{3}z_1\right) \\
		& +\cos \left(\frac{2\pi\sqrt{3} \nu}{3}z_2-\delta_3\right) \cdot \cos \left(\frac{2\pi(\zeta-\mu)}{3}z_1\right),
	\end{aligned}
\end{equation} and an antisymmetric part \begin{equation}
	\begin{aligned}
		T_a^{m,n}(z) & =\cos \left(\frac{2\sqrt{3}\pi \zeta}{3}z_2-\delta_1\right) \cdot \sin \left(\frac{2\pi(\mu-\nu)}{3}z_1\right) \\
		& +\cos \left(\frac{2\pi\sqrt{3} \mu}{3 }z_2-\delta_2\right) \cdot \sin \left(\frac{2\pi(\nu-\zeta)}{3}z_1\right) \\
		& +\cos \left(\frac{2\pi\sqrt{3} \nu}{3}z_2-\delta_3\right) \cdot \sin \left(\frac{2\pi(\zeta-\mu)}{3}z_1\right),
	\end{aligned}
\end{equation}
with eigenvalue
\begin{equation}
	\lambda=\frac{2}{27}\left(\frac{\pi}{r}\right)^2\left[\zeta^2+\mu^2+\nu^2\right]=\frac{4}{27}\left(\frac{\pi}{r}\right)^2\left[\mu^2+\mu \nu+\nu^2\right]=\frac{16\pi^2}{9}\left[\mu^2+\mu \nu+\nu^2\right] .
\end{equation}
These functions depend on $m$ and $n$ through the parameters $\zeta,\mu,\nu,\delta_1,\delta_2$ and $\delta_3$; which are determined as
\begin{equation}
	\delta_1=L-M-N ;\ \delta_2=-L+M-N ;\ \delta_3=-L-M+N,\text{ and }
\end{equation}
\begin{equation}
	\zeta=\frac{2 L-M-N}{\pi}-m-n ;\ \mu=\frac{2 M-N-L}{\pi}+m ;\ \nu=\frac{2 N-L-M}{\pi}+n,
\end{equation}
taking into account that $L,M,N$ are auxiliary variables solving\begin{equation}
	\begin{aligned}
		{[2 L-M-N-(m+n) \pi] \cdot \tan L } & =\frac{6}{\sqrt{3}} \Sigma  \\
		{[2 M-N-L+m \pi] \cdot \tan M } & =\frac{6}{\sqrt{3}} \Sigma  \\
		{[2 N-L-M+n \pi] \cdot \tan N } & =\frac{6}{\sqrt{3}} \Sigma 
	\end{aligned}
\end{equation}
for $m=0,1,2,\ldots$ and $n=m,m+1,\ldots$ and for $-\frac{\pi}{2}<L\leq 0$, $0\leq M<\frac{\pi}{2}$ and $0\leq N<\frac{\pi}{2}$.

 For any pair $(m,n)$ we have associated two linearly independent eigenfunctions, $T_a^{m,n}$ and $T_s^{m,n}$. We call this multiplicity the ``systematic doubling''. For any $\lambda$ we can have one or more couples $(m,n)$ such that $\lambda_{mn}=\lambda$, so we may have further multiplicities. 

\subsubsection{The hemiequilateral triangle}
The hemiequilateral
triangle $\mathcal{T}_\mathrm{hemi}$ is the right triangle obtained by subdividing the equilateral triangle along
an altitude (as shown in Figure \ref{polygons}). It may be characterized as a right triangle whose hypotenuse has precisely twice the length of one of its legs. Following Lamé \cite{Lamé} it is immediate that all of its eigenfunctions for the Dirichlet problem (resp. Neumann problem) are obtained by simply restricting the antisymmetric (resp. symmetric) modes of the equilateral triangle to this domain, thereby having a nodal (resp. antinodal) line along said altitude.

Therefore, the eigenvalues are the same as in the equilateral triangle and all the eigenfunctions associated to such an eigenvalue $\lambda=\frac{16\pi^2}{9}\mu$ are given by
\begin{equation}
	\begin{aligned}
		&	u_\lambda^{\mathcal{T}_\mathrm{hemi},\mathrm{D}}\left(\overline{z}_1,\overline{z}_2\right)=\sum_{(m,n)\in	\mathcal{N}_\lambda^{\mathcal{T}_\mathrm{hemi},\mathrm{D}}}c_{mn}\sin\left(\frac{2\pi\sqrt{3}n}{3}\overline{z}_2\right)\sin\left(\frac{2\pi m}{3}\overline{z}_1\right)=\\&=\sum_{(m,n)\in	\mathcal{N}_\lambda^{\mathcal{T}_\mathrm{hemi},\mathrm{D}}}c_{mn}u_{mn}^{\mathcal{T}_\mathrm{hemi},\mathrm{D}}\left(\overline{z}_1,\overline{z}_2\right),
	\end{aligned}
\end{equation}for the Dirichlet problem, and by 
\begin{equation}
	\begin{aligned}
		&	u_\lambda^{\mathcal{T}_\mathrm{hemi},\mathrm{N}}\left(\overline{z}_1,\overline{z}_2\right)=\sum_{(m,n)\in	\mathcal{N}_\lambda^{\mathcal{T}_\mathrm{hemi},\mathrm{N}}}c_{mn}\cos\left(\frac{2\pi\sqrt{3}n}{3}\overline{z}_2\right)\cos\left(\frac{2\pi m}{3}\overline{z}_1\right)\\&=\sum_{(m,n)\in	\mathcal{N}_\lambda^{\mathcal{T}_\mathrm{hemi},\mathrm{D}}}c_{mn}u_{mn}^{\mathcal{T}_\mathrm{hemi},\mathrm{N}}\left(\overline{z}_1,\overline{z}_2\right),
	\end{aligned}
\end{equation} for the Neumann problem.
Here we are using the following notation:\begin{equation}
	\begin{aligned}
		&	u_{mn}^{\mathcal{T}_\mathrm{hemi},\mathrm{D}}\left(\overline{z}_1,\overline{z}_2\right)=	\sin\left(\frac{2\pi\sqrt{3}n}{3}\overline{z}_2\right)\sin\left(\frac{2\pi m}{3}\overline{z}_1\right),\\&	u_{mn}^{\mathcal{T}_\mathrm{hemi},\mathrm{N}}\left(\overline{z}_1,\overline{z}_2\right)=	\cos\left(\frac{2\pi\sqrt{3}n}{3}\overline{z}_2\right)\cos\left(\frac{2\pi m}{3}\overline{z}_1\right),\\&\mathcal{N}_\lambda^{\mathcal{T}_\mathrm{hemi},\mathrm{D}}=\left\{N=(m,n)\in\mathbb{N}^2, 3n^2+m^2=4\mu=\frac{9\lambda}{\pi^24}\right\},\\&\mathcal{N}_\lambda^{\mathcal{T}_\mathrm{hemi},\mathrm{N}}=\mathcal{N}_\lambda^{\mathcal{T}_{\mathrm{equi}},\mathrm{N}}\text{ and }\mathcal{N}_\lambda^{\mathcal{T}_\mathrm{hemi}}=\mathcal{N}_\lambda^{\mathcal{T}_{\mathrm{equi}}}.
	\end{aligned}
\end{equation}$c_{mn}$ are real constants again, and $\overline{z}_1,\overline{z}_2$ are defined as in the previous section. 
\begin{remark}
	Since the spectral properties are not modified by applying an isometry or an homothety, we can use the new variable $\overline{z}$ or the original one $z$ indistinguishably during the proof of all the theorems in this paper.  
\end{remark}
\subsection{Balls and ellipses}
We consider now the eigenvalue problem with Dirichlet or Neumann conditions inside ellipses of $\RR^2$. When studying the particular case of balls, we can consider the more general case of higher dimensions $d\geq2$ since the extra symmetry of these domains allows to make this generalization without much difficulty. 

\subsubsection{Balls in any dimension}\label{ball}
Without losing generality we can assume that the ball $B$ is centered at the origin $z=0$ and of radius $1$. In any other ball, the eigenvalues and eigenfunctions differ only by a scaling factor. Moreover, since it does not entail a very large increase in difficulty, we study balls not only in $\RR^2$ but in arbitrary dimension $d\geq2$.

 To study all the possible eigenfunctions, we first consider separated variables solutions to the eigenvalue problem, with both Dirichlet or Neumann boundary conditions. Here $\rho=|z|$ and $\theta=z/|z|\in\mathbb{S}^{d-1}$, are spherical coordinates in $\mathbb{R}^d$. The solutions of the equation \begin{equation*}
	\partial_{\rho\rho}u+\frac{d-1}{\rho}\partial_{\rho}u+\frac{1}{\rho^2}\Delta_{\mathbb{S}^{d-1}}u+\lambda u=0,
\end{equation*}have the form  $u(z)=R(\rho)\Theta(\theta)$,
 with spherical harmonics for the angular part $\Theta=Y_{lm}$ (associated to the eigenvalues of the form $l(l+d-2)$ for any $l\in\mathbb{N}\cup\{0\}$ with multiplicity $M(d,l)= \binom{d+l-1}{d-1}-\binom{d+l-3}{d- 1}$) and the radial part satisfying the ordinary differential equation\begin{equation*}
	R''+\frac{d-1}{\rho}R'-\frac{l(l+d-2)}{\rho^2}R+\lambda R=0. 
\end{equation*}
This is a Bessel type ordinary differential equation with solutions of the form \begin{equation*}
	R(\rho)=c_1\rho^{1-d/2}J_{d/2+l-1}\left(\sqrt{\lambda}\rho\right)+c_2\rho^{1-d/2}Y_{d/2+l-1}\left(\sqrt{\lambda}\rho\right),
\end{equation*}
where $J_\alpha$ denotes the Bessel functions of the first kind and $Y_\alpha$ denotes the Bessel functions of the second kind. 
For integer or positive $\alpha$, Bessel functions of the first kind are finite at the origin ($\rho=0$); while for negative non-integer $\alpha$, Bessel functions of the first kind diverge as $\rho$ approaches zero. On the other hand, Bessel functions of the second kind have a singularity at the origin $\rho=0$. As we need the solution to be smooth on the origin we can discard the second kind Bessel functions and we get some conditions on $l$. We next impose the boundary conditions. 

For Dirichlet eigenfunctions, we ask $\sqrt{\lambda}$ to be a zero of the function $J_{d/2+l-1}$, so that we have $R(1)=c_11^{1-d/2}J_{d/2+l-1}\left(\sqrt{\lambda}\right)=0$. Since $d/2+l-1$ is real, $J_{d/2+l-1}$ has an infinite number of positive real zeros which are simple (see, for example, \cite[p.370]{AS}). We denote the $n$th positive zero of this function as $\alpha_{ln}$. 

On the other hand, for the Neumann eigenfunctions, we ask  $\sqrt{\lambda}$ to be a zero of the function $S_l(\rho):=\frac{d}{d\rho}\left(\rho^{1-d/2}J_{d/2+l-1}\left(\rho\right)\right)$, so that we have $R'(1)=0$. Again, as $d/2+l-1$ is real, $S_l$ has an infinite number of positive (non-negative if $l=0$) real zeros which are simple (see now \cite [Proposition 2.3.]{Helf}). We denote the $n$th positive zero of this function as $\beta_{ln}$. 

Then, all the separated variables eigenfunctions have the form \begin{equation*}
	u_{lnm}^\mathrm{D}(\rho,\theta)=c_1\rho^{1-d/2}J_{d/2+l-1}\left(\alpha_{ln}\rho \right)Y_{lm}\left(\theta\right),
\end{equation*}with $l\in\mathbb{N}\cup\{0\}$, $n\in\mathbb{N}$, $1\leq m\leq M(d,l)$ and the  associated eigenvalue being $\lambda_{ln}=\alpha_{ln}^2$ for Dirichlet boundary conditions; and, on the other hand, \begin{equation*}
	u_{lnm}^\mathrm{N}(\rho,\theta)=c_1\rho^{1-d/2}J_{d/2+l-1}\left(\beta_{ln}\rho \right)Y_{lm}\left(\theta\right),
\end{equation*}with $l\in\mathbb{N}\cup\{0\}$, $n\in\mathbb{N}$, $1\leq m\leq M(d,l)$, with  associated eigenvalue $\lambda_{ln}=\beta_{ln}^2$ for Neumann boundary conditions.

It could be the case that different values of $l\neq l'\in\NN\cup\{0\}$ and $n\neq n'\in\NN$ produce the same eigenvalue, $\alpha_{ln}^2=\alpha_{l'n'}^2$ or  $\beta_{ln}^2=\beta_{l'n'}^2$. However, the two following lemmas show that this does not occur and that the only linearly independent eigenfuntions that we have associated to the same eigenvalue differ on the value of $m$. We see first the Dirichlet case. This result was conjectured by Bourget and it follows (see for example \cite[p.464]{ISK}) from a number theory result obtained by Siegel \cite{Siegel}:

\begin{lemma} 
	Let us take $J_\alpha(\rho)$ and $J_{\alpha+m}(\rho)$ for $\alpha$ rational and $m$ natural numbers. If there exists a $\rho_0$ such that $J_\alpha(\rho_0)=0$ and $J_{\alpha+m}(\rho_0)=0$, then $\rho_0=0$.
\end{lemma} 

Applying this result we see that for $l_0\neq l_1$, $J_{d/2+l_0-1}$ and $J_{d/2+l_1-1}$ have no common zeros other than $\rho=0$. So for $\lambda_{ln}\neq 0$, the multiplicity of the eigenvalue $\lambda_{ln}$ is exactly $M(d,l)$, and all the eigenfunctions associated to an eigenvalue $\lambda_{ln}$ are \begin{equation*}
	u_{ln}(\rho,\theta)=\sum_{m=1}^{M(d,l)}c_m\rho^{1-d/2}J_{d/2+l-1}\left(\rho\sqrt{\lambda_{ln}}\right)Y_{lm}\left(\theta\right).
\end{equation*}

The analogous result for the Neumann problem is way more recent. It can be found in \cite[Lemma 2.4.]{Helf} and it uses results of number theory, given in \cite{Shi}.
\begin{lemma}
	For any values of $l\in\mathbb{N}\cup\{0\}$ and $m\in\mathbb{N}$, then the functions $\frac{d}{d\rho}\left(\rho^{1-d/2}J_{d/2+l-1}\left(\rho\right)\right)$ and $\frac{d}{d\rho}\left(\rho^{1-d/2}J_{d/2+l+m-1}\left(\rho\right)\right)$ have no common positive zeros. 
\end{lemma}
As in the Dirichlet case, from this result we see that for $l_0\neq l_1$, $S_{l_0}$ and $S_{l_1}$ have no common zeros other than $\rho=0$. So for $\lambda_{ln}\neq 0$, the multiplicity of the eigenvalue $\lambda_{ln}$ is exactly $M(d,l)$, and all the eigenfunctions associated to an eigenvalue $\lambda_{ln}$ are again of the form \begin{equation*}
	u_{ln}(\rho,\theta)=\sum_{m=1}^{M(d,l)}c_m\rho^{1-d/2}J_{d/2+l-1}\left(\rho\sqrt{\lambda_{ln}}\right)Y_{lm}\left(\theta\right).
\end{equation*}Notice that between the Dirichlet and Neumann case the only difference is the family of eigenvalues but not the expression of the associated eigenfunctions.
\subsubsection{Ellipses}
\begin{figure}\renewcommand\thefigure{2.3}
	\includegraphics[width=7cm]{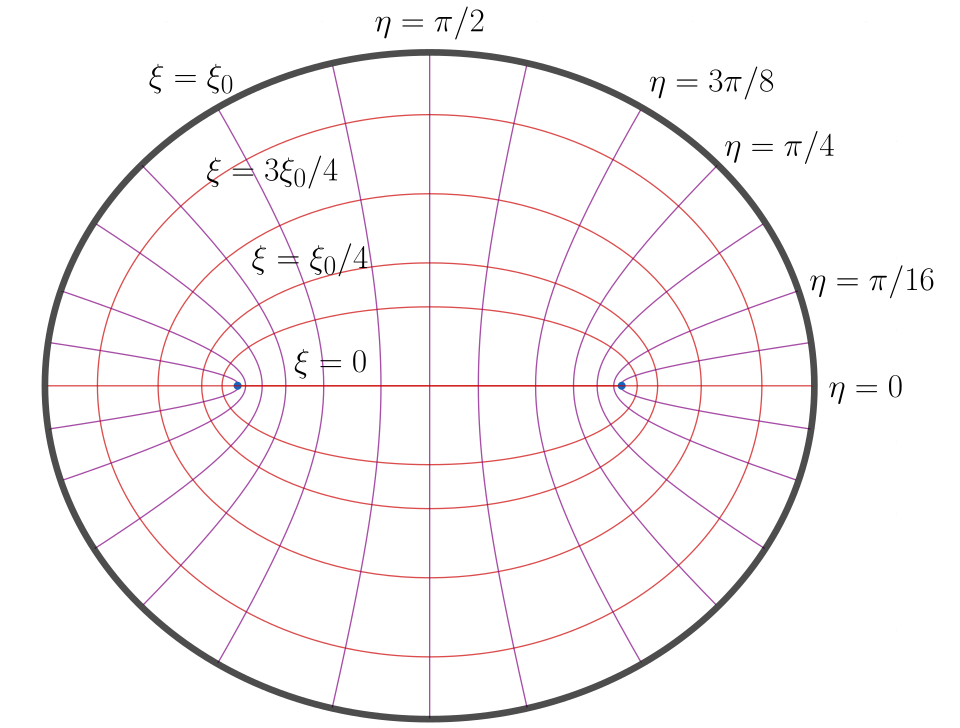}
	\caption{Elliptical coordinates system}\label{elli}
\end{figure} 
To study these domains, we consider the classical elliptical coordinates adapted to the kind of ellipses that we are considering \begin{equation}
	\left\{	\begin{aligned}
		&x =c\cdot \cosh(\xi)\cos(\eta)\\
		&y =c\cdot \sinh(\xi)\sin(\eta)
	\end{aligned}\right.
\end{equation}
Where $0\leq \xi<\infty$, $0\leq \eta<2\pi$, $0<b\leq 1$ and $c=\sqrt{1-b^2}$ is a parameter that represents half the distance between the foci of the ellipse. Our domain is the ellipse $\mathcal{E}_{b}$ with foci at $(\pm c,0)$ and length of the major and minor axes \begin{equation}
	\begin{aligned}
		&1=c\cdot \cosh(\xi_b)\\
		&b=c\cdot \sinh(\xi_b),
	\end{aligned}
\end{equation}i.e., $\xi=\xi_b$ is the boundary of our domain.
In cartesian coordinates, this is \begin{equation}
	\mathcal{E}_{b}=\left\{(z_1,z_2)\in\mathbb{R}^{2},\ \left(z_1\right)^2+\left(\frac{z_2}{b}\right)^2\leq 1\right\}
\end{equation} The quantity \begin{equation}
	e=\frac{1}{\cosh(\xi_b)}=\sqrt{1-b^2}=c
\end{equation}is called the eccentricity of the ellipse. It coincides with $c$ when, as done in this paper, we restrict ourselves to the case $a=1$.

In these new coordinates, the Laplacian reads as \begin{equation}
	\Delta=\frac{1}{c^2(\cosh^2(\xi)-\cos^2(\eta))}\left(\frac{\partial^2}{\partial\xi^2}+\frac{\partial^2}{\partial\eta^2}\right).
\end{equation}The eigenvalue problem is then:\begin{equation}\left\{
	\begin{aligned}
		&\frac{\partial^2u}{\partial\xi^2}+\frac{\partial^2u}{\partial\eta^2}+c^2\lambda\left(\cosh^2(\xi)-\cos^2(\eta)\right)u=0\text{ in }\mathcal{E}_b,\\
		&u(\xi_b,\eta)=0 \text{ or }\partial_\nu u(\xi_b,\eta)=0, \text{ in }\partial\mathcal{E}_b,
	\end{aligned}\right.
\end{equation}depending on whether we consider Dirichlet or Neumann boundary conditions. Here $\partial_\nu u$ is the derivative in the direction of the outward pointing normal to $\partial\mathcal{E}_b$.	Separating variables $\xi$ and $\eta$ in the previous equation, $u(\xi,\eta)=S(\xi)E(\eta)$, we reach
\begin{equation}
	\begin{aligned}
		\frac{1}{S} \frac{d^2 S}{d \xi^2}+c^2 \lambda \cosh ^2 \xi & =\al+\frac{1}{2} c^2 \lambda, \\
		c^2 \lambda \cos ^2 \eta-\frac{1}{E} \frac{d^2 E}{d \eta^2} & =\al+\frac{1}{2} c^2\lambda.
	\end{aligned}
\end{equation}
We have chosen $\al+c^2 \lambda / 2$ as the separation constant. Doing some computations we get the terms $c^2 \lambda\left(\cosh ^2 \xi-\frac{1}{2}\right)$ and $c^2 \lambda\left(\cos ^2 \eta-\frac{1}{2}\right)$, which can then be written in the respective forms $\frac{1}{2} \cosh 2 \xi$ and $\frac{1}{2} \cos 2 \eta$.  We now define a quantity $q=\frac{c^2 \lambda}{4}$ and get the following: \begin{equation}\label{sep}
	\left\{\begin{aligned}
&	\frac{d^2S}{d\xi^2}-\left(\alpha-2q\cosh(2\xi)\right)S=0	\\&\frac{d^2E}{d\eta^2}+\left(\alpha-2q\cos(2\eta)\right)E=0.
	\end{aligned}\right.
\end{equation}

The first ordinary differential equation is known as the Mathieu equation and the second one, is the Modified Mathieu equation. The first equation has associated periodic boundary conditions, $E(0)=E(2\pi)$, and the second one has Dirichlet boundary condition $S(\xi_b)=0$, or Neumann boundary conditions $S'(\xi_b)=0$. The nonsingular and periodic solutions (with period $2\pi$ or $\pi$) to the Mathieu equation are known as Mathieu functions and are elliptic analogues of the sine and cosine functions. In particular, these functions are either even or odd. If we make the change of variable $\xi\mapsto i\xi$, we get the Modified Mathieu equation. The solutions that corresponds with periodic solutions of the Mathieu are called Modified Mathieu functions. These functions are elliptical analogues of the Bessel functions and are oscillatory but not periodic for real $\xi$. We can notice the following, which is \cite[Remark 8]{Hez}.

\begin{remark}\label{remark}
	In order to obtain solutions well-defined on the line segment joining the foci, i.e., at $\mu=0$, the solutions must satisfy the boundary condition $S'(0)=0$ in the case when $E$ is even and $S(0)=0$ in the case when $E$ is odd. In these cases the solutions $S$ are also respectively even and odd functions. 
\end{remark}

We follow \cite{Math} and \cite{Hez} to understand better the solutions to these equations.

For a fixed $q>0$, the first equation of \eqref{sep} has two real valued sequences $\left\{a_n(q)\right\}_{n=0}^\infty$ and $\left\{b_n(q)\right\}_{n=0}^\infty$ so that there exist associated non-trivial solutions that are appropriate for our problem, i.e., periodic with period $\pi$ or $2\pi$. These solutions are even with respect to the transformation $\eta\mapsto-\eta$ (or, equivalently, $y\mapsto -y$) if, and only if, they are associated to a value $\alpha=a_n(q)$; and are odd if they are associated to a value $\alpha=b_n(q)$. We denote these functions as $E_n^e(q)$ and $E_n^o(q)$ respectively. We also know that the solutions which are $\pi$ periodic (i.e., invariant under the transformation $\eta\mapsto\pi-\eta$ or, equivalently, even with respect to $x\mapsto -x$) are those functions with $n$ even, and those solutions that have anti-period $\pi$ (i.e. odd functions in the $x$ variable) are the ones associated to odd values of $n$.

Similarly, for the second equation in \eqref{sep}, for each boundary condition (Dirichlet $\mathrm{D}$ or Neumann $\mathrm{N}$) and for each $q>0$ there exist sequences $\left\{c_n^\mathrm{BC}(q)\right\}_{n=0}^\infty$ and $\left\{d_n^\mathrm{BC}(q)\right\}_{n=0}^\infty$ ($\mathrm{BC}$ is either $\mathrm{D}$ or $\mathrm{N}$) such that the equation with the boundary condition imposed has non-trivial even solutions $S^{e,\mathrm{BC}}_m(q)$ if $\alpha=c_m^\mathrm{BC}(q)$ and odd solutions $S^{o,\mathrm{BC}}_m(q)$ if $\alpha=d_m^\mathrm{BC}(q)$.

In order to have eigenfunctions of the ellipse, one has to search values of $q$ such that both problems possess non-trivial solutions for the same value of $\alpha$. By Remark~\ref{remark}, we only consider the couples $E^e_n(q)S^{e,\mathrm{BC}}_m(q)$ and $E^o_n(q)S^{o,\mathrm{BC}}_m(q)$ associated to eigenvalues $q_{m,n}^{e,\mathrm{BC}}$ and $q_{m,n}^{o,\mathrm{BC}}$ respectively that satisfy\begin{equation}
	a_n\left(q_{m,n}^{e,\mathrm{BC}}\right)=c_m^\mathrm{BC}\left(q_{m,n}^{e,\mathrm{BC}}\right),
	\text{ or }
	b_n\left(q_{m,n}^{o,\mathrm{BC}}\right)=d_m^\mathrm{BC}\left(q_{m,n}^{o,\mathrm{BC}}\right).
\end{equation}

The existence of point of intersection of the curves $a_n(q)$ with $c_m^\mathrm{BC}(q)$ and $b_m(q)$ with $d_m^\mathrm{BC}(q)$ is guaranteed  by the main theorem in \cite{Neves}.
\begin{theorem}
	For each pair $\left(m,n\right)\in\left(\mathbb{N}\cup\left\{0\right\}\right)^2$, there is a unique positive solution $q$ to each of the equations $a_n\left(q_{m,n}^{e,\mathrm{BC}}\right)=c_m^{\mathrm{BC}}\left(q_{m,n}^{e,\mathrm{BC}}\right)$ and $ b_n\left(q_{m,n}^{o,\mathrm{BC}}\right)=d_m^\mathrm{BC}\left(q_{m,n}^{o,\mathrm{BC}}\right).$
\end{theorem}

The frequencies $\lambda^\mathrm{BC}$ of $\mathcal{E}_{a,\mathrm{BC}}$ with $\mathrm{BC}$ boundary condition are obtained by sorting $\left\{\frac{4q_{m,n}^{e,\mathrm{BC}}}{c^2},\frac{4q_{m,n}^{o,\mathrm{BC}}}{c^2}; \left(m,n\right)\in\left(\mathbb{N}\cup\left\{0\right\}\right)^2\right\}$. 
A study of the irreducible representations of the $\mathbb{Z}_2\times\mathbb{Z}_2$ symmetry group reveals that there exists an orthonormal basis of eigenfunctions of the ellipse which are even or odd with respect to each $\mathbb{Z}_2$ symmetry. The above discussion gives us that the symmetric eigenfunctions conforming that basis are:
\begin{equation}
	\left\{\begin{array}{lll}
		\text { (even, even) : } & u_{m, 2 k}^{e,\mathrm{BC}}=S_m^{e,\mathrm{BC}}(\xi,q) E_{2 k}^{e}(\eta,q) ; & q=q_{m, 2 k}^{e,\mathrm{BC}}, \\
		\text { (even, odd) : } & u_{m, 2 k}^{o,\mathrm{BC}}=S_m^{o,\mathrm{BC}}(\xi,q) E_{2 k}^{o}(\eta,q); & q=q_{m, 2 k}^{o,\mathrm{BC}}, \\
		\text { (odd, even) : } & u_{m, 2 k+1}^{e,\mathrm{BC}}=S_m^{e,\mathrm{BC}}(\xi,q) E_{2 k+1}^{e}(\eta,q) ; & q=q_{m, 2 k+1}^{e,\mathrm{BC}}, \\
		\text { (odd, odd) : } & u_{m, 2 k+1}^{o,\mathrm{BC}}=S_m^{o,\mathrm{BC}}(\xi,q) E_{2 k+1}^{o}(\eta,q) ; & q=q_{m, 2 k+1}^{o,\mathrm{BC}} .
	\end{array}\right.
\end{equation}

It is possible that two symmetric eigenfunctions correspond to the same eigenvalue, or to different eigenvalues. The next result (which is \cite[Theorem 1.1]{Hil}) shows that, for most ellipses, two different symmetric eigenfunctions cannot be associated to the same $q$.
\begin{theorem}\label{Hillariet}
	There exists a countable subset $\mathcal{C}\subset\left]0,1\right]$ so that if $b\notin\mathcal{C}$, then each eigenvalue of the Dirichlet or Neumann Laplace operator of the ellipse $\mathcal{E}_{b}$ is simple, i.e., it has multiplicity one.
\end{theorem}

As a consequence, if we choose one of those $\mathcal{E}_{b}$, with $b\notin\mathcal{C}$, and Dirichlet or Neumann conditions, we know that any eigenfunction is a multiple of the following form\begin{equation}
	u_{mn}\left(\xi,\eta\right)= S^p_m\left(\xi,q_{m,n}\right)E^p_n\left(\eta,q_{m,n}\right),
\end{equation}with associated eigenvalue $\lambda_{mn}=\frac{4q_{m,n}}{c^2}$ where $m,n\in\mathbb{N}\cup\left\{0\right\}$ and $p=e,o$. Notice that we have removed the boundary condition dependence to alleviate the notation.
\section{Inverse localization without a fixed based point}\label{pnf}

    In this chapter, we prove the positive part of Theorem \ref{BT.polygons}, except for the case where $\varphi$ has extra symmetries. Therefore, we assume that $\mathcal{P} \subset \mathbb{R}^2$ is an integrable triangle or a rational rectangle, as defined in Definition~\ref{IntPol}. Only when $\mathcal{P} = \mathcal{Q}_l$ is a rectangle, we consider dimensions $d \geq 2$. We also consider the eigenvalue problem with Dirichlet or Neumann boundary conditions on $\mathcal{P}$ and prove the following result on inverse localization, where the point around which we localize may depend on the choice of monochromatic wave.

\begin{theorem}\label{pnft}
	For any $\ep>0$, any ball $B$ of $\mathbb{R}^2$, any $k\in\mathbb{N}$ and any $\vp\in\MW$, there exist a sequence $\lambda_n$ of eigenvalues, a sequence of associated eigenfunctions $u_{n}$ and a sequence of open sets $O_{n,\varphi}\subseteq \mathcal{P}$, depending on $\varphi$ and $\ep$, such that \begin{equation}
		|O_{n,\varphi}|=\int_{O_{n,\varphi}}1\ dz\xrightarrow[n\rightarrow\infty]{}\left|\mathcal{P}\right|\text{, along the sequence}
	\end{equation}and \begin{equation}
		\left\|\vp-u_{n}\left(z^0+\frac{\cdot}{\sqrt{\lambda_n}}\right)\right\|_{C^k(B)}<\ep,
	\end{equation}for all the $z^0\in O_{n,\varphi}$ and any $n$ large enough.
	
\end{theorem}
We organize this chapter in two sections: in the first part, one can find the whole proof of Theorem \ref{pnft}, with the exception of a technical lemma concerning the bound of an error term, Lemma \ref{BoundOfNonTrans}, whose proof is postponed until the second part of the Chapter. 
\subsection{Proof of the inverse localization theorem}
\begin{proof}
	Along the whole proof we will be writing $\mathbb{R}^d$ and $d$ should be understood as $d=2$ for all the cases except for $\mathcal{P}=\mathcal{Q}_l$. We split the proof into several parts. Some of them are the same for all $\mathcal{P}$ but for others we will need to differentiate cases.
	
\textbf{	\underline{Step 1:}}	To begin with, we follow \cite[Proposition 2.2]{EPT21} to find a function \begin{equation}
		\phi(z):=\sum_{\gamma=1}^{L} \frac{\tilde{c}_\gamma}{\left|z-z^\gamma\right|^{\frac{d}{2}-1}} J_{\frac{d}{2}-1}\left(\left|z-z^\gamma\right|\right)
	\end{equation}
	with large $L>0$, some constants $\tilde{c}_\gamma\in\mathbb{R}$, and points $z^\gamma\in B_R\subset\mathbb{R}^d$ for all $1\leq \gamma\leq L$, where the radius $R\geq 1$ depends on $\vp$ and the approximation error; the function $\phi$	approximates the function $\vp$ in the unit ball as
	\begin{equation}
		\|\phi-\varphi\|_{C^k(B)}<\ep/3 .
	\end{equation}
	We invoke \cite[Proposition 2.1]{EPR20} to write 
	\begin{equation}
		\phi(z)=\sum_{\gamma=1}^{L} c_\gamma\int_{\mathbb{S}^{d-1}}e^{i(z-z^\gamma)\xi}d\sigma_{\mathbb{S}^{d-1}}(\xi)=\sum_{\gamma=1}^{L}c_\gamma\int_{\mathbb{S}^{d-1}}\cos\left((z-z^\gamma)\xi\right)d\sigma_{\mathbb{S}^{d-1}}(\xi), 
	\end{equation}with $c_\gamma=\tilde{c}_\gamma\cdot\left(2\pi\right)^{-d/2}$. To pass to the second equality we have used that~$\sin((z-z^\gamma)\xi)=-\sin(-(z-z^\gamma)\xi)$, and the integral of the imaginary part then vanishes. We conclude that
	\begin{equation}
		\left\|\vp-\sum_{\gamma=1}^{L}c_\gamma\int_{\mathbb{S}^{d-1}}\cos\left((\cdot-z^\gamma)\xi\right)d\sigma_{\mathbb{S}^{d-1}}(\xi)\right\|_{C^k(B)}<\ep/3 .
	\end{equation}
	
\textbf{	\underline{Step 2:}} For the next step we choose a point $z^0\in \mathcal{P}$ that will be fixed later and we define an eigenfunction $u_\lambda$, with the following general form: 
	\begin{equation}
		\begin{aligned}
			&u_\lambda^{\mathcal{P},\mathrm{BC}}\left(z^0+\frac{z}{\sqrt{\lambda}}\right)=\sum_{\gamma=1}^{L}c_\gamma u_{\lambda,\gamma}^{\mathcal{P},\mathrm{BC}}\left(z^0+\frac{z}{\sqrt{\lambda}}\right),
		\end{aligned}
	\end{equation}where we set
    \begin{equation}\label{eigenPNF}
		\begin{aligned}
			& u_{\lambda,\gamma}^{\mathcal{P},\mathrm{BC}}\left(z^0+\frac{z}{\sqrt{\lambda}}\right):=\frac{4^{d}}{\#\mathcal{N}_\lambda^{\mathcal{P}}}\sum_{N\in\mathcal{N}_\lambda^{\mathcal{P},\mathrm{BC}}}c_{N,\gamma}^{\mathcal{P},\mathrm{BC}}u_N^{\mathcal{P},\mathrm{BC}}\left(z^0+\frac{z}{\sqrt{\lambda}}\right),
		\end{aligned}
	\end{equation}whenever $\mathcal{P}\in\left\{\mathcal{Q}_l,\mathcal{T}_\mathrm{iso},\mathcal{T}_\mathrm{hemi}\right\}$. The case of $\mathcal{P}=\mathcal{T}_\mathrm{iso}$ is slightly different, and it will be defined later.
    
	For any $\gamma$, the precise definitions of all $	c_{N,\gamma}^{\mathcal{P},\mathrm{BC}}$, $u_N^{\mathcal{P},\mathrm{BC}}$, $\mathcal{N}_\lambda^{\mathcal{P}}$, and $\mathcal{N}_\lambda^{\mathcal{P},\mathrm{BC}}$ depend on the polygon $\mathcal{P}$ and the boundary conditions. We refer the reader to Chapter \ref{formulas} for the definitions of the sets $\mathcal{N}_\lambda^{\mathcal{P}}$ and $\mathcal{N}_\lambda^{\mathcal{P},\mathrm{BC}}$. The introduced functions and constants are as follows:
	\begin{itemize}
		\item For $\mathcal{P}=\mathcal{Q}_l$, we recall that
		\begin{equation}
			\begin{aligned}
				&u_{N}^{\mathcal{Q}_l,\mathrm{D}}\left(z^0+\frac{z}{\sqrt{\lambda}}\right)=\prod_{j=1}^{d}\sin\left(\frac{z^0_jN_j\pi}{l_j}+\frac{z_jN_j\pi}{l_j\sqrt{\lambda}}\right)\text{ and }\\	&u_{N}^{\mathcal{Q}_l,\mathrm{N}}\left(z^0+\frac{z}{\sqrt{\lambda}}\right)=\prod_{j=1}^{d}\cos\left(\frac{z^0_jN_j\pi}{l_j}+\frac{z_jN_j\pi}{l_j\sqrt{\lambda}}\right).
			\end{aligned}
		\end{equation}
		And the coefficients are defined as 
		\begin{equation}
			c_{N,\gamma}^{\mathcal{Q}_l,\mathrm{BC}}=u_N^{\mathcal{Q}_l,\mathrm{BC}}\left(z^0+\frac{z^\gamma}{\sqrt{\lambda}}\right),
		\end{equation}
		\item	In the case of $\mathcal{P}=\mathcal{T}_\mathrm{iso}$, where $N=(m,n)$, we have from the previous chapter that
		\begin{equation}
			\begin{aligned}
				&u_{N}^{\mathcal{T}_\mathrm{iso},\mathrm{D}}\left(z^0+\frac{z}{\sqrt{\lambda}}\right)=\sin\left(\pi m\left(z^0_1+\frac{z_1}{\sqrt{\lambda}}\right)\right)\sin\left(\pi n\left(z^0_2+\frac{z_2}{\sqrt{\lambda}}\right)\right)\text{ and }\\	&u_{N}^{\mathcal{T}_\mathrm{iso},\mathrm{N}}\left(z^0+\frac{z}{\sqrt{\lambda}}\right)=\cos\left(\pi m\left(z^0_1+\frac{z_1}{\sqrt{\lambda}}\right)\right)\cos\left(\pi n\left(z_2^0+\frac{z_2}{\sqrt{\lambda}}\right)\right).
			\end{aligned}
		\end{equation}
		We define the coefficients as follows:
		\begin{equation}
			\begin{aligned}
				c_{N,\gamma}^{\mathcal{T}_\mathrm{iso},\mathrm{D}}&=\sin\left(\pi m\left(z_1^0+\frac{z_1^\gamma}{\sqrt{\lambda}}\right)\right)\sin\left(\pi n\left(z_2^0+\frac{z_2^\gamma}{\sqrt{\lambda}}\right)\right)\\&-\sin\left(\pi n\left(z_1^0+\frac{z_1^\gamma}{\sqrt{\lambda}}\right)\right)\sin\left(\pi m\left(z_2^0+\frac{z_2^\gamma}{\sqrt{\lambda}}\right)\right)\text{ and }\\
				c_{N,\gamma}^{\mathcal{T}_\mathrm{iso},\mathrm{N}}&=\cos\left(\pi m\left(z_1^0+\frac{z_1^\gamma}{\sqrt{\lambda}}\right)\right)\cos\left(\pi n\left(z_2^0+\frac{z_2^\gamma}{\sqrt{\lambda}}\right)\right)\\&+\cos\left(\pi n\left(z_1^0+\frac{z_1^\gamma}{\sqrt{\lambda}}\right)\right)\cos\left(\pi m\left(z_2^0+\frac{z_2^\gamma}{\sqrt{\lambda}}\right)\right).
			\end{aligned}
		\end{equation}
		We remark that the coefficients so defined satisfy conditions \eqref{condD} and \eqref{condN}, respectively, and therefore equation \eqref{eigenPNF} defines a proper eigenfunction.
		
		\item  When $\mathcal{P}=	\mathcal{T}_\mathrm{equi}$, as already stated, the structure is a bit different: both $c_{N,\gamma}^{\mathcal{P},\mathrm{BC}}$ and $u_N^{\mathcal{P},\mathrm{BC}}$ are understood as vector valued with two components (one symmetric $S$ and another antisymmetric $A$) and their product as a scalar product. We recall that for any pair $N=(m,n)$, there are now two linearly independent eigenfunctions:
		\begin{equation}
			\begin{aligned}
				&u^{	\mathcal{T}_\mathrm{equi},\mathrm{D}}_{N,S}\left(z^0+\frac{z}{\sqrt{\lambda}}\right)=\sin\left(\frac{2\pi\sqrt{3}n}{3}\left(z^0_2+\frac{z_2}{\sqrt{\lambda}}\right)\right)\cos\left(\frac{2\pi m}{3}\left(z^0_1+\frac{z_1}{\sqrt{\lambda}}\right)\right),\\
				&u^{	\mathcal{T}_\mathrm{equi},\mathrm{D}}_{N,A}\left(z^0+\frac{z}{\sqrt{\lambda}}\right)=\sin\left(\frac{2\pi\sqrt{3}n}{3}\left(z^0_2+\frac{z_2}{\sqrt{\lambda}}\right)\right)\sin\left(\frac{2\pi m}{3}\left(z^0_1+\frac{z_1}{\sqrt{\lambda}}\right)\right),\\
				&u^{	\mathcal{T}_\mathrm{equi},\mathrm{N}}_{N,S}\left(z^0+\frac{z}{\sqrt{\lambda}}\right)=\cos\left(\frac{2\pi\sqrt{3}n}{3}\left(z^0_2+\frac{z_2}{\sqrt{\lambda}}\right)\right)\cos\left(\frac{2\pi m}{3}\left(z^0_1+\frac{z_1}{\sqrt{\lambda}}\right)\right),\\
				&u^{	\mathcal{T}_\mathrm{equi},\mathrm{N}}_{N,A}\left(z^0+\frac{z}{\sqrt{\lambda}}\right)=\cos\left(\frac{2\pi\sqrt{3}n}{3}\left(z^0_2+\frac{z_2}{\sqrt{\lambda}}\right)\right)\sin\left(\frac{2\pi m}{3}\left(z^0_1+\frac{z_1}{\sqrt{\lambda}}\right)\right).\\
			\end{aligned}
		\end{equation} So we define the following different constants
		\begin{equation}
			\begin{aligned}
				&c_{N,\gamma,S}^{	\mathcal{T}_\mathrm{equi},\mathrm{D}}=u^{	\mathcal{T}_\mathrm{equi},\mathrm{D}}_{N,S}\left(z^0+\frac{z^\gamma}{\sqrt{\lambda}}\right),\ c_{N,\gamma,A}^{	\mathcal{T}_\mathrm{equi},\mathrm{D}}=u^{	\mathcal{T}_\mathrm{equi},\mathrm{D}}_{N,A}\left(z^0+\frac{z^\gamma}{\sqrt{\lambda}}\right),\\
				&c_{N,\gamma,S}^{	\mathcal{T}_\mathrm{equi},\mathrm{N}}=u^{	\mathcal{T}_\mathrm{equi},\mathrm{N}}_{N,S}\left(z^0+\frac{z^\gamma}{\sqrt{\lambda}}\right),\ c_{N,\gamma,A}^{	\mathcal{T}_\mathrm{equi},\mathrm{N}}=u^{	\mathcal{T}_\mathrm{equi},\mathrm{N}}_{N,A}\left(z^0+\frac{z^\gamma}{\sqrt{\lambda}}\right).
			\end{aligned}
		\end{equation}
		And, finally, for any $1\leq\gamma\leq L$ and both boundary conditions, the eigenfunction is defined as \begin{equation}
			\begin{aligned}
				&u_{\lambda,\gamma}^{	\mathcal{T}_\mathrm{equi},\mathrm{BC}}\left(z^0+\frac{z}{\sqrt{\lambda}}\right)=\\&\frac{8}{\#\mathcal{N}_\lambda^{	\mathcal{T}_\mathrm{equi},\mathrm{BC}}}\sum_{N\in\mathcal{N}_\lambda^{	\mathcal{T}_\mathrm{equi},\mathrm{BC}}}\left(c_{N,\gamma,A}^{	\mathcal{T}_\mathrm{equi},\mathrm{BC}}\cdot u_{N,A}^{	\mathcal{T}_\mathrm{equi},\mathrm{BC}}+c_{N,\gamma,S}^{	\mathcal{T}_\mathrm{equi},\mathrm{BC}}\cdot u_{N,S}^{	\mathcal{T}_\mathrm{equi},\mathrm{BC}}\right).
			\end{aligned}
		\end{equation}
		Notice also that there is a small change in the constant in front.
		
		\item Finally, for $\mathcal{P}=\mathcal{T}_\mathrm{hemi}$ and $N=(m,n)$, the eigenfunctions are
		\begin{equation}
			\begin{aligned}
				&u^{\mathcal{T}_\mathrm{hemi},\mathrm{D}}_{N}\left(z^0+\frac{z}{\sqrt{\lambda}}\right)=\sin\left(\frac{2\pi\sqrt{3}n}{3}\left(z^0_2+\frac{z_2}{\sqrt{\lambda}}\right)\right)\sin\left(\frac{2\pi m}{3}\left(z^0_1+\frac{z_1}{\sqrt{\lambda}}\right)\right)\text{ and }\\
				&u^{\mathcal{T}_\mathrm{hemi},\mathrm{N}}_{N}\left(z^0+\frac{z}{\sqrt{\lambda}}\right)=\cos\left(\frac{2\pi\sqrt{3}n}{3}\left(z^0_2+\frac{z_2}{\sqrt{\lambda}}\right)\right)\cos\left(\frac{2\pi m}{3}\left(z^0_1+\frac{z_1}{\sqrt{\lambda}}\right)\right).
			\end{aligned}
		\end{equation}
		We again choose coefficients to be
		\begin{equation}
			\begin{aligned}
				&c_{N,\gamma}^{\mathcal{T}_\mathrm{hemi},\mathrm{D}}=u^{\mathcal{T}_\mathrm{hemi},\mathrm{D}}_{N}\left(z^0+\frac{z^\gamma}{\sqrt{\lambda}}\right)\text{ and }\\
				&c_{N,\gamma}^{\mathcal{T}_\mathrm{hemi},\mathrm{N}}=u^{\mathcal{T}_\mathrm{hemi},\mathrm{N}}_{N}\left(z^0+\frac{z^\gamma}{\sqrt{\lambda}}\right)\end{aligned}
		\end{equation}
	\end{itemize}
	
\textbf{	\underline{Step 3:}}	In the previous step we have chosen the constants in such a way that we can apply one of the following trigonometric identities  \begin{equation}
		\begin{aligned}
			&\sin(a+b)\sin(a+c)=\frac{1}{2}\left(\cos(b-c)-\cos(2a+(b+c))\right),\\&
			\cos(a+b)\cos(a+c)=\frac{1}{2}\left(\cos(2a+(b+c))+\cos(b-c)\right).
		\end{aligned}
	\end{equation} With these trigonometric double-angle relations, we can split, for all $1\leq \gamma\leq L$, the eigenfunctions as
	\begin{equation}\label{parts}
		u_{\lambda,\gamma}^{\mathcal{P},\mathrm{BC}}\left(z^0+\frac{z}{\sqrt{\lambda}}\right)=K_{\lambda}^\mathcal{P}(z-z^\gamma)+\mathcal{E}_{\lambda,\mathcal{P},\mathrm{BC}}^{z^0}(z,z^\gamma),
	\end{equation} where $K_{\lambda}^\mathcal{P}$ is a translation invariant part, which does not depend on the base point $z^0$, and $\mathcal{E}_{\lambda,\mathcal{P},\mathrm{BC}}^{z^0}$ is some translation-dependent function that will be bounded by a small quantity depending on $\lambda$ in the following step.
	
	In this step we pay attention to the translation invariant part. We restrict ourselves to eigenvalues $\lambda$ that are neither a perfect square $\lambda\neq m^2$ nor three times a perfect square $\lambda\neq 3\cdot n^2$. It is easy to see that then the translation invariant part depends only on the domain $\mathcal{P}$ but not on the boundary conditions.  A list of the functions that we get is as follows:
	\begin{itemize}
		\item $\mathcal{P}=\mathcal{Q}_l$ gives us	\begin{equation}
			K_{\lambda}^{\mathcal{Q}_l}(z-z^\gamma)=\frac{2^d}{\#\mathcal{N}_\lambda^{\mathcal{Q}_l}}\sum_{N\in\mathcal{N}_\lambda^{\mathcal{Q}_l,\mathrm{D}}}\prod_{j=1}^{d}\cos\left(\frac{N_j\pi}{l_j\sqrt{\lambda}}(z_j-z^\gamma_j)\right).
		\end{equation}
		\item For $\mathcal{P}=\mathcal{T}_\mathrm{iso}$ we have
		\begin{equation}
			K_{\lambda}^{\mathcal{T}_\mathrm{iso}}(z-z^\gamma)=\frac{4}{\#\mathcal{N}_\lambda^{\mathcal{T}_\mathrm{iso}}}\sum_{N\in\mathcal{N}_\lambda^{\mathcal{T}_\mathrm{iso},\mathrm{D}}}\cos\left(\frac{\pi m}{\sqrt{\lambda}}(z_1-z^\gamma_1)\right)\cos\left(\frac{\pi n}{\sqrt{\lambda}}(z_2-z^\gamma_2)\right).
		\end{equation}Notice that this is just a particular case of the previous one $(d=2,l_1=1)$. 
		\item When considering $\mathcal{P}=	\mathcal{T}_\mathrm{equi}$ or $\mathcal{P}=\mathcal{T}_\mathrm{hemi}$, we get the same:
		\begin{equation}
			\begin{aligned}	&K_{\lambda}^{\mathcal{T}_\mathrm{hemi}}(z-z^\gamma)=K_{\lambda}^{	\mathcal{T}_\mathrm{equi}}(z-z^\gamma)=\\&\frac{4}{\#\mathcal{N}_\lambda^{\mathcal{T}_\mathrm{equi}}}\sum_{N\in\mathcal{N}_\lambda^{\mathcal{T}_\mathrm{hemi},\mathrm{D}}}\cos\left(\frac{2\pi m}{3\sqrt{\lambda}}(z_1-z^\gamma_1)\right)\cos\left(\frac{2\pi n\sqrt{3}}{3\sqrt{\lambda}}(z_2-z^\gamma_2)\right)
			\end{aligned}
		\end{equation}
	\end{itemize}
	The symmetry of $\mathcal{N}_\lambda^{\mathcal P}$ and the extra hypotheses on $\lambda$ allow us to see these as:
	\begin{equation}
		\begin{aligned}
			&K_{\lambda}^{\mathcal{Q}_l}(z-z^\gamma)=\frac{1}{\#\mathcal{N}_\lambda^{\mathcal{Q}_l}}\sum_{N\in\mathcal{N}_\lambda^{\mathcal{Q}_l}}\cos\left(\sum_{j=1}^d\frac{N_j\pi}{l_j\sqrt{\lambda}}(z_j-z^\gamma_j)\right),\\
			&	K_{\lambda}^{\mathcal{T}_\mathrm{iso}}(z-z^\gamma)=\frac{1}{\#\mathcal{N}_\lambda^{\mathcal{T}_\mathrm{iso}}}\sum_{N\in\mathcal{N}_\lambda^{\mathcal{T}_\mathrm{iso}}}\cos\left(\frac{\pi m}{\sqrt{\lambda}}(z_1-z^\gamma_1)+\frac{\pi n}{\sqrt{\lambda}}(z_2-z^\gamma_2)\right),\\
			&	K_{\lambda}^{\mathcal{T}_\mathrm{equi}}(z-z^\gamma)=\frac{1}{\#\mathcal{N}_\lambda^{\mathcal{T}_\mathrm{equi}}}\sum_{N\in\mathcal{N}_\lambda^{\mathcal{T}_\mathrm{hemi}}}\cos\left(\frac{2\pi m}{3\sqrt{\lambda}}(z_1-z^\gamma_1)+\frac{2\pi n\sqrt{3}}{3\sqrt{\lambda}}(z_2-z^\gamma_2)\right)\text{ and }\\&K_{\lambda}^{\mathcal{T}_\mathrm{hemi}}(z-z^\gamma)=K_{\lambda}^{\mathcal{T}_\mathrm{equi}}(z-z^\gamma).
		\end{aligned}
	\end{equation}
	
	To continue, we will follow \cite{ILtori} and use equidistribution properties of integer points on ellipsoids. Before introducing the following key definition, we notice that there exist constants $C^\mathcal{P}$ such that $C^\mathcal{P}\cdot\lambda=\mu\in\mathbb{N}$ for all $\lambda$ eigenvalues of the problems that we are considering here. With this in mind, we can now introduce the following definition:\begin{definition}\label{aequid}
		We shall say that
		\begin{equation}
			\mathcal{N}^\mathcal{P}_\lambda=\left\{N\in\mathbb{Z}^d: \mathcal{Q}^\mathcal{P}(N)=\mu=C^\mathcal{P}\cdot\lambda\right\},
		\end{equation}
		is {\em asymptotically equidistributed}\/ over the ellipsoid
		\begin{equation}\label{defcE}
			\mathcal{M}^{\mathcal{P}}=\left\{\xi\in\mathbb{R}^d, \mathcal{Q}^{\mathcal{P}}(\xi)=1\right\}
		\end{equation}
		along a certain sequence of $\lambda_n$ if, for every function $f\in C^\infty(\RR^d)$,
		\begin{equation}
			\lim_{n\to\infty}\frac1{\#	\mathcal{N}^\mathcal{P}_{\lambda_n}}\sum_{N\in	\mathcal{N}^\mathcal{P}_{\lambda_n}} f\left(\frac{N}{\sqrt{C^\mathcal{P}\lambda}}\right) = \frac1{\left|	\mathcal{M}^{\mathcal{P}}\right|} \int_{	\mathcal{M}^{\mathcal{P}}} f(\eta)\, d\si_{	\mathcal{M}^{\mathcal{P}}}(\eta)\,,
		\end{equation}
		where $d\sigma_{	\mathcal{M}^{\mathcal{P}}}$ and $\left|	\mathcal{M}^{\mathcal{P}}\right|$ are the hypersurface measure form and hypersurface total measure of $	\mathcal{M}^{\mathcal{P}}$.
	\end{definition} We find in \cite{Iwaniek,Cill} that there exist sequences of $\lambda_n$ with the asymptotically equidistribution property for all the quadratic forms that we are considering here. Therefore, restricting ourselves to one of those sequences (notice that the sequence to choose will depend on $\mathcal{P}$), we get one of the following:
	
	\begin{equation}
		\begin{aligned}
			&K_{\lambda}^{\mathcal{Q}_l}(z-z^\gamma)=\frac{1}{|\mathcal{M}^{\mathcal{Q}_l}|}\int_{\mathcal{M}^{\mathcal{Q}_l}}\cos\left(\frac{\xi \sqrt{p}}{l}(z-z^\gamma)\right)d\sigma_{\mathcal{M}^{\mathcal{Q}_l}}(\xi)+\mathbb{E}(\lambda_n),\\
			&K_{\lambda}^{\mathcal{T}_\mathrm{iso}}(z-z^\gamma)=\frac{1}{|\mathcal{M}^{\mathcal{T}_\mathrm{iso}}|}\int_{\mathcal{M}^{\mathcal{T}_\mathrm{iso}}}\cos\left(\xi_1(z_1-z_1^\gamma)+\xi_2(z_2-z_2^\gamma)\right)d\sigma_{\mathcal{M}^{\mathcal{T}_\mathrm{iso}}}(\xi)+\mathbb{E}(\lambda_n),\\ &K_{\lambda}^{\mathcal{T}_\mathrm{equi}/\mathcal{T}_\mathrm{hemi}}(z-z^\gamma)=\mathbb{E}(\lambda_n)+\\&\frac{1}{|\mathcal{M}^{\mathcal{T}_\mathrm{equi}}|}\int_{\mathcal{M}^{\mathcal{T}_\mathrm{equi}}}\cos\left(\xi_1(z_1-z_1^\gamma)+\sqrt{3}\xi_2(z_2-z_2^\gamma)\right)d\sigma_{\mathcal{M}^{\mathcal{T}_\mathrm{equi}}}(\xi),
		\end{aligned}
	\end{equation}where $\mathbb{E}(\lambda_n)$ is an error term that satisfies $\mathbb{E}(\lambda_n)\xrightarrow[n\rightarrow \infty ]{}0$.
	
	All of them satisfy that 
	\begin{equation}
		\begin{aligned}
			K_{\lambda}^\mathcal{P}(z-z^\gamma)=&\frac{1}{|\mathcal{M}^\mathcal{P}|\cdot C_\mathcal{P}}\int_{\mathbb{S}^{d-1}}\cos\left(\xi\cdot(z-z^\gamma)\right)d\sigma_{\mathbb{S}^{d-1}}(\xi)+\mathbb{E}(\lambda_n),
		\end{aligned}
	\end{equation}where $C_\mathcal{P}$ is the determinant of the matrix that defines the change of variables to go from $\mathbb{S}^{d-1}$ to $\mathcal{M}^\mathcal{P}$. Precisely $C_{\mathcal{Q}_l}=\frac{p^{d/2}}{l_1\cdot l_2\cdot \ldots\cdot l_{d-1}}$, $C_{\mathcal{T}_\mathrm{iso}}=1$ and $C_{\mathcal{T}_\mathrm{equi}}=C_{\mathcal{T}_\mathrm{hemi}}=\frac{1}{\sqrt{3}}$.
	
	We then choose $\lambda_n$ large enough in the sequence to ensure that $\mathbb{E}(\lambda_n)<\ep/3$ and notice that\begin{equation}
		\left|	\vp(x)-	C_{\mathcal{P}}|\mathcal{M}^\mathcal{P}|	u_\lambda^{\mathcal{P},\mathrm{BC}}\left(z^0+\frac{z}{\sqrt{\lambda}}\right)\right|\leq\frac{2\ep}{3}+	C_\mathcal{P}|\mathcal{M}^\mathcal{P}|\sum_{\gamma=1}^{L}|c_\gamma|4^d\left|\mathcal{E}_{\lambda,\mathcal{P},\mathrm{BC}}^{z^0}(z,z^\gamma)\right|.
	\end{equation}
	
\textbf{	\underline{Step 4:}}
	In this last step we bound the terms $\mathcal{E}_{\lambda,\mathcal{P},\mathrm{BC}}^{z^0}(z,z^\gamma)$ by a small value for most $z\in\mathcal{P}$. We get it by combining the following lemmas:
	\begin{lemma}\label{BoundOfNonTrans}
		For $R>0$, set $\mathcal{R}^{2d}=[-R,R]^{2d}$. Then, for all $\delta>0$, \begin{equation}
			\int_{\mathcal{P}}\int_{\mathcal{R}^{2d}}\left(\mathcal{E}_{\lambda,\mathcal{P},\mathrm{BC}}^{z^0}(z,z^\gamma)\right)^2dzdz^\gamma dz^0\leq\frac{C_\delta R^{2d}}{\lambda^{1/2-\delta}},
		\end{equation}when $\lambda\rightarrow \infty$ along a sequence that gives us asymptotic equidistribution.
	\end{lemma}
	In the proof of this lemma different tools are needed depending on the domain $\mathcal{P}$ we are considering. Notice, however, that there is no difference in considering Neumann o Dirichlet conditions because error terms differ (at most) in a sign that does not affect the bound. The proof is postponed until Section \ref{lemmas}, where we consider the four different cases. The other essential lemma is the following.
	
	\begin{lemma}\label{finalbound}
		Let $\left(f_L\right)_{L=1}^\infty$ be a sequence of continuously differentiable functions $f_L: \mathcal{P}\times \mathcal{R}^{2d} \rightarrow \mathbb{R}$ and write $f_{z^0, L}(u):=f_L(z^0, u)$. Suppose that $\left\|f_L\right\|_{\mathcal{P} \times \mathcal{R}^{2d}}^{\mathrm{Lip}} \leq D$ is bounded uniformly with respect to $R$ and $L$ and that
		
		\begin{equation}
			\int_{\mathcal{P}}\int_{\mathcal{R}^{2d}}\left(f_{z^0, L}(u)\right)^2 d u d z^0 \leq C \frac{R^{2d}}{\Gamma(L)},
		\end{equation}
		
		where $\Gamma(L)$ is some positive real function of $L$. Then
		
		\begin{equation}
			\int_{\mathcal{P}}\left\|f_{z^0, L}\right\|_{\mathcal{R}^{2d}}^{\infty} d z^0 \leq C \frac{R^{2}}{(\Gamma(L))^{\frac{1}{d+3}}}.
		\end{equation}
		
	\end{lemma}
	This is \cite[lemma 2.18]{Cann}. In the statement we have used the notation \begin{equation}
		\|g\|_{\mathcal{D}}^{\mathrm{Lip}}=\sup _{w, v \in \mathcal{D}} \frac{|g(w)-g(v)|}{|w-v|}
	\end{equation} and \begin{equation}
	\|g\|_{\mathcal{D}}^{\infty}:=\sup _{w \in \mathcal{D}}|g(w)|.
	\end{equation} We notice that the condition $\left\|f_L\right\|_{\mathcal{P} \times \mathcal{R}^{2d}}^{\mathrm{Lip}} \leq D$ follows from the smoothness of eigenfunctions. Lemma \ref{finalbound} together with Lemma \ref{BoundOfNonTrans} gives us that, for all $\delta>0$,\begin{equation}
		\int_{\mathcal{P}}\left\|\mathcal{E}_{\lambda, \mathcal{P},\mathrm{BC}}^{z^0}\right\|_{\mathcal{R}^{2d}}^{\infty} d z^0 \leq C_\delta \frac{R^{2}}{\lambda^{\frac{1/2-\delta}{d+3}}}.
	\end{equation}
	Choosing, e.g., $\delta=1/4$, we have \begin{equation}
		\int_{\mathcal{P}}\left\|\mathcal{E}_{\lambda,\mathcal{P},\mathrm{BC}}^{z^0}\right\|_{\mathcal{R}^{2d}}^{\infty} d z^0 \leq C \frac{R^{2}}{\lambda^{\frac{1}{4d+12}}}.
	\end{equation}
	
	Let $\mathbb{M}=\max_{1\leq \gamma\leq L}|c_\gamma|$ and define $\tilde{\ep}=\frac{\ep}{3C_l2^d|\mathcal{M}^\mathcal{Q}|L\mathbb{M}}$. We now use Markov's inequality to see that the set 
	\begin{equation}\label{set}
		O_\lambda:=\left\{z^0\in \mathcal{P},\left\|\mathcal{E}_{\lambda,\mathcal{P},\mathrm{BC}}^{z^0}\right\|_{\mathcal{R}^{2d}}^{\infty}<\frac{\tilde{\ep}}{3} \right\}
	\end{equation}
	satisfy that \begin{equation}
		\begin{aligned}
			&\left|O_\lambda\right|=1-\left|\left\{z^0\in \mathcal{P},\left\|\mathcal{E}_{\lambda,\mathcal{P},\mathrm{BC}}^{z^0}\right\|_{\mathcal{R}^{2d}}^{\infty}\geq\frac{\tilde{\ep}}{3} \right\}\right|\\&\geq1-\frac{	\int_{\mathcal{P}}\left\|\mathcal{E}_{\lambda,\mathcal{P},\mathrm{BC}}^{z^0}\right\|_{\mathcal{R}^{2d}}^{\infty} d z^0}{\tilde{\ep}/3}\geq 1-3C \frac{R^{2}}{\tilde{\ep}\lambda^{\frac{1}{4d+12}}}.
		\end{aligned}
	\end{equation} Particularly, for any fixed $\tilde{\ep}$, letting $\lambda_n\rightarrow\infty$ along the chosen sequence (the one with the equidistribution property) we have $\left|O_{\lambda_n}\right|\xrightarrow[\lambda\rightarrow\infty]{}1$. Notice also that $O_\lambda$ is an open set for any $\lambda$. 
    
    Finally, let's denote $O_{n,\varphi}=O_{\lambda_n}$ and notice that for any $z^0\in O_{n,\varphi} $, we have  \begin{equation}\label{error}
		\begin{aligned}
			&	\left|	\vp(x)-	C_l|\mathcal{M}^\mathcal{Q}|u_{\lambda_n}^{\mathcal{P},\mathrm{BC}}\left(z^0+\frac{z}{\sqrt{\lambda}}\right)\right|\leq\frac{2\ep}{3}+	C_l|\mathcal{M}^\mathcal{Q}|\sum_{\gamma=1}^{L}|c_\gamma|2^d\left|\mathcal{E}_{\lambda,\mathcal{P},\mathrm{BC}}^{z^0}(z,z^\gamma)\right|\\
			&\leq\frac{2\ep}{3}+	C_l|\mathcal{M}^\mathcal{Q}|L\mathbb{M}2^d\left\|\mathcal{E}_{\lambda,\mathcal{P},\mathrm{BC}}^{z^0}\right\|_{\mathcal{R}^{2d}}^{\infty}<\ep,
		\end{aligned}
	\end{equation}for any $x\in B_R$.
	Finally, setting $u_n$ as \begin{equation}
		u_n(x)=C_l|\mathcal{M}^\mathcal{Q}|u_{\lambda_n}^{\mathcal{P},\mathrm{BC}}(x),
	\end{equation}we get \begin{equation}
		\left\|\vp-u_\lambda\left(z^0+\frac{z}{\sqrt{\lambda}}\right)\right\|_{C^0(B_R)}<\ep.
	\end{equation} Using standard elliptic arguments we have the desired bound for the $C^k$ norm in any ball $B$ choosing $R>0$ big enough such that $B\subset B_R$ and all the $z^\gamma\in B_R$ for $1\leq\gamma\leq L$.  
\end{proof}

\subsection{Proof of the lemma: bounding the error term}\label{lemmas}
In this section we prove Lemma \ref{BoundOfNonTrans} in all the different considered cases. It is based in many computations and in some other smaller lemmas that are also proved along this section. We again split the proof in some steps for convenience.

\subsubsection{The case of rectangles $\mathcal{P}=\mathcal{Q}_l$}\label{QlNF}
\begin{proof}
	
\textbf{	\underline{Step I:}}
	To write the error term, we consider two kind of functions,\begin{equation}
		\begin{aligned}
			&C_j^{2,\mathrm{N}}(z,z^\gamma)=\cos\left(\frac{N_j\pi}{l_j\sqrt{\lambda}}(z_j-z^\gamma_j)\right),\\
			&C_j^{1,\mathrm{N}}(z,z^\gamma)=\cos\left(\frac{2N_j\pi z^0_j}{l_j}+\frac{N_j\pi }{l_j\sqrt{\lambda}}(z_j+z^\gamma_j)\right).
		\end{aligned}
	\end{equation}We also consider $i=(i_1,\ldots,i_d)$ with $i_j\in\{1,2\}$ for all $1\leq j\leq d$ and $|i|=i_1+\ldots+i_d$. In eigenfunctions we have all the possible products of $C_j^{i_j,\mathrm{N}}(z,z^\gamma)$, so with this notation, the function to bound is, in the Dirichlet case,
	\begin{equation}\label{DE}
		\begin{aligned}
			\mathcal{E}_{\lambda,\mathcal{Q}_l,\mathrm{D}}^{z^0}(z,z^\gamma)&=\frac{1}{\#\mathcal{N}_\lambda^{\mathcal{Q}_l}}\sum_{N\in\mathcal{N}_\lambda^{\mathcal{Q}_l,\mathrm{D}}}\left[\prod_{j=1}^d\left(C_j^{2,\mathrm{N}}(z,z^\gamma)-C_j^{1,\mathrm{N}}(z,z^\gamma)\right)-\prod_{j=1}^dC_j^{2,\mathrm{N}}(z,z^\gamma)\right]\\
			&=\frac{1}{\#\mathcal{N}_\lambda^{\mathcal{Q}_l}}\sum_{N\in\mathcal{N}_\lambda^{\mathcal{Q}_l,\mathrm{D}}}\sum_{\substack{1\leq i\leq 2\\|i|<2d}}(-1)^{|i|}\prod_{j=1}^dC_j^{i_j,\mathrm{N}}(z,z^\gamma),
		\end{aligned}
	\end{equation}and, in the Neumann case, 
	\begin{equation}\label{NE}
		\begin{aligned}
			\mathcal{E}_{\lambda,\mathcal{Q}_l,\mathrm{N}}^{z^0}(z,z^\gamma)&=\frac{1}{\#\mathcal{N}_\lambda^{\mathcal{Q}_l}}\sum_{N\in\mathcal{N}_\lambda^{\mathcal{Q}_l,\mathrm{N}}}\left[\prod_{j=1}^d\left(C_j^{2,\mathrm{N}}(z,z^\gamma)+C_j^{1,\mathrm{N}}(z,z^\gamma)\right)-\prod_{j=1}^dC_j^{2,\mathrm{N}}(z,z^\gamma)\right]\\
			&=\frac{1}{\#\mathcal{N}_\lambda^{\mathcal{Q}_l}}\sum_{N\in\mathcal{N}_\lambda^{\mathcal{Q}_l,\mathrm{N}}}\sum_{\substack{1\leq i\leq 2\\|i|<2d}}\prod_{j=1}^dC_j^{i_j,\mathrm{N}}(z,z^\gamma).
		\end{aligned}
	\end{equation} 
	Since we are only interested in bounding this function, we can forget about the sign of each term and the proof works the same for Dirichlet and Neumann conditions.
	
\textbf{	\underline{Step II:}}	We first fix $(z,z^\gamma)\in\mathcal{R}^{2d}$ and integrate over $\mathcal{Q}_l$. So with $z,z^\gamma\in [-R,R]^d$ fixed, recalling \eqref{DE} (and \eqref{NE} for the Neumann case, where only the sign changes), we get the following
	
	\begin{equation}
		\begin{aligned}
			&\int_{\mathcal{Q}_l}\left(\mathcal{E}_{\lambda,\mathcal{Q}_l,\mathrm{D}}^{z^0}(z,z^\gamma)\right)^2dz^0=\int_{\mathcal{Q}_l}\left(\sum_{N\in\mathcal{N}_\lambda^{\mathcal{Q}_l,\mathrm{D}}}\sum_{\substack{1\leq i\leq 2\\|i|<2d}}\frac{(-1)^{|i|}}{\#\mathcal{N}_\lambda^{\mathcal{Q}_l}}\prod_{j=1}^dC_j^{i_j,\mathrm{N}}(z,z^\gamma)\right)^2dz^0\\
			&=\sum_{\substack{N\in\mathcal{N}_\lambda^{\mathcal{Q}_l,\mathrm{D}}\\M\in\mathcal{N}_\lambda^{\mathcal{Q}_l,\mathrm{D}}}}\sum_{\substack{1\leq i,k\leq 2\\|i|,|k|<2d}}\frac{(-1)^{|i|+|k|}}{\left(\#\mathcal{N}_\lambda^{\mathcal{Q}_l}\right)^2} \int_{\mathcal{Q}_l}\left(\prod_{j=1}^dC_j^{i_j,\mathrm{N}}(z,z^\gamma)\right)\left(\prod_{j=1}^dC_j^{k_j,M}(z,z^\gamma)\right)dz^0\\
			&=\sum_{\substack{N\in\mathcal{N}_\lambda^{\mathcal{Q}_l,\mathrm{D}}\\M\in\mathcal{N}_\lambda^{\mathcal{Q}_l,\mathrm{D}}}}\sum_{\substack{1\leq i,k\leq 2\\|i|,|k|<2d}} \frac{(-1)^{|i|+|k|}}{\left(\#\mathcal{N}_\lambda^{\mathcal{Q}_l}\right)^2}\prod_{j=1}^d\left(\int_{0}^{l_j}C_j^{i_j,\mathrm{N}}(z,z^\gamma)C_j^{k_j,M}(z,z^\gamma)dz^0\right).
		\end{aligned}
	\end{equation}We study now each $\int_{0}^{l_j}C_j^{i_j,\mathrm{N}}(z,z^\gamma)C_j^{k_j,M}(z,z^\gamma)dz^0$ separately. There are three different cases:
	\begin{itemize}
		\item None of the terms depend on $z^0$, $(i_j,k_j)=(2,2)$, \begin{equation}
			\begin{aligned}
				&\int_{0}^{l_j}C_j^{2,\mathrm{N}}(z,z^\gamma)C_j^{2,M}(z,z^\gamma)dz^0=	\\&\int_{0}^{l_j}\cos\left(\frac{N_j\pi}{l_j\sqrt{\lambda}}(z_j-z^\gamma_j)\right)\cos\left(\frac{M_j\pi}{l_j\sqrt{\lambda}}(z_j-z^\gamma_j)\right)dz^0\\
				&=	l_j\cdot \cos\left(\frac{N_j\pi}{l_j\sqrt{\lambda}}(z_j-z^\gamma_j)\right)\cos\left(\frac{M_j\pi}{l_j\sqrt{\lambda}}(z_j-z^\gamma_j)\right).
			\end{aligned}
		\end{equation} \item Only one of the terms depends on $z^0$, $(i_k,k_j)=(1,2)$ or $(i_k,k_j)=(2,1)$,\begin{equation}
			\begin{aligned}
				&\int_{0}^{l_j}C_j^{1,\mathrm{N}}(z,z^\gamma)C_j^{2,M}(z,z^\gamma)dz^0=	\\&\int_{0}^{l_j}\cos\left(\frac{2N_j\pi z^0_j}{l_j}+\frac{N_j\pi }{l_j\sqrt{\lambda}}(z_j+z^\gamma_j)\right)C_j^{2,M}(z,z^\gamma)dz^0\\
				&=	\cos\left(\frac{M_j\pi}{l_j\sqrt{\lambda}}(z_j-z^\gamma_j)\right)	\int_{0}^{l_j}\cos\left(\frac{2N_j\pi z^0_j}{l_j}+\frac{N_j\pi }{l_j\sqrt{\lambda}}(z_j+z^\gamma_j)\right)dz^0=0.
			\end{aligned}
		\end{equation}
		\item both terms depend on $z^0$, $(i_k,k_j)=(1,1)$,\begin{equation}
			\begin{aligned}
				&\int_{0}^{l_j}C_j^{1,\mathrm{N}}(z,z^\gamma)C_j^{1,M}(z,z^\gamma)dz^0=\\
				&	\int_{0}^{l_j}\cos\left(\frac{2N_j\pi z^0_j}{l_j}+\frac{N_j\pi }{l_j\sqrt{\lambda}}(z_j+z^\gamma_j)\right)\cos\left(\frac{2M_j\pi z^0_j}{l_j}+\frac{M_j\pi }{l_j\sqrt{\lambda}}(z_j+z^\gamma_j)\right)dz^0\\
				&=\left\{\begin{aligned}
					&\frac{l_j}{2}\text{ if }M_j=N_j,\\
					&0\text{ if }M_j\neq N_j.
				\end{aligned}\right.
			\end{aligned}
		\end{equation}
	\end{itemize}
	We can then rewrite, 
	\begin{equation}
		\begin{aligned}
			&\int_{\mathcal{Q}_l}\left(\mathcal{E}_{\lambda,\mathcal{Q}_l,\mathrm{BC}}^{z^0}(z,z^\gamma)\right)^2dz^0=\\&\frac{1}{\left(\#\mathcal{N}_\lambda^{\mathcal{Q}_l}\right)^2}\sum_{\substack{N,M\\ \in\mathcal{N}_\lambda^{\mathcal{Q}_l,\mathrm{D}}}}\sum_{\substack{1\leq i\leq 2\\|i|<2d}} \prod_{j=1}^d\left(\int_{0}^{l_j}C_j^{i_j,M}(z,z^\gamma)C_j^{i_j,\mathrm{N}}(z,z^\gamma)dz^0\right)\\
			&=\frac{1}{\left(\#\mathcal{N}_\lambda^{\mathcal{Q}_l}\right)^2}\sum_{\substack{N,M\\ \in\mathcal{N}_\lambda^{\mathcal{Q}_l,\mathrm{D}}}}\sum_{\substack{1\leq i\leq 2\\|i|<2d}} \prod_{j=1}^d A_j^{i_j,N,M}(z,z^\gamma),
		\end{aligned}
	\end{equation}where we have computed $A_j^{i_j,N,M}(z,z^\gamma)$ as
	\begin{equation}
		A_j^{i_j,N,M}(z,z^\gamma)=\left\{\begin{aligned}
			&\frac{l_j}{2}\delta_{N_j,M_j},\text{ if }i_j=1\\
			&l_j\cdot\cos\left(\frac{N_j\pi}{l_j\sqrt{\lambda}}(z_j-z^\gamma_j)\right)\cos\left(\frac{M_j\pi}{l_j\sqrt{\lambda}}(z_j-z^\gamma_j)\right),\text{ if }i_j=2,
		\end{aligned}
		\right.
	\end{equation}
	where $\delta_{\cdot,\cdot}$ is the Kronecker delta. We now notice that if $N_j=M_j$ for all $1\leq j\leq d$ except for one, obviously it is true that $N=M$.
	Then we can split the integral in two parts:
    
	\begin{equation}
		\begin{aligned}
			&\int_{\mathcal{Q}_l}\left(\mathcal{E}_{\lambda,\mathcal{Q}_l,\mathrm{BC}}^{z^0}(z,z^\gamma)\right)^2dz^0=I(z,z^\gamma)+II(z,z^\gamma)=\\
			&=\frac{1}{\left(\#\mathcal{N}_\lambda^{\mathcal{Q}_l}\right)^2}\sum_{N\in\mathcal{N}_\lambda^{\mathcal{Q}_l,\mathrm{D}}}\sum_{\substack{1\leq i\leq 2\\|i|=d, d+1}} \prod_{j=1}^d A_j^{i_j,N,\mathrm{N}}(z,z^\gamma)+\\&\frac{1}{\left(\#\mathcal{N}_\lambda^{\mathcal{Q}_l}\right)^2}\sum_{\substack{N,M\\ \in\mathcal{N}_\lambda^{\mathcal{Q}_l,\mathrm{D}}}}\sum_{\substack{1\leq i\leq 2\\d+1<|i|<2d}} \prod_{j=1}^d A_j^{i_j,N,M}(z,z^\gamma),
		\end{aligned}
	\end{equation}
	
\textbf{	\underline{Step III:}}
	The next step is to integrate over $\mathcal{R}^{2d}$. We proceed to study the first term  $\int_{\mathcal{R}^{2d}}I(z,z^\gamma)dzdz^\gamma$.
	\begin{equation}
		\begin{aligned}
			&\int_{\mathcal{R}^{2d}}I(z,z^\gamma)dzdz^\gamma=\\&\frac{1}{\left(\#\mathcal{N}_\lambda^{\mathcal{Q}_l}\right)^2}\sum_{N\in\mathcal{N}_\lambda^{\mathcal{Q}_l,\mathrm{D}}}\sum_{\substack{1\leq i\leq 2\\|i|=d,d+1}} \prod_{j=1}^d \int_{-R}^R\int_{-R}^RA_j^{i_j,N,\mathrm{N}}(z,z^\gamma)dz_jdz^\gamma_j\\&=\frac{1}{\left(\#\mathcal{N}_\lambda^{\mathcal{Q}_l}\right)^2}\sum_{N\in\mathcal{N}_\lambda^{\mathcal{Q}_l,\mathrm{D}}}\sum_{\substack{1\leq i\leq 2\\|i|=d,d+1}} \prod_{j=1}^d B_j^{i_j,\mathrm{N}}.
		\end{aligned}
	\end{equation}
	Here we are defining \begin{equation}
		\begin{aligned}
			&B_j^{i_j,\mathrm{N}}=\int_{-R}^R\int_{-R}^RA_j^{i_j,N,\mathrm{N}}(z,z^\gamma)d(z_j,z^\gamma_j)=\\&\left\{\begin{aligned}
				&\frac{l_j (2R)^2}{2},\text{ if }i_j=1\\
				&l_j\int_{-R}^R\int_{-R}^R\cos^2\left(\frac{N_j\pi}{l_j\sqrt{\lambda}}(z_j-z^\gamma_j)\right),i_j=2
			\end{aligned}\right.\\&=\left\{\begin{aligned}
				&\frac{l_j \cdot (2R)^2}{2},\text{ if }i_j=1\\
				&2l_j R^2+\frac{l_j^3\lambda}{4N_j^2\pi^2}\left(1-\cos\left(\frac{4RN_j\pi}{l_j\sqrt{\lambda}}\right)\right),i_j=2.
			\end{aligned}\right.
		\end{aligned}
	\end{equation}
	We are only interested in bounding this integral, not in calculating it, so we can use an explicit error for the Taylor expansion of $\cos(x)$ at $x=0$ to have 
	\begin{equation}
		\begin{aligned}
			&B_j^{i_j,\mathrm{N}}\leq\left\{\begin{aligned}
				&\frac{l_j \cdot (2R)^2}{2},\text{ if }i_j=1\\
				&2l_j R^2+\frac{l_j^3\lambda}{4N_j^2\pi^2}\frac{1}{2}\left(\frac{4RN_j\pi}{l_j\lambda}\right)^2,\text{ if }i_j=2.
			\end{aligned}\right.=\left\{\begin{aligned}
				&\frac{l_j \cdot (2R)^2}{2},\text{ if }i_j=1\\
				&4l_jR^2,\text{ if }i_j=2.
			\end{aligned}\right.
		\end{aligned}
	\end{equation}
	Let's call $\mathcal{A}=\prod_{j=1}^dl_j$, then 
	\begin{equation}
		\begin{aligned}
			&\int_{\mathcal{R}^{2d}}I(z,z^\gamma)dzdz^\gamma=\frac{1}{\left(\#\mathcal{N}_\lambda^{\mathcal{Q}_l}\right)^2}\sum_{N\in\mathcal{N}_\lambda^{\mathcal{Q}_l,\mathrm{D}}}\sum_{\substack{1\leq i\leq 2\\|i|=d,d+1}} \prod_{j=1}^d B_j^{i_j,\mathrm{N}}.\\
			&\leq \frac{1}{\left(\#\mathcal{N}_\lambda^{\mathcal{Q}_l}\right)^2}\sum_{N\in\mathcal{N}_\lambda^{\mathcal{Q}_l,\mathrm{D}}}\sum_{\substack{1\leq i\leq 2\\|i|=d,d+1}} \prod_{j=1}^d 4l_j R^2=\\&\frac{1}{\mathcal{N}_\lambda^2}\sum_{N\in\mathcal{N}_\lambda^{\mathcal{Q}_l,\mathrm{D}}}\sum_{\substack{1\leq i\leq 2\\|i|=d,d+1}} (2R)^{2d} \mathcal{A}=\frac{C\cdot R^{2d}}{\mathcal{N}_\lambda}.
		\end{aligned}
	\end{equation}
	
\textbf{	\underline{Step IV:}}	We now study $\int_{\mathcal{R}^{2d}}II(z,z^\gamma)dzdz^\gamma$. Note that this term does not exist in dimension $d=2$ since there is no natural number $|i|$ such that $3=d+1<|i|<2d=4$. First, we are going to bound the function inside the integral by noting that 
	\begin{equation}
		A_j^{i_j,N,M}(z,z^\gamma)\leq\left\{\begin{aligned}
			&\frac{l_j}{2}\delta_{N_j,M_j},\text{ if }i_j=1\\
			&l_j,\text{ if }i_j=2.
		\end{aligned}
		\right.
	\end{equation}
	Since $d+1<|i|<2d$ we have $i$ with $2,3,4,\ldots,d-1$ entries $i_j=2$, so we write
	\begin{equation}\label{secondterm}
		\begin{aligned}
			&II(z,z^\gamma)=\frac{1}{\left(\#\mathcal{N}_\lambda^{\mathcal{Q}_l}\right)^2}\sum_{N\in\mathcal{N}_\lambda^{\mathcal{Q}_l,\mathrm{D}}}\sum_{M\in\mathcal{N}_\lambda^{\mathcal{Q}_l,\mathrm{D}}}\sum_{\substack{1\leq i\leq 2\\d+1<|i|<2d}} \prod_{j=1}^d A_j^{i_j,N,M}(z,z^\gamma)\\
			&=\frac{1}{\left(\#\mathcal{N}_\lambda^{\mathcal{Q}_l}\right)^2}\sum_{M\in\mathcal{N}_\lambda^{\mathcal{Q}_l,\mathrm{D}}}\sum_{k=d+2}^{2d-1}\sum_{\substack{1\leq i\leq 2\\|i|=k}}\left[ \sum_{N\in\mathcal{N}_\lambda^{\mathcal{Q}_l,\mathrm{D}}}\prod_{j=1}^d A_j^{i_j,N,M}(z,z^\gamma)\right].	\end{aligned}
	\end{equation}
	Paying attention only to $\sum_{M\in\mathcal{N}_\lambda^{\mathcal{Q}_l,\mathrm{D}}}\prod_{j=1}^d A_j^{i_j,N,M}(z,z^\gamma)$ for a fixed $N$ we see that
	\begin{equation}
		\sum_{M\in\mathcal{N}_\lambda^{\mathcal{Q}_l,\mathrm{D}}}\prod_{j=1}^d A_j^{i_j,N,M}(z,z^\gamma)\leq\mathcal{A}\sum_{\substack{M\in\mathcal{N}_\lambda^{\mathcal{Q}_l,\mathrm{D}}\\M_j=N_j\text{ if }i_j=1}}1.
	\end{equation}And applying this into \eqref{secondterm}, it is easy to see that
	\begin{equation}
		\begin{aligned}
			&II(z,z^\gamma)\leq\frac{\mathcal{A}}{\left(\#\mathcal{N}_\lambda^{\mathcal{Q}_l}\right)^2}\sum_{N\in\mathcal{N}_\lambda^{\mathcal{Q}_l,\mathrm{D}}}\sum_{k=d+2}^{2d-1}\sum_{\substack{1\leq i\leq 2\\|i|=k}}\left[ \sum_{\substack{M\in\mathcal{N}_\lambda^{\mathcal{Q}_l,\mathrm{D}}\\N_j=M_j\text{ if }i_j=1}}1\right].
		\end{aligned}
	\end{equation}
	To continue with the bounding, we define the following positive definite integral diagonal quadratic forms: 

		 For $|i|=d+2,$ with $i_{j_1}=i_{j_2}=2,$ we define \begin{equation}
		 	\mathcal{Q}^i_2(k_1,k_2)=q_{j_1}\frac{p}{p_{j_1}}k_1^2+q_{j_2}\frac{p}{p_{j_2}}k_2^2;
		 \end{equation} for $|i|=d+3,$ with $i_{j_1}=i_{j_2}=i_{j_3}=2,$ we define \begin{equation}
		 \mathcal{Q}^i_3(k_1,k_2, k_3)=q_{j_1}\frac{p}{p_{j_1}}k_1^2+q_{j_2}\frac{p}{p_{j_2}}k_2^2+q_{j_3}\frac{p}{p_{j_3}}k_3^2;
		 \end{equation}and we continue like this, until for $|i|=2d-1,$ with $i_{j}=1,$ we have \begin{equation}\begin{aligned}
		 	\mathcal{Q}^i_{d-1}(k_1,\ldots,k_{j-1},k_{j+1},\ldots, k_d)=&q_{1}\frac{p}{p_{1}}k_1^2+\ldots+q_{j-1}\frac{p}{p_{j-1}}k_{j-1}^2+\\&q_{j+1}\frac{p}{p_{j+1}}k_{j+1}^2+\ldots+q_{d}\frac{p}{p_{d}}k_d^2.
		 \end{aligned}
	\end{equation}
	Note that each of them satisfy  $\mathcal{Q}^i_{|i|-d}$ and call \begin{equation}
		r^{\mathcal{Q},i}_{|i|-d}(N)=\left|\left\{(k_1,\ldots,k_{|i|-d})\in\mathbb{N}^{|i|-d}, \mathcal{Q}^i_{|i|-d}(k_1,\ldots,k_{|i|-d})=N\right\}\right|,
	\end{equation} so we have\begin{equation}
		\begin{aligned}
			&II(z,z^\gamma)\leq\frac{\mathcal{A}}{\left(\#\mathcal{N}_\lambda^{\mathcal{Q}_l}\right)^2}\sum_{N,k,i}r^{\mathcal{Q},i}_{k-d}\left(\lambda-q_{j_1}\frac{p}{p_{j_1}}N_{j_1}^2-\ldots-q_{j_{2d-k}}\frac{p}{p_{j_{2d-k}}}N_{j_{2d-k}}^2\right),
		\end{aligned}
	\end{equation}where $\sum_{N,k,i}$ represents the following:\begin{equation}\sum_{N,k,i}=\sum_{N\in\mathcal{N}_\lambda^{\mathcal{Q}_l,\mathrm{D}}}\sum_{k=d+2}^{2d-1}\sum_{\substack{1\leq i\leq 2\\|i|=k}}.\end{equation}For the next step we introduce the following lemma:
	
	 \begin{lemma}\label{algebralemma}
		Let $Q\left(k_1, \ldots, k_d\right)$ be a positive definite integral diagonal quadratic form in $d \geq 2$ variables.
		The number of integral representations $Q\left(k_1, \ldots, k_d\right)=n$ satisfies, for all $\delta>0$ small enough,
			\begin{equation}
			r_Q(n)\leq C_{d,\delta} n^{d / 2-1+\delta} .
		\end{equation}
		
			The constant depends only on $d$ and $\ep$, so it is independent of the actual coefficients of $Q$ and of $n$.
	\end{lemma}
	\begin{proof}
		
		We proceed by induction on $d$. The case $d=2$ is well known and based on limiting bounds of the divisors function: for all $\delta>0$,
		
		\begin{equation}
			r^2_Q(n)\leq C_{2,\delta} n^{\delta} .
		\end{equation} 
		The induction hypothesis is for all $\delta>0$,
		
		\begin{equation}
			r_Q^{d-1}(n)\leq C_{d-1,\delta} n^{(d-1) / 2-1+\delta} . 
		\end{equation}
		By changing coordinates we can assume that $Q(z_1,\ldots,k_d)=a_1k_1^2+\ldots+a_dk_d^2$ with $0\leq a_1\leq a_2\leq \ldots \leq a_d$. Since $\det(Q)=a_1\cdot\ldots\cdot a_d$ it is true that $a_d\geq \det Q^{1/d}$.
		
		Using the hypothesis of induction, the number of representations with $k_d^2=0$ is $\leq C_{d-1,\delta}n^{(d-3)/2+\delta}$, so we focus on the representations with $k_d^2>0$. Notice that \begin{equation}
			k_d^2=\frac{n}{a_d}-\sum_{i=1}^{d-1}\frac{a_i}{a_d}k_i^2\leq \frac{n}{ \det Q^{1/d}}.
		\end{equation} And then $k_d\leq C_Q\sqrt{n}$. It is then obvious that $k_{d-1}\leq C_Q\sqrt{n}$, $\ldots$, $k_{3}\leq C_Q\sqrt{n}$. Then we have $n^{\frac{d-2}{2}}\det Q^{-\frac{d-2}{2d}}$ options for the $(d-2)$-tuple $(k_3,\ldots,k_d)$. Fixing this, we are left with \begin{equation}
			\begin{aligned}
			    &a_1k_1^2+a_2k_2^2=n-\left(\sum_{i=3}^{d}a_ik_i^2\right) \iff\\& a_1k_1^2+a_2k_2^2+e=0,\text{ with }e=\left(\sum_{i=3}^{d}a_ik_i^2\right)-n\leq c_Q n.
			\end{aligned}
		\end{equation}
		Using a \cite[Lemma 8]{Blomer}, we know that the choices for $(k_1,k_2)$ are for any $\delta>0$, at most $\leq C_\delta \det Q^\delta n^\delta$. 
		Finally, \begin{equation}
			r_Q^d(n)\leq C_\delta n^{\frac{d-3}{2}+\delta}+C_\delta n^{\frac{d-2}{2}+\delta}\det Q^{-\frac{d-2}{2d}+\delta}\leq C_{d,\delta}n^{\frac{d}{2}-1+\delta}.
		\end{equation}The last constant is independent of $Q$ because the quadratic form is positive definite and of integer coefficients and, consequently, $\det Q^{-\frac{d-2}{2d}+\delta}$ is bounded by one for any $\delta<\frac{d-2}{2d}$ (recall that this is needed only for $d>2$ so $\frac{d-2}{2d}>0$).
	\end{proof}			
	
	Using this for some $\delta>0$ (to be fixed later), we get \begin{equation}
		\begin{aligned}
			&\frac{\mathcal{A}}{\left(\#\mathcal{N}_\lambda^{\mathcal{Q}_l}\right)^2}\sum_{N\in\mathcal{N}_\lambda^{\mathcal{Q}_l,\mathrm{D}}}\sum_{k=d+2}^{2d-1}\sum_{\substack{1\leq i\leq 2,|i|=k\\ i_{j_1}=\ldots=i_{j_{2d-k}=1}}}r^{\mathcal{Q},i}_{k-d}\left(\lambda-\frac{q_{j_1}p}{p_{j_1}}N_{j_1}^2-\ldots-\frac{q_{j_{2d-k}}p}{p_{j_{2d-k}}}N_{j_{2d-k}}^2\right)\\
			&\leq \frac{\mathcal{A}}{\left(\#\mathcal{N}_\lambda^{\mathcal{Q}_l}\right)^2}\sum_{N\in\mathcal{N}_\lambda^{\mathcal{Q}_l,\mathrm{D}}}\sum_{k=d+2}^{2d-1}\sum_{\substack{1\leq i\leq 2,|i|=k\\ i_{j_1}=\ldots=i_{j_{2d-k}=1}}}C_{\delta}\left(\lambda-\frac{q_{j_1}p}{p_{j_1}}N_{j_1}^2-\ldots\right)^{\frac{k-d}{2}-1+\delta}\\
			&\leq \frac{\mathcal{A}}{\left(\#\mathcal{N}_\lambda^{\mathcal{Q}_l}\right)^2}\sum_{N\in\mathcal{N}_\lambda^{\mathcal{Q}_l,\mathrm{D}}}\sum_{k=d+2}^{2d-1}\sum_{\substack{1\leq i\leq 2,|i|=k\\ i_{j_1}=\ldots=i_{j_{2d-k}=1}}}C_{\delta}\lambda ^{\frac{k-d}{2}-1+\delta}\\
			&\leq \frac{\mathcal{A}}{\left(\#\mathcal{N}_\lambda^{\mathcal{Q}_l}\right)^2}\sum_{N\in\mathcal{N}_\lambda^{\mathcal{Q}_l,\mathrm{D}}}\sum_{k=d+2}^{2d-1}\binom{d}{2d-k}C_{\delta}\lambda^{\frac{k-d}{2}-1+\delta}\leq \frac{\mathcal{A}}{\#\mathcal{N}_\lambda^{\mathcal{Q}_l}}C_\delta\lambda^{\frac{d-3}{2}+\delta}.
		\end{aligned}
	\end{equation}
	The next step is to bound from below, what requires another lemma. 
	\begin{lemma}\label{Iwaniec}
		If we choose $\lambda$ large enough along a sequence that satisfies the  asymptotic equidistribution as in Definition \ref{aequid}, then we have
		\begin{equation}
			\begin{aligned}
				&\#\mathcal{N}_\lambda^{\mathcal{Q}_l}\geq c\cdot \lambda^{d/2-1},\text{ if }d\geq 5,\\
				&\#\mathcal{N}_\lambda^{\mathcal{Q}_l}\geq c\cdot \lambda\log\left(\log(3\lambda)\right)^{-1},\text{ if }d=4,\\
				&\#\mathcal{N}_\lambda^{\mathcal{Q}_l}\geq c_\gamma\cdot \lambda^{1/2-\gamma},\text{ for all }\gamma>0\text{ and }\text{ if }d=3,\\
				&\#\mathcal{N}_\lambda^{\mathcal{Q}_l}\geq c_\gamma\cdot \lambda^{\gamma},\text{ for all }\gamma>0\text{ and }\text{ if }d=2,
			\end{aligned}
		\end{equation}when $\lambda\rightarrow\infty$. In particular, for $d\geq 3$, \begin{equation}
			\#\mathcal{N}_\lambda^{\mathcal{Q}_l}\geq c_{\gamma,\mathrm{D}}\cdot \lambda^{\frac{d-2}{2}-\gamma},\text{ for all }\gamma>0.
		\end{equation}
	\end{lemma}	The proof for $d\geq 3$ appears in the book of Iwaniec \cite[chapter 11]{Iwaniek}. For the case $d=2$, it is easy to see using estimates of the divisor functions that $\#\mathcal{N}_\lambda^{\mathcal{Q}_l}=\mathcal{O}\left(\lambda^\gamma\right)$ for every $\gamma>0$.
	
	Using both lemmas together, for any $\delta,\gamma>0$,
	\begin{equation}
		\begin{aligned}
			II(z,z^\gamma)\leq\frac{\mathcal{A}}{\#\mathcal{N}_\lambda^{\mathcal{Q}_l}}C_\delta\lambda^{\frac{d-3}{2}+\delta}\leq \mathcal{A}C_\delta c_\gamma \lambda^{\frac{d-3}{2}+\delta-\frac{d-2}{2}+\gamma}=\mathcal{A}C_\delta c_\gamma\lambda^{-\frac{1}{2}+\delta+\gamma}.
		\end{aligned}
	\end{equation}We choose $\ep=\delta+\gamma<1/2$.
	It is now trivial to check that 
	\begin{equation}
		\int_{\mathcal{R}^{2d}}II(z,z^\gamma)dzdz^\gamma\leq C_\ep\lambda^{-\frac{1}{2}+\ep}R^{2d}.
	\end{equation}
	
\textbf{	\underline{Step V:}} To conclude, let us put all together. For all $\ep>0$,
	
	\begin{equation}
		\begin{aligned}
			&\int_{\mathcal{Q}_l}\int_{\mathcal{R}^{2d}}\left(\mathcal{E}_{\lambda,\mathcal{Q}_l,\mathrm{BC}}^{z^0}(z,z^\gamma)\right)^2dzdz^\gamma dz^0=\int_{\mathcal{R}^{2d}}I(z,z^\gamma)dzdz^\gamma+\int_{\mathcal{R}^{2d}}II(z,z^\gamma)dzdz^\gamma \\
			&\leq \frac{C\cdot R^{2d}}{\#\mathcal{N}_\lambda^{\mathcal{Q}_l}}+ C_\ep\lambda^{-\frac{1}{2}+\ep}R^{2d}\leq C\cdot R^{2d}\lambda^{\ep-\frac{d-2}{2}}+ C_\ep\lambda^{-\frac{1}{2}+\ep}R^{2d}\leq \frac{C_\ep R^{2d}}{\lambda^{1/2-\ep}},
		\end{aligned}
	\end{equation}as $\lambda\rightarrow\infty$ along a sequence that gives us equidistribution.
	
\end{proof}

\subsubsection{The case of the right isosceles triangle $\mathcal{P}=\mathcal{T}_\mathrm{iso}$}
\begin{proof}
\textbf{	\underline{Step I:}}	To make computations easier, we will consider the error term $\mathcal{E}_{\lambda,\mathcal{T}_\mathrm{iso},\mathrm{BC}}^{z^0}(z,z^\gamma)$ split into two terms: $\mathcal{E}_{\lambda,\mathcal{T}_\mathrm{iso},\mathrm{BC}}^{z^0, (1)}(z,z^\gamma)$ and $\mathcal{E}_{\lambda,\mathcal{T}_\mathrm{iso},\mathrm{BC}}^{z^0, (2)}(z,z^\gamma)$. In this subsection, we will prove lemma \ref{BoundOfNonTrans} for both of them. 
	
	First, notice that
		\begin{equation}
		\int_{\mathcal{T}_\mathrm{iso}}\int_{\mathcal{R}^{4}}F^{(i)}(z,z^\gamma,z^0)d(z,z^\gamma,z^0)\leq	\int_{Q_1}\int_{\mathcal{R}^{4}}F^{(i)}(z,z^\gamma,z^0)d(z,z^\gamma,z^0),
	\end{equation} for $i=1,2$ and \begin{equation}F^{(i)}(z,z^\gamma,z^0)=\left(\mathcal{E}_{\lambda,\mathcal{T}_\mathrm{iso},\mathrm{BC}}^{z^0,(i)}(z,z^\gamma)\right)^2>0.\end{equation} 
	
	 Therefore, it is enough to study the right hand side integral, which is easier because of the symmetry of $Q_1$.
	
	We then	define sets   $O_\lambda^{(1)}$ and $O_\lambda^{(2)}$ as in \eqref{set} for both error terms, choosing  \begin{equation}
		\tilde{\ep}=\frac{\ep}{24C_l|\mathcal{M}^\mathcal{Q}|L\mathbb{M}}.
	\end{equation} Applying \ref{finalbound} and the later analysis to both sets we get, instead of equation \eqref{error}, 
	\begin{equation}
		\begin{aligned}
			&	\left|	\vp(x)-	C_l|\mathcal{M}^\mathcal{Q}|u_\lambda^{\mathcal{T}_\mathrm{iso},\mathrm{BC}}\left(z^0+\frac{z}{\sqrt{\lambda}}\right)\right|\leq\frac{2\ep}{3}+	C_l|\mathcal{M}^\mathcal{Q}|\sum_{\gamma=1}^{L}|c_\gamma|4\left|\mathcal{E}_{\lambda,\mathcal{T}_\mathrm{iso},\mathrm{BC}}^{z^0,(1)}(z,z^\gamma)\right|\\
			&+C_l|\mathcal{M}^\mathcal{Q}|\sum_{\gamma=1}^{L}|c_\gamma|4\left|\mathcal{E}_{\lambda,\mathcal{T}_\mathrm{iso},\mathrm{BC}}^{z^0,(2)}(z,z^\gamma)\right|\leq\frac{2\ep}{3}+	C_l|\mathcal{M}^\mathcal{Q}|L\mathbb{M}\left\|\mathcal{E}_{\lambda,\mathcal{T}_\mathrm{iso},\mathrm{BC}}^{z^0,(1)}\right\|_{\mathcal{R}^{4}}^{\infty}\\&+C_l|\mathcal{M}^\mathcal{Q}|L\mathbb{M}\left\|\mathcal{E}_{\lambda,\mathcal{T}_\mathrm{iso},\mathrm{BC}}^{z^0,(2)}\right\|_{\mathcal{R}^{4}}^{\infty}<\frac{2\ep}{3}+\frac{\ep}{6}+\frac{\ep}{6}=\ep,
		\end{aligned}
	\end{equation}
	for any $x\in B_R$ and any $z^0\in O_\lambda^{(1)}\cap O_\lambda^{(2)}$. 
	
	It is also easy to notice that
	\begin{equation}
		\begin{aligned}
			&\left|O_\lambda^{(1)}\cap O_\lambda^{(2)}\right|=1-\left|\mathcal{T}_\mathrm{iso}\backslash\left(O_\lambda^{(1)}\cap O_\lambda^{(2)}\right)\right|\geq1-\left(\left|\mathbb{T_{IR}}\backslash O_\lambda^{(1)}\right|+\left|\mathbb{T_{IR}}\backslash O_\lambda^{(2)}\right|\right) \\&\geq1-\left(\left\{z^0\in \mathcal{P},\left\|\mathcal{E}_{\lambda,\mathcal{T}_\mathrm{iso},\mathrm{BC}}^{z^0,(1)}\right\|_{\mathcal{R}^{4}}^{\infty}\geq\frac{\tilde{\ep}}{3} \right\}+\left\{z^0\in \mathcal{P},\left\|\mathcal{E}_{\lambda,\mathbb{T}_ {IR},\mathrm{BC}}^{z^0,(2)}\right\|_{\mathcal{R}^{4}}^{\infty}\geq\frac{\tilde{\ep}}{3} \right\}\right)\\&\geq1-\left(\frac{	\int_{\mathcal{P}}\left\|\mathcal{E}_{\lambda,\mathcal{T}_\mathrm{iso},\mathrm{BC}}^{z^0,(1)}\right\|_{\mathcal{R}^{4}}^{\infty} d z^0}{\tilde{\ep}/3}+\frac{	\int_{\mathcal{P}}\left\|\mathcal{E}_{\lambda,\mathcal{T}_\mathrm{iso},\mathrm{BC}}^{z^0,(2)}\right\|_{\mathcal{R}^{4}}^{\infty} d z^0}{\tilde{\ep}/3}\right)\geq 1-3C \frac{R^{2}}{\tilde{\ep}\lambda^{\frac{1}{4d+12}}}.
		\end{aligned}
	\end{equation}
	
	The first term  $\mathcal{E}_{\lambda,\mathcal{T}_\mathrm{iso},\mathrm{BC}}^{z^0, (1)}(z,z^\gamma)$ is defined to be a particular case of the previously introduced $\mathcal{E}_{\lambda,\mathcal{Q}_l,\mathrm{BC}}^{z^0}(z,z^\gamma)$, whenever $d=2$ and $l=1$. Using the same notation as in Subsection \ref{QlNF},
	we have
	
	\begin{equation}
		\begin{aligned}
			&\mathcal{E}_{\lambda,\mathcal{T}_\mathrm{iso},\mathrm{D}}^{z^0,(1)}(z,z^\gamma)
			=\frac{1}{\#\mathcal{N}_\lambda^{\mathcal{T}_\mathrm{iso}}}\sum_{N\in\mathcal{N}_\lambda^{\mathcal{T}_\mathrm{iso},\mathrm{D}}}\sum_{\substack{1\leq i\leq 2\\|i|=2,3}}(-1)^{|i|}c_1^{i_1,\mathrm{N}}(z,z^\gamma)c_2^{i_2,\mathrm{N}}(z,z^\gamma)\\&=\frac{1}{\#\mathcal{N}_\lambda^{\mathcal{T}_\mathrm{iso}}}\sum_{N\in\mathcal{N}_\lambda^{\mathcal{T}_\mathrm{iso},\mathrm{D}}}\left(c_1^{1,\mathrm{N}}c_2^{1,\mathrm{N}}-c_1^{1,\mathrm{N}}c_2^{2,\mathrm{N}}-c_1^{2,\mathrm{N}}c_2^{1,\mathrm{N}}\right)(z,z^\gamma),
		\end{aligned}
	\end{equation}for the Dirichlet case, and, in the Neumann case, 
	\begin{equation}
		\begin{aligned}
			&\mathcal{E}_{\lambda,\mathcal{T}_\mathrm{iso},\mathrm{N}}^{z^0,(1)}(z,z^\gamma)
			=\frac{1}{\#\mathcal{N}_\lambda^{\mathcal{T}_\mathrm{iso}}}\sum_{N\in\mathcal{N}_\lambda^{\mathcal{T}_\mathrm{iso},\mathrm{N}}}\sum_{\substack{1\leq i\leq 2\\|i|=2,3}}c_1^{i_1,\mathrm{N}}(z,z^\gamma)c_2^{i_2,\mathrm{N}}(z,z^\gamma)\\&=\frac{1}{\#\mathcal{N}_\lambda^{\mathcal{T}_\mathrm{iso}}}\sum_{N\in\mathcal{N}_\lambda^{\mathcal{T}_\mathrm{iso},\mathrm{N}}}\left(c_1^{1,\mathrm{N}}c_2^{1,\mathrm{N}}+c_1^{1,\mathrm{N}}c_2^{2,\mathrm{N}}+c_1^{2,\mathrm{N}}c_2^{1,\mathrm{N}}\right)(z,z^\gamma).
		\end{aligned}
	\end{equation}
	Therefore, Lemma \ref{BoundOfNonTrans} in this case follows form the previous subsection (Steps II and III).
	
\textbf{	\underline{Step II:} }For the second term, we introduce a similar notation:
	\begin{equation}
		\begin{aligned}
			&	P_1^{1,\mathrm{N}}(z,z^\gamma,z^0)=\cos\left(z_1^0\pi(n-m)+\frac{\pi nz_1^\gamma-\pi mz_1}{\sqrt{\lambda}}\right) \\
			&	P_1^{2,\mathrm{N}}(z,z^\gamma,z^0)=\cos\left(z_1^0\pi(n+m)+\frac{\pi nz_1^\gamma+\pi mz_1}{\sqrt{\lambda}}\right) \\
			&	P_2^{1,\mathrm{N}}(z,z^\gamma,z^0)= \cos\left(z_2^0\pi(m-n)+\frac{\pi mz_2^\gamma-\pi nz_2}{\sqrt{\lambda}}\right)\\
			&	P_2^{2,\mathrm{N}}(z,z^\gamma,z^0)=\cos\left(z_2^0\pi(m+n)+\frac{\pi mz_2^\gamma+\pi nz_2}{\sqrt{\lambda}}\right) .
		\end{aligned}
	\end{equation} 
	So then, in the Dirichlet case, it is true that
	\begin{equation}
		\begin{aligned}
			&\mathcal{E}_{\lambda,\mathcal{T}_\mathrm{iso},\mathrm{D}}^{z^0,(2)}(z,z^\gamma)
			=\frac{1}{\#\mathcal{N}_\lambda^{\mathcal{T}_\mathrm{iso}}}\sum_{N\in\mathcal{N}_\lambda^{\mathcal{T}_\mathrm{iso},\mathrm{D}}}\sum_{1\leq i,j\leq 2}(-1)^{i+j}P_1^{i,\mathrm{N}}(z,z^\gamma,z^0)P_2^{j,\mathrm{N}}(z,z^\gamma,z^0)\\&=\frac{1}{\#\mathcal{N}_\lambda^{\mathcal{T}_\mathrm{iso}}}\sum_{N\in\mathcal{N}_\lambda^{\mathcal{T}_\mathrm{iso},\mathrm{D}}}\left(P_1^{1,\mathrm{N}}P_2^{1,\mathrm{N}}+P_1^{2,\mathrm{N}}P_2^{2,\mathrm{N}}-P_1^{1,\mathrm{N}}P_2^{2,\mathrm{N}}-P_1^{2,\mathrm{N}}P_2^{1,\mathrm{N}}\right)(z,z^\gamma,z^0),
		\end{aligned}
	\end{equation} while for the Neumann case we get, 
	\begin{equation}
		\begin{aligned}
			&\mathcal{E}_{\lambda,\mathcal{T}_\mathrm{iso},\mathrm{N}}^{z^0,(2)}(z,z^\gamma)
			=\frac{1}{\#\mathcal{N}_\lambda^{\mathcal{T}_\mathrm{iso}}}\sum_{N\in\mathcal{N}_\lambda^{\mathcal{T}_\mathrm{iso},\mathrm{N}}}\sum_{1\leq i,j\leq 2}P_1^{i,\mathrm{N}}(z,z^\gamma,z^0)P_2^{j,\mathrm{N}}(z,z^\gamma,z^0)\\&=\frac{1}{\#\mathcal{N}_\lambda^{\mathcal{T}_\mathrm{iso}}}\sum_{N\in\mathcal{N}_\lambda^{\mathcal{T}_\mathrm{iso},\mathrm{N}}}\left(P_1^{1,\mathrm{N}}P_2^{1,\mathrm{N}}+P_1^{2,\mathrm{N}}P_2^{2,\mathrm{N}}+P_1^{1,\mathrm{N}}P_2^{2,\mathrm{N}}+P_1^{2,\mathrm{N}}P_2^{1,\mathrm{N}}\right)(z,z^\gamma,z^0),
		\end{aligned}
	\end{equation} 
	We fix $(z,z^\gamma)\in \mathcal{R}^4$ and integrate over $Q_1$ (we do it in the Dirichlet case, but the Neumann case is analogous),
	
	\begin{equation}
		\begin{aligned}
			&\int_{Q_1}\left(\mathcal{E}_{\lambda,\mathcal{T}_\mathrm{iso},\mathrm{D}}^{z^0,(2)}(z,z^\gamma)\right)^2dz^0=\int_{Q_1}\left(\frac{1}{\#\mathcal{N}_\lambda^{\mathcal{T}_\mathrm{iso}}}\sum_{N\in\mathcal{N}_\lambda^{\mathcal{T}_\mathrm{iso},\mathrm{D}}}\sum_{1\leq i,j\leq 2}(-1)^{i+j}P_1^{i,\mathrm{N}}P_2^{j,\mathrm{N}}\right)^2dz^0\\
			&=\frac{1}{\left(\#\mathcal{N}_\lambda^{\mathcal{T}_\mathrm{iso}}\right)^2}\sum_{N,M\in\mathcal{N}_\lambda^{\mathcal{T}_\mathrm{iso},\mathrm{D}}}\sum_{\substack{1\leq i,j\leq 2\\1\leq k,l\leq 2}}(-1)^{i+j+k+l} \int_{Q_1}\left(P_1^{i,\mathrm{N}}P_2^{j,\mathrm{N}}\right)\left(P_1^{k,M}P_2^{l,M}\right)dz^0\\
			&=\frac{1}{\left(\#\mathcal{N}_\lambda^{\mathcal{T}_\mathrm{iso}}\right)^2}\sum_{N,M\in\mathcal{N}_\lambda^{\mathcal{T}_\mathrm{iso},\mathrm{D}}}\sum_{\substack{1\leq i,j\leq 2\\1\leq k,l\leq 2}}(-1)^{i+j+k+l} \int_{0}^1\left(P_1^{i,\mathrm{N}}P_1^{k,M}\right)dz_1^0\int_0^1\left(P_2^{j,\mathrm{N}}P_2^{l,M}\right)dz_2^0.
		\end{aligned}
	\end{equation}
    We study now each possible combination of $\int_{0}^{1}\left(P_j^{i,\mathrm{N}}P_j^{k,M}\right)dz_j^0$, $j,i,k=1,2$. Notice that all of them are of the form \begin{equation}
		\int_{0}^{1}\left(P_j^{i,\mathrm{N}}P_j^{k,M}\right)dz_j^0=\int_0^1\cos\left(\pi \alpha x+c_1\right)\cos\left(\pi \beta x+c_2\right)dz,
	\end{equation} with constants $\alpha,\beta\in\mathbb{Z}$ and $ c_1,c_2\in \mathbb{R}$. This kind of integral is known to be
	\begin{equation}
		\begin{aligned}
			&	\int_0^1\cos\left(\pi \alpha x+c_1\right)\cos\left(\pi \beta x+c_2\right)dz=\\&\frac{\sin\left(c_1-c_2+\pi(\alpha-\beta)\right)-\sin\left(c_1-c_2\right)}{2\pi(\alpha-\beta)}+\frac{\sin\left(c_1+c_2+\pi(\alpha+\beta)\right)-\sin\left(c_1+c_2\right)}{2\pi(\alpha+\beta)},
		\end{aligned}
	\end{equation}if $\left|\alpha\right|\neq\left|\beta\right|$,
	\begin{equation}
		\begin{aligned}
			\int_0^1\cos\left(\pi \alpha x+c_1\right)\cos\left(\pi \beta x+c_2\right)dz=\frac{\sin\left(\pi \alpha \right)\cos\left(c_1+c_2+\pi \alpha\right)}{2\pi \alpha}+\frac{\cos\left(c_1-c_2\right)}{2},
		\end{aligned}
	\end{equation}if $\alpha=\beta$ and, finally, 
	\begin{equation}
		\begin{aligned}
			\int_0^1\cos\left(\pi \alpha x+c_1\right)\cos\left(\pi \beta x+c_2\right)dz=\frac{\sin\left(\pi \alpha\right)\cos\left(c_1-c_2+\pi \alpha\right)}{2\pi \alpha}+\frac{\cos\left(c_1+c_2\right)}{2},
		\end{aligned}
	\end{equation}if $\alpha=-\beta$.
	
	By definition, $\alpha,\beta\equiv 0\mod(2)$ if $m,n$ (resp. $a,b$) have the same parity. As seen in \cite{CillDisk}, the sequence of $\lambda$ along which we have asymptotic equidistribution runs along even values of $\lambda$, being $m\equiv n \mod(2)$ (resp. $a\equiv b\mod(2)$). This simplifies the integral:
	\begin{equation}
		\int_0^1\cos\left(\pi \alpha x+c_1\right)\cos\left(\pi \beta x+c_2\right)dz=\left\{\begin{aligned}
			&0\text{ if }|\alpha|\neq|\beta|,\\
			&\frac{1}{2}\cos\left(c_1-c_2\right)\text{ if }\alpha=\beta\text{ and }\\
			&\frac{1}{2}\cos\left(c_1+c_2\right)\text{ if }\alpha=-\beta.
		\end{aligned}\right.
	\end{equation}
	In particular, paying attention to the definition of $\alpha$ and $\beta$, for $i,j=1,2$ \begin{equation}
		\int_{0}^{1}\left(P_j^{1,\mathrm{N}}P_j^{2,M}\right)dz_j^0=0,
	\end{equation}and also \begin{equation}
		\int_{0}^{1}\left(P_j^{i,\mathrm{N}}P_j^{i,M}\right)dz_j^0\leq \frac{1}{2}\left(\delta(n-b)+\delta(n-a)\right).
	\end{equation}In the end, we get the existence of a constant $C$ such that	\begin{equation}
		\begin{aligned}
			&\int_{Q_1}\left(\mathcal{E}_{\lambda,\mathcal{T}_\mathrm{iso},\mathrm{BC}}^{z^0,(2)}(z,z^\gamma)\right)^2dz^0\leq \frac{1}{\left(\#\mathcal{N}_\lambda^{\mathcal{T}_\mathrm{iso}}\right)^2}\sum_{N\in\mathcal{N}_\lambda^{\mathcal{T}_\mathrm{iso},\mathrm{BC}}}C\leq\frac{C}{\#\mathcal{N}_\lambda^{\mathcal{T}_\mathrm{iso}}}.
		\end{aligned}
	\end{equation}
	Lemma \ref{Iwaniec} gives us, for any $\gamma>0$ and $\lambda$ in the sequence that gives asymptotic equidistribution, the bound	\begin{equation}
		\begin{aligned}
			&\int_{Q_1}\left(\mathcal{E}_{\lambda,\mathcal{T}_\mathrm{iso},\mathrm{BC}}^{z^0,(2)}(z,z^\gamma)\right)^2dz^0\leq\frac{C}{\#\mathcal{N}_\lambda^{\mathcal{T}_\mathrm{iso}}}\leq c_\gamma\lambda^{-\gamma}.
		\end{aligned}
	\end{equation}We now compute the other integral as follows:
	\begin{equation}
		\int_{\mathcal{R}^{4}}\int_{Q_1}\left(\mathcal{E}_{\lambda,\mathcal{T}_\mathrm{iso},\mathrm{BC}}^{z^0,(2)}(z,z^\gamma)\right)^2dz^0d(z^\gamma,x)\leq c_\gamma\lambda^{-\gamma} R^4,
	\end{equation}for any $\gamma>0$, concluding Lemma \ref{BoundOfNonTrans} in this case.
\end{proof}
\subsubsection{The case of the equilateral triangle $\mathcal{P}=	\mathcal{T}_\mathrm{equi}$.}
\begin{proof}
	In this easier case the error function is, for the Neumann case,	\begin{equation}
		\begin{aligned}
			&\mathcal{E}_{\lambda,\mathcal{T}_\mathrm{equi},\mathrm{N}}^{z^0}(z,z^\gamma)
			=\\&\frac{1}{\#\mathcal{N}_\lambda^{\mathcal{T}_\mathrm{equi}}}\sum_{N\in\mathcal{N}_\lambda^{\mathcal{T}_\mathrm{equi},\mathrm{N}}}\cos\left(\frac{2\pi m\left(z_1-z_1^\gamma\right)}{3\sqrt{\lambda}}\right)\cos\left(\frac{2\pi\sqrt{3}n}{3}\left(2z_2^0+\frac{z_2+z_2^\gamma}{\sqrt{\lambda}}\right)\right),
		\end{aligned}
	\end{equation} while for Dirichlet is
	\begin{equation}
		\begin{aligned}
			&\mathcal{E}_{\lambda,\mathcal{T}_\mathrm{equi},\mathrm{D}}^{z^0}(z,z^\gamma)
			=\\&\frac{-1}{\#\mathcal{N}_\lambda^{\mathcal{T}_\mathrm{equi}}}\sum_{N\in\mathcal{N}_\lambda^{\mathcal{T}_\mathrm{equi},\mathrm{D}}}\cos\left(\frac{2\pi m\left(z_1-z_1^\gamma\right)}{3\sqrt{\lambda}}\right)\cos\left(\frac{2\pi\sqrt{3}n}{3}\left(2z_2^0+\frac{z_2+z_2^\gamma}{\sqrt{\lambda}}\right)\right).
		\end{aligned}
	\end{equation} 
	As in the previous case, we note that \begin{equation}
		\int_{	\mathcal{T}_\mathrm{equi}}\left(\mathcal{E}_{\lambda,\mathcal{T}_\mathrm{equi},\mathrm{BC}}^{z^0}(z,z^\gamma)\right)^2dz^0\leq \int_{Q_E}\left(\mathcal{E}_{\lambda,\mathcal{T}_\mathrm{equi},\mathrm{BC}}^{z^0}(z,z^\gamma)\right)^2dz^0,
	\end{equation}where $Q_E=\left\{(z_1,z_2),\ -1/2\leq z_1\leq 1/2,\ 0\leq z_2\leq \sqrt{3}/2\right\}$. We then have, for fixed $(z,z^\gamma)\in\mathcal{R}^4$, 
	\begin{equation}
		\begin{aligned}
			&\int_{Q_E}\left(\mathcal{E}_{\lambda,\mathcal{T}_\mathrm{equi},\mathrm{BC}}^{z^0}(z,z^\gamma)\right)^2dz^0=\\&\frac{1}{\left(\#\mathcal{N}_\lambda^{\mathcal{T}_\mathrm{equi}}\right)^2}\sum_{N,M\in\mathcal{N}_\lambda^{\mathcal{T}_\mathrm{equi},\mathrm{BC}}}\cos\left(\frac{2\pi m\left(z_1-z_1^\gamma\right)}{3\sqrt{\lambda}}\right)\cos\left(\frac{2\pi a\left(z_1-z_1^\gamma\right)}{3\sqrt{\lambda}}\right)\cdot\\&\int_0^{\sqrt{3}/2}\int_{-1/2}^{1/2}\cos\left(\frac{2\pi\sqrt{3}n}{3}\left(2z_2^0+\frac{z_2+z_2^\gamma}{\sqrt{\lambda}}\right)\right)\cos\left(\frac{2\pi\sqrt{3}b}{3}\left(2z_2^0+\frac{z_2+z_2^\gamma}{\sqrt{\lambda}}\right)\right)dz_1^0dz_2^0.\end{aligned}\end{equation}The last integral is $0$ if $n\neq b$ and $\frac{3}{4\sqrt{3}}$ otherwise. Since $3n^2+m^2=3b^2=a^2$, $n=b$ implies that also $a=m$, so we can simplify it as: 
	\begin{equation}
		\begin{aligned}
			&\int_{Q_E}\left(\mathcal{E}_{\lambda,\mathcal{T}_\mathrm{equi},\mathrm{BC}}^{z^0}(z,z^\gamma)\right)^2dz^0=\frac{3}{4\sqrt{3}\left(\#\mathcal{N}_\lambda^{\mathcal{T}_\mathrm{equi}}\right)^2}\sum_{N\in\mathcal{N}_\lambda^{\mathcal{T}_\mathrm{equi},\mathrm{BC}}}\cos\left(\frac{2\pi m\left(z_1-z_1^\gamma\right)}{3\sqrt{\lambda}}\right)^2\\&\leq\frac{3}{4\sqrt{3}\#\mathcal{N}_\lambda^{\mathcal{T}_\mathrm{equi}}}. \end{aligned}\end{equation}
	We conclude like in the last case: applying Lemma \ref{BoundOfNonTrans} and integrating over $\mathcal{R}^4$ we get that, for any $\gamma>0$ and $\lambda$ large enough along the sequence with asymptotic equidistribution, 
	\begin{equation}
		\int_{\mathcal{R}^{4}}\int_{	\mathcal{T}_\mathrm{equi}}\left(\mathcal{E}_{\lambda,\mathcal{T}_\mathrm{equi},\mathrm{BC}}^{z^0}(z,z^\gamma)\right)^2dz^0d(z^\gamma,x)\leq c_\gamma\lambda^{-\gamma} R^4.
	\end{equation}
\end{proof}

\subsubsection{The case of the hemiequilateral triangle $\mathcal{P}=\mathcal{T}_\mathrm{hemi}$.}
\begin{proof}
\textbf{	\underline{Step I:}}	In this final case, the error function is, for Dirichlet conditions, \begin{equation}
		\begin{aligned}
			&\mathcal{E}_{\lambda,\mathcal{T}_\mathrm{hemi},\mathrm{D}}^{z^0}(z,z^\gamma)
			=\frac{1}{\#\mathcal{N}_\lambda^{\mathcal{T}_\mathrm{equi}}}\sum_{N\in\mathcal{N}_\lambda^{\mathcal{T}_\mathrm{equi},\mathrm{D}}}\left(P^{1,\mathrm{N}}_1P^{1,\mathrm{N}}_2-P^{2,\mathrm{N}}_1P^{1,\mathrm{N}}_2-P^{1,\mathrm{N}}_1P^{2,\mathrm{N}}_2\right)(z,z^\gamma,z^0),
		\end{aligned}
	\end{equation} and for Neumann, 
	\begin{equation}
		\begin{aligned}
			&\mathcal{E}_{\lambda,\mathcal{T}_\mathrm{hemi},\mathrm{N}}^{z^0}(z,z^\gamma)
			=\frac{1}{\#\mathcal{N}_\lambda^{\mathcal{T}_\mathrm{equi}}}\sum_{N\in\mathcal{N}_\lambda^{\mathcal{T}_\mathrm{equi},\mathrm{N}}}\left(P^{1,\mathrm{N}}_1P^{1,\mathrm{N}}_2+P^{2,\mathrm{N}}_1P^{1,\mathrm{N}}_2+P^{1,\mathrm{N}}_1P^{2,\mathrm{N}}_2\right)(z,z^\gamma,z^0),
		\end{aligned}
	\end{equation}where now we define
	\begin{equation}
		\begin{aligned}
			&P_1^{1,\mathrm{N}}(z,z^\gamma,z^0)= \cos\left(\frac{2\pi m}{3}\left(2z_1^0+\frac{z_1+z_1^\gamma}{\sqrt{\lambda}}\right)\right) ,\\
			&P_1^{2,\mathrm{N}}(z,z^\gamma,z^0)= \cos\left(\frac{2\pi m\left(z_1-z_1^\gamma\right)}{3\sqrt{\lambda}}\right) ,\\
			&P_2^{1,\mathrm{N}}(z,z^\gamma,z^0)= \cos\left(\frac{2\pi\sqrt{3}n}{3}\left(2z_2^0+\frac{z_2+z_2^\gamma}{\sqrt{\lambda}}\right)\right) ,\\
			&P_2^{2,\mathrm{N}}(z,z^\gamma,z^0)=\cos\left(\frac{2\pi\sqrt{3}n\left(z_2-z_2^\gamma\right)}{3\sqrt{\lambda}}\right)  .
		\end{aligned}
	\end{equation} 
	One more time, we notice that  \begin{equation}
		\int_{\mathcal{T}_\mathrm{hemi}}\left(\mathcal{E}_{\lambda,\mathcal{T}_\mathrm{hemi},\mathrm{BC}}^{z^0}(z,z^\gamma)\right)^2dz^0\leq \int_{Q_{HE}}\left(\mathcal{E}_{\lambda,\mathcal{T}_\mathrm{hemi},\mathrm{BC}}^{z^0}(z,z^\gamma)\right)^2dz^0,
	\end{equation}where $Q_E=\left\{(z_1,z_2),\ 0\leq z_1\leq 3/2,\ 0\leq z_2\leq \sqrt{3}/2\right\}$. We then have, for fixed $(z,z^\gamma)\in\mathcal{R}^4$ and for Dirichlet (the Neumann case is similar without negative signs), 
	
	\begin{equation}
		\begin{aligned}
			&\int_{Q_{HE}}\left(\mathcal{E}_{\lambda,\mathcal{T}_\mathrm{hemi},\mathrm{D}}^{z^0}(z,z^\gamma)\right)^2dz^0=\\&\frac{1}{\left|\mathcal{N}_\lambda^{\mathcal{T}_\mathrm{hemi}}\right|^2}\sum_{N,M\in\mathcal{N}_\lambda^{\mathcal{T}_\mathrm{hemi},\mathrm{D}}}\sum_{\substack{1\leq i,j\leq 2\\1\leq k,l\leq 2\\i+j, k+l<4}}(-1)^{i+j+k+l}\int_{Q_{HE}}\left(P^{i,\mathrm{N}}_1P^{j,\mathrm{N}}_2\right)\left(P^{k,M}_1P^{l,M}_2\right)dz^0=\\&\frac{(-1)^{i+j+k+l}}{\left|\mathcal{N}_\lambda^{\mathcal{T}_\mathrm{hemi}}\right|^2}\sum_{N,M\in\mathcal{N}_\lambda^{\mathcal{T}_\mathrm{hemi},\mathrm{D}}}\sum_{\substack{1\leq i,j\leq 2\\1\leq k,l\leq 2\\i+j, k+l<4}}\int_0^{\sqrt{3}/2}\left(P^{j,\mathrm{N}}_2P^{l,M}_2\right)dz_2^0\int_{0}^{3/2}\left(P^{i,\mathrm{N}}_1P^{k,M}_1\right)dz_1^0.\end{aligned}\end{equation}
	An analysis of the integrals as the one in previous subsections gives us that:
	\begin{equation}
		\int_0^{\sqrt{3}/2}\left(P^{2,\mathrm{N}}_2P^{1,M}_2\right)dz_2^0=\int_{0}^{3/2}\left(P^{2,\mathrm{N}}_1P^{1,M}_1\right)dz_1^0=0,
	\end{equation}
	\begin{equation}
		\begin{aligned}
		   & B_1^{2,N,M}=\int_0^{3/2}\left(P^{2,\mathrm{N}}_1P^{2,M}_1\right)dz_1^0\\&=\frac{3}{2}\cos\left(\frac{2\pi m\left(z_1-z_1^\gamma\right)}{3\sqrt{\lambda}}\right)\cos\left(\frac{2\pi a\left(z_1-z_1^\gamma\right)}{3\sqrt{\lambda}}\right),
		\end{aligned}
	\end{equation}\begin{equation}
		\begin{aligned}
		    &B_2^{2,N,M}=\int_0^{\sqrt{3}/2}\left(P^{2,\mathrm{N}}_2P^{2,M}_2\right)dz_2^0\\&=\frac{\sqrt{3}}{2}\cos\left(\frac{2\pi\sqrt{3} n\left(z_2-z_2^\gamma\right)}{3\sqrt{\lambda}}\right)\cos\left(\frac{2\pi\sqrt{3} b\left(z_2-z_2^\gamma\right)}{3}\right),
		\end{aligned}
	\end{equation}
	\begin{equation}
		B_2^{1,N,M}=	\int_0^{\sqrt{3}/2}\left(P^{1,\mathrm{N}}_2P^{1,M}_2\right)dz_2=\left\{\begin{aligned}
			& 0\text{ if }n\neq b,\\
			&\frac{3}{4\sqrt{3}}\text{ if }n=b;
		\end{aligned}\right.
	\end{equation}and, finally, 
	\begin{equation}
		B_1^{1,N,M}=	\int_0^{3/2}\left(P^{1,\mathrm{N}}_1P^{1,M}_1\right)dz_1=\left\{\begin{aligned}
			& 0\text{ if }a\neq m,\\
			&\frac{3}{4}\text{ if }n=b.
		\end{aligned}\right.
	\end{equation}When considering the product of two terms, it is always true that it is zero except for the case $N=M$. Therefore, the error satisfies, for any boundary condition,
	\begin{equation}
		\begin{aligned}
			&\int_{Q_{HE}}\left(\mathcal{E}_{\lambda,\mathcal{T}_\mathrm{hemi},\mathrm{BC}}^{z^0}(z,z^\gamma)\right)^2dz^0=\\&\frac{1}{\left|\mathcal{N}_\lambda^{\mathcal{T}_\mathrm{hemi}}\right|^2}\sum_{N,M\in\mathcal{N}_\lambda^{\mathcal{T}_\mathrm{hemi},\mathrm{BC}}}\left(B_1^{1,N,M}B_2^{1,N,M}+B_1^{2,N,M}B_2^{1,N,M}+B_1^{1,N,M}B_2^{2,N,M}\right)(z,z^\gamma)=\\&\frac{1}{\left|\mathcal{N}_\lambda^{\mathcal{T}_\mathrm{hemi}}\right|^2}\sum_{N\in\mathcal{N}_\lambda^{\mathcal{T}_\mathrm{hemi},\mathrm{BC}}}\left(B_1^{1,N,\mathrm{N}}B_2^{1,N,\mathrm{N}}+B_1^{2,N,\mathrm{N}}B_2^{1,N,\mathrm{N}}+B_1^{1,N,\mathrm{N}}B_2^{2,N,\mathrm{N}}\right)(z,z^\gamma).\end{aligned}\end{equation}
	Moreover, notice that $\left|B_j^{i,N,\mathrm{N}}(z,z^\gamma)\right|\leq \frac{3}{2}$ for all $1\leq i,j \leq 2$ and all $(z,z^\gamma)\in\mathcal{R}^4$.
	
\textbf{	\underline{Step II:}} The next step is to integrate over $\mathcal{R}^4$ and bound that integral: 
	\begin{equation}
		\begin{aligned}
			&\int_{\mathcal{R}^{4}}\int_{Q_{HE}}\left(\mathcal{E}_{\lambda,\mathcal{T}_\mathrm{hemi},\mathrm{BC}}^{z^0}(z,z^\gamma)\right)^2dz^0d(z,z^\gamma)\leq\\&\frac{1}{\left|\mathcal{N}_\lambda^{\mathcal{T}_\mathrm{hemi}}\right|^2}\sum_{N\in\mathcal{N}_\lambda}\int_{\mathcal{R}^{4}}\left|B_1^{1,N,\mathrm{N}}B_2^{1,N,\mathrm{N}}+B_1^{2,N,\mathrm{N}}B_2^{1,N,\mathrm{N}}+B_1^{1,N,\mathrm{N}}B_2^{2,N,\mathrm{N}}\right|d(z,z^\gamma)\leq\\&\frac{1}{\left|\mathcal{N}_\lambda^{\mathcal{T}_\mathrm{hemi}}\right|^2}\sum_{N\in\mathcal{N}_\lambda^{\mathcal{T}_\mathrm{hemi},\mathrm{BC}}}\frac{27}{4}\int_{\mathcal{R}^{4}}d(z,z^\gamma)\leq\frac{27R^4}{4\left|\mathcal{N}_\lambda^{\mathcal{T}_\mathrm{hemi}}\right|}.\end{aligned}\end{equation}
	To conclude, we apply Lemma \ref{BoundOfNonTrans}, to get that, for any $\gamma>0$, \begin{equation}
		\int_{\mathcal{R}^{4}}\int_{\mathcal{T}_\mathrm{hemi}}\left(\mathcal{E}_{\lambda,\mathcal{T}_\mathrm{hemi},\mathrm{BC}}^{z^0}(z,z^\gamma)\right)^2dz^0d(z,z^\gamma)\leq c_\gamma\lambda^{-\gamma}R^{4}.
	\end{equation}
	The study of these four cases concludes the proof of the lemma and, therefore, of Theorem \ref{pnft}.
\end{proof}

Let us now make a comment about the proof of this theorem that will be specially useful in Chapter \ref{ILtil}. It deepens in the way of choosing the sequence of eigenvalues $\lambda_n$ taken in the Theorem.
\begin{remark}\label{lambda}
	A closer look at the proof of Theorem \ref{pnft} reveals that the sequence $\lambda_n$ utilized for the approximation does not depend on $\varphi$: we are only imposing that it satisfies the asymptotic equidistribution property (see Definition \ref{aequid}). However, the smallest $n$ that we need to reach the bound $\ep$ does depend on $\varphi$.
\end{remark}

\section{Inverse localization around a fixed base point}\label{fp}

    In this chapter, we prove the positive part of Theorem \ref{BT.polygons} concerning symmetric monochromatic waves. As stated in Chapter \ref{Intro}, if we wish to fix in advance the point of $\mathcal{P}$ around which we localize the eigenfunctions, we must impose some symmetry condition on the solution to the Helmholtz equation that we seek to approximate. Therefore, there exist particular points in $\overline{\mathcal{P}}$ where certain solutions to the Helmholtz equation cannot be approximated, and we cannot expect a better result. 

In the following remark, we illustrate this with an example: we examine the simplest points in the specific case of the two-dimensional square $Q_1$ with Dirichlet boundary conditions.

\begin{figure}\renewcommand\thefigure{4.1}
	\includegraphics[width=12cm]{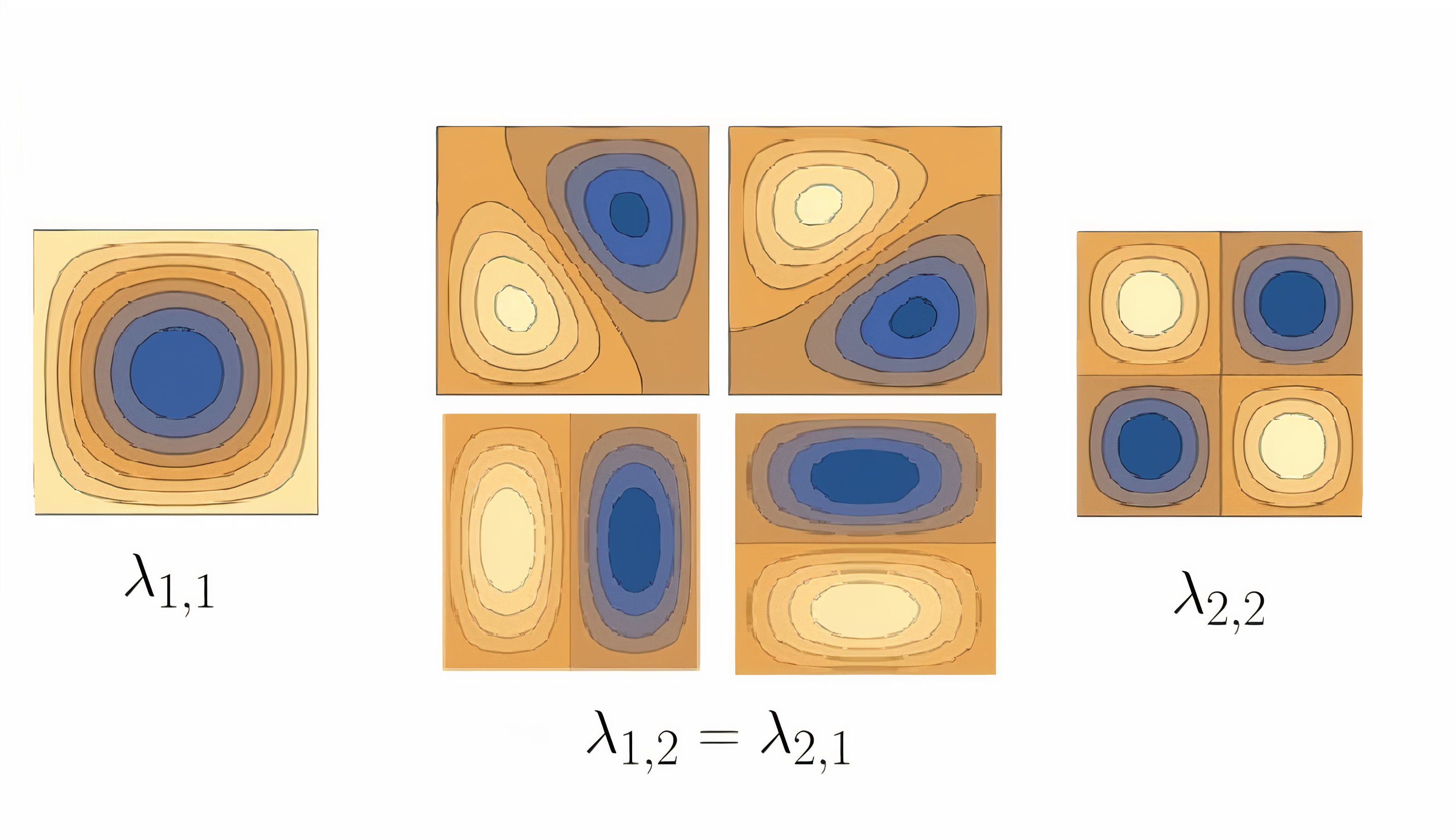}
	\caption{Some Dirichlet eigenfunctions on $\mathcal{Q}_1$ associated to the first eigenvalues. Here we can observe the symmetry condition discussed in Remark \ref{RemarkSymm}. Original images created with the software Mathematica.  }
\end{figure}

\begin{remark}\label{RemarkSymm}We consider the two easiest kind of points with symmetry restrictions on localized eigenfunctions: 
	
	\begin{enumerate}
	\item Points in the boundary: if the point is in the horizontal or vertical boundaries of $Q_1$, $(p_1,p_2)\in \partial V\cup \partial H=\left\{(0,p_2),(1,p_2)\right\}\cup \left\{(p_1,0),(p_1,1)\right\}$, we have that any eigenfunction \begin{equation}u(z_1,z_2)=\sum_{N\in\mathbb{Z}^2, |N|=\sqrt{\lambda}}c_N\sin\left(N_1\pi\left(p_1+\frac{z_1}{\sqrt{\lambda}}\right)\right)\sin\left(N_2\pi\left(p_2+\frac{z_2}{\sqrt{\lambda}}\right)\right)\end{equation} satisfies that either  \begin{equation}u(z_1,z_2)=\sum_{N\in\mathbb{Z}^2, |N|=\sqrt{\lambda}}\tilde{c}_N\sin\left(N_1\pi\frac{z_1}{\sqrt{\lambda}}\right)\sin\left(N_2\pi\left(p_2+\frac{z_2}{\sqrt{\lambda}}\right)\right)\end{equation} or, in the other case, \begin{equation}u(z_1,z_2)=\sum_{N\in\mathbb{Z}^2, |N|=\sqrt{\lambda}}\tilde{c}_N\sin\left(N_1\pi\left(p_1+\frac{z_1}{\sqrt{\lambda}}\right)\right)\sin\left(N_2\pi\frac{z_2}{\sqrt{\lambda}}\right)\end{equation}
	We can see that we have odd symmetry with respect to one of the axes, and we can then approximate only solutions to Helmholtz with such a symmetry condition. Of course, not all of them satisfy it.
	
	\item Some points in the interior: the easiest case is $(p_1,p_2)=(1/2,1/2)$, where we have  \begin{equation}
		u(z_1,z_2)=\sum_{N\in\mathbb{Z}^2, |N|=\sqrt{\lambda}}c_k\sin\left(\frac{N_1\cdot\pi }{2}+\frac{z_1\cdot N_1}{\sqrt{\lambda}}\right)\sin\left(\frac{N_2\cdot\pi }{2}+\frac{N_2\cdot z_2}{\sqrt{\lambda}}\right).
	\end{equation}
	Notice now that for $\lambda=N_1^2+N_2^2$, we have the following options:\begin{itemize}
		\item $\lambda$ is odd, $N_1$ is odd and $N_2$ is even.
		\item  $\lambda$ is odd, $N_1$ is even and $N_2$ is odd.
		\item $\lambda$ is even and both $N_1,N_2$ are odd.
		\item $\lambda$ is even and both $N_1,N_2$ are even. 
	\end{itemize}
	So eigenfunctions have two possibilities: for $\lambda$ even,  \begin{equation}\begin{aligned}
	    \label{evenpi2}
		u(z_1,z_2)=&\sum_{\substack{|N|=\sqrt{\lambda}\\ N_1,N_2\text{ even}}}c_N\sin\left(\frac{N_1 z_1}{\sqrt{\lambda}}\right)\sin\left(\frac{N_2 z_2}{\sqrt{\lambda}}\right)+\\&\sum_{\substack{|N|=\sqrt{\lambda}\\ N_1,N_2\text{ odd}}}c_N\cos\left(\frac{N_1 z_1}{\sqrt{\lambda}}\right)\cos\left(\frac{N_2 z_2}{\sqrt{\lambda}}\right),
	\end{aligned}
	\end{equation}and for $\lambda$ odd,
	\begin{equation}
		\begin{aligned}
		    u(z_1,z_2)=&\sum_{\substack{|N|=\sqrt{\lambda}\\ N_1\text{ even }\\ N_2\text{ odd}}}c_N\sin\left(\frac{N_1 z_1}{\sqrt{\lambda}}\right)\cos\left(\frac{N_2 z_2}{\sqrt{\lambda}}\right)+\\&\sum_{\substack{|N|=\sqrt{\lambda}\\ N_1\text{ odd}\\ N_2\text{ even }}}c_N\cos\left(\frac{N_1 z_1}{\sqrt{\lambda}}\right)\sin\left(\frac{N_2 z_2}{\sqrt{\lambda}}\right).
		\end{aligned}
	\end{equation} In the first case, $u$ is even (i.e. $u(-z_1,-z_2)=u(z_1,z_2)$) while in the second one, it is odd ($u(-z_1,-z_2)=-u(z_1,z_2)$). Therefore, we can only approximate solutions to Helmholtz with these symmetry properties. For example, we can consider the solution to Helmholtz given by\begin{equation}
		\vp(z_1,z_2)=J_0(\left|(z_1,z_2)\right|)+J_0(\left|(z_1,z_2)-(1,1)\right|)-J_0(\left|(z_1,z_2)+(1,1)\right|).
	\end{equation}It is neither symmetric nor antisymmetric, so we can not approximate it by eigenfunctions localized in $(1/2,1/2)$. Notice that, for any symmetric eigenfunction $u$,\begin{equation}
		\left\|u-\vp\right\|_{C^0(B)}\geq\frac{1}{2}\sup_{z\in B}\left|\vp(z)-\vp(-z)\right|>0;
	\end{equation}while for any antisymmetric $u$, \begin{equation}
		\left\|u-\vp\right\|_{C^0(B)}\geq\frac{1}{2}\sup_{z\in B}\left|\vp(z)+\vp(-z)\right|>0.
	\end{equation}
\end{enumerate}

\end{remark}
We now recall that we are considering $\mathcal{P}\subset\RR^2$ to be be an integrable triangle or a rational rectangle as in Definition \ref{IntPol}. Only if $\mathcal{P}=\mathcal{Q}_l$, we consider $d\geq2$ instead of $2$ for the dimension. Let us also consider the eigenvalue problem with Dirichlet or Neumann boundary conditions on $\mathcal{P}$. Then the theorem is as follows. 
\begin{theorem}\label{Allpoints}

Set any $z^0\in \mathcal{P}$, any $k\in\mathbb{N}$, any $\ep>0$, any ball $B\subset \mathbb{R}^2$ and any $\vp\in\MW$ that satisfies a certain symmetry condition depending on $\mathcal{P}$ and the boundary conditions. 
Then we can find a sequence of eigenvalues $\lambda_n$ and an associated sequence of eigenfunctions $u_{n}$ that satisfy  \begin{equation}
	\left\|\vp-u_{n}\left(z^0+\frac{\cdot}{\sqrt{\lambda_n}}\right)\right\|_{C^k(B)}<\ep,
\end{equation}for all $n$ large enough. The conditions over $\vp$ are collected in the Table \ref{TableA} of Appendix \ref{AppendixC}.

\end{theorem}
\begin{proof}
	
\textbf{	\underline{Step I:}}
	We just follow Step I in the proof of Theorem \ref{pnft} to get a function $\phi$ that	approximates the function $\vp$ in the ball as
	\begin{equation}
		\|\phi-\varphi\|_{C^k(B)}<\ep/2,
	\end{equation}
	and that can be written as 
	\begin{equation}
		\phi(z)=\sum_{\gamma=1}^{L} c_\gamma\int_{\mathbb{S}^{d-1}}e^{i(z-z^\gamma)\xi}d\sigma_{\mathbb{S}^{d-1}}(\xi),
	\end{equation}
	with large $L>0$, $R\geq 1$, ${c}_\gamma\in\mathbb{R}$ and $z^\gamma\in B_R\subset\mathbb{R}^d$ for all $1\leq \gamma\leq L$.
	Choosing \begin{equation}
		p(\xi)=\sum_{\gamma=1}^{L} c_\gamma e^{-i\xi z^\gamma},
	\end{equation}we get also the expression
	\begin{equation}\label{HermitianpolyII}
		\phi(z)=\int_{\mathbb{S}^{d-1}} e^{i z\cdot\xi} p(\xi) \, d\si_{\mathbb{S}^{d-1}}(\xi),
	\end{equation}
	where $p$ is a complex-valued Hermitian polynomial (i.e., $p(-\xi)=\overline{p(\xi)}$).
	
	As already studied in Appendix \ref{AppendixC}, all the different conditions that we ask on $\vp$ can be seen as follows:\begin{equation}
		\vp\left(S^{\mathcal{P}}_j\left(z\right)\right)=\left(-1\right)^{\sigma_\mathrm{BC}}\vp(z),\ 1\leq j\leq J_\mathcal{P},
	\end{equation}
	where $\sigma_D=1$ and $\sigma_N=2$ and $S^\mathcal{P}_j$ is an orthogonal symmetry for any $1\leq j\leq J_\mathcal{P}$ and all $\mathcal{P}$ (see Remark \ref{conditions} for a precise description of these transforms). Notice that it was seen in Appendix \ref{AppendixC} there there exist such monochromatic waves.
	
	Thanks to these extra conditions, it is not difficult to see (use e.g. Herglotz's theorem \cite[Theorem 7.1.27]{Hor}) that, for any $1\leq j\leq J_\mathcal{P}$, $p$ can be chosen so that it satisfies the same conditions:\begin{equation}
		p\left(S^{\mathcal{P}}_j\left(z\right)\right)=\left(-1\right)^{c_\mathrm{BC}}p(z),\ 1\leq j\leq J_\mathcal{P}.
	\end{equation}To do this, we notice that if the polynomial $p$ does not satisfy these conditions we can consider a new monochromatic wave $\tilde{\phi}$ given by\begin{equation}
\tilde{	\phi}(z)=\int_{\mathbb{S}^{d-1}} e^{i z\cdot\xi} \tilde{p}(\xi) \, d\si_{\mathbb{S}^{d-1}}(\xi)=\int_{\mathbb{S}^{d-1}} e^{i z\cdot\xi} \frac{\left(-1\right)^{\sigma_{\mathrm{BC}}}}{J_\mathcal{P}}\sum_{j=1}^{J_\mathcal{P}}p\left(S^\mathcal{P}_j\left(\xi\right)\right) \, d\si_{\mathbb{S}^{d-1}}(\xi).
	\end{equation}It is clear that $\tilde{p}$ satisfies the desired conditions and notice that
	\begin{equation}
	\begin{aligned}
	&\|\tilde{\phi}-\varphi\|_{C^k(B)}\leq\left\|\tilde{\phi}-\frac{\left(-1\right)^{\sigma_{\mathrm{BC}}}}{J_\mathcal{P}}\sum_{j=1}^{J_\mathcal{P}}\left(\varphi\circ S^\mathcal{P}_j\right)\right\|_{C^k(B)}\leq\\&\left\|\frac{\left(-1\right)^{\sigma_{\mathrm{BC}}}}{J_\mathcal{P}}\sum_{j=1}^{J_\mathcal{P}}\left(\int_{\mathbb{S}^{d-1}} e^{i z\cdot\xi} p\left(S^\mathcal{P}_j\left(\xi\right)\right) \, d\si_{\mathbb{S}^{d-1}}(\xi)-\left(\varphi\circ S^\mathcal{P}_j(z)\right)\right)\right\|_{C^k(B)}\\&\leq \frac{1}{J_\mathcal{P}}\sum_{j=1}^{J_\mathcal{P}}\left\|\phi\circ S^\mathcal{P}_j-\varphi\circ S^\mathcal{P}_j\right\|_{C^k(B)}\leq\|\phi-\varphi\|_{C^k(B)}<\ep/2.
	\end{aligned}
	\end{equation}
	Notice now that these conditions allow us to rewrite \eqref{HermitianpolyII} as
	\begin{equation}
		\phi(z)=\int_{\mathbb{S}^{d-1}}p(\xi)\mathbb{T}_\mathcal{P}(z,\xi)d\si_{\mathbb{S}^{d-1}}(\xi),
	\end{equation}where $\mathbb{T}_\mathcal{P}$ is a trigonometric function that depends on the polygon and boundary conditions under consideration. Note also that in the particular case of $\mathcal{P}=	\mathcal{T}_\mathrm{equi}$, we get something slightly different:\begin{equation}
		\phi(z)=\int_{\mathbb{S}^{1}}\left(p_s(\xi)\mathbb{T}_{	\mathcal{T}_\mathrm{equi}}^s(z,\xi)+p_a(\xi)\mathbb{T}_{	\mathcal{T}_\mathrm{equi}}^a(z,\xi)\right)d\si_{\mathbb{S}^{d-1}}(\xi),
	\end{equation}where both $\mathbb{T}_{	\mathcal{T}_\mathrm{equi}}^s(z,\xi)$ and $\mathbb{T}_{	\mathcal{T}_\mathrm{equi}}^a(z,\xi)$ are trigonometric functions and $p_s$ and $p_a $ are the symmetric and antisymmetric part of $p(\cdot,\xi_2)$, i.e., $p(\xi)=p_a(\xi)+p_s(\xi),$ with \begin{equation}
		 p_a(-\xi_1,\xi_2)=-p_a(\xi),\text{ and } p_s(-\xi_1,\xi_2)=p_s(\xi).
	\end{equation}
	The precise functions are as in the following tables: for Dirichlet boundary condition, we have  
	\begin{center}
		\begin{tabular}{cc} 
			\hline
			\rule{0pt}{3ex}Polygon & $\mathbb{T}_\mathcal{P}$ \\
			\hline
			 \rule{0pt}{3ex}$\mathcal{Q}_l$& $\prod_{j=1}^d\sin\left(z_j\xi_j\right)$\\
			 \rule{0pt}{3ex}$\mathcal{T}_\mathrm{iso}$ & $\sin\left(z_1\xi_1\right)\sin\left(z_2\xi_2\right)-\sin\left(z_1\xi_2\right)\sin\left(z_2\xi_1\right)$ \\
			
			 \rule{0pt}{3ex}$	\mathcal{T}_\mathrm{equi}$ & $\mathbb{T}_{	\mathcal{T}_\mathrm{equi}}^s(z,\xi)= \cos\left(z_1\xi_1\right)\sin\left(z_2\xi_2\right)$ and $\mathbb{T}_{	\mathcal{T}_\mathrm{equi}}^a(z,\xi)= \sin\left(z_1\xi_1\right)\sin\left(z_2\xi_2\right)$\\
			 \rule{0pt}{3ex}$\mathcal{T}_\mathrm{hemi}$ & $\sin\left(z_1\xi_1\right)\sin\left(z_2\xi_2\right)$ \\\hline
		\end{tabular} 
	\end{center}While, on the other hand, for Neumann boundary conditions, we have 
	\begin{center}
		\begin{tabular}{cc} 
			\hline
			\rule{0pt}{3ex}Polygon & $\mathbb{T}_\mathcal{P}$ \\
		
			\hline
		\rule{0pt}{3ex}$\mathcal{Q}_l$ & $\prod_{j=1}^d\cos\left(z_j\xi_j\right)$ \\
			\rule{0pt}{3ex} $\mathcal{T}_\mathrm{iso}$ & $\cos\left(z_1\xi_1\right)\cos\left(z_2\xi_2\right)-\cos\left(z_1\xi_2\right)\cos\left(z_2\xi_1\right)$  \\
			 \rule{0pt}{3ex}$	\mathcal{T}_\mathrm{equi}$ & $\mathbb{T}_{	\mathcal{T}_\mathrm{equi}}^s(z,\xi)= \cos\left(z_1\xi_1\right)\cos\left(z_2\xi_2\right)$ and $\mathbb{T}_{	\mathcal{T}_\mathrm{equi}}^a(z,\xi)= \sin\left(z_1\xi_1\right)\cos\left(z_2\xi_2\right)$ \\
			\rule{0pt}{3ex} $\mathcal{T}_\mathrm{hemi}$ & $\cos\left(z_1\xi_1\right)\cos\left(z_2\xi_2\right)$ \\
			\hline
		\end{tabular} 
	\end{center}
\textbf{	\underline{Step II:}} We now make some computations (that will depend of the polygon under consideration) to approximate this function by eigenfunctions.
	
	For the case $\mathcal{P}=\mathcal{Q}_l$ we notice that, for Dirichlet conditions,
	\begin{equation}
		\int_{\mathbb{S}^{d-1}} p(\xi) \prod_{j=1}^d\sin\left(z_j\xi_j\right) d\sigma_{\mathbb{S}^{d-1}}(\xi)=C_{\mathcal{Q}_l}\int_{\mathcal{M}^{\mathcal{Q}_l}}p(\xi) \prod_{j=1}^d\sin\left(\frac{z_j\pi}{l_j}\xi_j\right)d\sigma_{\mathcal{M}^{\mathcal{Q}_l}}(\xi),
	\end{equation}and, in the Neumann case, 
	\begin{equation}
		\int_{\mathbb{S}^{d-1}} p(\xi) \prod_{j=1}^d\cos\left(z_j\xi_j\right) d\sigma_{\mathbb{S}^{d-1}}(\xi)=C_{\mathcal{Q}_l}\int_{\mathcal{M}^{\mathcal{Q}_l}}p(\xi) \prod_{j=1}^d\cos\left(\frac{z_j\pi}{l_j}\xi_j\right)d\sigma_{\mathcal{M}^{\mathcal{Q}_l}}(\xi).
	\end{equation}
	
	We now choose $\lambda$ along a sequence in which we have asymptotic equidistribution as introduced in Definition \ref{aequid} and satisfies the same extra conditions as in Step 3 of the proof of Theorem \ref{pnft}. Thus we have, for Dirichlet,
	\begin{equation}
		\begin{aligned}
			&\frac{C_{\mathcal{Q}_l}|\mathcal{M}^{\mathcal{Q}_l}|}{\#\mathcal{N}_\lambda^{\mathcal{Q}_l}}\sum_{N\in \mathcal{N}_\lambda^{\mathcal{Q}_l}}p\left(\frac{N}{\sqrt{\lambda}}\right) \prod_{j=1}^d\sin\left(\frac{z_jN_j\pi}{l_j\sqrt{\lambda}}\right)+\mathbb{E}(\lambda)\\
			&	=\frac{2^dC_{\mathcal{Q}_l}|\mathcal{M}^{\mathcal{Q}_l}|}{\#\mathcal{N}_\lambda^{\mathcal{Q}_l}}\sum_{N\in \mathcal{N}_\lambda^{\mathcal{Q}_l,\mathrm{D}}}c_N \prod_{j=1}^d\sin\left(\frac{z_jN_j\pi}{l_j\sqrt{\lambda}}\right)+\mathbb{E}(\lambda);
		\end{aligned}
	\end{equation}and in the Neumann case
	\begin{equation}
		\begin{aligned}
			&\frac{C_{\mathcal{Q}_l}|\mathcal{M}^{\mathcal{Q}_l}|}{\#\mathcal{N}_\lambda^{\mathcal{Q}_l}}\sum_{N\in \mathcal{N}_\lambda^{\mathcal{Q}_l}}p\left(\frac{N}{\sqrt{\lambda}}\right) \prod_{j=1}^d\cos\left(\frac{z_jN_j\pi}{l_j\sqrt{\lambda}}\right)+\mathbb{E}(\lambda)\\
			&	=\frac{2^dC_{\mathcal{Q}_l}|\mathcal{M}^{\mathcal{Q}_l}|}{\#\mathcal{N}_\lambda^{\mathcal{Q}_l}}\sum_{N\in \mathcal{N}_\lambda^{\mathcal{Q}_l,\mathrm{N}}}c_N \prod_{j=1}^d\cos\left(\frac{z_jN_j\pi}{l_j\sqrt{\lambda}}\right)+\mathbb{E}(\lambda);
		\end{aligned}
	\end{equation}where	$\mathbb{E}(\lambda)$ is an error function that satisfies $\lim_{\lambda_n\rightarrow\infty}\mathbb{E}(\lambda_n)=0$, when $\lambda_n$ is a subsequence with the asymptotic equidistribution property (see again Definition \ref{aequid}), and $c_m$ is given by all the possible changes of sign in $p\left(\frac{N}{\sqrt{\lambda}}\right)$.
	
	We then have \begin{equation}
		\vp(z)=\frac{2^dC_{\mathcal{Q}_l}|\mathcal{M}^{\mathcal{Q}_l}|}{\#\mathcal{N}_\lambda^{\mathcal{Q}_l}}\sum_{N\in \mathcal{N}_\lambda^{\mathcal{Q}_l,\mathrm{D}}}c_N \prod_{j=1}^d\sin\left(\frac{z_jN_j\pi}{l_j\sqrt{\lambda}}\right)+\mathbb{E}(\lambda)+\frac{\ep}{2},
	\end{equation}for $\mathrm{BC}=D$ and, for $\mathrm{BC}=N$, 
	\begin{equation}
		\vp(z)=\frac{2^dC_{\mathcal{Q}_l}|\mathcal{M}^{\mathcal{Q}_l}|}{\#\mathcal{N}_\lambda^{\mathcal{Q}_l}}\sum_{N\in \mathcal{N}_\lambda^{\mathcal{Q}_l,\mathrm{N}}}c_N \prod_{j=1}^d\cos\left(\frac{z_jN_j\pi}{l_j\sqrt{\lambda}}\right)+\mathbb{E}(\lambda)+\frac{\ep}{2}.
	\end{equation}
	
	When $\mathcal{P}=\mathcal{T}_\mathrm{iso}$, we first note that
	\begin{equation}
		\begin{aligned}
			&\int_{\mathbb{S}^{1}}p(\xi)\left(\sin(z_1\xi_1)\sin(z_2\xi_2)-\sin(z_1\xi_2)\sin(z_2\xi_1)\right)d\si_{\mathbb{S}^{1}}(\xi)\\&=\int_{\mathbb{S}^{1}}\left(p(\xi)-p(\xi_2,\xi_1)\right)\sin(z_1\xi_1)\sin(z_2\xi_2)d\si_{\mathbb{S}^{1}}(\xi), 
		\end{aligned}
	\end{equation}and also \begin{equation}
		\begin{aligned}
			&\int_{\mathbb{S}^{1}}p(\xi)\left(\cos(z_1\xi_1)\cos(z_2\xi_2)+\cos(z_1\xi_2)\cos(z_2\xi_1)\right)d\si_{\mathbb{S}^{1}}(\xi)\\&=\int_{\mathbb{S}^{1}}\left(p(\xi)+p(\xi_2,\xi_1)\right)\cos(z_1\xi_1)\cos(z_2\xi_2)d\si_{\mathbb{S}^{1}}(\xi).
		\end{aligned}
	\end{equation}Therefore, a similar approach allows us to get 
	\begin{equation}
		\vp(z)=\frac{4|\mathbb{S}^1|}{\#\mathcal{N}_\lambda^{\mathcal{T}_\mathrm{iso}}}\sum_{N\in \mathcal{N}_\lambda^{\mathcal{T}_\mathrm{iso},\mathrm{D}}}c_N \sin\left(\frac{z_1N_1\pi}{\sqrt{\lambda}}\right)\sin\left(\frac{z_2N_2\pi}{\sqrt{\lambda}}\right)+\mathbb{E}(\lambda)+\frac{\ep}{2},
	\end{equation} for the Dirichlet case, where $c_N$ is defined with all the possible signs in 
    $$
    p\left(\frac{\pi N_1}{\sqrt{\lambda}},\frac{\pi N_2}{\sqrt{\lambda}}\right)-p\left(\frac{\pi N_2}{\sqrt{\lambda}},\frac{\pi N_1}{\sqrt{\lambda}}\right)
    $$ and satisfy $c_{N_2,N_1}=-c_{N_1,N_2}$, and for Neumann,
	\begin{equation}
		\vp(z)=\frac{4|\mathbb{S}^1|}{\#\mathcal{N}_\lambda^{\mathcal{T}_\mathrm{iso}}}\sum_{N\in \mathcal{N}_\lambda^{\mathcal{T}_\mathrm{iso},\mathrm{N}}}c_N \cos\left(\frac{z_1N_1\pi}{\sqrt{\lambda}}\right)\cos\left(\frac{z_2N_2\pi}{\sqrt{\lambda}}\right)+\mathbb{E}(\lambda)+\frac{\ep}{2},
	\end{equation}where $c_N$ is now defined with $p\left(\frac{\pi N_1}{\sqrt{\lambda}},\frac{\pi N_2}{\sqrt{\lambda}}\right)+p\left(\frac{\pi N_2}{\sqrt{\lambda}},\frac{\pi N_1}{\sqrt{\lambda}}\right)$ and satisfy $c_{N_2,N_1}=c_{N_1,N_2}$.
	
	When $\mathcal{P}=	\mathcal{T}_\mathrm{equi}$, choosing correctly $c_N^a$ and $c_N^s$, we get to similar expressions:
	\begin{equation}
		\begin{aligned}
			&\vp(z)=\mathbb{E}(\lambda)+\frac{\ep}{2}+\\&\frac{4|\mathcal{M}^{	\mathcal{T}_\mathrm{equi}}|}{\sqrt{3}\#\mathcal{N}_\lambda^{\mathcal{T}_\mathrm{equi}}}\sum_{N\in \mathcal{N}_\lambda^{	\mathcal{T}_\mathrm{equi}},\mathrm{D}} \sin\left(\frac{2n\pi\sqrt{3}z_2}{3\sqrt{\lambda}}\right)\left[c_N^s\cos\left(\frac{2\pi mz_1}{3\sqrt{\lambda}}\right)+c_N^a\sin\left(\frac{2\pi mz_1}{3\sqrt{\lambda}}\right)\right],
		\end{aligned}
	\end{equation}when considering Dirichlet and\begin{equation}
		\begin{aligned}
			&\vp(z)=\mathbb{E}(\lambda)+\frac{\ep}{2}+\\&\frac{4|\mathcal{M}^{	\mathcal{T}_\mathrm{equi}}|}{\sqrt{3}\#\mathcal{N}_\lambda^{\mathcal{T}_\mathrm{equi}}}\sum_{N\in \mathcal{N}_\lambda^{	\mathcal{T}_\mathrm{equi}},\mathrm{N}} \cos\left(\frac{2n\pi\sqrt{3}z_2}{3\sqrt{\lambda}}\right)\left[c_N^s\cos\left(\frac{2\pi mz_1}{3\sqrt{\lambda}}\right)+c_N^a\sin\left(\frac{2\pi mz_1}{3\sqrt{\lambda}}\right)\right],
		\end{aligned}
	\end{equation}for Neumann.
	Finally, the same works for $\mathcal{P}=\mathcal{T}_\mathrm{hemi}$, with Dirichlet, 
	\begin{equation}
		\begin{aligned}
			&\vp(z)=\mathbb{E}(\lambda)+\frac{\ep}{2}+\frac{4|\mathcal{M}^{	\mathcal{T}_\mathrm{equi}}|}{\sqrt{3}\#\mathcal{N}_\lambda^{\mathcal{T}_\mathrm{equi}}}\sum_{N\in \mathcal{N}_\lambda^{	\mathcal{T}_\mathrm{equi}},\mathrm{D}}c_N \sin\left(\frac{2n\pi\sqrt{3}z_2}{3\sqrt{\lambda}}\right)\sin\left(\frac{2\pi mz_1}{3\sqrt{\lambda}}\right),
		\end{aligned}
	\end{equation}and for Neumann,
	\begin{equation}
		\begin{aligned}
			&\vp(z)=\mathbb{E}(\lambda)+\frac{\ep}{2}+\frac{4|\mathcal{M}^{	\mathcal{T}_\mathrm{equi}}|}{\sqrt{3}\#\mathcal{N}_\lambda^{\mathcal{T}_\mathrm{equi}}}\sum_{N\in \mathcal{N}_\lambda^{	\mathcal{T}_\mathrm{equi}},\mathrm{N}}c_N \cos\left(\frac{2n\pi\sqrt{3}z_2}{3\sqrt{\lambda}}\right)\cos\left(\frac{2\pi mz_1}{3\sqrt{\lambda}}\right).
		\end{aligned}
	\end{equation}
	For the next part of the proof, we differentiate into two different cases depending on the point $z^0\in\mathcal{P}$ around which we want to approximate.
	
\textbf{	\underline{Step III:}} In this step we consider those points with all components being rational with respect to the sides of $\mathcal{P}$. This means, when $\mathcal{P}=\mathcal{Q}_l$, that
	\begin{equation}
		z^0=(z^0_1,\ldots,z_d^0)=\left(\frac{r_1}{s_1}l_1,\ldots,\frac{r_{d-1}}{s_{d-1}}l_{d-1},\frac{r_d}{s_d}\right);
	\end{equation}for $r_j,s_j\in\mathbb{N}$ for any $1\leq j\leq d$; when $\mathcal{P}=\mathcal{T}_\mathrm{iso}$ that\begin{equation}
		z^0=(z^0_1,z_2^0)=\left(\frac{r_1}{s_1},\frac{r_2}{s_2}\right);
	\end{equation} and, finally, when $\mathcal{P}=	\mathcal{T}_\mathrm{equi}$ or $\mathcal{P}=\mathcal{T}_\mathrm{hemi}$ that 
	\begin{equation}
		z^0=(z^0_1,z_2^0)=\left(\frac{3r_1}{2s_1},\frac{\sqrt{3}r_2}{2s_2}\right),
	\end{equation}
	both with $r_1,r_2,s_1,s_2\in\mathbb{N}$. In any case, let $s=\LCM(s_1,\ldots,s_d)$ (with $d=2$ whenever $\mathcal{P}\neq \mathcal{Q}_l$). We choose $\lambda$ in the asymptotically equidistributed subsequence and large enough so that $\mathbb{E}(\lambda)<\frac{\ep}{2}$ and consider the eigenvalue $\lambda s^2$ and an associated eigenfunction  that depends on the polygon\begin{equation}
		u_{s^2\lambda}^{\mathcal{P},\mathrm{BC}}\left(z^0+\frac{z}{s\sqrt{\lambda}}\right). 
	\end{equation}The eigenfunctions are as in the following tables: for Dirichlet boundary conditions,
	
		\vspace{0.1cm}
		
	\begin{tabular}{cc} 
		\hline
		\rule{0pt}{3ex}Polygon & $u_{s^2\lambda}^{\mathcal{P},\mathrm{BC}}(z)$\vspace{0.1cm} \\
		\hline
		\rule{0pt}{5ex}$\mathcal{Q}_l$& $\frac{2^dC_{\mathcal{Q}_l}|\mathcal{M}^{\mathcal{Q}_l}|}{\left|\mathcal{N}_{\lambda}^{\mathcal{Q}_l}\right|}\sum_{M\in \mathcal{N}_{s^2\lambda}^{\mathcal{Q}_l,\mathrm{D}}}C_M \prod_{j=1}^d\sin\left(\frac{z_jM_j\pi}{l_j}\right)$\\
		\rule{0pt}{5ex} $\mathcal{T}_\mathrm{iso}$ & $\frac{4|\mathbb{S}^1|}{\#\mathcal{N}_\lambda^{\mathcal{T}_\mathrm{iso}}}\sum_{M\in \mathcal{N}_{s^2\lambda}^{\mathcal{T}_\mathrm{iso},\mathrm{D}}}C_M \sin\left(z_1M_1\pi\right)\sin\left(z_2M_2\pi\right)$ \\
		
		\rule{0pt}{5ex} $	\mathcal{T}_\mathrm{equi}$ & $\frac{4|\mathcal{M}^{	\mathcal{T}_\mathrm{equi}}|}{\sqrt{3}\#\mathcal{N}_\lambda^{\mathcal{T}_\mathrm{equi}}}\sum_{M\in \mathcal{N}_{s^2\lambda}^{	\mathcal{T}_\mathrm{equi}},\mathrm{D}} \sin\left(\frac{2M_2\pi\sqrt{3}z_2}{3}\right)\left[C_M^s\cos\left(\frac{2\pi M_1z_1}{3}\right)+C_M^a\sin\left(\frac{2\pi M_1z_1}{3}\right)\right]$\\
		\rule{0pt}{5ex} $\mathcal{T}_\mathrm{hemi}$ & $\frac{4|\mathcal{M}^{	\mathcal{T}_\mathrm{equi}}|}{\sqrt{3}\#\mathcal{N}_\lambda^{\mathcal{T}_\mathrm{equi}}}\sum_{M\in \mathcal{N}_{s^2\lambda}^{	\mathcal{T}_\mathrm{equi}},\mathrm{D}}C_M \sin\left(\frac{2M_2\pi\sqrt{3}z_2}{3}\right)\sin\left(\frac{2\pi M_1z_1}{3}\right)$ \\ 
\vspace{0.1cm}
	\end{tabular} 
	
\begin{flushleft}
			while for Neumann boundary condition,
	\end{flushleft}
	\vspace{0.1cm}
	
		\begin{tabular}{cc} 
		\hline
		\rule{0pt}{3ex} Polygon & $u_{s^2\lambda}^{\mathcal{P},\mathrm{BC}}(z)$ \vspace{0.1cm}\\
		\hline
			\rule{0pt}{5ex}   $\mathcal{Q}_l$ & $\frac{2^dC_{\mathcal{Q}_l}|\mathcal{M}^{\mathcal{Q}_l}|}{\left|\mathcal{N}_{\lambda}^{\mathcal{Q}_l}\right|}\sum_{M\in \mathcal{N}_{s^2\lambda}^{\mathcal{Q}_l,\mathrm{N}}}C_M \prod_{j=1}^d\cos\left(\frac{z_jM_j\pi}{l_j}\right)$ \\
			\rule{0pt}{5ex}  $\mathcal{T}_\mathrm{iso}$ & $\frac{4|\mathbb{S}^1|}{\#\mathcal{N}_\lambda^{\mathcal{T}_\mathrm{iso}}}\sum_{M\in \mathcal{N}_{s^2\lambda}^{\mathcal{T}_\mathrm{iso},\mathrm{N}}}C_M \cos\left(z_1M_1\pi\right)\cos\left(z_2M_2\pi\right)$  \\
			\rule{0pt}{5ex}  $	\mathcal{T}_\mathrm{equi}$ & $\frac{4|\mathcal{M}^{	\mathcal{T}_\mathrm{equi}}|}{\sqrt{3}\#\mathcal{N}_\lambda^{\mathcal{T}_\mathrm{equi}}}\sum_{M\in \mathcal{N}_{s^2\lambda}^{	\mathcal{T}_\mathrm{equi}},\mathrm{N}} \cos\left(\frac{2M_2\pi\sqrt{3}z_2}{3}\right)\left[C_M^s\cos\left(\frac{2\pi M_1z_1}{3}\right)+C_M^a\sin\left(\frac{2\pi M_1z_1}{3}\right)\right]$ \\
			\rule{0pt}{5ex}  $\mathcal{T}_\mathrm{hemi}$ & $\frac{4|\mathcal{M}^{	\mathcal{T}_\mathrm{equi}}|}{\sqrt{3}\#\mathcal{N}_\lambda^{\mathcal{T}_\mathrm{equi}}}\sum_{M\in \mathcal{N}_{s^2\lambda}^{	\mathcal{T}_\mathrm{equi}},\mathrm{N}}C_M \cos\left(\frac{2M_2\pi\sqrt{3}z_2}{3}\right)\cos\left(\frac{2\pi M_1z_1}{3}\right)$ \\\vspace{0.1cm}	
	\end{tabular} 
	
	To conclude, we have defined \begin{equation}
		C_{M}^*=\left\{\begin{aligned}
			&c_{N}^*(-1)^{\sum_{j=1}^d r_j N_j\frac{s}{s_j}}\text{ if }M=sN
			\\
			&0, \text{ otherwise.}
		\end{aligned}\right.
	\end{equation}Here $*$ represents $s$ and $a$ in the equilateral triangle case and nothing in all the other cases. We have also used the trigonometric identities \begin{equation}
		\begin{aligned}
			&\cos\left(x+n\pi\right)=(-1)^n\cos\left(x\right),\\
			&\sin\left(x+n\pi\right)=(-1)^n\sin\left(x\right),
		\end{aligned}
	\end{equation}
	so that we get \begin{equation}
		\left\|u_{s^2\lambda}^{\mathcal{P},\mathrm{BC}}\left(z^0+\frac{\cdot}{s\sqrt{\lambda}}\right)-\vp\right\|_{C^0(B)}=\mathbb{E}(\lambda)+\frac{\ep}{2}<\ep.
	\end{equation}
\textbf{	\underline{Step IV:}}  In this final step, we consider points with at least one component being irrational with respect to the sides of $\mathcal{P}$. This means, when $\mathcal{P}=\mathcal{Q}_l$, 
	\begin{equation}
		z^0=\left(z^0_1,\ldots,z^0_d\right)=\left(\alpha_1l_1,\ldots,\alpha_{d-1}l_{d-1},\alpha_d\right)
	\end{equation} with some $\alpha_j\notin\mathbb{Q}$; when $\mathcal{P}=\mathcal{T}_\mathrm{iso}$ that\begin{equation}
		z^0=(z^0_1,z_2^0)=\left(\alpha_1,\alpha_2\right);
	\end{equation} and, finally, when $\mathcal{P}=	\mathcal{T}_\mathrm{equi}$ or $\mathcal{P}=\mathcal{T}_\mathrm{hemi}$ that 
	\begin{equation}
		z^0=(z^0_1,z_2^0)=\left(\frac{3\alpha_1}{2},\frac{\sqrt{3}\alpha_2}{2}\right),
	\end{equation}with the same condition: either $\alpha_1$ or $\alpha_2$ is not in $\mathbb{Q}$.

	We choose $\lambda$ large enough so that $\mathbb{E}(\lambda)<\ep/4$. Let us also consider the quantity \begin{equation}
		F(\lambda)=\frac{2^d\pi C_{\mathcal{P}}|\mathcal{M}^{\mathcal{P}}|}{\left|\mathcal{N}_{\lambda}^{\mathcal{P}}\right|}\sum_{N\in \mathcal{N}_{\lambda}^{\mathcal{P},\mathrm{BC}}}|c_N|\left(N_1+\ldots+N_d\right). 
	\end{equation}It might be large but it is finite for a fixed $\lambda$. We recall that $d=2$ except when $\mathcal{P}=\mathcal{Q}_l$. We use now Dirichlet's approximation theorem \cite[Theorem 1B]{Sch}.
	\begin{theorem}[Dirichlet's approximation theorem]
		
		Given $\alpha_1,\ldots,\alpha_d$ in $\mathbb{R}$ with at least one of them being irrational, then there are infinitely many $d-$tuples $\left(\frac{r_1}{s},\ldots,\frac{r_d}{s}\right)$ with $\GCD(s,r_1,\ldots,r_d)=1$ and \begin{equation}
			\left|\alpha_i-\frac{r_i}{s}\right|<\frac{1}{s^{1+(1/d)}},\ \ \ i=1,\ldots,d.
		\end{equation}
	\end{theorem}
	Since there are infinitely many $d-$tuples, we can choose a sequence of integer numbers $\left\{(s^n,r_1^n,\ldots,r_d^n)\right\} _{n=1}^\infty$ satisfying that $\lim_{n\rightarrow\infty}s^n=\infty$. Therefore, choosing $N$ large enough, we can assure that for $s:=s^N$ \begin{equation}
		\frac{F(\lambda)}{s^{1/d}}<\ep/2.
	\end{equation}
	Let us also define $r_j:=r_j^N$ for all $1\leq j\leq d$. We then consider the point $y^0$ rational with respect to the sides of $\mathcal{P}$ as in the previous step: when $\mathcal{P}=\mathcal{Q}_l$, 
	\begin{equation}
		y^0=(y^0_1,\ldots,y_d^0)=\left(\frac{r_1}{s}l_1,\ldots,\frac{r_{d-1}}{s}l_{d-1},\frac{r_d}{s}\right);
	\end{equation}for $\mathcal{P}=\mathcal{T}_\mathrm{iso}$,\begin{equation}
		y^0=(y^0_1,y_2^0)=\left(\frac{r_1}{s},\frac{r_2}{s}\right);
	\end{equation} and, finally, when $\mathcal{P}=	\mathcal{T}_\mathrm{equi}$ or $\mathcal{P}=\mathcal{T}_\mathrm{hemi}$ that 
	\begin{equation}
		y^0=(y^0_1,y_2^0)=\left(\frac{3r_1}{2s},\frac{\sqrt{3}r_2}{2s}\right).
	\end{equation}
	
	Doing just like in the previous step, we get to  \begin{equation}
		\left\|u_{s^2\lambda}^{\mathcal{P},\mathrm{BC}}\left(y^0+\frac{\cdot}{s\sqrt{\lambda}}\right)-\vp\right\|_{C^0(B)}<\ep/2.
	\end{equation}For any $z\in B$, using the Mean Value Theorem, we have the following \begin{equation}\begin{aligned}
			&\left|u_{s^2\lambda}^{\mathcal{P},\mathrm{BC}}\left(y^0+\frac{z}{s\sqrt{\lambda}}\right)-u_{s^2\lambda}^{\mathcal{P},\mathrm{BC}}\left(z^0+\frac{z}{s\sqrt{\lambda}}\right)\right| \\&\leq
			\frac{2^d\pi C_{\mathcal{P}}|\mathcal{M}^{\mathcal{P}}|}{\left|\mathcal{N}_{\lambda}^{\mathcal{P}}\right|}\sum_{N\in \mathcal{N}_{\lambda}^{\mathcal{P},\mathrm{BC}}}|c_N| \sum_{j=1}^d\left(sN_j\left|\frac{r_j}{s_j}-\alpha_j\right|\right)\\
			&<s^{-1/d}\frac{2^d\pi C_{\mathcal{P}}|\mathcal{M}^{\mathcal{P}}|}{\left|\mathcal{N}_{\lambda}^{\mathcal{P}}\right|}\sum_{N\in \mathcal{N}_{\lambda}^{\mathcal{P},\mathrm{BC}}}|c_N| \sum_{j=1}^d\left(N_j\right)=\frac{F(\lambda)}{s^{1/d}}<\frac{\ep}{2}.
		\end{aligned}
	\end{equation}Finally, 
	\begin{equation}\begin{aligned}
			&\left\|u_{s^2\lambda}^{\mathcal{P},\mathrm{BC}}\left(z^0+\frac{\cdot}{s\sqrt{\lambda}}\right)-\vp\right\|_{C^0(B)}\leq\left\|u_{s^2\lambda}^{\mathcal{P},\mathrm{BC}}\left(y^0+\frac{\cdot}{s\sqrt{\lambda}}\right)-\vp\right\|+ \\&\left\|u_{s^2\lambda}^{\mathcal{P},\mathrm{BC}}\left(y^0+\frac{\cdot}{s\sqrt{\lambda}}\right)-u_{s^2\lambda}^{\mathcal{P},\mathrm{BC}}\left(z^0+\frac{\cdot}{s\sqrt{\lambda}}\right)\right\|_{C^0(B)}<\ep/2+\ep/2=\ep.
	\end{aligned}\end{equation}
	In any of the two cases, using standard elliptic estimates we get the same for the $C^k\left(B\right)$ norm, which completes the proof. 
\end{proof}

\begin{remark}
In the boundary points $\partial \mathcal{P}$ we have the same result and with essentially the same proof; we just need to choose $\frac{r_j}{s_j}=\frac{0}{s_j}$ or $\frac{r_j}{s_j}=\frac{s_j}{s_j}$ for any $s_j$ that we want. The rest of the proof follows the same.  
\end{remark}
\section{No inverse localization for typical integrable billiards} 
In this chapter we gather three results of strong failure of inverse localization in three different contexts: some Robin conditions in the case of $Q_1$ and $\mathcal{T}_{\mathrm{equi}}$ (in Section \ref{Robin}), irrational rectangle like domains $\mathcal{Q}_l$ with Dirichlet or Neumann conditions (see Section \ref{Nopara}) and generic ellipses with again Dirichlet or Neumann conditions (in Section \ref{Noelli}, where we also study the ball as a particular case). In all these cases, we prove that the inverse localization fails in a strong way, following Definition \ref{NOIL}. 

Notice also that, as already commented in Chapter \ref{Intro}, these results show that the inverse localization property proved in the previous chapter is, in some sense, exceptional and that it should fail if we consider a generic billiard. It is important also to notice that in all the cases here studied, the strong failure of the inverse localization property implies that the version of this property introduced in Theorem \ref{ILBerry} cannot hold and, consequently, Berry's random wave model property (as introduced in Definition \ref{BerryLWL}) fails, as expected for a generic integrable billiard. 

Before proving any of the results, let us recall from Chapter \ref{formulas} the following notation. We denote by $\Lambda_{\mathcal{P}}^{\mathrm{BC}}$ the set of all eigenvalues of the problem in $\mathcal{P}$ with boundary conditions given by $\mathrm{BC}\in\left\{\mathrm{D},\mathrm{N},\mathrm{P},\mathrm{R}_\Sigma\right\}$ (this represents, Dirichlet, Neumann, periodic and Robin of parameter $\Sigma$ conditions) and we use $\mathcal{V}_{\mathcal{P},\mathrm{BC}}^\lambda$ for the eigenspace associated to the eigenvalue $\lambda$, i.e., the set of all the eigenfunctions associated to the eigenvalue $\lambda$ with boundary condition $\mathrm{BC}$. 

\subsection{Typical Robin boundary conditions}\label{Robin}
In this first section, we are able to prove the no inverse localization result of Theorem \ref{T.noILrobin}. To do that, we consider $\mathcal{P}=\mathcal{Q}_1$ or $\mathcal{P}=\mathcal{T}_\mathrm{equi}$ and we impose Robin boundary conditions as in \eqref{BCsi} with the property that $\Sigma:=\frac{\sin(\sigma)}{\cos(\sigma)}$ is well defined, positive and small enough. The main result of this section shows that this property fails strongly in the sense of Definition \ref{NOIL}. To do that, we use the fact that, for the eigenproblem that we consider, the multiplicity of all eigenvalues is uniformly bounded (by two, in both cases) and that we know explicit trigonometric formulas for all eigenfunctions.

\begin{theorem}
	We consider $\mathcal{P}=\mathcal{Q}_1\subset\RR^2$ or $\mathcal{P}=\mathcal{T}_{\mathrm{equi}}\subset\RR^2$ with the Robin condition~\eqref{BCsi}. Let any $k\in\mathbb{N}$ and any ball $B\subset\mathbb{R}^2$. If we ask  $\Sigma=\frac{\sin(\sigma)}{\cos(\sigma)}$ to be well defined, positive and small enough, then in $\Omega$ with the Robin condition associated to $\Sigma$ the inverse localization property fails strongly.
\end{theorem}
\begin{proof}
	We first see the case of $\mathcal{P}=Q_1$ and recall the formulas for the eigenfunctions that we saw in Section \ref{rect}. We will use also the following result proved in~\cite[Theorem 1.1.]{RudWig}:
	\begin{theorem}\label{Rud}
		There exists $\Sigma_0^{Q_1}>0$ so that for all $0<\Sigma<\Sigma_0^{Q_1}$ there are no spectral multiplicities other than the trivial ones $\lambda_{nm}^\Sigma=\lambda_{mn}^\Sigma$.
	\end{theorem}
	Using this, we know that for any eigenvalue $\lambda_{nm}^\Sigma$ all the eigenfunctions associated to the Robin condition given by $\Sigma$ have the following form: 
	\begin{equation}
		\begin{aligned}
			u_{nm}(z_1,z_2)=&c_1\left(k_n\cos\left(k_nz_1\right)+\Sigma\sin(k_nz_1)\right)\cdot \left(k_m\cos\left(k_mz_2\right)+\Sigma\sin(k_mz_2)\right)+\\&c_2\left(k_m\cos\left(k_mz_1\right)+\Sigma\sin(k_mz_1)\right)\cdot \left(k_n\cos\left(k_nz_2\right)+\Sigma\sin(k_nz_2)\right),
		\end{aligned}
	\end{equation}
	with $c_1,c_2$ arbitrary real constants. Localizing on any $(z_1^0,z_2^0)\in Q_1$ we have the following formula.
	
	\begin{equation}
		\begin{aligned}
			u_{nm}\left(z_1^0+\frac{z_1}{\sqrt{\lambda}},z_2^0+\frac{z_2}{\sqrt{\lambda}}\right)=&c_1\left[k_n\cos\left(k_n\left(z_1^0+\frac{z}{\sqrt{\lambda}}\right)\right)+\Sigma\sin\left(k_n\left(z_1^0+\frac{z}{\sqrt{\lambda}}\right)\right)\right]\cdot\\&\cdot \left[k_m\cos\left(k_m\left(z_2^0+\frac{z_2}{\sqrt{\lambda}}\right)\right)+\Sigma\sin\left(k_m\left(z_2^0+\frac{z_2}{\sqrt{\lambda}}\right)\right)\right]+\\&c_2\left[k_m\cos\left(k_m\left(z_1^0+\frac{z}{\sqrt{\lambda}}\right)\right)+\Sigma\sin\left(k_m\left(z_1^0+\frac{z}{\sqrt{\lambda}}\right)\right)\right]\cdot\\&\cdot \left[k_n\cos\left(k_n\left(z_2^0+\frac{z_2}{\sqrt{\lambda}}\right)\right)+\Sigma\sin\left(k_n\left(z_2^0+\frac{z_2}{\sqrt{\lambda}}\right)\right)\right].
		\end{aligned}
	\end{equation} Using trigonometric simplifications this is equal to 
	\begin{equation}
		\begin{aligned}
			&	u_{nm}\left(z_1^0+\frac{z_1}{\sqrt{\lambda}},z_2^0+\frac{z_2}{\sqrt{\lambda}}\right)=\sqrt{\left(k_n^2+\Sigma^2\right)\left(k_m^2+\Sigma^2\right)}\cdot\\&\left[c_1\sin\left(\frac{k_nz_1}{\sqrt{\lambda}}+z_1^{0,n}\right)\sin\left(\frac{k_mz_2}{\sqrt{\lambda}}+z_2^{0,m}\right)\right.+\left.c_2\sin\left(\frac{k_mz_1}{\sqrt{\lambda}}+z_1^{0,m}\right)\sin\left(\frac{k_nz_2}{\sqrt{\lambda}}+z_2^{0,n}\right)\right],
		\end{aligned}
	\end{equation}
	taking $z_1^{0,n}=\arctan\left(\frac{k_n}{\Sigma}\right)+k_nz_1^0$ and similarly for $z_1^{0,m}$, $z_2^{0,n}$ and $z_2^{0,m}$. Applying the Euler formula, we see that \begin{equation}
		u_{nm}\left(z_1^0+\frac{z_1}{\sqrt{\lambda}},z_2^0+\frac{z_2}{\sqrt{\lambda}}\right)\in \mathcal{A}_T=\left\{\sum_{j=1}^T\alpha_je^{\theta_j\cdot(z_1,z_2)}:\ \alpha_j\in\mathbb{C}, \theta_j\in\mathbb{R}^2,|\theta_j|=1\right\},
	\end{equation} for $T\in\mathbb{N}$ with $T\geq 8$.	
	
	On the other hand, for the case of $\mathcal{P}=\mathcal{T}_{\mathrm{equi}}$ we recall a similar result in \cite[Theorem 1.5]{RudWig2}
	\begin{theorem}
		There exists $\Sigma_0^{\mathcal{T}_{\mathrm{equi}}}>0$ so that for all $0<\Sigma<\Sigma_0^{\mathcal{T}_{\mathrm{equi}}}$ there are no spectral multiplicities other than systematic doubling, i.e., other than having symmetric $T_s^{mn}$ and antisymmetric $T_a^{mn}$  eigenfunctions.
	\end{theorem}
	This result together with the study done in Section \ref{equi} about eigenfunctions for the Robin problem allow us, in an analogous way, to know that for any eigenvalue $\lambda_{mn}^\Sigma$, all the eigenfunctions are of the form
	\begin{equation}
		\begin{aligned}
			u_{mn}(z_1,z_2)=c_sT_s^{mn}(z_1,z_2)+c_aT_a^{mn}(z_1,z_2),
		\end{aligned}
	\end{equation}where $c_s$ and $c_a$ are real constants and $T_s^{mn}$ and $T_a^{mn}$ are trigonometric functions of the form:
	
	\begin{equation}
			T_s^{mn}(z_1,z_2)=\left[\begin{aligned}
				& \cos \left(\frac{2\sqrt{3}\pi \zeta}{3}z_2-\delta_1\right) \cos \left(\frac{2\pi(\mu-\nu)}{3}z_1\right) \\
				& +\cos \left(\frac{2\pi\sqrt{3} \mu}{3 }z_2-\delta_2\right) \cos \left(\frac{2\pi(\nu-\zeta)}{3}z_1\right) \\
				& +\cos \left(\frac{2\pi\sqrt{3} \nu}{3}z_2-\delta_3\right) \cos \left(\frac{2\pi(\zeta-\mu)}{3}z_1\right)
			\end{aligned}\right],\end{equation}and\begin{equation}
		T_a^{mn}(z_1,z_2)=\left[\begin{aligned}
				& \cos \left(\frac{2\sqrt{3}\pi \zeta}{3}z_2-\delta_1\right)  \sin \left(\frac{2\pi(\mu-\nu)}{3}z_1\right) \\
				& +\cos \left(\frac{2\pi\sqrt{3} \mu}{3 }z_2-\delta_2\right) \sin \left(\frac{2\pi(\nu-\zeta)}{3}z_1\right) \\
				& +\cos \left(\frac{2\pi\sqrt{3} \nu}{3}z_2-\delta_3\right) \sin \left(\frac{2\pi(\zeta-\mu)}{3}z_1\right)
			\end{aligned}\right],
	\end{equation}where $\zeta,\mu,\nu,\delta_1,\delta_2$ and $\delta_3$ depend on $m$ and $n$ (see Section \ref{equi}).
	For any $z^0\in\mathcal{T}_{\mathrm{equi}}$, localized eigenfunctions can, again by Euler formula, be seen as elements of $\mathcal{A}_T$, 
	\begin{equation}
		u_{mn}\left(z_1^0+\frac{z_1}{\sqrt{\lambda}},z_2^0+\frac{z_2}{\sqrt{\lambda}}\right)\in	\mathcal{A}_T=\left\{\sum_{j=1}^{T}\alpha_je^{\theta_j\cdot(z_1,z_2)}:\ \alpha_j\in\mathbb{C}, \theta_j\in\mathbb{R}^2,|\theta_j|=1\right\},
	\end{equation}with now $T\geq 24$.

	To finish the proof for both cases, we notice that the set $\mathcal{A}_T|_B$ is a linear subspace of real dimension $3T$ of the space $L^2(B)$, and that $\MW^s|_B\subset L^2(B)$ is an infinite dimensional linear space. Notice also that, for any $\varphi\in\MW^s$,  and for $\mathcal{P}\in\left\{\mathcal{Q}_1,\mathcal{T}_\mathrm{equi}\right\}$, 
	\begin{equation}
		\inf_{z^0\in\mathcal{P}}\inf_{\lambda\in\Lambda_{\mathcal{P}}^{\mathrm{R}_\Sigma}}\inf_{u_\lambda \in\mathcal{V}_{\mathcal{P},\mathrm{R}_\Sigma}^\lambda}\left\|\varphi-u_\lambda\left(z^0+\frac{\cdot}{\sqrt{\lambda}}\right)\right\|_{C^0\left(B\right)}\geq C_B\dist_{L^2(B)}\left(\varphi,\mathcal{A}_T\right),
	\end{equation}
    for some constant $C_B$  that only depends on $B$. Therefore, \begin{equation}(\MW^s\backslash\mathcal{N}_{\mathcal{P},\mathrm{R}_\Sigma})|_B\subset\mathcal{A}_T|_B\end{equation} for $\mathcal{P}\in\left\{\mathcal{Q}_1,\mathcal{T}_\mathrm{equi}\right\}$ and $\Sigma$ small enough.  Consequently, being $\mathcal A_T$ finite dimensional,  $\mathcal{N}_{\mathcal{P},\mathrm{R}_\Sigma}$ is dense in $\MW^s$. To conclude, we see that $\mathcal{N}_{\mathcal{P},\mathrm{R}_\Sigma}$ is open.
	
	To do that, we just consider a $\varphi\in\mathcal{N}_{\mathcal{P},\mathrm{R}_\Sigma}$. This means that there exists $\delta>0$ such that \begin{equation}		\inf_{z^0\in\mathcal{P}}\inf_{\lambda\in\Lambda_{\mathcal{P}}^{\mathrm{R}_\Sigma}}\inf_{u_\lambda \in\mathcal{V}_{\mathcal{P},\mathrm{R}_\Sigma}^\lambda}\left\|\varphi-u_\lambda\left(z^0+\frac{\cdot}{\sqrt{\lambda}}\right)\right\|_{C^0\left(B\right)}>\delta>0.\end{equation}If we consider those $\psi\in\MW^s$ such that $\left\|\varphi-\psi\right\|<\delta/2$, we can easily check that it is true the following
	\begin{equation}
		\left\|\varphi-\psi\right\|\geq \left\|\varphi-\psi\right\|	_{C^0\left(B\right)},
	\end{equation}and so \begin{equation}\begin{aligned}	&	\inf_{z^0\in\mathcal{P}}\inf_{\lambda\in\Lambda_{\mathcal{P}}^{\mathrm{R}_\Sigma}}\inf_{u_\lambda \in\mathcal{V}_{\mathcal{P},\mathrm{R}_\Sigma}^\lambda}\left\|\psi-u_\lambda\left(z^0+\frac{\cdot}{\sqrt{\lambda}}\right)\right\|_{C^0\left(B\right)}\geq\\&\inf_{z^0\in\mathcal{P}}\inf_{\lambda\in\Lambda_{\mathcal{P}}^{\mathrm{R}_\Sigma}}\inf_{u_\lambda \in\mathcal{V}_{\mathcal{P},\mathrm{R}_\Sigma}^\lambda}\left\|\varphi-u_\lambda\left(z^0+\frac{\cdot}{\sqrt{\lambda}}\right)\right\|_{C^0\left(B\right)}-\delta/2>\delta/2.\end{aligned}\end{equation}Therefore, $B_{\left\|\cdot\right\|}\left(\varphi,\delta/2\right)\subset\mathcal{N}_{\mathcal{P},\mathrm{R}_\Sigma}$, proving that it is open thus finishing the proof.
	
\end{proof}
\subsection{Generic parallelepipeds}\label{Nopara}
The next negative result of inverse localization that we prove shows that for typical parallelepipeds in $\RR^d$ with Dirichlet or Neumann boundary conditions the property fails strongly. More precisely, we study all the irrational rectangle like domains, i.e., orthogonal parallelepipeds with at least one of their sides not being a square root of a rational number. In this section we then prove the negative part of Theorem \ref{BT.polygons} where we see that inverse localization fails strongly in the sense of Definition \ref{NOIL}.
\begin{theorem}
For any irrational rectangle $\mathcal{Q}_l\subset\RR^d$, i.e., at least one $l_j^2\notin\mathbb{Q}$ for $1\leq j\leq d-1$, with Dirichlet or Neumann boundary conditions and any ball $B\subset\RR^d$, the inverse localization property fails strongly. 
\end{theorem}\begin{proof}
	
	We start by considering the jet space of order $3$ at the point $0$ of solutions to the Helmholtz equation \begin{equation}
		\widetilde{\mathcal{H}}:=\left\{\left(\varphi(0), \left(\partial_i\varphi(0)\right)_{1\leq i\leq d},\left(\partial_{i}\partial_{j}\varphi(0)\right)_{1\leq i,j\leq d},\left(\partial_{i}\partial_{j}\partial_{l}\varphi(0)\right)_{1\leq i, j, l\leq d}\right): \varphi\in \MW^s\right\}.
	\end{equation}This set is well defined since $\MW^s\subset C^\infty\left(\mathbb{R}^d\right)$. It is easy to check that $\widetilde{\mathcal{H}}\subset\mathbb{R}^{1+d+d^2+d^3}$ is a linear space and its dimension satisfies that $\dim\left(\widetilde{\mathcal{H}}\right)\leq d_H:=\binom{d+2}{3} +\binom{d+1}{2}$. This can be seen by the fact that, thanks to the equation, we have \begin{equation}
		\varphi(0)=\sum_{j=1}^d\partial_{j}^2\varphi(0),
	\end{equation} and also, for any $1\leq i\leq d$,\begin{equation}
		\partial_{i}\varphi(0)=\sum_{j=1}^d\partial_{j}^2\partial_{i}\varphi(0).
	\end{equation}Therefore, the four first entries of each vector are determined by the others. Moreover, since $\varphi$ is smooth, crossed derivatives are symmetric, and so many other entries are already determined. Consequently, we can find an isomorphism between $\widetilde{\mathcal{H}}$ and a new space $\mathcal{H}$ given by\begin{equation}
		\mathcal{H}:=\left\{\left(\left(\partial_{i}\partial_{j}\varphi(0)\right)_{1\leq i\leq j\leq d},\left(\partial_{i}\partial_{j}\partial_{l}\varphi(0)\right)_{1\leq i\leq j\leq l\leq d}\right): \varphi\in \MW^s\right\}\subset\mathbb{R}^{d_H}.
	\end{equation}
	
	Therefore, if we define the function \begin{equation}
		\begin{aligned}
			\mathcal{J}:&\MW^s\longrightarrow \mathbb{R}^{d_H}\\&\varphi\mapsto\left(\left(\partial_{z_i}\partial_{z_j}\varphi(0)\right)_{1\leq i\leq j\leq d},\left(\partial_{z_i}\partial_{z_j}\partial_{z_l}\varphi(0)\right)_{1\leq i\leq j\leq l\leq d}\right),
		\end{aligned}
	\end{equation}then $\mathcal{H}=\mathcal{J}\left(\MW^s\right)$. Notice that this is a map  that is continuous and linear. 
	
	To continue with the proof, we recall the next lemma that follows from \cite[Theorem 3.2]{Damon}:\begin{lemma}\label{Damon}
		For any $p\in \mathbb{R}^{d_H}$, there exist $\delta_p>0$, a ball $B_{\delta_p}\subset\mathbb{R}^d$ and a function $\phi\in C^\infty\left(B_{\delta_p}\right)$ such that $\Delta \phi+\phi=0$ in $B_{\delta_p}$ and \begin{equation}
			p=\left(\left(\partial_{z_i}\partial_{z_j}\phi(0)\right)_{1\leq i\leq j\leq d},\left(\partial_{z_i}\partial_{z_j}\partial_{z_l}\phi(0)\right)_{1\leq i\leq j\leq l\leq d}\right).
		\end{equation}
	\end{lemma}In other words, for any $p\in\RR^{d_H}$, there exists a locally defined monochromatic wave $\phi$ such that $\mathcal{J}(\phi)=p$. Moreover, the global approximation theory with decay for solutions to Helmholtz equation (see \cite[Chapter 1]{Mangeles}) allows us to find another monochromatic wave $\varphi\in\MW^s$, globally defined, as close as we want in norm $\left\|\cdot\right\|_{C^3\left(B\right)}$ to $\phi$. Therefore, since by classical elliptic regularity bound, the norms $\left\|\cdot\right\|_{C^3\left(B\right)}$ and $\left\|\cdot\right\|$ as defined in \eqref{norm} are equivalent, we have $\overline{\mathcal{H}}=\mathbb{R}^{d_H}$.
	
	 However, since $\mathcal{H}=\mathcal{J}\left(\MW^s\right)$ is a linear subspace of a finite dimensional linear space, it is closed and we infer that $\mathcal{J}$ is surjective, i.e., $\mathcal H=\mathbb R^{d_H}$.
	
	Now we are ready to show that $\mathcal{N}_{\mathcal{Q}_l,\mathrm{BC}}$, for $\mathrm{BC}=\mathrm{D}$ or $\mathrm{BC}=\mathrm{N}$, is open. To do that, we just consider $\varphi\in\mathcal{N}_{\mathcal{Q}_l,\mathrm{BC}}$. This means that there exists $\delta>0$ such that \begin{equation}	\inf_{z^0\in\mathcal{Q}_l}\inf_{\lambda\in\Lambda_{\mathcal{Q}_l}^{\mathrm{BC}}}\inf_{u_\lambda \in\mathcal{V}_{\mathcal{Q}_l,\mathrm{BC}}^\lambda}\left\|\varphi-u_\lambda\left(z^0+\frac{\cdot}{\sqrt{\lambda}}\right)\right\|_{C^0\left(B\right)}>\delta>0.\end{equation}If we consider those $\psi\in\MW^s$ such that $\left\|\varphi-\psi\right\|<\delta/2$, we can easily check that it is true the following
	\begin{equation}
		\left\|\varphi-\psi\right\|\geq \left\|\varphi-\psi\right\|	_{C^0\left(B\right)},
	\end{equation}and so \begin{equation}\begin{aligned}	&\inf_{z^0\in\mathcal{Q}_l}\inf_{\lambda\in\Lambda_{\mathcal{Q}_l}^{\mathrm{BC}}}\inf_{u_\lambda \in\mathcal{V}_{\mathcal{Q}_l,\mathrm{BC}}^\lambda}\left\|\psi-u_\lambda\left(z^0+\frac{\cdot}{\sqrt{\lambda}}\right)\right\|_{C^0\left(B\right)}\geq\\&	\inf_{z^0\in\mathcal{Q}_l}\inf_{\lambda\in\Lambda_{\mathcal{Q}_l}^{\mathrm{BC}}}\inf_{u_\lambda \in\mathcal{V}_{\mathcal{Q}_l,\mathrm{BC}}^\lambda}\left\|\varphi-u_\lambda\left(z^0+\frac{\cdot}{\sqrt{\lambda}}\right)\right\|_{C^0\left(B\right)}-\delta/2>\delta/2.\end{aligned}\end{equation}Therefore, $B_{\left\|\cdot\right\|}\left(\varphi,\delta/2\right)\subset\mathcal{N}_{\mathcal{Q}_l,\mathrm{BC}}$, and we conclude that it is open.
	
	To see the density of $\mathcal{N}_{\mathcal{Q}_l,\mathrm{BC}}$, we first prove the following lemma:\begin{lemma}
		The set\begin{equation}
			\mathcal{V}:=\left\{\begin{aligned}&\left(\left(\partial_{i}\partial_{j}f(0)\right)_{1\leq i\leq j\leq d},\left(\partial_{i}\partial_{j}\partial_{l}f(0)\right)_{1\leq i\leq j\leq l\leq d}\right),
				\\& f=u_\lambda\left(z^0+\frac{\cdot}{\sqrt{\lambda}}\right),\ z^0\in\mathcal{Q}_l,\ \lambda \in \Lambda_{\mathcal{Q}_l}^{\mathrm{BC}},\ u_\lambda\in\mathcal{V}_{\mathcal{Q}_l,\mathrm{BC}}^\lambda
			\end{aligned}\right\}\subset\mathbb{R}^{d_H}
		\end{equation} satisfies that its closure has empty interior. 
	\end{lemma}
	\begin{proof}
		Let us start by fixing an eigenvalue $\lambda>0$. For now, we will assume that $\lambda=\sum_{r=1}^m\beta_r\lambda_r$ with $\lambda_1\neq 0$. We have already seen in Section \ref{rect} that any eigenfunction associated to $\lambda$ with Dirichlet, Neumann or periodic boundary condition has the form 
		\begin{equation}
			u_\lambda(z)=\sum_{\substack{N\in\mathbb{Z}^d\\ \mathcal{Q}_l(N)=\lambda}}c_N\exp\left(i\pi\sum_{j=1}^d\frac{N_jz_j}{l_j}\right), \ c_N\in\mathbb{C},\text{ with } \overline{c_N}=c_N.
		\end{equation}For Dirichlet or Neumann conditions we have some extra hypothesis on $c_N$. 
        
        We define the set \begin{equation}
			\mathcal{V}_\lambda:=\left\{\left(\left(\partial_{i}\partial_{j}f(0)\right)_{1\leq i\leq j\leq d},\left(\partial_{i}\partial_{j}\partial_{l}f(0)\right)_{1\leq i\leq j\leq l\leq d}\right),\begin{aligned}
				& f=u_\lambda\left(z^0+\frac{\cdot}{\sqrt{\lambda}}\right)\\& u_\lambda\in\mathcal{V}_{\mathcal{Q}_l,\mathrm{BC}}^\lambda,\ z^0\in\mathcal{Q}_l
			\end{aligned}\right\}.
		\end{equation}Since we saw thanks to Lemma \ref{Damon} that $\mathcal{H}=\mathbb{R}^{d_H}$, we immediately have that   $\mathcal{V}_\lambda\subset\mathcal{H}$. By Lemma \ref{independence}, we also have that for any $u_\lambda\in\mathcal{V}_{\mathcal{Q}_l,\mathrm{BC}}^\lambda$, \begin{equation}
			\sum_{i\in I_1}\partial_{i}^2u_\lambda(z)+\lambda_1u_\lambda(z)=0.
		\end{equation} 
		Together with the Helmholtz equation and the hypothesis of $\lambda_1\neq 0$, we get \begin{equation}
			\left(1-\frac{\lambda}{\lambda_1}\right)\sum_{i\in I_1}\partial_{i}^2u_\lambda\left(z^0+\frac{z}{\sqrt{\lambda}}\right)+\sum_{i\notin I_1}\partial_{i}^2u_\lambda\left(z^0+\frac{z}{\sqrt{\lambda}}\right)=0;
		\end{equation}and, taking derivatives, for any $1\leq j\leq d$, \begin{equation}
			\left(1-\frac{\lambda}{\lambda_1}\right)\sum_{i\in I_1}\partial_{i}^2\partial_{j}u_\lambda\left(z^0+\frac{z}{\sqrt{\lambda}}\right)+\sum_{i\notin I_1}\partial_{i}^2\partial_{j}u_\lambda\left(z^0+\frac{z}{\sqrt{\lambda}}\right)=0.
		\end{equation}So, the set $\mathcal{V}_\lambda$ satisfies \begin{equation}
			\mathcal{V}_\lambda=\left\{\begin{aligned}&\left(\left(\partial_{i}\partial_{j}f(0)\right)_{1\leq i\leq j\leq d},\left(\partial_{i}\partial_{j}\partial_{l}f(0)\right)_{1\leq i\leq j\leq l\leq d}\right),\begin{aligned}
					& f=u_\lambda\left(z^0+\frac{\cdot}{\sqrt{\lambda}}\right)\\& u_\lambda\in \mathcal{V}_{\mathcal{Q}_l,\mathrm{BC}}^\lambda,\ z^0\in\mathcal{Q}_l
				\end{aligned},\\&
				\left(1-\frac{\lambda}{\lambda_1}\right)\sum_{i\in I_1}\partial_{i}^2f(0)+\sum_{i\notin I_1}\partial_{i}^2f(0)=0,\\&
				\left(1-\frac{\lambda}{\lambda_1}\right)\sum_{i\in I_1}\partial_{i}^2\partial_{j}f(0)+\sum_{i\notin I_1}\partial_{i}^2\partial_{j}f(0)=0,\text{ for any $1\leq j\leq d$.}
			\end{aligned}\right\}.
		\end{equation}
		Doing some computations in the equations that define $\mathcal{V}_\lambda$, we can define a new set \begin{equation}
			\mathcal{A}_1=\left\{\begin{aligned}&\left(\left(\partial_{i}\partial_{j}f(0)\right)_{1\leq i\leq j\leq d},\left(\partial_{i}\partial_{j}\partial_{k}f(0)\right)_{1\leq i\leq j\leq k\leq d}\right),\text{ such that $\forall\ 1\leq j\leq d$}\\& \left(\sum_{i\in I_1}\partial_{i}^2\partial_{j}f(0)\right)\left(\sum_{i\notin I_1}\partial_{i}^2f(0)\right)=\left(\sum_{i\notin I_1}\partial_{i}^2\partial_{j}f(0)\right)\left(\sum_{i\in I_1}\partial_{i}^2f(0)\right)
			\end{aligned}\right\}.
		\end{equation}
		Since $\mathcal{A}_1$ does not depend on $\lambda$, we have the following inclusion. 
		\begin{equation}
			\bigcup_{\lambda>0, \lambda_1\neq 0}\mathcal{V}_\lambda\subset\mathcal{A}_1.
		\end{equation}
		Since any $\lambda>0$ needs to have some $\lambda_r\neq 0$ and we can do the same reasoning for that one, we have\begin{equation}
			\mathcal{V}=	\bigcup_{\lambda>0}\mathcal{V}_\lambda=\bigcup_{r=1}^m	\bigcup_{\lambda>0,\lambda_r\neq 0}\mathcal{V}_\lambda\subset\bigcup_{r=1}^m\mathcal{A}_r.
		\end{equation}It is immediate to prove that $\bigcup_{r=1}^m\mathcal{A}_r$ is a subset of a closed submanifold $\mathcal{A}\subset\mathbb{R}^{d_H}$ of codimension bigger or equal than one, and consequently, it is closed and has empty interior. One can now just notice that  $\overline{\mathcal{V}}$ is a subset of $\mathcal{A}$ and, consequently, it also has empty interior in $\mathbb R^{d_H}$, as we wanted to see. 	\end{proof}
	
	To finish the proof we notice that, for any $\varphi\in\MW^s$,  we have \begin{equation}
	\begin{aligned}
		&\inf_{z^0\in \mathcal{Q}_l}\inf_{\lambda\in\Lambda_{\mathcal{Q}_l}^{\mathrm{BC}}}\inf_{u_\lambda \in\mathcal{V}_{\mathcal{Q}_l,\mathrm{BC}}^\lambda}\left\|\varphi-u_\lambda\left(z^0+\frac{\cdot}{\sqrt{\lambda}}\right)\right\|_{C^0\left(B\right)}\geq\\&C \inf_{z^0\in \mathcal{Q}_l}\inf_{\lambda\in\Lambda_{\mathcal{Q}_l}^{\mathrm{BC}}}\inf_{u_\lambda \in\mathcal{V}_{\mathcal{Q}_l,\mathrm{BC}}^\lambda}\left\|\varphi-u_\lambda\left(z^0+\frac{\cdot}{\sqrt{\lambda}}\right)\right\|_{C^3\left(B/2\right)}\geq\\&C \inf_{z^0\in \mathcal{Q}_l}\inf_{\lambda\in\Lambda_{\mathcal{Q}_l}^{\mathrm{BC}}}\inf_{u_\lambda \in\mathcal{V}_{\mathcal{Q}_l,\mathrm{BC}}^\lambda}\left|\mathcal{J}\left(\varphi\right)-\mathcal{J}^*\left(u_\lambda\left(z^0+\frac{\cdot}{\sqrt{\lambda}}\right)\right)\right|\geq C \dist\left(\mathcal{J}\left(\varphi\right),\mathcal{V}\right).
	\end{aligned}
	\end{equation}Here $\mathcal{J}^*$ is just the extension of $\mathcal{J}$ to the set $C^3\left(\left\{0\right\},\mathbb{R}\right)$ of three times differentiable functions defined in a neighborhood of $0$. 
    
    In the first inequality we have used again elliptic regularity estimates. Therefore, we have $\left(\mathcal{N}_{\mathcal{Q}_l,\mathrm{BC}}\right)^c\subset\mathcal{J}^{-1}\left(\overline{\mathcal{V}}\right)$. Since $\mathcal{J}$ is open and continuous, we have that $\inte\left(\mathcal{J}^{-1}\left(\overline{\mathcal{V}}\right)\right)=\mathcal{J}^{-1}\left(\inte\left(\overline{\mathcal{V}}\right)\right)=\mathcal{J}^{-1}\left(\emptyset\right)=\emptyset$. Accordingly, $\inte\left(\left(\mathcal{N}_{\mathcal{Q}_l,\mathrm{BC}}\right)^c\right)=\emptyset$ and so $\mathcal{N}_{\mathcal{Q}_l,\mathrm{BC}}$ is dense in $\MW^s$. This concludes the proof of the Theorem.
\end{proof} 
\subsection{Balls and generic ellipses}\label{Noelli}
In this last section of the chapter we study the lack of inverse localization in elliptical billiards with Dirichlet or Neumann boundary conditions. Once more we are able to prove that, for a generic ellipse, the inverse localization property fails strongly when considering these boundary conditions. More precisely, the result that we get works for any ellipse $\mathcal{E}_b$ except for those with $b$ in a countable subset $\mathcal{C}\subset]0,1]$. We have no more information about the set $\mathcal{C}$ except for the fact that $1\notin\mathcal{C}$, i.e., in any ball $\mathcal{B}\subset\mathbb{R}^2$ with Dirichlet or Neumann conditions the inverse localization property fails strongly. Moreover, we prove this to be also true for any ball in higher dimensions $\mathcal{B}\subset\mathbb{R}^{d}$, $d\geq 2$.
\subsubsection{No inverse localization in balls}\label{NOILBalls}
We start by proving the negative results of inverse localization on balls. Since translations and rescalings do not change the spectral properties of a domain, we can restrict ourselves to the case in which $\mathcal{B}$ is the ball of radius one centered at the origin. 
\begin{theorem}
	Let $\mathcal{B}\subset\mathbb{R}^d$ be the centered unit ball, $d\geq2$, and let $B\subset\mathbb{R}^d$ be any ball. If we consider Dirichlet or Neumann boundary conditions, then the inverse localization property fails strongly. 
\end{theorem}

\begin{proof}	 
	To begin, we recall the definition of the set \begin{equation}
		\mathcal{N}_{\mathcal B,\mathrm{BC}}=\left\{\varphi\in \MW^s,\ \inf_{ z^0\in\mathcal B}\inf_{\lambda\in \Lambda_\mathcal{B}^{\mathrm{BC}}}\inf_{u_\lambda\in\mathcal{V}_{\mathcal{B},\mathrm{BC}}^\lambda}
       \left\|\varphi-u_\lambda\left(z^0+\frac{\cdot}{\sqrt{\lambda}}\right)\right\|_{C^0\left(B\right)}>0\right\},
	\end{equation} 
   We also recall from Section 2 that eigenfunctions for the Dirichlet and Neumann problem have the same formal expression and they differ only in the eigenvalues, therefore as far as we consider eigenvalues $\lambda_{ln}$ and we do not use particular properties of them, we can again address both problems simultaneously.

    
    We now split the proof in four steps: 
     
	\textbf{	\underline{Step I:}} In this first step, we will see that $\mathcal{N}_{\mathcal{B},\mathrm{BC}}$ is open. To do that, we do just like in the previous section and consider a $\varphi\in\mathcal{N}_{\mathcal{B},\mathrm{BC}}$. This means that there exists $\delta>0$ such that \begin{equation}	\inf_{ z^0\in\mathcal B}\inf_{\lambda\in\Lambda_{\mathcal{B}}^{\mathrm{BC}}}\inf_{u_\lambda \in\mathcal{V}_{\mathcal{B},\mathrm{BC}}^\lambda}\left\|\varphi-u_\lambda\left(z^0+\frac{\cdot}{\sqrt{\lambda}}\right)\right\|_{C^0\left(B\right)}>\delta>0.\end{equation}If we consider those $\psi\in\MW^s$ such that $\left\|\varphi-\psi\right\|_{C^0\left(B\right)}<\delta/2$, we can easily check that it is true the following \begin{equation}\begin{aligned}	&\inf_{ z^0\in\mathcal B}\inf_{\lambda\in\Lambda_{\mathcal{B}}^{\mathrm{BC}}}\inf_{u_\lambda \in\mathcal{V}_{\mathcal{B},\mathrm{BC}}^\lambda}\left\|\psi-u_\lambda\left(z^0+\frac{\cdot}{\sqrt{\lambda}}\right)\right\|_{C^0\left(B\right)}\geq\\&	\inf_{ z^0\in\mathcal B}\inf_{\lambda\in\Lambda_{\mathcal{B}}^{\mathrm{BC}}}\inf_{u_\lambda \in\mathcal{V}_{\mathcal{B},\mathrm{BC}}^\lambda}\left\|\varphi-u_\lambda\left(z^0+\frac{\cdot}{\sqrt{\lambda}}\right)\right\|_{C^0\left(B\right)}-\delta/2>\delta/2.\end{aligned}\end{equation}So, $B_{C^0\left(B\right)}\left(\varphi,\delta/2\right)\subset\mathcal{N}_{\mathcal{B},\mathrm{BC}}$, and we conclude that it is open. In the rest of the proof, we will see that this set is dense in $\MW^s$. 
    
    
    \textbf{	\underline{Step II:}} In this step, we will introduce some notations and maps that we will need in the rest of the proof, and we will comment some of their properties.

We start by  considering a point $z^0=\left(r_0,\theta_0\right)\in[0,1]\times\mathbb{S}^{d-1}$. For any $0<R<1$ that we will fix later, we associate to any $r_0\in[0,1]$ a smooth map as follows\begin{equation}
       \begin{aligned}
          I: &[0,1]\rightarrow C^\infty\left([0,R],\mathbb{R}\right)\\
           &\begin{aligned}r_0\mapsto 
              I_{r_0}: [0,R]&\rightarrow\mathbb{R}\\
              r&\mapsto r_0\left(1-R\right)+r.
           \end{aligned}
       \end{aligned}
   \end{equation}Notice that $I_{r_0}\left([0,R]\right)=\left[r_0\left(1-R\right),r_0\left(1-R\right)+R\right]\subset [0,1]$ is always an interval of length $R$ and notice that $I$ and $I_{r_0}$ are continuous maps. 

   We can then define, for any pair $\left(r_0,\theta_0\right)\in[0,1]\times\mathbb{S}^{d-1}$, the following  map: \begin{equation}
       \begin{aligned}
           R_{\theta_0}^{r_0}:&\MW^s\rightarrow C^2\left([0,R]\right)\\&\varphi\mapsto\varphi\left(I_{r_0}\left(\cdot\right),\theta_0\right).
       \end{aligned}
   \end{equation}This map is the restriction (of a monochromatic wave) to a segment of length $R$ around $\left(r_0,\theta_0\right)$ in the direction of $\theta_0$. It is easy to notice that any $  R_{\theta_0}^{r_0}$ is continuous and so it is the map \begin{equation}
       \begin{aligned}
           R:&[0,1]\times\mathbb{S}^{d-1}\rightarrow C\left(\MW^s,C^2\left([0,R]\right)\right)\\&\left(r_0,\theta_0\right)\mapsto R_{\theta_0}^{r_0}:=R(r_0,\theta_0).
       \end{aligned}
   \end{equation}
We also define the map  $\tilde{R}_{\theta_0}^{r_0}$ as the extension of $R_{\theta_0}^{r_0}$ to the set of two times differentiable functions in a neighborhood of the segment \begin{equation}L_{\theta_0}^{r_0}:=\left\{(r,\theta_0)\in\mathcal{B},\ r\in I_{r_0}([0,R])\right\}.\end{equation} We then use the explicit expression of the eigenfunctions to notice that any localized eigenfunction associated to an eigenvalue $\lambda_{ln}$, defined as $\tilde{u}_{ln}^{z^0}:=u_{ln}\left(z^0+\frac{\cdot}{\sqrt{\lambda_{ln}}}\right)$ when $z^0=\left(r_0,\theta_0\right)$, has the following form: \begin{equation}
		\begin{aligned}
			& \tilde{R}_{\theta_0}^{r_0}\left(\tilde{u}_{ln}^{z^0}\right)(r)=u_{ln}\left(r_0+\frac{I_{r_0}(r) }{\sqrt{\lambda_{ln}}},\theta_0\right)=\\&\sum_{m=1}^{M(d,l)}c_m\left(r_0+\frac{r_0\left(1-R\right)+r}{\sqrt{\lambda_{ln}}}\right)^{1-d/2}J_{\frac{d}{2}+l-1}\left(\sqrt{\lambda_{ln}}r_0+r_0\left(1-R\right)+r\right)Y_{lm}\left(\theta_0\right)=\\&c_{l,n,z^0}\left(\sqrt{\lambda_{ln}}r_0+r_0\left(1-R\right)+r\right)^{1-d/2}J_{d/2+l-1}\left(\sqrt{\lambda_{ln}}r_0+r_0\left(1-R\right)+r\right)=\\&c_{l,n,z^0}f_{\sqrt{\lambda_{ln}}r_0+r_0\left(1-R\right),l}(r)
            \,,
		\end{aligned}
	\end{equation}
	with $r\in [0,R]$ and $f_{t,l}(r):=\left(t+r\right)^{1-d/2}J_{d/2+l-1}\left(t+r\right)$. A straightforward computation allows us to check that, for any $t\in\mathbb{R}^+_0$ and for any $l\in\mathbb{N}\cup\left\{0\right\}$, the function $f_{t,l}$ satisfies the following Bessel type differential equation on the whole $\mathbb{R}$:
	\begin{equation}
		\left(t+r\right)^2 f_{t,l}''(r)+(d-1)(t+r)f_{t,l}'(r)+\left(\left(t+r\right)^2-l(l+d-2)\right)f_{t,l}(r)=0.
	\end{equation}
	We define now the evaluation map as follows:
	\begin{equation}
		\begin{aligned}
			& \mathcal{EM}^M: C^2\left(\left[0,R\right]\right)\rightarrow\mathbb{R}^{12M}\\&f\mapsto [\mathcal{EM}^M(f)]_i=\left(\begin{array}{c c c}
				f\left(0+\frac{iR}{M}\right) & f'\left(0+\frac{iR}{M}\right) & f''\left(0+\frac{iR}{M}\right) \\ f\left(\frac{R}{4M}+\frac{iR}{M}\right) &
				f'\left(\frac{R}{4M}+\frac{iR}{M}\right) & f''\left(\frac{R}{4M}+\frac{iR}{M}\right) \\ f\left(\frac{R}{2M}+\frac{iR}{M}\right) & f'\left(\frac{R}{2M}+\frac{iR}{M}\right) &
				f''\left(\frac{R}{2M}+\frac{iR}{M}\right)	\\ f\left(\frac{3R}{4M}+\frac{iR}{M}\right) & f'\left(\frac{3R}{4M}+\frac{iR}{M}\right) & f''\left(\frac{3R}{4M}+\frac{iR}{M}\right) \\
			\end{array}\right)
		\end{aligned}
	\end{equation}
    for $i=0,1,\cdots, M-1$.
	Here $M$ is an integer that will be fixed later. This map is easily checked to be continuous and linear.	On the other hand, by interpolation, we know that for any $P\in \mathbb{R}^{12M}$ there exists a polynomial $\mathbb{P}(P)\in C^2\left([0,R]\right) $ such that $\mathcal{EM}^M\left(\mathbb{P}(P)\right)=P$. Therefore, $\mathcal{EM}^M$ is also surjective and open. For a point $P=\left(P_1,\ldots,P_{12M}\right)$ let us denote the coordinates as: \begin{equation}
	    \begin{aligned}
	        &P_{j}=:P^0_j,\text{ whenever }1\leq j\leq 12,\\
            &P_{j}=:P^1_{j-12},\text{ whenever }13\leq j\leq 24=12\cdot 2,\\
            &\ldots\\
            &P_{j}=:P^{M-1}_{j-12(M-1)},\text{ whenever }12(M-1)+1\leq j\leq 12M.\\
             \end{aligned}
	\end{equation}For any $0\leq i\leq M-1$, we denote $P^i:=\left(P^i_1,\ldots,P^i_{12}\right)\in\mathbb{R}^{12}.$ We also define the following six functions from $\mathbb{R}^{12}$ (let us say $P\in\mathbb{R}^{12}$) to $\mathbb{R}$. Notice that some of them depend on $M$ and $0\leq j\leq M-1$ through the parameter $\al_j^M:=\frac{jR}{M}$.
    
	\begin{equation}
    \begin{aligned}
        &a(P)=P_{3}P_{4}-P_{6}P_{1},\ d(P)=P_{9}P_{10}-P_{12}P_{7},\\ 
        &b^M_j(P)=2\al_j^M(P_{3}P_{4}+P_{1}P_{4})+(d-1)(P_{2}P_{4}-P_{5}P_{1})\\&-\left(2\al_j^M+\frac{R}{2}\right)(P_{6}P_{1}+P_{4}P_{1}),\\
        &c^M_j(P)=(\al_j^M)^2(P_{3}P_{4}+P_{1}P_{4})+\al_j^M(d-1)P_{2}P_{4}\\
        &-\left(\al_j^M+\frac{R}{2}\right)(d-1)P_{5}P_{1}-\left(\al_j^M+\frac{R}{4}\right)^2(P_{6}P_{1}+P_{4}P_{1}),\\
        &e^M_j(P)=\left(2\al_j^M+R\right)(P_{9}P_{10}+P_{7}P_{10})+(d-1)(P_{8}P_{10}-P_{11}P_{7})\\
        &-\left(2\al_j^M+\frac{3R}{2}\right)(P_{12}P_{7}+P_{10}P_{7}),\\
        &f^M_j(P)=\left(\al_j^M+\frac{R}{2}\right)^2(P_{9}P_{10}+P_{7}P_{10})+\left(\al_j^M+\frac{R}{2}\right)(d-1)P_{8}P_{10}\\
        &-\left(\al_j^M+\frac{3R}{4}\right)(d-1)P_{11}P_{7}-\left(\al_j^M+\frac{3R}{4}\right)^2(P_{12}P_{7}+P_{10}P_{7}).
    \end{aligned}
\end{equation}All of them are homogeneous polynomials of degree two.
	We define now the polynomials:\begin{equation}
    \begin{aligned}
        &\tilde{Q}^M_j(P):=\left(b^M_j(P)^2d(P)-2a(P)c^M_j(P)d(P)-a(P)b^M_j(P)e(P)+2a(P)^2f^M_j(P)\right)^2\\
        &-\left(a(P)e^M_j(P)-b^M_j(P)d(P)\right)^2\left(b^M_j(P)^2-4a(P)c^M_j(P)\right).
    \end{aligned}
\end{equation}
    With all this, we can finally define the function\begin{equation}
        \begin{aligned}
            &Q_M:\mathbb{R}^{12M}\rightarrow\mathbb{R}^{M}\\
            &P=(P^0,P^1,\ldots,P^{M-1})\mapsto Q_M(P)=\left(\tilde{Q}^M_0\left(P^0\right), \tilde{Q}^M_1\left(P^1\right),\ldots,\tilde{Q}^M_{M-1}\left(P^{M-1}\right)\right).
        \end{aligned}
    \end{equation}
	 
      Notice that $Q_M^{-1}\left(\left\{0\right\}\right)$ is an (affine) algebraic variety (in particular, a closed set with empty interior because $Q_M$ is surjective) and, consequently, $\mathbb{R}^{12M}\backslash Q_M^{-1}\left(\left\{0\right\}\right)$ is open and dense.  Moreover, since any entry of $Q_M$ is evaluated in different coordinates, we notice that all of them are independent and so $Q_M^{-1}\left(\left\{0\right\}\right)$ is of codimension at least $M$ on $\mathbb{R}^{12M}$. 

       On the other hand, it is straightforward to check that, for any $t\in\mathbb{R}^+$ and for any $l\in\mathbb{N}\cup\left\{0\right\}$, $Q_M\left(\mathcal{EM}^M\left(f_{t,l}\right)\right)=0$ and, consequently, for any eigenfunction $u_{ln}$ and any $z^0\in\mathcal{B}$,
	\begin{equation}
		\mathcal{EM}^M\left(\tilde{R}_{\theta_0}^{r_0}\left(\tilde{u}_{ln}^{z^0}\right)\right)\in Q_M^{-1}\left(\left\{0\right\}\right).
	\end{equation}
	
  \textbf{	\underline{Step III:}} To continue the proof, we will see that, for any pair $\left(r_0,\theta_0\right)\in[0,1]\times\mathbb{S}^{d-1}$, the set $\left(\mathcal{EM}^M\circ	R_{\theta_0}^{r_0}\right)^{-1}\left( Q_M^{-1}\left(\left\{0\right\}\right)\right)$ has at least codimension $M$ in $\MW^s$, and in particular it has empty interior and it is closed.
  
  For any $P\in \RR^{12M}$, we consider the polynomial $\mathbb{P}(P)$ introduced before. An easy application of the Cauchy–Kovalevskaya theorem ensures the existence of a neighborhood $N_P$ of $L_{\theta_0}^{r_0}$ in $\mathbb{R}^{d}$ and a function $H_P:N_P\rightarrow\mathbb{R}$ satisfying the Helmholtz equation $\Delta H_P+H_P=0$ in $N_P$ such that $ \tilde{R}_{\theta_0}^{r_0}\left(H_P\right)=\mathbb{P}(P)$. Moreover, thanks to the global approximation theory with decay \cite{APDE}, for any $\ep>0$ there exists another function $\varphi_P\in\MW^s$ such that \begin{equation}
		\left\|\varphi_P-H_P\right\|_{C^2\left(N_P\right)}<\ep.
	\end{equation}If we notice that
	\begin{equation}
		\left\|\varphi_P-H_P\right\|_{C^2\left(N_P\right)}\geq\frac{1}{4M} \left|\mathcal{EM}^M\left( R_{\theta_0}^{r_0}\left(\varphi_P\right)\right)-\mathcal{EM}^M\left(\tilde{ R}_{\theta_0}^{r_0}\left(H_P\right)\right)\right|,
	\end{equation}it follows that for any $P\in\RR^{12M}$ and any $\de>0$ there exists $\varphi_P\in \MW^s$ such that $\left|\mathcal{EM}^M\left( R_{\theta_0}^{r_0}\left(\varphi_P\right)\right)-P\right|<\de.$ Equivalently, $\left(\mathcal{EM}^M\circ R_{\theta_0}^{r_0}\right)\left(\MW^s\right)$ is dense in $\RR^{12M}$.  Notice that the map $\mathcal{EM}^M\circ R_{\theta_0}^{r_0}$ is linear and continuous (between two Banach spaces). 
    
    Therefore, we just proved that the map $\mathcal{EM}^M\circ R_{\theta_0}^{r_0}:\MW^s\rightarrow\mathbb{R}^{12M}$ is a linear and continuous map between two Banach spaces with image dense in $\mathbb{R}^{12M}$ and, therefore, surjective (because all linear subspace of a finite dimensional linear space are closed).  By the open mapping theorem, we can conclude that it is also open and so $\left(\mathcal{EM}^M\circ R_{\theta_0}^{r_0}\right)^{-1}\left(\mathbb{R}^{12M}\backslash Q_M^{-1}\left(\left\{0\right\}\right)\right)=\MW^s\backslash\left(\mathcal{EM}^M\circ R_{\theta_0}^{r_0}\right)^{-1}\left( Q_M^{-1}\left(\left\{0\right\}\right)\right)$ is open and dense in $\MW^s$. 
    
    Taking into account the codimension of $Q_M^{-1}\left(\{0\}\right)$, we can also infer that the set $\left(Q_M\circ\mathcal{EM}^M\circ R_{\theta_0}^{r_0}\right)^{-1}\left(\{0\}\right)$ has codimension at least $M$ in $\MW^s$. Indeed, since the codimension of $Q_M^{-1}\left(\{0\}\right)$ in $\mathbb{R}^{12M}$ is at least $M$ and $\mathcal{EM}^M\circ R^{r_0}_{\theta_0}$ is surjective, we easily conclude that a generic $(M-1)$-parameter family of functions in $\MW^s$ does not intersect $\left(Q_M\circ\mathcal{EM}^M\circ R^{r_0}_{\theta_0}\right)^{-1}\left(\{0\}\right)$, and hence its codimension is at least $M$. Here, generic family means belonging to an open and dense subset of the set of families determined by at most $M-1$ parameters.
	
	 \textbf{	\underline{Step IV:}} The last step in the proof is to consider what happens whenever we change the point $(r_0,\theta_0)\in\mathcal{B}$ around which we are doing the localization.  We first notice that, for any $\varphi\in \MW^s$, and whenever the previously introduced $R>0$ is small enough, it is true that	\begin{equation}
		\begin{aligned}
			&\inf_{z^0\in\mathcal{B}}\inf_{\lambda\in\Lambda_{\mathcal{B}}^{\mathrm{BC}}}\inf_{u_\lambda \in\mathcal{V}_{\mathcal{B},\mathrm{BC}}^\lambda}\left\|\varphi-u_\lambda\left(z^0+\frac{\cdot}{\sqrt{\lambda}}\right)\right\|_{C^0\left(B\right)}\geq\\&C\inf_{z^0\in\mathcal{B}}\inf_{\lambda\in\Lambda_{\mathcal{B}}^{\mathrm{BC}}}\inf_{u_\lambda \in\mathcal{V}_{\mathcal{B},\mathrm{BC}}^\lambda}\left\|\varphi-u_\lambda\left(z^0+\frac{\cdot}{\sqrt{\lambda}}\right)\right\|_{C^2\left(B_R\right)}\geq\\&\frac{C}{4M}\inf_{z^0\in\mathcal{B}}\inf_{\lambda\in\Lambda_{\mathcal{B}}^{\mathrm{BC}}}\inf_{u_\lambda \in\mathcal{V}_{\mathcal{B},\mathrm{BC}}^\lambda}\left|\mathcal{EM}^M\circ R_{\theta_0}^{r_0}\left(\varphi\right)-\mathcal{EM}^M\circ\tilde{R}_{\theta_0}^{r_0}\left(\tilde{u}_\lambda^{z^0}\right)\right|\geq\\ &\frac{C}{4M}\inf_{z^0\in\mathcal{B}} \dist\left(\mathcal{EM}^M\circ R_{\theta_0}^{r_0}(\varphi),Q_M^{-1}\left(\{0\}\right)\right).
		\end{aligned}
	\end{equation}
  Therefore, the set \begin{equation}
      \mathcal{N}_{\mathcal B,\mathrm{BC}}^c:=\left\{\varphi\in \MW^s,\ 0=\inf_{z^0\in\mathcal{B}}\inf_{\lambda\in\Lambda_{\mathcal{B}}^{\mathrm{BC}}}\inf_{u_\lambda \in\mathcal{V}_{\mathcal{B},\mathrm{BC}}^\lambda}\left\|\varphi-u_\lambda\left(z^0+\frac{\cdot}{\sqrt{\lambda}}\right)\right\|_{C^0\left(B\right)}\right\},
  \end{equation}is contained in the set \begin{equation}\mathcal{A}:=\bigcup_{\substack{r_0\in[0,1]\\\theta_0\in\mathbb{S}^{d-1}}}\left(Q_M\circ\mathcal{EM}^M\circ R_{\theta_0}^{r_0}\right)^{-1}\left(\{0\}\right).\end{equation} The set $\mathcal{A}$ is the union of a family of codimension $M$ sets that is parametrized by the pair $(r_0,\theta_0)\in\mathcal{B}$. Since these parameters live in a space of dimension $d$ and the dependence of them is continuos, we can ensure that $\mathcal{A}$ is a subset of $\MW^s$ of codimension at least $M-d$. Taking $M>d$, we can ensure that the set of monochromatics waves $\vp\in\MW^s$ that do not belong to $ \mathcal{N}_{\mathcal B,\mathrm{BC}}^c$ is dense in $\MW^s$.  Finally, we just recall from Step I that it is also open, which completes the proof.\end{proof}

\subsubsection{No inverse localization on generic ellipses}
In this subsection, we study the property of the inverse localization in the family of all ellipses in $\RR^2$ and we manage to see that it fails strongly for almost all monochromatic waves, when considering Dirichlet or Neumann boundary conditions. The proof follows similar arguments as the one in balls.

\begin{theorem}
	Let us fix any ball $B\subset\RR^2$. There exists a countable subset $\mathcal{C}\subset\left]0,1\right]$ so that if $b\notin\mathcal{C}$, then in the ellipse $\mathcal{E}_{b}$ with Dirichlet or Neumann boundary conditions the inverse localization property fails strongly. 
\end{theorem}
\begin{proof}	
	
	We first recall from Section \ref{elli} that eigenfunctions for the Dirichlet and Neumann problem have the same formal expression and that they differ only in the eigenvalues. Therefore as far as we consider eigenvalues $\lambda_{mn}$ and we do not use particular properties of them, we can address both problems simultaneously. We will then write $\mathrm{BC}$ for the dependence on the boundary conditions and it can be $\mathrm{BC}=\mathrm{D}$ or $\mathrm{BC}=\mathrm{N}$. We split the proof in four steps: 
     
	\textbf{	\underline{Step I:}} In this first step, we show that $\mathcal{N}_{\mathcal{E}_b,\mathrm{BC}}$ is open. We proceed exactly as in the  previous section: let us consider $\varphi\in\mathcal{N}_{\mathcal{E}_b,\mathrm{BC}}$. This means that there exists $\delta>0$ such that \begin{equation}	\inf_{ z^0\in\mathcal{E}_b}\inf_{\lambda\in\Lambda_{\mathcal{E}_b}^{\mathrm{BC}}}\inf_{u_\lambda \in\mathcal{V}_{\mathcal{E}_b,\mathrm{BC}}^\lambda}\left\|\varphi-u_\lambda\left(z^0+\frac{\cdot}{\sqrt{\lambda}}\right)\right\|_{C^0\left(B\right)}>\delta.\end{equation}If we consider those $\psi\in\MW^s$ such that $\left\|\varphi-\psi\right\|_{C^0\left(B\right)}<\delta/2$, we can easily check that the following holds \begin{equation}\begin{aligned}	&\inf_{ z^0\in\mathcal{E}_b}\inf_{\lambda\in\Lambda_{\mathcal{E}_b}^{\mathrm{BC}}}\inf_{u_\lambda \in\mathcal{V}_{\mathcal{E}_b,\mathrm{BC}}^\lambda}\left\|\psi-u_\lambda\left(z^0+\frac{\cdot}{\sqrt{\lambda}}\right)\right\|_{C^0\left(B\right)}\geq\\&	\inf_{ z^0\in\mathcal{E}_b}\inf_{\lambda\in\Lambda_{\mathcal{E}_b}^{\mathrm{BC}}}\inf_{u_\lambda \in\mathcal{V}_{\mathcal{E}_b,\mathrm{BC}}^\lambda}\left\|\varphi-u_\lambda\left(z^0+\frac{\cdot}{\sqrt{\lambda}}\right)\right\|_{C^0\left(B\right)}-\delta/2>\delta/2.\end{aligned}\end{equation}So, $B_{C^0\left(B\right)}\left(\varphi,\delta/2\right)\subset\mathcal{N}_{\mathcal{E}_b,\mathrm{BC}}$, and we conclude that it is open. In the rest of the proof, we will see that this set is dense in $\MW^s$. 

      \textbf{	\underline{Step II:}} In this step, we introduce some notations and maps that we will need in the proof, and we will comment on some of their properties.
      
	We first fix a point $z^0\in\mathcal{E}_{b}$ given by the elliptical coordinates $\left(\xi_0,\eta_0\right)\in[0,\xi_b]\times [0,2\pi]$. For any small $0<R<1$ that we will fix later, we associate to any $\xi_0\in[0,\xi_b]$ a smooth map: 
    \begin{equation}
       \begin{aligned}
          I: &[0,\xi_b]\rightarrow C^\infty\left([-R,R],\mathbb{R}\right)\\
           &\begin{aligned}\xi_0\mapsto 
              I_{\xi_0}: [-R,R]&\rightarrow\mathbb{R}\\
              \xi&\mapsto \xi_0\left(1-R\right)+\xi_b\xi.
           \end{aligned}
       \end{aligned}
   \end{equation}Notice that $I_{\xi_0}\left([-R,R]\right)=\left[\xi_0-R\left(\xi_0+\xi_b\right),\xi_0+R\left(\xi_b-\xi_0\right)\right]$ is always an interval of length $2R\xi_b$ contained in $[0,\xi_b]$ for any $\xi_0$ and $R$ small enough, and notice that the maps $I$ and $I_{\xi_0}$ are continuous maps. We can also define, for any pair $\left(\xi_0,\eta_0\right)\in[0,\xi_b]\times[0,2\pi]$, the following  map: \begin{equation}
       \begin{aligned}
           R_{\eta_0}^{\xi_0}:&\MW^s\rightarrow C^2\left([-R,R]\right)\\&\varphi\mapsto\varphi\left(c\cdot\cosh\left(I_{\xi_0}\left(\cdot\right)\right)\cos(\eta_0),c\cdot\sinh\left(I_{\xi_0}\left(\cdot\right)\right)\sin(\eta_0)\right).
       \end{aligned}
   \end{equation}Recall that $c=\sqrt{1-b^2}$ is the eccentricity of the ellipse. This map is the restriction of $\varphi$ to the set: \begin{equation}
		L_{z^0}=\left\{\left(c\cdot\cosh\left(
		\xi\right)\cos\left(\eta_0\right),c\cdot\sinh\left(\xi\right)\sin\left(\eta_0\right)\right),\ \xi\in I_{\xi_0}([-R,R])\right\},
	\end{equation} which is contained in $\mathcal{E}_b$. $L_{z^0}$ is easily checked to be a segment of a straight line (if $\eta_0=\pi/2$ or $\eta_0=3\pi/2$ ) or an arc of a hyperbola (if $\xi_0\neq\pi/2,\ 3\pi/2 $). It is easy to notice that  $  R_{\eta_0}^{\xi_0}$ is continuous and so it is the map \begin{equation}
       \begin{aligned}
           R:&[0,\xi_b]\times[0,2\pi]\rightarrow C\left(\MW^s,C^2\left([-R,R]\right)\right)\\&\left(\xi_0,\eta_0\right)\mapsto R_{\eta_0}^{\xi_0}:=R(\xi_0,\eta_0).
       \end{aligned}
   \end{equation}
   
 We also define the map  $\tilde{R}_{\eta_0}^{\xi_0}$ as the extension of $R_{\eta_0}^{\xi_0}$ to the set of twice differentiable functions in a neighborhood of $L_{z^0}$.  Following the discussion in Section \ref{elli} and taking $b\notin\mathcal{C}$, we notice that any localized eigenfunction associated to an eigenvalue $\lambda_{mn}$, defined as $\tilde{u}_{mn}^{z^0}:=u_{mn}\left(z^0+\frac{\cdot}{\sqrt{\lambda_{ln}}}\right)$ when $z^0=\left(\xi_0,\eta_0\right)$, has the following form: \begin{equation}
		\begin{aligned}
			& \tilde{R}_{\eta_0}^{\xi_0}\left(\tilde{u}_{mn}^{z^0}\right)(\xi)=u_{mn}\left(\xi_0+\frac{I_{\xi_0}(\xi) }{\sqrt{\lambda_{mn}}},\eta_0\right)=\\&d\cdot S^p_m\left(\xi_0+\frac{\xi_0(1-R)+\xi_b\xi}{\sqrt{\lambda_{mn}}},q_{m,n}\right)E^p_n\left(\eta_0,q_{m,n}\right).
		\end{aligned}
	\end{equation}
    Here  $d\in\mathbb{R}$ is a constant, $p\in\left\{e,o\right\}$, $m,n\geq 0$, and we are not including the dependence on the boundary conditions. Some computations (including formulas for the $\cosh$) allow us to check that, for any $s,t\in\mathbb{R}^+_0$ and any $m,n\in\mathbb{N}\cup\{0\}$, the function $f_{s,t,m,n}(\xi):=S_m^p\left(t+\frac{\xi}{s},q_{m,n}\right)$ satisfies the following ordinary differential equation for any $\xi\in\mathbb{R}$:
	\begin{equation}
		s^2f''_{s,t,m,n}(\xi)-\left(\al_{m,n}-2q_{m,n}\cosh\left(2t+\frac{2\xi}{s}\right)\right)f_{s,t,m,n}(\xi)=0,
	\end{equation}for certain $q_{m,n}>0$ and $\al_{m,n}=a_{m,n}$ or $\alpha_{m,n}=b_{m,n}$.
    
	 We define now the evaluation map as follows: 
	\begin{equation}
		\begin{aligned}
			&\mathcal{EM}^M:C^2\left(\left[-R,R\right]\right)\rightarrow\mathbb{R}^{14M}\\&f\mapsto \left(\begin{array}{c c | c }
				f\left(\frac{jR}{M}\right) & f''\left(\frac{jR}{M}\right) &\multirow{4}{*}{$\begin{aligned} f\left(-\frac{R}{8M}-\frac{jR}{M}\right)&   f''\left(-\frac{R}{8M}-\frac{jR}{M}\right)\\
        f\left(-\frac{R}{4M}-\frac{jR}{M}\right)&  f''\left(-\frac{R}{4M}-\frac{jR}{M}\right)\rule{0pt}{18pt}\\  f\left(-\frac{R}{2M}-\frac{jR}{M}\right)&  f''\left(-\frac{R}{2M}-\frac{jR}{M}\right)\rule{0pt}{18pt}\end{aligned}$}\rule{0pt}{15pt} \\ f\left(\frac{R}{8M}+\frac{jR}{M}\right)&  f''\left(\frac{R}{8M}+\frac{jR}{M}\right) \rule{0pt}{15pt}\\f\left(\frac{R}{4M}+\frac{jR}{M}\right)&  f''\left(\frac{R}{4M}+\frac{jR}{M}\right)\rule{0pt}{15pt}\\f\left(\frac{R}{2M}+\frac{jR}{M}\right)&  f''\left(\frac{R}{2M}+\frac{jR}{M}\right)\rule{0pt}{15pt}
			\end{array}\right)_{j=0}^{M-1}
		\end{aligned}
       \end{equation} Here $M$ is an integer that will be fixed later. This map is easily checked to be continuous and linear.	If we proceed as in the previous section, we can see by interpolation, that for any $P\in \mathbb{R}^{14M}$ there exists a polynomial $\mathbb{P}(P)$ such that $\mathcal{EM}^M\left(\mathbb{P}(P)\right)=P$. Therefore, it is surjective and open, because of the open mapping theorem.
	
	For a point $P=\left(P_1,\ldots,P_{14M}\right)$ let us denote the coordinates as: \begin{equation}
	    \begin{aligned}
	        &P_{j}=:P^0_j,\text{ whenever }1\leq j\leq 14,\\
            &P_{j}=:P^1_{j-14},\text{ whenever }15\leq j\leq 28=14\cdot 2,\\
            &\ldots\\
            &P_{j}=:P^{M-1}_{j-14(M-1)},\text{ whenever }14(M-1)+1\leq j\leq 14M.\\
             \end{aligned}
	\end{equation}For any $0\leq i\leq M-1$, we set $P^i:=\left(P^i_1,\ldots,P^i_{14}\right)\in\mathbb{R}^{14}.$ We also define the following six polynomial functions from $\mathbb{R}^{14}$ (let us say $P\in\mathbb{R}^{14}$) to $\mathbb{R}$. Notice that all of them depend on certain parameters: $\alpha$, $q$, $\beta^j$, $a$, that correspond to $\alpha_{m,n}$, $q_{m,n}$, $\cosh\left(2t+\frac{2j}{Ms}\right)$, $\cosh\left(\frac{R}{4sM}\right)$, respectively:
    \begin{equation}
        \begin{aligned}
            &a(P)=\alpha\left(P_{13}P_{2}-P_{1}P_{14}\right)-2q\beta\left(aP_{13}P_{2}-P_{1}P_{14}\right)+2q\sqrt{\beta^2-1}\sqrt{a^2-1}P_{13}P_2,\\  &b(P)=\alpha\left(P_{11}P_{2}-P_{1}P_{12}\right)-2q\beta\left(aP_{11}P_{2}-P_{1}P_{12}\right)-2q\sqrt{\beta^2-1}\sqrt{a^2-1}P_{11}P_2,\\  &c(P)=\alpha\left(P_{9}P_{2}-P_{1}P_{10}\right)-2q\beta\left(\left(2a^2-1\right)P_{9}P_{2}-P_{1}P_{10}\right)\\&+2q\sqrt{\beta^2-1}2a\sqrt{a^2-1}P_{9}P_2,\\  &d(P)=\alpha\left(P_{7}P_{2}-P_{1}P_{8}\right)-2q\beta\left(\left(2a^2-1\right)P_{7}P_{2}-P_{1}P_{8}\right)\\&-2q\sqrt{\beta^2-1}2a\sqrt{a^2-1}P_{7}P_2,\\  &e(P)=\alpha\left(P_{5}P_{2}-P_{1}P_{6}\right)-2q\beta\left(\left(8a^2\left(a^2-1\right)+1\right)P_{5}P_{2}-P_{1}P_{6}\right)\\&+2q\sqrt{\beta^2-1}2\left(2a^2-1\right)\sqrt{\left(a^2-1\right)^2+1}P_{5}P_2,\\  &f(P)=\alpha\left(P_{3}P_{2}-P_{1}P_{4}\right)-2q\beta\left(\left(8a^2\left(a^2-1\right)+1\right)P_{3}P_{2}-P_{1}P_{4}\right)\\&-2q\sqrt{\beta^2-1}2\left(2a^2-1\right)\sqrt{\left(a^2-1\right)^2+1}P_{3}P_2.
        \end{aligned}
    \end{equation}
    
	All of them are homogeneous polynomials of degree two. We define the system of equations given by $a(P^j)=b(P^j)=c(P^j)=d(P^j)=e(P^j)=f(P^j)=0$ for each $0\leq j\leq M-1$ (where the dependence on $j$ appears also in the functions through $\beta^j$ even if it is not explicitly indicated). Some computations allow us to see that if this system is satisfied, so it is the equation $\tilde{Q}_j^M(P)=0$ for each $0\leq j\leq M-1$, where $\tilde{Q}_j^M$ is an homogeneous polynomial independent now of the parameters $\alpha,\beta,q,a$, analogous to the one we found in the previous section. Essentially, we are using the different equations to remove the parameters from them. This can always be done since the dependence on this parameters is polygonal or through square roots. 
    
      With all this, we can finally define the function\begin{equation}
        \begin{aligned}
            &Q_M:\mathbb{R}^{14M}\rightarrow\mathbb{R}^{M}\\
            &P=(P^0,P^1,\ldots,P^{M-1})\mapsto Q_M(P)=\left(\tilde{Q}^M_0\left(P^0\right), \tilde{Q}^M_1\left(P^1\right),\ldots,\tilde{Q}^M_{M-1}\left(P^{M-1}\right)\right).
        \end{aligned}
    \end{equation}
	 
      Notice that $Q_M^{-1}\left(\left\{0\right\}\right)$ is an (affine) algebraic variety (in particular, a closed set with empty interior because $Q_M$ is surjective) and, consequently, $\mathbb{R}^{14M}\backslash Q_M^{-1}\left(\left\{0\right\}\right)$ is open and dense.  Moreover, since any entry of $Q_M$ is evaluated at different coordinates, we notice that all of them are independent and so $Q_M^{-1}\left(\left\{0\right\}\right)$ is of codimension at least $M$ in $\mathbb{R}^{14M}$. 

       On the other hand, it comes from an easy computation that , for any $s,t\in\mathbb{R}^+_0$ and any $m,n\in\mathbb{N}\cup\{0\}$, $Q_M\left(\mathcal{EM}^M\left(f_{s,t,m,n}\right)\right)=0$ and, consequently, for any eigenfunction $u_{mn}$ and any $z^0\in\mathcal{E}_b$,
	\begin{equation}
		\mathcal{EM}^M\left(\tilde{R}_{\eta_0}^{\xi_0}\left(\tilde{u}_{mn}^{z^0}\right)\right)\in Q_M^{-1}\left(\left\{0\right\}\right).
	\end{equation}

 \textbf{	\underline{Step III:}}  We continue the proof in a similar way as the one in the previous section:  for any $P\in \RR^{14M}$, we consider the polynomial $\mathbb{P}(P)$ given by interpolation. 	An easy application of the Cauchy–Kovalevskaya theorem ensures the existence of a neighborhood $N_P$ of $L_{z_0}$ in $\mathbb{R}^{2}$ and a function $H_P:N_P\rightarrow\mathbb{R}$ satisfying the Helmholtz equation $\Delta H_P+H_P=0$ in $N_P$ such that $\restr{H_P}{L_{z_0}}=\mathbb{P}(P)$. Moreover, by the global approximation theory with decay \cite{APDE}, this means that for any $\ep>0$, there exists another function $\varphi_P\in\MW^s$ such that \begin{equation}
		\left\|\varphi_P-H_P\right\|_{C^2\left(N_P\right)}<\ep.
	\end{equation}If we notice that
	\begin{equation}
		\left\|\varphi_P-H_P\right\|_{C^2\left(N_P\right)}\geq\frac{1}{7M} \left|\mathcal{EM}^M\left(R_{\eta_0}^{\xi_0}\left(\varphi_P\right)\right)-\mathcal{EM}^M\left(\tilde{R}_{\eta_0}^{\xi_0}\left(H_P\right)\right)\right|,
	\end{equation}we are saying that for any $P\in\mathbb{R}^{14M}$ and any $\de>0$ there exists $\varphi_P\in \MW^s$ such that $\left|\mathcal{EM}^M\left(R_{\eta_0}^{\xi_0}\left(\varphi_P\right)\right)-P\right|<\de.$ Accordingly $\left(\mathcal{EM}^M\circ R_{\eta_0}^{\xi_0}\right)\left(\MW^s\right)$ is dense in $\mathbb{R}^{14M}$. Since the composition map $\mathcal{EM}^M\circ	R_{\eta_0}^{\xi_0}:\MW^s\rightarrow\mathbb{R}^{14M}$ is linear and continuous and has a dense image in $\mathbb{R}^{14M}$ and, it is surjective. By the open mapping theorem, we can conclude that it is also open and, consequently, $\left(\mathcal{EM}^M\circ	R_{\eta_0}^{\xi_0}\right)^{-1}\left(\mathbb{R}^{14M}\backslash Q^{-1}\left(\left\{0\right\}\right)\right)=\MW^s\backslash\left(\mathcal{EM}^M\circ	R_{\eta_0}^{\xi_0}\right)^{-1}\left( Q^{-1}\left(\left\{0\right\}\right)\right)$ is open and dense in $\MW^s$. Reasoning as in the previous section,  we can then conclude that the set $\left(Q_M\circ\mathcal{EM}^M\circ R_{\eta_0}^{\xi_0}\right)^{-1}\left(\{0\}\right)$ has codimension at least $M$ in $\MW^s$.
	
	 \textbf{	\underline{Step IV:}} The last step in the proof is to consider what occurs whenever we change the point $(\xi_0,\eta_0)\in\mathcal{E}_b$ around which we are doing the localization.  We first notice that, for any $\varphi\in \MW^s$, taking $R>0$ small enough, it is true that	\begin{equation}
		\begin{aligned}
			&\inf_{z^0\in\mathcal{E}_b}\inf_{\lambda\in\Lambda_{\mathcal{E}_b}^{\mathrm{BC}}}\inf_{u_\lambda \in\mathcal{V}_{\mathcal{E}_b,\mathrm{BC}}^\lambda}\left\|\varphi-u_\lambda\left(z^0+\frac{\cdot}{\sqrt{\lambda}}\right)\right\|_{C^0\left(B\right)}\geq\\&C\inf_{z^0\in\mathcal{E}_b}\inf_{\lambda\in\Lambda_{\mathcal{E}_b}^{\mathrm{BC}}}\inf_{u_\lambda \in\mathcal{V}_{\mathcal{E}_b,\mathrm{BC}}^\lambda}\left\|\varphi-u_\lambda\left(z^0+\frac{\cdot}{\sqrt{\lambda}}\right)\right\|_{C^2\left(B_R\right)}\geq\\&\frac{C}{7M}\inf_{z^0\in\mathcal{E}_b}\inf_{\lambda\in\Lambda_{\mathcal{E}_b}^{\mathrm{BC}}}\inf_{u_\lambda \in\mathcal{V}_{\mathcal{E}_b,\mathrm{BC}}^\lambda}\left|\mathcal{EM}^M\circ R_{\eta_0}^{\xi_0}\left(\varphi\right)-\mathcal{EM}^M\circ\tilde{R_{\eta_0}^{\xi_0}}\left(\tilde{u}_\lambda^{z^0}\right)\right|\geq\\ &\frac{C}{7M}\inf_{z^0\in\mathcal{E}_b} \dist\left(\mathcal{EM}^M\circ R_{\eta_0}^{\xi_0}(\varphi),Q_M^{-1}\left(\{0\}\right)\right).
		\end{aligned}
	\end{equation}
  Therefore, the set \begin{equation}
      \mathcal{N}_{\mathcal{E}_b,\mathrm{BC}}^c:=\left\{\varphi\in \MW^s,\ 0=\inf_{z^0\in\mathcal{E}_b}\inf_{\lambda\in\Lambda_{\mathcal{E}_b}^{\mathrm{BC}}}\inf_{u_\lambda \in\mathcal{V}_{\mathcal{E}_b,\mathrm{BC}}^\lambda}\left\|\varphi-u_\lambda\left(z^0+\frac{\cdot}{\sqrt{\lambda}}\right)\right\|_{C^0\left(B\right)}\right\},
  \end{equation}is contained in the set \begin{equation}\mathcal{A}:=\bigcup_{\substack{\xi_0\in[0,\xi_b]\\\eta_0\in[0,2\pi]}}\left(Q_M\circ\mathcal{EM}^M\circ R_{\eta_0}^{\xi_0}\right)^{-1}\left(\{0\}\right).\end{equation} The set $\mathcal{A}$ is the union of a family of codimension $M$ sets that is parametrized by the pair $(\xi_0,\eta_0)\in\mathcal{E}_d$. Since these parameters live in a space of dimension $2$ and the dependence of them is continuous, we can ensure that $\mathcal{A}$ is a subset of $\MW^s$ of codimension at least $M-2$. Taking $M>2$ we can then ensure that the set of monochromatic waves $\vp\in\MW^s$ that do not belong to $ \mathcal{N}_{\mathcal{E}_b,\mathrm{BC}}^c$ is dense in $\MW^s$.  Finally, we just recall from Step I that it is also open and we conclude the proof.
	
\end{proof}
\begin{remark}
If we consider three dimensional ellipsoids, that is,
in cartesian coordinates,  \begin{equation}
	\mathcal{E}_{a,b}=\left\{(z_1,z_2,z_3)\in\mathbb{R}^{2},\ \left(z_1\right)^2+\left(\frac{z_2}{a}\right)^2+\left(\frac{z_3}{b}\right)^2\leq 1\right\},
\end{equation} then a similar approach allows us to show that for generic ellipsoid the inverse localization property fails. This follows from an analogous study of the eigenfunctions and the fact that, for almost any $(a,b)\in [0,1]$, the ellipsoid $\mathcal{E}_{a,b}$ has simple spectrum, cf. \cite[Theorem 7.3]{Hil}.
\end{remark}

\section{Inverse localization for almost integrable polygonal billiards}\label{ILtil}
In this section we will extend the positive inverse localization results of Theorem \ref{BT.polygons}, imposing some extra hypotheses, to the more general case of rational almost integrable polygons. We first recall from Chapter \ref{Intro} the definition of a rational almost integrable polygon: 
\begin{definition}
	We say that $\Omega\subset\RR^2$ is a rational almost integrable polygon if it is drawn on one of the following four lattices:  $L_{\mathcal{T}_{\mathrm{equi}}}$, $L_{\mathcal{T}_{\mathrm{iso}}}$,  $L_{\mathcal{T}_{\mathrm{hemi}}}$ or $L_{\mathcal{Q}_{l}}$ for a rational rectangle 
	(i.e., $l^2\in\mathbb{Q}$).
\end{definition}

 Another way of seeing this is to consider $\Omega$ as given by \begin{equation}
	\Omega=\bigcup_{j=1}^J\mathbb{A}_j\left(\mathcal{P}\right),
\end{equation}for some $J\in\mathbb{N}$, some integrable polygon \begin{equation}\mathcal{P}\in\left\{L_{\mathcal{T}_{\mathrm{equi}}}, L_{\mathcal{T}_{\mathrm{iso}}},  L_{\mathcal{T}_{\mathrm{hemi}}}, L_{\mathcal{Q}_{l}}\text{ with } l^2\in\mathbb{Q}\right\}\end{equation}  and functions \begin{equation}
	\mathbb{A}_j:\mathbb{R}^{2}\rightarrow\mathbb{R}^{2}
\end{equation}which are compositions of affine reflections with respect to lines of the lattice $L_\mathcal{P}$. This is the same as saying that $\Omega$ is drawn in the family of congruent and symmetrically placed polygons obtained
by reflection about the sides of $\mathcal{P}$.

\begin{figure}\renewcommand\thefigure{6.1}
	\includegraphics[width=7cm]{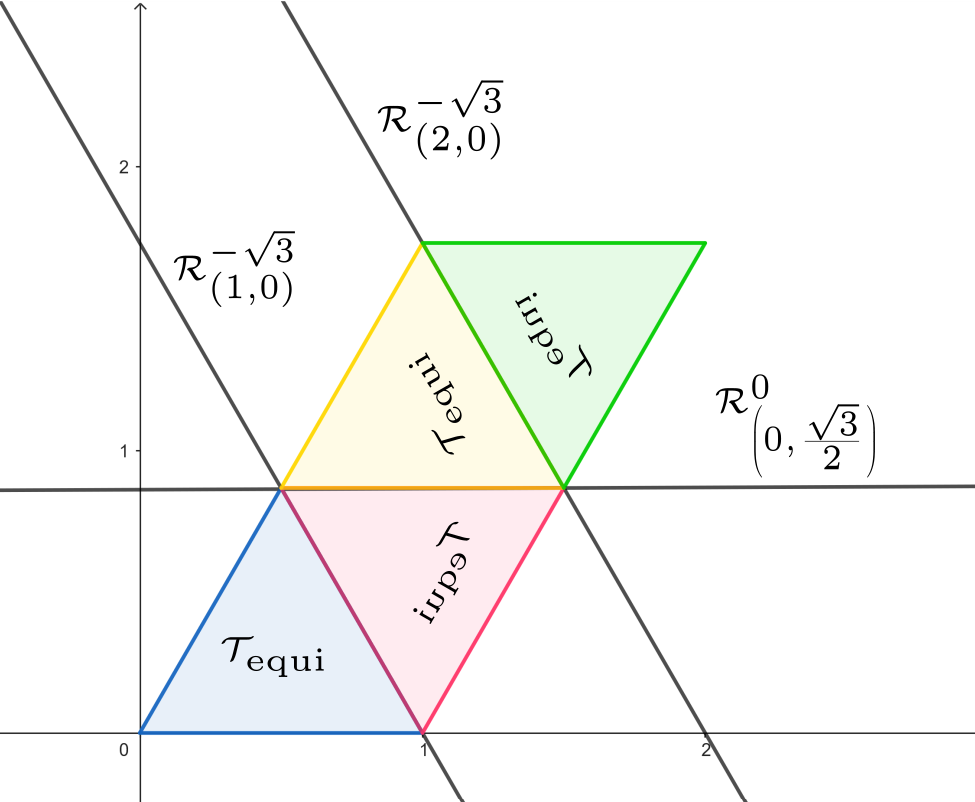}
	\caption{Example of a rational almost integrable polygonal billiard. We can also see the affine lines with respect to which we reflex to create this polygon. }\label{til}
\end{figure} 

\begin{remark}
In Figure \ref{til} we can see an example of a rational almost integrable polygon $\Omega$ given by \begin{equation}
	\Omega=\bigcup_{j=1}^4\mathbb{A}_j\left(\mathcal{T}_\mathrm{equi}\right),
\end{equation}where $\mathbb{A}_1=\I$ is the identity, $\mathbb{A}_2$ is the affine reflection with respect to $\mathcal{R}_{(1,0)}^{-\sqrt{3}}$ (i.e., the affine line that goes through $(1,0)$ with a slope of $-\sqrt{3}$), $\mathbb{A}_3$ is the composition of $\mathbb{A}_2$ and the affine reflection with respect to $\mathcal{R}_{\left(0,\frac{\sqrt{3}}{2}\right)}^{0}$ (i.e., the affine line that passes horizontally by $\left(0,\frac{\sqrt{3}}{2}\right)$); and, finally, $\mathbb{A}_4$ is the composition of $\mathbb{A}_2$, $\mathbb{A}_3$ and the affine reflection with respect to $\mathcal{R}_{(2,0)}^{-\sqrt{3}}$. Therefore, it is easy to see that $\mathcal{T}_\mathrm{equi}$ is the blue part of $\Omega$, $\mathbb{A}_1\left(\mathcal{T}_\mathrm{equi}\right)$ is the pink part of it, $\mathbb{A}_2\left(\mathcal{T}_\mathrm{equi}\right)$ is the yellow and, finally, $\mathbb{A}_3\left(\mathcal{T}_\mathrm{equi}\right)$ correspond to the green part of the polygon.

 \end{remark}
 
In this context, we can prove the following inverse localization result:

\begin{theorem}\label{ILlat}
	Let $\Omega\subset\RR^2$ be a rational almost integrable polygon. Let us consider the eigenvalue problem with Dirichlet or Neumann boundary conditions on $\Omega$. Then the following two results of inverse localization are satisfied:
	
\textbf{$\bullet$ inverse localization around a fixed base point}

Pick any $z^0\in \Omega\backslash L_\mathcal{P}$, any $k\in\mathbb{N}$, any $\ep>0$, any ball $B\subset \mathbb{R}^2$ and any $\vp\in\MW$ that satisfies a certain symmetry condition depending on $\mathcal{P}$ and the boundary conditions. 
Then we can find a sequence of eigenvalues $\lambda_n$ and an associated sequence of eigenfunctions $u_{n}$ that satisfies  \begin{equation}
	\left\|\vp-u_{n}\left(z^0+\frac{\cdot}{\sqrt{\lambda_n}}\right)\right\|_{C^k(B)}<\ep,
\end{equation}for all $n$ large enough. The conditions over $\vp$ are collected in Table \eqref{TableB} of  Appendix \ref{AppendixC}.

\textbf{$\bullet$ inverse localization without a fixed base point}

For any $\ep>0$, any ball $B$ of $\mathbb{R}^2$, any $k\in\mathbb{N}$ and any $\vp\in\MW$, there exists a sequence $\lambda_n$ of eigenvalues, a finite amount of sequences of associated eigenfunctions $u_{n}^j$ , where $1\leq j\leq J$ and $J$ depends only on $\Omega$, and a sequence of open sets $O_{n,\varphi}\subseteq \Omega$, depending on $\varphi$ and $\ep$, such that \begin{equation}
	|O_{n,\varphi}|=\int_{O_{n,\varphi}}1\ dz\xrightarrow[n\rightarrow\infty]{}\left|\Omega\right|\text{, along the sequence}
\end{equation}and for any $z^0\in O_{n,\varphi}$ there is one $1\leq j_0\leq J$, such that  \begin{equation}
	\left\|\vp-u^{j_0}_{n}\left(z^0+\frac{\cdot}{\sqrt{\lambda_n}}\right)\right\|_{C^k(B)}<\ep,
\end{equation}for any $n$ large enough.

\end{theorem}
\begin{proof}
	\textbf{	\underline{Step I:}} We start by proving the first part of this result. We notice that, for any point $z^0\in\Omega\backslash L_\mathcal{P}$, there exists $1\leq j_0\leq J$ such that $z^0\in \mathbb{A}_{j_0}(\mathcal{P})$ and there exists some $\delta>0$ small enough such that $B\left(\mathbb{A}_{j_0}^{-1}\left(z^0\right),\delta\right)\subset\mathcal{P}$. Recall that, for any $1\leq j\leq J$, $\mathbb{A}_j$ is a composition of affine reflections with respect some of the lines of the lattice $L_{\mathcal{P}}$ and, therefore, we have 
	\begin{equation}
		\left(\mathbb{A}_j\right)^{-1}\left(z^0+\frac{z}{\sqrt{\lambda_n}}\right)=\left(\mathbb{A}_j\right)^{-1}\left(z^0\right)+\frac{\mathbb{S}^j\left(z\right)}{\sqrt{\lambda_n}},
	\end{equation}where $\mathbb{S}^j$ is a composition of linear reflections with respect to different lines that goes through the origin, as the ones presented in \eqref{reflex} in Appendix \ref{AppendixC}. More precisely, every $\mathbb{S}^j$ is a composition of reflections, possibly repeated, of the following table:
		\begin{equation}\label{TableS}
		\begin{array}{cc} 
			\hline
			\rule{0pt}{3ex}\text{Polygon} & \text{Reflections} \\
			\hline
			\rule{0pt}{3ex}\mathcal{Q}_l& \mathbb{S}_0,\ \mathbb{S}_\infty\\
			\rule{0pt}{3ex}\mathcal{T}_\mathrm{iso} & \mathbb{S}_0,\ \mathbb{S}_\infty,\ \mathbb{S}_1\\
			
			\rule{0pt}{3ex}	\mathcal{T}_\mathrm{equi} & \mathbb{S}_0,\ \mathbb{S}_{\sqrt{3}},\ \mathbb{S}_{\sqrt{-3}} \\
			\rule{0pt}{3ex}\mathcal{T}_\mathrm{hemi} & \mathbb{S}_0,\ \mathbb{S}_\infty,\ \mathbb{S}_{\sqrt{3}},\ \mathbb{S}_{\sqrt{-3}}\\
		\end{array} 
	\end{equation}

We now choose a monochromatic wave $\varphi\in\MW$ satisfying the extra conditions on the Table \ref{TableB} associated with $\mathcal{P}$ (it was seen in Appendix \ref{AppendixC} there there exist such monochromatic waves), and notice that $-\varphi\in\MW$ also satisfies them. Since we have seen that $\mathbb{S}^j$ is a composition of symmetries of Table \ref{TableS} (which are the same appearing in the conditions of Table \ref{TableB}) we know that, for any $1\leq j\leq J$,  $\varphi\circ \left(\mathbb{S}^j\right)^{-1}=\left(-1\right)^{\sigma_\mathrm{BC}^j}\varphi$, where $\sigma_\mathrm{BC}^j$ is either $1$ or $2$ and depends on the boundary conditions and on the amount of reflections that form $\mathbb{S}^j$.
 
  The next step consists in applying the positive part of Theorem \ref{BT.polygons}. We first observe that, since the conditions in Table \ref{TableB} are more restrictive than the ones in Table \ref{TableA} (see Appendix \ref{AppendixC} for more details), we will be able to apply it to any $\varphi$ satisfying the conditions in Table \ref{TableB}. We fix now $k\in\mathbb{N}$, $\ep>0$ and $B\subset\RR^2$ and we recall that  $\mathbb{A}_{j_0}^{-1}\left(z^0\right)\in\mathcal{P}$. The positive part of Theorem \ref{BT.polygons} then gives us the existence of a sequence of eigenvalues $\lambda_n$ and eigenfunctions $\tilde{u}_{n}$ of $\mathcal{P}$ with Dirichlet or Neumann boundary conditions, respectively, such that \begin{equation}
		\left\|\left(-1\right)^{\sigma^{j_0}_\mathrm{BC}}\vp-\tilde{u}_{n}\left(\mathbb{A}_{j_0}^{-1}\left(z^0\right)+\frac{\cdot}{\sqrt{\lambda_n}}\right)\right\|_{C^k(B)}<\ep,
	\end{equation}for all $n$ large enough.
	If the ball $B$ has radius $R$, we choose $n$ large enough so that $\frac{R}{\sqrt{\lambda_n}}\leq \delta$ and therefore we are only considering $\tilde{u}_{n}$ inside the ball $B\left(\mathbb{A}_{j_0}^{-1}\left(z^0\right),\delta\right)$. 
	
	We can now define the following function $	u_{n}:\ \Omega\rightarrow\mathbb{R}$. For any $z\in\Omega$ there exist some $1\leq j_z\leq J$ such that $z\in\overline{\mathbb{A}_{j_z}\left(\mathcal{P}\right)}$, then we define $u_n(z):=\tilde{u}_{n}\left(\left(\mathbb{A}_{j_z}\right)^{-1}\left(z\right)\right).$ Notice that, by Lamé's Fundamental Theorem (see \cite[Theorem 1]{trigIV}), the boundary conditions satisfied by $\tilde{u}_n$ ensure that this function is well defined in $L_\mathcal{P}$ and smooth in $\Omega$. Moreover, $u_n$ satisfies the equation $\Delta u_n+\lambda_n u_n=0$ in the whole $\Omega$ and the same boundary condition in $\partial\Omega$ that $\tilde{u}_n$ satisfies in $\partial\mathcal{P}$.
		
		  With this definition we can then estimate \begin{equation}
		\begin{aligned}
			&\left\|\vp-u_{n}\left(z^0+\frac{\cdot}{\sqrt{\lambda_n}}\right)\right\|_{C^k(B)}=	\left\|\vp-\tilde{u}_{n}\left(\left(\mathbb{A}_{j_0}\right)^{-1}\left(z^0+\frac{\cdot}{\sqrt{\lambda_n}}\right)\right)\right\|_{C^k(B)}=\\&\left\|\vp-\tilde{u}_{n}\left(\left(\mathbb{A}_{j_0}\right)^{-1}\left(z^0\right)+\frac{\mathbb{S}^{j_0}\left(\cdot\right)}{\sqrt{\lambda_n}}\right)\right\|_{C^k(B)}=\\&\left\|\left(-1\right)^{\sigma^{j_0}_\mathrm{BC}}\vp-\tilde{u}_{n}\left(\left(\mathbb{A}_{j_0}\right)^{-1}\left(z^0\right)+\frac{\cdot}{\sqrt{\lambda_n}}\right)\right\|_{C^k(B)}<\ep.
		\end{aligned}
	\end{equation}
	
	We conclude the proof taking the sequence of eigenvalues $\lambda_n$ given by the positive part of Theorem~\ref{BT.polygons}, the sequence of eigenfunctions given by $u_{n}$ and imposing the extra condition of  $n$ being big enough such that  $\frac{R}{\sqrt{\lambda_n}}\leq \delta$.
	
\textbf{	\underline{Step II:}} 	We now prove the second part of the theorem. We will apply Theorem \ref{pnft} to the base polygon $\mathcal{P}$ but first, we recall from Remark \ref{lambda} the fact that the sequence of eigenvalues $\lambda_n$ that we use to make the approximation does not depend on the monochromatic wave $\vp\in\MW$ that we choose.

On the other hand, we allude to the fact that, since the Laplace operator commutes with any isometry of $\RR^2$, it is obvious that for any $\vp\in\MW$ and any isometry $\mathbb{S}:\mathbb{R}^2\rightarrow\mathbb{R}^2$, we have $\vp\circ\mathbb{S}\in\MW$. As in the previous step we know that for any $1\leq j\leq J$, we can find an isometry $\mathbb{S}^j$ given as a composition of reflections of the Table \ref{TableS} such that 	\begin{equation}
	\left(\mathbb{A}_j\right)^{-1}\left(z^0+\frac{z}{\sqrt{\lambda_n}}\right)=\left(\mathbb{A}_j\right)^{-1}\left(z^0\right)+\frac{\mathbb{S}^j\left(z\right)}{\sqrt{\lambda_n}}.
\end{equation}

We can then apply Theorem \ref{pnft} to the basic polygon $\mathcal{P}$ of the lattice considering some $\ep>0$, some $B\subset\RR^2$ and some $k\in\mathbb{N}$. As monochromatic wave we consider different cases: $\vp_1=\vp=\vp\circ \I=\vp\circ\mathbb{S}^1$, $\vp_2=\vp\circ\mathbb{S}^2$, $\ldots$,  $\vp_J=\vp\circ\mathbb{S}^J$. Therefore we get a sequence of eigenvalues $\lambda_n$, different sequences of eigenfunctions $\tilde{u}_{n}^j$ and different sequences of sets $\tilde{O}_n^j\subset\mathcal{P}$, with $1\leq j\leq J$, that satisfy Theorem \ref{pnft}. 
	
	Let us now consider the set $\tilde{O}_n:=\bigcap_{j=1}^J \tilde{O}_n^j\subset\mathcal{P}$. Since any $\tilde{O}_n^j$ is open, this set is open and, for $n$ big enough, it is non empty. Since, for any $1\leq j\leq J$, we have $\left|\tilde{O}_n^j\right|\xrightarrow[n\rightarrow\infty]{}\left|\mathcal{P}\right|$, then it is also true that $\left|\tilde{O}_n\right|\xrightarrow[n\rightarrow\infty]{}\left|\mathcal{P}\right|$, by just considering the complementary sets.
	
	We then define the finite set of eigenfunction $	u_{n}^j:\ \Omega\rightarrow\mathbb{R}$ as in the previous step:  for any $z\in\Omega$ there exist some $1\leq i\leq J$ such that $z\in\overline{\mathbb{A}_i\left(\mathcal{P}\right)}$, then we define $u_n^j(z):=\tilde{u}_{n}^j\left(\left(\mathbb{A}_i\right)^{-1}\left(z\right)\right).$ Recall that, this function is well defined in $L_\mathcal{P}$, smooth in $\Omega$ and satisfies the equation $\Delta u_n^j+\lambda_n u_n^j=0$ in the whole $\Omega$ with the same boundary condition in $\partial\Omega$ that $\tilde{u}_n^j$ satisfies in $\partial\mathcal{P}$.
	
	Since $\tilde{O}_n^j\subset\mathcal{P}$, we can define a new set given by \begin{equation}
	     \overline{O}_n:=\left\{z\in\tilde{O}_n \text{ such that }B\left(z,\frac{R}{\sqrt{\lambda_n}}\right)\subset\mathcal{P}\right\},
	\end{equation}where $R$ is the radius of $B$. We can conclude the proof taking the sequences $\lambda_n$, $O_n=\bigcup_{i=1}^J\mathbb{A}_i \left(\overline{O}_n\right)$, which obviously satisfies, $\left|O_n\right|\xrightarrow[n\rightarrow\infty]{}\left|\Omega\right|$, and $u_{n}^j:\Omega\rightarrow\mathbb{R}$ where  $1\leq j\leq N$.
	
	To see that this concludes the proof, we notice that for any $z^0\in O_{n}$, in particular, there exists a $j_0\in\left\{1,2,\ldots,J\right\}$ such that $z^0\in  \mathbb{A}_{j_0}\left(\tilde{O}_{n}\right)$. Therefore we have that, for any $n$ large enough so that Theorem \ref{pnft} holds and that  $B\left(\mathbb{A}_{j_0}^{-1}\left(z^0\right),\frac{R}{\sqrt{\lambda_n}}\right)\subset\mathcal{P}$, it is true that \begin{equation}
		\left\|\vp_{j_0}-\tilde{u}^{j_0}_{n}\left(\mathbb{A}_{j_0}^{-1}\left(z^0\right)+\frac{\cdot}{\sqrt{\lambda_n}}\right)\right\|_{C^k(B)}<\ep.
	\end{equation}
	Finally, we see that 
	\begin{equation}
		\begin{aligned}
		&	\left\|\vp-u^{j_0}_{n}\left(z^0+\frac{\cdot}{\sqrt{\lambda_n}}\right)\right\|_{C^k(B)}=\left\|\vp_{j_0}\circ\left(\mathbb{S}^{j_0}\right)^{-1}-\tilde{u}^{j_0}_{n}\circ\mathbb{A}_{j_0}^{-1}\left(z^0+\frac{\cdot}{\sqrt{\lambda_n}}\right)\right\|_{C^k(B)}=\\&\left\|\vp_{j_0}-\tilde{u}^{j_0}_{n}\circ\mathbb{A}_{j_0}^{-1}\left(z^0+\frac{\mathbb{S}^{j_0}\left(\cdot\right)}{\sqrt{\lambda_n}}\right)\right\|_{C^k(B)}=\left\|\vp_{j_0}-\tilde{u}^{j_0}_{n}\left(\mathbb{A}_{j_0}^{-1}\left(z^0\right)+\frac{\cdot}{\sqrt{\lambda_n}}\right)\right\|_{C^k(B)}<\ep.
		\end{aligned}
	\end{equation}
\end{proof}
\begin{remark}\label{higdem}
	If we consider $\mathcal{P}=\mathcal{Q}_l\subset\RR^d$, we could also tile $\mathbb{R}^d$ with copies of $\mathcal{Q}_l$, for any value of $l\in\mathbb{R}^{d-1}_+$. In this context we would consider reflections of $\mathcal{Q}_l$ with respect to its faces (which are now hyperplanes, instead of lines). Since Theorem \ref{BT.polygons} holds in this case also for dimension $d\geq2$, we can prove an analogous version of Theorem \ref{ILlat}. The conditions to impose on $\varphi$ would be the ones in Table \ref{TableA}. More precisely, $\mathbb{S}^j$ would be a composition of a finite amount of $\mathbb{S}_i^{\mathcal{Q}_l}$ (as defined in Remark \ref{conditions}) with $1\leq i\leq d$ and \begin{equation}
		\begin{aligned}
		\mathbb{S}_i^{\mathcal{Q}_l}&:\mathbb{R}^d\rightarrow\mathbb{R}^d\\	&z\mapsto\mathbb{S}_i^{\mathcal{Q}_l}\left(z_1,\ldots,z_{i-1},z_i,z_{i+1},\ldots,z_d\right)=\left(z_1,\ldots,z_{i-1},-z_i,z_{i+1},\ldots,z_d\right).
		\end{aligned}
	\end{equation}
\end{remark}
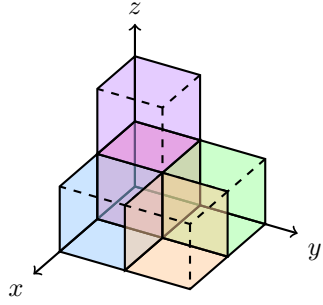
\begin{figure}\renewcommand\thefigure{6.2}
	\tdplotsetmaincoords{60}{120}
	\begin{tikzpicture}[tdplot_main_coords]
		
		\definecolor{cubeA}{RGB}{255, 102, 153} 
		\definecolor{cubeB}{RGB}{153, 204, 255} 
		\definecolor{cubeC}{RGB}{153, 255, 153} 
		\definecolor{cubeD}{RGB}{255, 204, 153} 
		\definecolor{cubeE}{RGB}{204, 153, 255} 
		
		\draw[thick, ->] (0,0,0) -- (2.7,0,0) node[anchor=north east]{$x$};
		\draw[thick, ->] (0,0,0) -- (0,2.5,0) node[anchor=north west]{$y$};
		\draw[thick, ->] (0,0,0) -- (0,0,2.5) node[anchor=south]{$z$};
		
		\filldraw[fill=cubeA, fill opacity=0.5, draw=black, thick] (0,0,0) -- (1,0,0) -- (1,1,0) -- (0,1,0) -- cycle; 
		\filldraw[fill=cubeA, fill opacity=0.5, draw=black, thick] (0,0,0) -- (1,0,0) -- (1,0,1) -- (0,0,1) -- cycle; 
		\filldraw[fill=cubeA, fill opacity=0.5, draw=black, thick] (0,0,0) -- (0,1,0) -- (0,1,1) -- (0,0,1) -- cycle; 
		
		\draw[black, thick, dashed] (0,1,1) -- (1,1,1) -- (1,0,1); 
		
		\filldraw[fill=cubeB, fill opacity=0.5, draw=black, thick] (1,0,0) -- (2,0,0) -- (2,1,0) -- (1,1,0) -- cycle; 
		\filldraw[fill=cubeB, fill opacity=0.5, draw=black, thick] (1,0,0) -- (2,0,0) -- (2,0,1) -- (1,0,1) -- cycle; 
		\filldraw[fill=cubeB, fill opacity=0.5, draw=black, thick] (1,0,0) -- (1,1,0) -- (1,1,1) -- (1,0,1) -- cycle; 
		
		\draw[black, thick, dashed] (1,1,1) -- (2,1,1) -- (2,0,1); 
		
		\filldraw[fill=cubeC, fill opacity=0.5, draw=black, thick] (0,1,0) -- (1,1,0) -- (1,2,0) -- (0,2,0) -- cycle; 
		\filldraw[fill=cubeC, fill opacity=0.5, draw=black, thick] (0,1,0) -- (1,1,0) -- (1,1,1) -- (0,1,1) -- cycle; 
		\filldraw[fill=cubeC, fill opacity=0.5, draw=black, thick] (0,1,0) -- (0,2,0) -- (0,2,1) -- (0,1,1) -- cycle; 
		
		\draw[black, thick, dashed] (0,2,1) -- (1,2,1) -- (1,1,1); 
		
		\filldraw[fill=cubeD, fill opacity=0.5, draw=black, thick] (1,1,0) -- (2,1,0) -- (2,2,0) -- (1,2,0) -- cycle; 
		\filldraw[fill=cubeD, fill opacity=0.5, draw=black, thick] (1,1,0) -- (2,1,0) -- (2,1,1) -- (1,1,1) -- cycle; 
		\filldraw[fill=cubeD, fill opacity=0.5, draw=black, thick] (1,1,0) -- (1,2,0) -- (1,2,1) -- (1,1,1) -- cycle; 
		
		\draw[black, thick, dashed] (1,2,1) -- (2,2,1) -- (2,1,1); 
		\draw[black, thick, dashed] (2,2,1) -- (2,2,0) ; 
		
		\filldraw[fill=cubeE, fill opacity=0.5, draw=black, thick] (0,0,1) -- (1,0,1) -- (1,1,1) -- (0,1,1) -- cycle; 
		\filldraw[fill=cubeE, fill opacity=0.5, draw=black, thick] (0,0,1) -- (1,0,1) -- (1,0,2) -- (0,0,2) -- cycle; 
		\filldraw[fill=cubeE, fill opacity=0.5, draw=black, thick] (0,0,1) -- (0,1,1) -- (0,1,2) -- (0,0,2) -- cycle; 
		
		\draw[black, thick, dashed] (0,1,2) -- (1,1,2) -- (1,0,2); 
		\draw[black, thick, dashed] (1,1,2) -- (1,1,1) ; 
		
		
	\end{tikzpicture}
	\caption{Following Remark \ref{higdem}, we can also have the inverse localization property in some polytope (i.e. polygonal like domains in $\RR^d$) like the one of the figure.}\label{til3d}
\end{figure} 
\section{Applications to nodal set and critical points of eigenfunctions}\label{nsandcpSect}
In this section we present some applications of the inverse localization property to the study of nodal sets and critical points of eigenfunctions on the (almost) integrable billiards studied in this paper. We start by recalling two important definitions: for any smooth function $f$, we define its nodal set as the points \begin{equation}
	\mathcal{N}\left(f\right):=\left\{z\in\mathbb{R}^d,\ f(z)=0\right\}.
\end{equation} On the other hand, the connected components of $\RR^d\backslash\mathcal{N}\left(f\right)$ are known as \emph{nodal domains} of $f$. 

In this chapter, we prove two results: in the first one, we construct eigenfunctions which have some connected components of their nodal set with a prescribed topological class (therefore, this result will be of interest mostly when $d\geq2$, i.e., when we consider $\mathcal{Q}_l$) and with at least $N$ non degenerate critical points, for any fixed $N\in\NN$. On the other hand, we construct eigenfunctions with any nesting
arrangement among their zero set components. We introduce each one of these results in two separated sections.

\subsection{Topologically complicated nodal sets and critical points}
Let us begin by introducing some notation. If $X$ is a subset of~$\RR^d$, we denote its dilation by factor $c>0$ and its translation by a point $p\in\RR^d$ by
\begin{equation}
cX:=\{cx: x\in X\}\,,\qquad X+ p:=\{x+p: x\in X\}\,.
\end{equation}
Also, here $B_R$ denotes the ball centered at the origin of radius~$R$.

In the statement of the following theorem we consider an arbitrary number of hypersurfaces. There is, however, a topological condition that we must impose on these hypersurfaces, which is described in the following

\begin{definition}\label{link}
Let $\{\Si_j\}_{j=1}^N\subset\mathbb R^d$ be a collection of smoothly embedded codimension one compact orientable submanifolds without boundary that are pairwise disjoint. We say that the collection $\{\Si_j\}_{j=1}^N\subset\mathbb R^d$ is not linked if there are $N$ pairwise disjoint contractible domains $V_j$ such that each hypersurface is contained in one of them $\Sigma_j\subset V_j$.
\end{definition}
We are now ready to state and prove the main result of this section.
\begin{theorem}\label{nsandcp}
	Let $\Omega\subset\mathbb{R}^{2}$ be a rational almost integrable polygon in the lattice $L_\mathcal{P}$ (see Definition \ref{ai})  or $\Omega=\mathcal{Q}_l\subset\mathbb{R}^{d}$ with $ l_i^2\in \mathbb{Q}$ for all $1\leq i\leq d-1$, and consider the eigenvalue problem  with Dirichlet or Neumann conditions. Take any~$M,N\in\NN$, a collection of compact embedded hypersurfaces~$\Si_j$, where $1\leq j\leq N$, of~$\RR^d$ that are not linked, and any positive integer~$k$. Let us consider any $\ep>0$ and any  $z^0\in\Omega$. Then there exist a sequence of eigenvalues $\la_n\to\infty$ and  some~$R>0$ such that for all large enough~$n$ there is an eigenfunction $u_n$ of $\lambda_n$ that has in the ball $B\left(z^0,R\la_n^{-1/2}\right)$ at least~$N$ nodal components of the form
	\begin{equation}
	\widetilde\Si_j^n:= \la_n^{-1/2} \,  \Phi_{n}(c_j \Si_j + p_j)
	\end{equation}
	and at least $M$ nondegenerate local extrema. Here $c_j>0$, $p_j\in\RR^d$, and $\Phi_n$ is a diffeomorphism of~$\RR^d$ which is close to the identity:
	$  \|\Phi_n-\id\|_{C^k(\RR^d)}<\ep$.
\end{theorem}

\begin{proof}
For the benefit of the reader, we will split the proof in two parts. 

\textbf{	\underline{Step I:}}
	We first construct a set of analytic hypersurfaces diffeomorphic to the ones we want to construct and some additional ones that will yield the critical points. Moreover, we make this set invariant with respect to the symmetry conditions of Appendix \ref{AppendixC}.
    
    Indeed, an easy application of Whitney's approximation theorem ensures that, by perturbing the hypersurfaces a little if necessary, we can assume that $\Sigma_j$ is a real analytic hypersurface of $\mathbb{R}^d$.  Let us choose the constants $c_j$ so that the first Dirichlet eigenvalue of the domain~$\Om_j$ bounded by the rescaled hypersurface $c_j\Si_j$ is~1. 	To ensure that we also have at least $M$ nondegenerate local extrema, we consider $M$ new hypersurfaces $\Si_j$ with $N+1\leq j\leq N+M$ that we take to be hyperspheres of radius $r$ so that the first Dirichlet eigenvalue of the corresponding balls that they bound is $1$. This eigenvalue is simple and the corresponding eigenfunction is spherically symmetric around the center of the ball and has a nondegenerate local extremum at the center, as desired.
	
	Let us now pick any vectors $p_j\in\RR^d$ so that the translated domains $\Om_j':=\Om_j+p_j$ have disjoint closures (which is possible because $\Sigma_j$ are not linked), and take some large ball~$B_R$ so that $\Om_j'\subset B_R$ for all $j$. Moreover, if $\Omega$ is an almost integrable polygon, we take the vectors $p_j$ so that the ball $B_R$ is contained in the fundamental sector $\mathbb{F}_\mathcal{P}\subset\mathbb{R}^2$ associated to $\mathcal{P}$, as defined in Remark \ref{nod}. On the other hand, whenever $\Omega=\mathcal{Q}_l$, we take $B_R$ contained in the first orthant $\mathbb{O}^+$, i.e., those $z\in\mathbb{R}^d$ such that $z_i>0$ for any $1\leq i\leq d$. This implies that all the $\Om_j'$ are contained in $\mathbb{F}_\mathcal{P}$ or $\mathbb{O}^+$.

In the case that $\Om$ is an almost integrable polygon, we can now apply all the reflections of the group of isometries presented in Remark \ref{nod} to any $\Om_j'$. Let us define the sets $\tilde{\Om}_j^0=\Om_j'$, $\tilde{\Om}_j^i=\tilde{\Om}_j^{i-1}\cup\mathbb{S}_i^{\pi/I}\left(\tilde{\Om}_j^{i-1}\right)$, where $I\in\{2,3,4,6\}$ depends on $\mathcal{P}$ and $1\leq i\leq I$. The maps $\mathbb{S}_i^{\pi/I}$ are defined in Remark \ref{nod}. Taking $\tilde{\Om}_j:=\tilde{\Om}_j^I$, we have a set formed by $2I$ copies of $\Om_j'$  and that is invariant with respect to the finite group of isometries that we are considering. 

On the other hand, when $\Om=\mathcal{Q}_l$, we define $\tilde{\Om}_j^0=\Om_j'$ and  $\tilde{\Om}_j^i=\tilde{\Om}_j^{i-1}\cup\mathbb{S}_i^{\mathcal{Q}_l}\left(\tilde{\Om}_j^{i-1}\right)$, with $\mathbb{S}_i^{\mathcal{Q}_l}$ defined as in \eqref{sim} and $1\leq i\leq d$. Again $\tilde{\Om}_j:=\tilde{\Om}_j^d$ is formed by $2^d$ copies of $\Om_j'$  and is invariant with respect to the finite group of reflections that we are considering. 

Following the same construction, in both cases, we define the symetrized hypersurfaces $\tilde{\Sigma}_j\subset\tilde{\Om}_j$ that are the symmetrized of $\Si'_j=c_j \Si_j+p_j$. In any case, this is also invariant with respect to the finite group of reflections that we are considering. 

\textbf{	\underline{Step II:}} Finally, we construct the eigenfunction that will have the desired properties. To do so, we first construct a monochromatic wave with the desired nodal set and then we approximate it by localized eigenfunctions, without losing the construction. 

It was proved in~\cite[Theorem 2.2.]{APDE} (see also \cite[Proposition 3.2.]{EP15}) that there exists a  solution~$\phi'$ to the Helmholtz equation in a neighborhood of a symmetric compact set $\mathcal N$ (by symmetric we mean with respect to the reflections introduced in the previous step), which has a union of structurally stable nodal components diffeomorphic to $\bigcup_{j=1}^{M+N}\tilde{\Si}_j$ inside $\mathcal N$. Moreover, $\phi'$ can be taken to be symmetric as well under the corresponding group. By the way of choosing $\Si_j$ with $N+1\leq j\leq N+M$, we notice that $\phi'$ has at least $M$ non-degenerate local extrema in $\mathcal N$. The global approximation theorem for monochromatic waves then implies that there is $\phi\in MW^s$ that approximates $\phi'$ in $\mathcal N$ as much as desired, i.e.,
\begin{equation}
\|\phi-\phi'\|_{C^k(\mathcal N)}\leq \delta\,.
\end{equation}

    
The next step of the proof is to define a symmetric function in $MW^s$ that still has a union of structurally stable nodal components diffeomorphic to $\bigcup_{j=1}^{N+M}\tilde{\Si}_j$. To do that, we distinguish the two cases:\begin{itemize}
    \item In the case of $\Omega$ being an almost integrable domain, we denote by $g_i$ the elements of the group of isometries of Remark \ref{nod}, where $1\leq i\leq \left|\mathbb{G}_\mathcal{P}\right|$ and $g_1=\I$ is the identity, and we define the function $\varphi\in MW^s$ as follows: \begin{equation}
        \varphi(z):=\frac{1}{\left|\mathbb{G}_\mathcal{P}\right|}\sum_{i=1}^{\left|\mathbb{G}_\mathcal{P}\right|}\left(-1\right)^{s_i^{\mathrm{BC}}}\phi\circ g_i,
    \end{equation}where $s_i^{\mathrm{N}}=0$ always and $s_i^{\mathrm{D}}$ is $0$ when $g_i$ is a composition of an even number of generators and $1$ when it is composed by an odd number. It is immediate to check then that $\vp$ satisfies the conditions on Table \ref{TableB}.
    
    \item In the case of $\Omega=\mathcal{Q}_l$, we denote by $h_i$ with $1\leq i\leq 2^d$ all the possible nontrivial composition of the reflections $\mathbb{S}_j^{\mathcal{Q}_l}$ with $1\leq j\leq d$. We can the define the function $\varphi\in MW^s$ as follows: 
    \begin{equation}
        \varphi(z):=\frac{1}{2^d}\sum_{i=1}^{2^d}\left(-1\right)^{s_i^{\mathrm{BC}}}\phi\circ h_i,
    \end{equation}where $s_i^{\mathrm{BC}}$ depends on the boundary conditions under consideration:  $s_i^{\mathrm{N}}=0$ always and $s_i^{\mathrm{D}}$ is $0$ when $h_i$ is composition of an even number of reflections and $1$ when it is defined by an odd number. Again, we can check that $\vp$ satisfies the conditions on the first line of Table \ref{TableA}.
    \end{itemize}

    We now notice that, in the case of $\Omega=\mathcal{Q}_l$ (the other one is analogous), the following bound holds:\begin{equation}
        \begin{aligned}
           & \left\|{\phi}-\varphi\right\|_{C^{k}(\mathcal{N})}\leq\left\|\frac{1}{2^d}\sum_{i=1}^{2^d}\left(-1\right)^{s_i^{\mathrm{BC}}}{\phi'}\circ h_i-\frac{1}{2^d}\sum_{i=1}^{2^d}\left(-1\right)^{s_i^{\mathrm{BC}}}\phi\circ h_i\right\|_{C^{k}(\mathcal{N})} \\&+\|\phi-\phi'\|_{C^k(\mathcal N)}\leq C_d\left\|{\phi}-\phi'\right\|_{C^{k}(\mathcal{N})}\leq C_d{\delta},
        \end{aligned}
    \end{equation}where we have used both the symmetry of ${\phi'}$ and the invariance of $\mathcal{N}$. Therefore, $\varphi$ also has a union of structurally stable nodal components diffeomorphic to $\bigcup_{j=1}^{M+N}\tilde{\Si}_j$ inside  $\mathcal N$.

    Since we have the needed symmetry conditions, both the positive part of Theorem~\ref{BT.polygons}, when $\Omega=\mathcal{Q}_l$, and Theorem \ref{T.LATICE}, when considering almost integrable polygons, apply, and we infer that for any point $z^0\in\Omega$ there exists a sequence of localized eigenfunctions $u_n\left(z^0+\frac{\cdot}{\sqrt{\la_n}}\right)$ with eigenvalues $\la_n\to\infty$ such that
	\begin{equation}
	\lim_{n\to\infty}\left\|u_n\left(z^0+\frac{\cdot}{\sqrt{\la_n}}\right)-\varphi\right\|_{C^{k}(B_R)}=0\,,
	\end{equation}
	where $B_R$ is any sufficiently large ball that contains $\mathcal N$.
    
	To conclude the result, we just apply the structural stability of the constructed components of the nodal set. 
   
\end{proof}
\begin{remark}
	The previous proof can be enormously simplified if we do not fix beforehand the point $z^0$ around which we construct the nodal components and the critical points. In that case, we can use the inverse localization property without a fixed base point and, consequently, we are not forced to impose any symmetry on the monochromatic wave $\varphi$ that has the hypersurfaces as structurally stable components of the nodal set.
\end{remark}

 This result can be seen to be related with another well known problem: the number of nodal domains. More precisely, we consider interesting the following remark:
 \begin{remark}
     For a Neumann eigenfunction in a domain whose boundary is piecewise analytic we can study the number of nodal domains of any given eigenfunction $u_n$ associated to the eigenvalue $\lambda_n$. Let us denote this quantity by $\mu\left(u_n\right)$. A famous result by Courant (see \cite{Courant}) states that $\mu\left(u_n\right)\leq n$. On the other hand, \cite[Proposition 10.6.]{HelfII}
  gives a rather explicit way to construct a specific family of eigenfunctions (corresponding to eigenvalues of multiplicity at most two) whose number of nodal domains increases to infinity whenever $n$ does so. An immediate consequence of Theorem \ref{nsandcp} is that we can find a family of eigenfunctions (that corresponds now to eigenvalues of increasing multiplicity) whose number of nodal domains also goes to infinity. To do so, we just apply the theorem to an increasing number of not linked hypersurfaces $\Sigma_k$ with $1\leq k\leq n$ and $n\rightarrow +\infty$.
  
  \end{remark}

\subsection{Nesting of nodal domains}
In this section we provide another application of the inverse localization results to the study of nodal sets of eigenfunctions. Let us start the section by introducing some useful definitions and notation. 

 Let us denote by $\mathcal{C}\left(f\right)$ the set of all compact connected components of the nodal set of a function $f$. We assume that the function $f$ is analytic (as is the case of the monochromatic waves) and the components in $\mathcal{C}\left(f\right)$ are regular (i.e., $\nabla f$ does not vanish at any point of the set). We first recall that, in graph theory, a tree is an undirected graph in which any two vertices are connected by exactly one path. Moreover, we recall that a rooted tree is a tree in which one vertex has been designated the root. Following \cite{Yaiza}, we see that, to each element $c\in\mathcal{C}\left(f\right)$, we can associate a finite connected rooted tree as follows:  by the Jordan-Brouwer separation Theorem each component $c\in\mathcal{C}\left(f\right)$ has an exterior and interior part and we can  choose the interior to be the compact one. The nodal domains of $f$  which are in the interior of $c$ are defined to be the vertices of a graph. Two vertices will be connected by an edge if the respective nodal domains have a common boundary component (unique if there is one). This gives a finite connected rooted tree denoted $e(c)$, whose root is the domain adjacent to $c$ (the biggest one contained in $c$, that we will label by $\emptyset$). Let $\mathrm{T}$ be the collection (countable and discrete) of finite connected rooted trees. The main result of this section is that, whenever $\Omega$ is a rational almost integrable polygon or $\Omega=\mathcal{Q}_l$, any rooted tree can be realized in the nodal set of an eigenfunction of $\Omega$ with either Dirichlet or Neumann boundary conditions. 

\begin{theorem}\label{Yaiz}
	Let $\Omega\subset\mathbb{R}^{2}$ be a rational almost integrable polygon in the lattice $L_\mathcal{P}$ (see Definition \ref{ai})  or $\Omega=\mathcal{Q}_l\subset\mathbb{R}^{d}$ with $ l_i^2\in \mathbb{Q}$ for all $1\leq i\leq d-1$, and consider the associated eigenvalue problem  with Dirichlet or Neumann conditions. Take any $t\in\mathrm{T}$ and any  $z^0\in\Omega$. Then there exists a sequence of eigenvalues $\la_n\to\infty$ and  some~$R>0$ such that for all large enough~$n$ there is an eigenfunction $u_n$ of $\lambda_n$ that has a nodal domain $c\in\mathcal{C}\left(u_n\right)$ with $e(c)=t$ in the ball $B\left(z^0,R\la_n^{-1/2}\right)$.
\end{theorem}
\begin{proof}
	 To the given tree $t\in\mathrm{T}$ we associate a new one, say $t'$, in which we just add a new vertex connected only with the root. This new vertex is now the root of the tree and represents the existence of a new nodal component surrounding all the nodal domains composing $t$. We can safely assume that $c'$ is contained in $\mathbb{F}_{\mathcal{P}}$, in the case of an almost integrable polygon, or in $\mathbb{O}^+$, in the case in which $\Omega=\mathcal{Q}_l$. Applying \cite[Theorem 2]{Yaiza} and proceeding as in the previous section, we easily obtain a function $\phi'$ that solves the Helmholtz equation in a neighborhood of a compact set $\mathcal N$ that contains $c'$, so that $c'\in\mathcal{C}\left(\phi'\right)$ and $e(c')=t'$. As before, the set $\mathcal N$ is symmetric (with respect to the corresponding reflections) as well as the function $\phi'$, and the components of the nodal set that exhibit the $t'$ structure are structurally stable. It follows from the global approximation theorem that there exists $\phi\in MW^s$ that approximates $\phi'$ in $\mathcal N$ as much as desired, i.e.,
\begin{equation}
\|\phi-\phi'\|_{C^k(\mathcal N)}\leq \delta\,.
\end{equation}

    The next step consists in constructing a new monochromatic wave that is symmetric under the corresponding group of reflections, and still has a set of nodal components with the desired tree structure. To do that, we define $\varphi$ following the same symmetrization process as in the Section before: \begin{itemize}
	 	\item In the case of $\Omega$ being an almost integrable domain, we recall that $g_i$ are the elements of the group of isometries of Remark \ref{nod}, where $1\leq i\leq \left|\mathbb{G}_\mathcal{P}\right|$, and we define the function $\varphi:\mathbb{R}^2\rightarrow\mathbb{R}$ as follows: \begin{equation}
	 		\varphi(z):=\frac{1}{\left|\mathbb{G}_\mathcal{P}\right|}\sum_{i=1}^{\left|\mathbb{G}_\mathcal{P}\right|}\left(-1\right)^{s_i^{\mathrm{BC}}}\phi\circ g_i,
	 	\end{equation}where $s_i^{\mathrm{N}}=0$ always and $s_i^{\mathrm{D}}$ is $0$ when $g_i$ is the composition of an even number of generators and $1$ when it is composed by an odd number. 
	 	
	 	\item In the case of $\Omega=\mathcal{Q}_l$, we recall that $h_i$ with $1\leq i\leq 2^d$ are all the possible nontrivial composition of the reflections $\mathbb{S}_j^{\mathcal{Q}_l}$ with $1\leq j\leq d$. We can the define the function $\varphi:\mathbb{R}^2\rightarrow\mathbb{R}$ as follows: 
	 	\begin{equation}
	 		\varphi(z):=\frac{1}{2^d}\sum_{i=1}^{2^d}\left(-1\right)^{s_i^{\mathrm{BC}}}\phi\circ h_i,
	 	\end{equation}where $s_i^{\mathrm{N}}=0$ always and $s_i^{\mathrm{D}}$ is $0$ when $h_i$ is composition of an even number of reflections and $1$ when it is defined by an odd number. 
	 \end{itemize}
	   In any case, ${\varphi}$ is still a solution to Helmholtz equation that now satisfies the symmetry conditions of Table \ref{TableA} or Table \ref{TableB}, respectively. Moreover, by structural stability, we notice that the nodal set of $\varphi$ contains a finite number of copies of $c'$ (one in each reflection of the fundamental sector), and $e(c')=t'$. These nodal sets are contains in a ball $B_R$ for large enough $R$. The extra vertex (resp. nodal domain) that we added when defining $t'$ ensures that there is a bounded convex domain in $B_R$ inside which we only have those nodal sets that allow us to construct~$t$. 
	   
	   We conclude the proof as in the previous section:  since we have the needed symmetry conditions, 
	   those of Theorem~\ref{BT.polygons}, when $\Omega=\mathcal{Q}_l$, or Theorem \ref{T.LATICE}, when considering almost integrable polygons, for any point $z^0\in\Omega$ there exists a sequence of localized eigenfunctions $u_n\left(z^0+\frac{\cdot}{\sqrt{\la_n}}\right)$ with eigenvalues $\la_n\to\infty$ such that
	   \begin{equation}
	   	\lim_{n\to\infty}\left\|u_n\left(z^0+\frac{\cdot}{\sqrt{\la_n}}\right)-\varphi\right\|_{C^{k}(B_R)}=0\,,
	   \end{equation}
	   Finally, we just apply again the structural stability of the constructed components of the nodal set, and the theorem follows. 
	   
\end{proof} 
\begin{remark}
    Similarly as in the previous section, we notice that this proof can be simplified if we just want the existence of a point $z^0$ around which we have nodal nesting (we can use the inverse localization theorem without a fixed base point, and hence we do not need to construct a symmetrized monochromatic wave).
    
\end{remark}
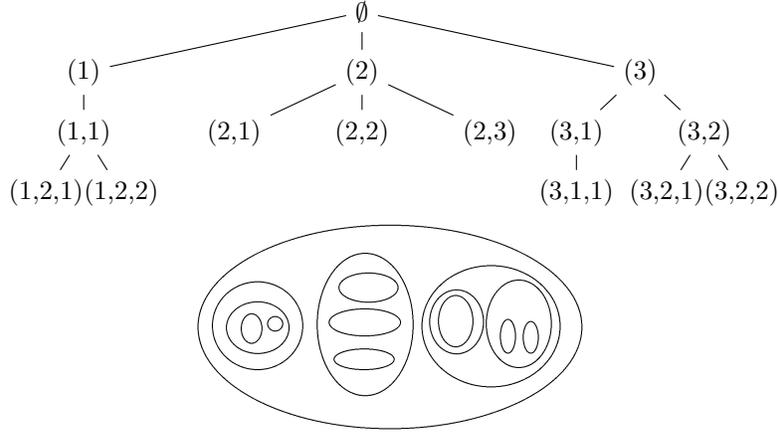
\begin{figure}\renewcommand\thefigure{7.1}
	\begin{tikzpicture}[
		level distance=0.8cm,
		level 1/.style={sibling distance=3.7cm},
		level 2/.style={sibling distance=1.7cm},
		level 3/.style={sibling distance=1cm}
		]
		
		\node {$\emptyset$}
		child { node {(1)}
			child { node {(1,1)}
				child { node {(1,2,1)} }
				child { node {(1,2,2)} }
			}
		}
		child { node {(2)}
			child { node {(2,1)} }
				child { node {(2,2)} }
					child { node {(2,3)} }
		}
		child { node {(3)}
			child { node {(3,1)}
				child { node {(3,1,1)} }
			}
			child { node {(3,2)}
				child { node {(3,2,1)} }
				child { node {(3,2,2)} }
			}
		};
		
	\end{tikzpicture}
    \includegraphics[width=5.6cm]{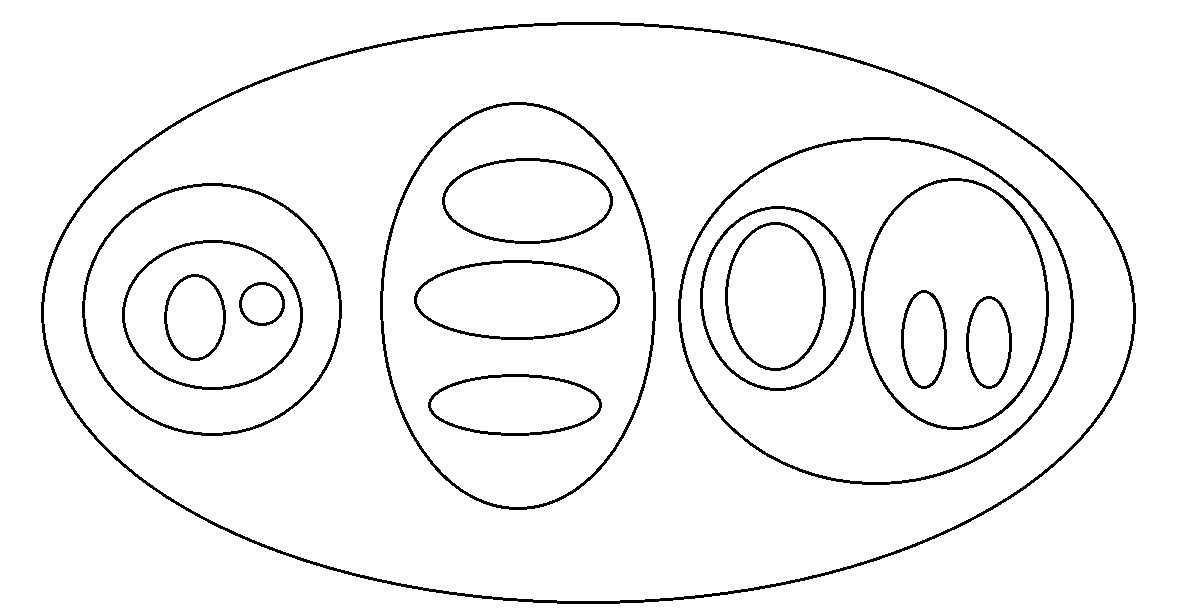}
	\caption{Example of a tree and a transversal cut of the corresponding
nesting of nodal domains.}
\end{figure}

\section*{Acknowledgements}
This work has received funding from the European Research Council (ERC) under the European Union's Horizon 2020 research and innovation program through the grant agreement~862342 (A.E.\ and A.G.-R.). It is partially supported by the grants CEX2023-001347-
S, RED2022-134301-T and PID2022-136795NB-I00 (A.E. and D.P.-S.) funded by
MCIN/AEI/10.13039/501100011033, and Ayudas Fundación BBVA a Proyectos de
Investigación Científica 2021 (D.P.-S.). A.G.-R. is also a postgraduate fellow of the Ministry of Science and Innovation at the Residencia de Estudiantes (2022–2024).
\appendix

\section{Local weak limit formulation of Berry's Random Wave Model}\label{AppendixA}
The existing relation between the local weak limit formulation of Berry's conjecture and the inverse localization property allows us to prove a slightly different version of this last property in a very particular case, $\mathcal{P}=\mathcal{Q}_1$. More precisely, we can prove the following result (the proof is presented in Section~\ref{app.3}).

\begin{theorem}\label{BILSquare}
	For any solution to Helmholtz $\vp$ in $\MW$, any $\ep>0$, any ball $B\subset\mathbb{R}^d$ and any $k\in\mathbb{N}$ there exists a number $J$ and a positive measure set $M=M(\varphi,k,\ep,J)\subset \mathcal{Q}_1$ such that for any point $z^0\in M$ we have \begin{equation}
		\left\|u_{n_J}\left(z^0+\frac{\cdot}{\sqrt{\lambda_{n_J}}}\right)-\vp\right\|_{C^k(B)}<\ep,
	\end{equation}where $\left\{n_j\right\}_{j=1}^\infty$  is a density one subsequence and $u_{n_j}$ are the subsequences that follow.
	\begin{itemize}
		\item Dirichlet boundary condition:\begin{equation}
	u_{n}^\mathrm{D}(z)=\frac{4}{\left|\mathcal{N}_\lambda^{\mathcal{Q}_1,\mathrm{D}}\right|^{1/2}}\sum_{N\in\mathcal{N}_\lambda^{\mathcal{Q}_1,\mathrm{D}}}\sin\left(z_1m\pi\right)\sin\left(z_2n\pi\right).
\end{equation}
		\item Neumann boundary condition:
		 \begin{equation}
	u_{n}^\mathrm{N}(z)=\frac{4}{\left|\mathcal{N}_\lambda^{\mathcal{Q}_1,\mathrm{N}}\right|^{1/2}}\sum_{N\in\mathcal{N}_\lambda^{\mathcal{Q}_1,\mathrm{N}}}\cos\left(z_1m\pi\right)\cos\left(z_2n\pi\right).
\end{equation}
	\end{itemize}
\end{theorem}
Here, we are using, $\mathcal{N}_\lambda^{\mathcal{Q}_1,\mathrm{D}}=\left\{N=(m,n)\in\mathbb{N}^2,\ m^2+n^2=\lambda\frac{1}{\pi^2}\right\}$ and also $\mathcal{N}_\lambda^{\mathcal{Q}_1,\mathrm{N}}=\left\{N=(m,n)\in\left(\mathbb{N}\cup\{0\right)^2,\ m^2+n^2=\lambda\frac{1}{\pi^2}\right\}$ for the sets of all possible pairs $(m,n)$ associated to the same eigenvalue $\lambda$.

The interest of this result is that the obtained version of the inverse localization property possesses a different flavor than the others introduced in this paper: here, the sequences of eigenvalues and eigenfunctions are fixed, while in the others theorems different sequences of eigenfunctions are used to approximate different monochromatic waves. In more general billiards (where we can expect Berry's property to hold) it is therefore more natural to study this approach, since eigenvalues in a generic billiard are known to be simple \cite{Uhlen}. 

The proof of Theorem \ref{BILSquare} comes from two facts: the existence of sequences in the square satisfying the Berry random wave model (this follows directly from a careful interpretation of \cite{Ingremeau}) and the fact that Berry's conjecture implies some inverse localization property (Theorem \ref{ILBerry}). To prove this latter result and, consequently, Theorem \ref{BILSquare}, we construct an appropriate continuous functional in $C^0\left(\RR^d\right)$ and we evaluate the local weak measures of the given sequences at this particular functional.

 In this Appendix, we first recall the rigorous formulation of Berry's conjecture proposed in \cite{Ingremeau}, then we establish and prove this close relation existing between this formulation and the inverse localization property and, finally, we see that the Berry's property holds in the unit square for Dirichlet and Neumann boundary conditions. 
 \subsection{Local weak limits formulation of Berry's Random Wave conjecture}
 Following \cite{Ingremeau} we introduce a new notion of convergence for the Laplace eigenfunctions in the high-energy limit, the local weak convergence. To do so, we associate a sequence of measures to a sequence of Laplace eigenfunctions, called \emph{local measures}. It is proved in Ingremeau's paper that we may always extract a subsequence of local measures which converge to a limit, called \emph{local weak limit} of the sequence of eigenfunctions, which is a measure on the space $\MW$. 
 
 Let us consider a bounded open set $\Omega\subset \mathbb{R}^d$, a sequence $u_n\in C^\infty(\Omega,\mathbb{R})$ and a sequence $\lambda_n$ going to infinity, such that \begin{equation}\label{eigeneq}
	\begin{aligned}
		&\left\|u_n\right\|_{L^2(\Omega)}^2=\Vol(\Omega),\text{ and }\\
		&\left\{\begin{aligned}
			&\Delta u_n+\lambda_{n}u_n=0, \text{ in }\Omega\\
			&u_n=0,\text{ in } \partial\Omega,
		\end{aligned}\right.
	\end{aligned}
\end{equation}

\begin{definition}\label{funcion}
	For $z^0\in\Omega$ and $n\in\mathbb{N}$ we define the function $\tilde{u}_{z^0,\mathrm{N}}\in C^\infty_c(\mathbb{R}^d)$,
	\begin{equation}
		\tilde{u}_{z^0,\mathrm{N}}(z)=u_n\left(z^0+\frac{z}{\sqrt{\lambda_n}}\right)\chi\left(\frac{|z|}{\sqrt{\lambda_{n}}\dist(z^0,\partial\Omega)}\right),
	\end{equation}where $\chi\in C^\infty_c\left([0,\infty),[0,1]\right)$ is a decreasing function taking value one in a neighborhood of the origin and vanishing outside $[0,1]$.
\end{definition}
\begin{remark}
	All this theory can be analogously developed when $\{u_n\}$ is a sequence of eigenfunctions on $\mathbb{T}^d=(\mathbb R/\mathbb Z)^d$ satisfying now \begin{equation}
		\begin{aligned}
			&\left\|u_n\right\|_{L^2(\mathbb T^d)}^2=1,\text{ and }\\
			&\left\{\begin{aligned}
				&\Delta u_n+\lambda_{n}u_n=0, \text{ in }\mathbb{T}^d\\
				&u_n\left(z+\ga\right)=u_n(z),\text{ for any } \ga\in\mathbb{Z}^d.
			\end{aligned}\right.
		\end{aligned}
	\end{equation}
	 In this case, the functions $ \tilde{u}_{z^0,\mathrm{N}}\in C^\infty_c(\mathbb{R}^d)$ are defined as \begin{equation}
		\tilde{u}_{z^0,\mathrm{N}}(z)=u_n\left(z^0+\frac{z}{\sqrt{\lambda_n}}\right)\chi\left(\frac{|z|}{\sqrt{\lambda_{n}}}\right).
	\end{equation}
	In the case of Neumann boundary conditions one can consider a definition analogous to Definition \ref{funcion}.
\end{remark}
We now introduce some notation that will be useful in what follows. Let us consider the space $C^0\left(\RR^d\right)$ equipped with the distance \begin{equation}
	\dz(f,g):=\inf\left\{\ep>0: \sup_{|z|<1/\ep}|f(z)-g(z)|<\ep  \right\}\,,\end{equation}
	which makes $\left(C^0\left(\RR^d\right),\dz\right)$ a Polish space. We denote by $\mathcal{M}_0$ the Banach space of finite signed measures on $C^0\left(\RR^d\right)$ and by $C_b\left(C^0\left(\RR^d\right)\right)$ the space of bounded real-valued continuous functions on the metric space $C^0\left(\RR^d\right)$ with the supremum norm. Notice that $\MW\subset C^0\left(\RR^d\right)$ is a closed subset of $C^0\left(\RR^d\right)$ with the topology given by $\dz$ so, in particular, it is a Polish space (as already mentioned in Chapter \ref{Intro}). In a similar way, we denote by $\mathcal{M}$ the Banach space of finite signed measures on $\MW$, by $C_b\left(\MW\right)$ the space of bounded real-valued continuous functions on $\MW$ equipped with the supremum norm and by $C_b\left(\MW\right)^*$ the topoligical dual of $C_b\left(\MW\right)$. 
	
	To conclude, we recall that by Tietze's extension theorem (see e.g. \cite[\S5]{Tie}), there exists a continuous linear map such that \begin{equation}\label{iota}
	\begin{aligned}
			\iota:& C_b\left(\MW\right)\rightarrow C_b\left(C^0\left(\RR^d\right)\right)\\&\restr{\iota G}{\MW} =G,\\&\left\|\iota G\right\|_{C_b\left(C^0\left(\RR^d\right)\right)}=\left\|G\right\|_{C_b\left(\MW\right)}.
	\end{aligned}
	\end{equation}
We are now ready to define the local measures.
\begin{definition}
	For each $n\in\mathbb{N}$ and $z^0\in\Omega$, we have $\tilde{u}_{z^0,\mathrm{N}}\in C^0\left(\RR^d\right)$, so we may define $\delta_{\tilde{u}_{z^0,\mathrm{N}}}\in\mathcal{M}_0$. For any set $U\subset\Omega$ of positive Lebesgue measure, we then define the \emph{local measure} of $u_n$ on $U$ as \begin{equation}
		\LM_{U}(u_n):=\frac{1}{\Vol(U)}\int_Udz^0\delta_{\tilde{u}_{z^0,\mathrm{N}}}.
	\end{equation}This defines a probability measure in $\mathcal{M}_0$. Taking $\iota$ as in \eqref{iota} we can also define $\LM_{U}^{\iota}(u_n)\in \left(C_b(\MW)\right)^*$ by
	\begin{equation}
		\forall G\in C_b(\MW),\ \left\langle\LM_{U}^{\iota}(u_n),G\right\rangle=\left\langle\LM_{U}(u_n),\iota G\right\rangle.
	\end{equation}
\end{definition}
\begin{remark}
	Alternatively, we can consider the space $C^k\left(\RR^d\right)$ for some $k\in\mathbb{N}$ and a distance similar to $d_0$ given by\begin{equation}
		\dk(f,g):=\inf\left\{\ep>0: \left\|f-g\right\|_{C^k\left(B\left(0,1/\ep\right)\right)}<\ep\right\}\,,\end{equation}which makes it a Polish space, and introduce a new definition of local measure in the Banach space of finite signed measures on $C^k\left(\RR^d\right)$. This construction is done in detail in \cite{Ingremeau}, where it was also proven that it does not depend on the value of $k$. 
		\end{remark}

The other essential definition for the formulation of Berry's conjecture is the following.
\begin{definition}
	We denote by $\sigma_{U}\left(\left\{u_n\right\}_n\right)$ to the set of accumulation points of $\LM_{U}(u_n)$ for the weak$-*$ topology, and by $\sigma_{U}^{\iota}\left(\{u_n\}_n\right)$ the set of accumulation points of $\LM_{U}^{\iota}\left(u_n\right)$ for the weak$-*$ topology.
\end{definition}
 Regarding these measures, it was proved in \cite[Corollary 1]{Ingremeau} that it is always possible to extract a subsequence of local measures that  converges to a measure $\mu\in\mathcal{M}_0$. It is also proved that this measure $\mu$ is supported in $\MW$, and moreover, that any measure $\nu\in \sigma_{U}^{\iota}\left(\{u_n\}_n\right)$ is in $\mathcal{M}$. This is summarized in the following lemma. 
\begin{lemma}
	Let  $u_n\in C^\infty(\Omega,\mathbb{R})$ be as in \eqref{eigeneq}. Let $U\subset\Omega$ be a Borel set of positive Lebesgue measure. There exists a subsequence $n_j$ and a probability measure $\mu\in\mathcal{M}_0$ such that $\LM_{U}(u_n)\stackrel{*}{\rightharpoonup} \mu$, i.e., for all $G\in C_b\left(\MW\right)$, we have \begin{equation}
		\lim_{n\to\infty}\left\langle\LM_{U}(u_n),G\right\rangle=\left\langle\mu,G\right\rangle.
	\end{equation}
	Moreover, $\mu$ is supported in $\MW$.
\end{lemma}
We now introduce three more definitions to finish the section and state this formulation of Berry's conjecture:
\begin{definition}
	For any measure $ \nu\in\mathcal{M}$ we say that $\nu$ is the local weak limit of $\{u_n\}_n$ on $U$ if $\sigma_U\left(\left\{u_n\right\}_n\right)=\{\nu\}$. 
\end{definition}
\begin{definition}[Berry's random monochromatic wave]
		
We call \emph{random isotropic monochromatic wave}, and denote by $\PsiR$, the unique centered stationary Gaussian field on $\mathbb{R}^d$ whose covariance function is\begin{equation}
	\mathbb{E}\left[\PsiR(x)\PsiR(y)\right]=\int_{\mathbb{S}^{d-1}}e^{2\pi i(x-y)\theta}d\omega_{d-1}(\theta),
\end{equation}where $d\omega_{d-1}$ is the uniform measure on $\mathbb{S}^{d-1}$.

We denote by  $\muRMW$ the Borel measure associated to $\PsiR$, i.e. \begin {equation}\begin{aligned}\mu_{Berry}:\ &\mathcal{B}(\MW)\rightarrow\mathbb{R}^+_0\\&A\mapsto P(\PsiR\in A).\end{aligned}\end{equation}
\end{definition}
\begin{definition}[Berry’s random wave property]\label{BerryLWL}
    Let $\Omega\subset\mathbb{R}^d$ be an open set with positive Lebesgue measure. We say that $\Omega$, with boundary conditions $\mathrm{BC}$, satisfies the Berry's Random Wave property if there exist a $\{\lambda_n\}$ a sequence of eigenvalues with full density, and $\{u_n\}_n$ an orthogonal sequence of real-valued eigenfunctions of the Laplacian in $\Omega$ with boundary conditions $\mathrm{BC}$, satisfying $\left\|u_n\right\|_{L^2(\Omega)}^2=\Vol(\Omega)$, and such that its local weak limit in $\Omega$ is $\muRMW$, i.e.,  \begin{equation}
		\sigma_\Omega\left(\left\{u_n\right\}_n\right)=\left\{\muRMW\right\}.
	\end{equation}
\end{definition}

\begin{conjecture}[Berry’s random wave conjecture]
	
	Let $\Omega\subset\mathbb{R}^d$ be an open set with positive Lebesgue measure and such that the billiard flow is chaotic in the sense of ~\cite[Section 2]{cbilliards}. Then $\Omega$ with Dirichlet boundary condition satisfies the Berry's random wave property.
\end{conjecture}

\subsection{Relation with the inverse localization property}
As stated in Chapter \ref{Intro}, it is possible to derive a strong version of the inverse localization property from this formulation of Berry's conjecture. We deduce it in the following result.
\begin{theorem}[Strong inverse localization where Berry's random wave property is true]\label{ILBerry}
	Let $\Omega$ be a chaotic billiard, with a certain boundary condition, where Berry's random wave property is true in the sense of Statement \ref{BerryLWL}. For any $\vp\in\MW$, any ball $B\subset\RR^d$, any $\ep>0$, any $U\subset\Omega$ open subset and any $k\in\mathbb{N}$, there exist a natural number $J\in\mathbb{N}$ and a positive measure set $M=M(k,\ep,J,\varphi,B, U)\subset U$ such that for any point $z^0\in M\subset U$ we have \begin{equation}
		\left\|u_{n_J}\left(z^0+\frac{\cdot}{\sqrt{\lambda_{n_J}}}\right)-\vp\right\|_{C^k(B)}<\ep,
	\end{equation}where $\left\{u_{n_j}\right\}$ is a subsequence of eigenfunctions on $\Om$ with the prescribed boundary condition associated to the sequence $\lambda_{n_j}$ and given by the Berry's random wave property, i.e., satisfying that $	\sigma_\Omega\left(\left\{u_{n_j}\right\}_j\right)=\left\{\muRMW\right\}.$
\end{theorem}
\begin{proof}
	Fix any open set $U\subseteq \Omega$.	Since we assume that Berry's conjecture holds, we have	\begin{equation}
		\sigma_U\left(\left\{u_{n_j}\right\}_j\right)=\left\{\muRMW\right\},
	\end{equation}which means \begin{equation}
		\forall F\in \mathcal{C}_b\left(C^0\left(\RR^d\right)\right), \ \lim_{j\rightarrow\infty}\left\langle \LM_{U}\left(\tilde{u}_{z^0,n_j}\right), F\right\rangle=\left\langle \muRMW, F\right\rangle.
	\end{equation}This can be understood as \begin{equation}
		\lim_{j\rightarrow\infty}\frac{1}{\Vol(U)}\int_{U}F\left(\tilde{u}_{z^0,n_j}\right)dz^0=\int_{C^0\left(\RR^d\right)}Fd\muRMW. 
	\end{equation}Now let $\chi\in C^\infty(\mathbb{R}^d)$ be an even decreasing function that is equal to $1$ on $(0,\ep/2)$ and is supported on $(0,\ep)$. For the chosen $\varphi,\ep,$ and $k$ we define the functional $F_0\in \mathcal{C}_b(C^\infty\left(\RR^d\right))$ given by \begin{equation}
		F_0(f):=\chi\left(\left\|f-\varphi\right\|_{C^k(B)}\right)
	\end{equation}for any $f\in C^\infty(\mathbb{R}^d)$. Explicitly, we have  \begin{equation}
		F_0(f)=\chi\left(\left\|f-\varphi\right\|_{C^k(B)}\right)=\left\{\begin{aligned}
			&1 \text{ if } \|f-\varphi\|_{C^k(B)}<\ep/2 \\
			&\text{Smooth in the middle}\\
			&	0 \text{ if } \|f-\varphi\|_{C^k(B)}>\ep.
		\end{aligned}\right.
	\end{equation}In particular,  \begin{equation}
		F_0\left(\tilde{u}_{z^0,n_j}\right)=\chi\left(\left\|\tilde{u}_{z^0,n_j}-\varphi\right\|_{C^k(B)}\right),
	\end{equation}for any $j\in\NN$. This non-linear functional is continuous and bounded.
	
	The next step is to see that  $\left\langle \muRMW, F_0\right\rangle>0$. This is easy and a proof can be found e.g. in \cite{Yo}.  Let us briefly recall the argument here.
	 
	First, notice that any solution to the Helmholtz equation in the ball $B\subset \mathbb R^d$ can be expanded as a Bessel-Fourier series \begin{equation}
		\varphi(z)= \sum_{l=0}^\infty \sum_{m=1}^{d_l} c_{lm}\,\frac{J_{l+\frac d2-1}(|z|)}{|z|^{\frac d2-1}}\, Y_{lm}\left(\frac z{|z|}\right).
	\end{equation}
	On the other hand, it is known that the random field $\PsiR$ associated to the measure $\muRMW$ can be written as \begin{equation}
		\PsiR(x)= \sum_{l=0}^\infty \sum_{m=1}^{d_l} a_{lm}\,\frac{J_{l+\frac d2-1}(|x|)}{|x|^{\frac d2-1}}\, Y_{lm}\left(\frac x{|x|}\right),
	\end{equation}where $a_{lm}$ are independent random Gaussian variables.
	Since the topology we are considering in $C^0\left(\mathbb{R}^d\right)$ is precisely the topology given by convergence of each coefficient in the expansion, the claim follows because for any $l\in\mathbb{N}$ and any $1\leq m\leq d_l$, $\mathbb P(\{|a_{lm}-c_{lm}|<\varepsilon\})>0$.

	 We conclude that $\frac{1}{\Vol(U)}\int_{U}F\left(\tilde{u}_{z^0,n_j}\right)dz^0$ converges to a positive number when $j\rightarrow\infty$, and hence for big enough $J$ there must be a positive measure set of points $M$ on $U\subseteq \Om$ for which $F\left(\tilde{u}_{z^0,n_J}\right)>0$. This clearly implies, by the definition of $F$, that for those points, we have  \begin{equation}
	 	\|\tilde{u}_{z^0,n_J}-\varphi\|_{C^k(B)}=	\left\|u_{n_J}\left(z^0+\frac{\cdot}{\sqrt{\lambda_{n_J}}}\right)-\vp\right\|_{C^k(B)}\leq\ep.
	 \end{equation}This concludes the proof.\end{proof}
	Regarding this result, a few remarks are in order.
\begin{remark}
	
		\begin{description}\  \\	\item[The set $M$] The set of admissible base points satisfies the condition of being ``asymptotically dense'', in the sense for any arbitrarily small open set $ U\subset \Om$ and any arbitrarily small $\ep>0$, there exists a $J\in\mathbb{N}$ big enough such that $M(k,\ep,J,\vp,B,U)$ has positive measure inside $U$. This agrees with the idea of equidistribution that appears in Berry's random wave property: all the subsets of $\Omega$ are indistinguishable with respect to the inverse localization property. 
        
		\item[Failure of the inverse localization property] If the inverse localization property fails strongly as in Definition \ref{NOIL}, this version of the inverse localization does not occur either. This can easily be seen by the fact that, for any $\varphi\in\MW$, Theorem \ref{ILBerry} gives us that \begin{equation}
		    \inf_{z^0\in\Omega}\inf_{\lambda\in \Lambda_\Omega^{\mathrm{BC}}}\inf_{u_\lambda\in\mathcal{V}_{\Omega,\mathrm{BC}}^\lambda}\left\|\varphi-u_\lambda\left(z^0+\frac{\cdot}{\sqrt{\lambda}}\right)\right\|_{C^0\left(B\right)}=0,
		\end{equation}since this quantity is seen to be smaller than any $\epsilon>0$. This is in contradiction with those  $\vp\in\mathcal{N}_{\Omega,\mathrm{BC}}$. As a consequence of  we get a list of integrable billiards in which the aforementioned Berry's random wave property cannot hold true: irrational rectangles,  balls  $\mathcal{B}\subset\mathbb{R}^d$ and generic ellipses $\mathcal{E}_b$, with Dirichlet or Neumann boundary condition, and $\mathcal{Q}_1$ or $\mathcal{T}_{\mathrm{equi}}$, with certain Robin boundary condition.
        
        These statements agree with the fact that Berry's random wave property is expected to be false in general non-chaotic billiards.
		\item[Different contexts] An analogous theorem was proved in \cite{Yo} for closed manifolds instead of Euclidean domains. The proof here follows the same strategy as the proof of Proposition 6.4 there. 	
        
		\item[Two different versions of inverse localization] The main differences between the inverse localization result that we get here and the definition of inverse localization without a fixed base point that we use in the rest of the paper are mostly two. To begin with, in the latter the subsequence of eigenvalue $\lambda_{n_j}$ and eigenfunctions $u_{n_j}$ that we utilize depend on the function $\varphi\in\MW$, while in the former, both subsequences are fixed beforehand and they are the ones whose associated local weak measures converge to $\muRMW$. The other main difference is related with the distribution of points around which we can localized. As already mentioned, in the strong formulation of the inverse localization, we can find such a point in any open set of the domain, while in the other one we do not ask to have such a control. 
        
        As a consequence of both of them, we are able to prove the following stronger form of inverse localization which allows us to approximate different monochromatic waves around different points on a billiard using the same sequence of eigenfunctions:\begin{corollary}
				Let $\Omega$ be a chaotic billiard, with a certain boundary condition, where Berry's conjecture is true in the sense of Statement \ref{BerryLWL}. For any $N\in\mathbb{N}$ and any $\vp_1,\ldots,\vp_N\in\MW$, any ball $B\subset\RR^d$, any $\ep>0$ and any $k\in\mathbb{N}$, there exist a natural number $J\in\mathbb{N}$ and $N$ positive measure sets $M_i=M_i(k,\ep,J)\subset \Om$ for $1\leq i\leq N$ such that for any point $z^i\in M_i$ we have \begin{equation}
				\left\|u_{n_J}\left(z^i+\frac{\cdot}{\sqrt{\lambda_{n_J}}}\right)-\vp_i\right\|_{C^k(B)}<\ep,
			\end{equation}where $\left\{u_{n_j}\right\}$ is a subsequence of eigenfunctions on $\Om$ with the prescribed boundary condition associated to the sequence $\lambda_{n_j}$ and given by the Berry's conjecture, i.e., satisfying that $	\sigma_\Omega\left(\left\{u_{n_j}\right\}_j\right)=\left\{\muRMW\right\}.$
		\end{corollary}
		The only known example of a context where a multiple inverse localization property like this is satisfied are $d-$dimensional spheres (see \cite[Section 5]{EPT21}).
	\end{description}
\end{remark}
\subsection{Random wave model for square eigenfunctions}\label{app.3}
  It was shown in \cite{Ingremeau} that certain families of eigenfunctions
  on $\mathbb{T}^2$ satisfy the conclusion of Berry’s conjecture (although no chaotic dynamics is present here). Using this result we can prove that a particular sequence of eigenvalues and eigenfunctions of the eigenvalue problem in the unit square $\mathcal{Q}_1=[0,1]\times[0,1]$ (both with Dirichlet and Neumann conditions) satisfies the local weal limit formulation of Berry's Random Wave Model. 

We first recall that, in Ingremeau's paper \cite{Ingremeau}, the following result is proven.

\begin{theorem}\label{BerryIngremeau}
	There exists a density one sequence $n_j$ such that we have \begin{equation}
		\sigma_{U}\left(\{u_{n_j}\}_j\right)=\left\{\muRMW\right\},
	\end{equation} for any Borel set $U\subset\mathbb{T}^2$, where 
	\begin{equation}
		u_{n_j}(z)=\frac{1}{\left|\mathcal{N}_{\lambda_{n_j}}^{\mathcal{Q}_1,\mathrm{P}}\right|^{1/2}}\sum_{\xi\in\mathcal{N}_{\lambda_{n_j}}^{\mathcal{Q}_1,\mathrm{P}}}e^{2i\pi z\cdot \xi},
	\end{equation} is an eigenfunction of the torus $\mathbb{T}^2$ associated to the eigenvalue $\lambda_{n_j}$.
\end{theorem}
\begin{remark}
	Let us briefly comment on the proof of this theorem. It begins by defining a measure on $\mathbb{S}^1$ given by \begin{equation}
		\mu_{u_n}:=\frac{1}{\left|\mathcal{N}_{\lambda_n}^{\mathcal{Q}_1,\mathrm{P}}\right|}\sum_{\xi\in\mathcal{N}_{\lambda_n}^{\mathcal{Q}_1,\mathrm{P}}}\delta_{\xi/\sqrt{\lambda_{n}}}.
	\end{equation} Actually, the theorem is valid for more general eigenfunctions of the form\begin{equation}
		u_{n_j}(z)=\frac{1}{\left|\mathcal{N}_{\lambda_{n_j}}^{\mathcal{Q}_1,\mathrm{P}}\right|^{1/2}}\sum_{\xi\in\mathcal{N}_{\lambda_{n_j}}^{\mathcal{Q}_1,\mathrm{P}}}a_\xi e^{2i\pi z\cdot \xi},
	\end{equation} with a more general coefficient $a_\xi$ satisfying  \cite[Hypothesis 1]{Ingremeau}. The proof then continue utilizing the ``derandomization'' method introduced by Bourgain, Buckley and Wigman in \cite{Bourgain} and \cite{Buckley}. The basic idea of this is that we can approximate $u_{n_j}$ by a linear combination of Gaussian random variables with a not too big error term. To conclude, we use this to see that also the local weak measure of $u_{n_j}$ is close to $\muRMW$ as $n_j$ tends to infinity.
\end{remark}
Now we would like to apply this theorem to different boundary conditions. For the eigenvalue problem in the square $\left[0,1\right]\times\left[0,1\right]$ with Dirichlet conditions we can consider the sequence of eigenfunctions given by 
\begin{equation}
	u_{n}^\mathrm{D}(z)=\frac{4}{\left|\mathcal{N}_\lambda^{\mathcal{Q}_1,\mathrm{D}}\right|^{1/2}}\sum_{N\in\mathcal{N}_\lambda^{\mathcal{Q}_1,\mathrm{D}}}\sin\left(z_1m\pi\right)\sin\left(z_2n\pi\right).
\end{equation}
We can then prove the following result: 
\begin{lemma}\label{BD}
	There exists a density one sequence $n_j$ such that  \begin{equation}
		\sigma_{U}\left((u_{n_j}^\mathrm{D})_j\right)=\left\{\mu_{\text{Berry}}\right\},
	\end{equation}for any Borel set $U\subset\mathcal{Q}_1$.
\end{lemma}\begin{proof}
	The eigenfunction can be rewritten as 
	\begin{equation}
		u_{n}(z)=\frac{1}{\sqrt{\left|\mathcal{N}_\lambda^{\mathcal{Q}_1,\mathrm{D}}\right|}}\sum_{\substack{N\in\mathcal{N}_\lambda^{\mathcal{Q}_1},\mathrm{P},\\ \lambda\text{ not a square}}}\sign\left(m\cdot n\right)e^{\pi i z\cdot N},
	\end{equation} where we have used that $\left|\mathcal{N}_\lambda^{\mathcal{Q}_1,\mathrm{P}}\right|=4\mathcal{N}_\lambda^{\mathcal{Q}_1,\mathrm{D}}$ (when $\lambda$ is not a perfect square) and some trigonometric formulas. In this case, the coefficients are given by the formula $a_\xi=\frac{1}{\sqrt{\left|\mathcal{N}_\lambda^{\mathcal{Q}_1,\mathrm{D}}\right|}}\sign\left(m\cdot n\right)$, for any $N=(m,n)\in\mathcal{N}_\lambda^{\mathcal{Q}_1,\mathrm{D}}$, and we get the same associated measure on $\mathbb{S}^1$. Therefore, it satisfies the  conditions of \cite[Hypothesis 1]{Ingremeau} and we can apply again Bourgain, Buckley and Wigman techniques.
	
To conclude the proof, one needs to take into account two things. The first one is that local weak measures in $\mathcal{Q}_1$ and in $\mathbb{T}^2$ differ in the presence of the cut-off function $\chi$. However, it is just introduced to ensure that $\tilde{u}_{z^0,\mathrm{N}}$ is defined in the whole $\RR^d$ and, by its definition, it is easy to check that it does not make any difference in the limit. The second one, is that any Borel set $U$ of $\mathcal{Q}_1$ is also a Borel set of $\mathbb{T}^2$. With all this, we can finally argue as in the proof of \cite[Theorem 2]{Ingremeau} and conclude that, taking as a full density sequence the one of \ref{BerryIngremeau} without the perfect squares (this is still of full density), for any Borel set $U\subset\mathcal{Q}_1$, \begin{equation}
		\sigma_{Q}\left((u_{n_j}^\mathrm{D})_j\right)=	\sigma_{\mathbb{T}^2}\left((u_{n_j})_j\right)=\left\{\muRMW\right\}.
	\end{equation}
\end{proof}

	Even though the definitions in \cite{Ingremeau} only work for the Dirichlet boundary condition, a similar analysis can be done for the Neumann condition. We therefore consider now the similar problem but with Neumann condition. The sequence of eigenfunctions is now defined as \begin{equation}
	u_{n}^\mathrm{N}(z)=\frac{4}{\left|\mathcal{N}_\lambda^{\mathcal{Q}_1,\mathrm{N}}\right|^{1/2}}\sum_{N\in\mathcal{N}_\lambda^{\mathcal{Q}_1,\mathrm{N}}}\cos\left(z_1m\pi\right)\cos\left(z_2n\pi\right).
\end{equation}
\begin{lemma}\label{BN}
	There exists a density one sequence $n_j$ such that \begin{equation}
		\sigma_{U}\left((u^\mathrm{N}_{n_j})_j\right)=\left\{\muRMW\right\},
	\end{equation}for any Borel set $U$ of $\mathcal{Q}_1$.\end{lemma}
	\begin{proof}
		The eigenfunction can be rewritten as
		\begin{equation}
		u_{n}^\mathrm{N}(z)=\frac{1}{\sqrt{\left|\mathcal{N}_\lambda^{\mathcal{Q}_1,\mathrm{N}}\right|}}\sum_{N\in\mathcal{N}_\lambda^{\mathcal{Q}_1,\mathrm{P}}}e^{\pi i x\cdot N},
		\end{equation}
		whenever $\lambda_n$ is not a perfect square, and the associated measure is again essentially the same. Consequently, the proof is analogous to the proof of the previous lemma, and the result follows.
	\end{proof}

We conclude this section by noticing that Theorem \ref{BILSquare} follows by just considering together Lemma \ref{BD} and \ref{BN}, respectively for Dirichlet or Neumann conditions, with Theorem \ref{ILBerry}. This result provides the only known example of a billiard where the conclusion of Berry's conjecture holds and, therefore, the stronger form of inverse localization.

\section{Tables of symmetry conditions}\label{AppendixC}
 In this appendix we collect the symmetry conditions that we need to impose to the monochromatic waves in the positive part of Theorem \ref{BT.polygons} and Theorem \ref{T.LATICE} in order to get the inverse localization property around any fixed but arbitrary point of $\mathcal{P}$. In this first table, we gather the conditions when $\mathcal{P}$ is a rational rectangle or an integrable triangle. Remember that only if $\mathcal{P}=\mathcal{Q}_l$, we consider $d\geq 2$.
 \subsection{Symmetry conditions in Theorem \ref{BT.polygons}}
The conditions that we need to ask to $\varphi\in\MW$ are given by:	\begin{center}
\begin{equation}\label{TableA}
		\begin{array}{ccc} 
		\hline
	\rule{0pt}{3.2ex}	\text{Boundary conditions}	& \text{Polygon} &\text{Symmetry condition} \\
		\hline
		\multirow{4}{*}{Dirichlet} & \rule{0pt}{3.2ex}\mathcal{Q}_l& \vp(z_1,\ldots,-z_j,\ldots,z_d)=-\vp(z)\text{ for all }1\leq j\leq d\\
		& \rule{0pt}{3.2ex}\mathcal{T}_{\mathrm{iso}} & \vp(-z_1,z_2)=\vp(z_1,-z_2)=\vp(z_2,z_1)=-\vp(z_1,z_2)\\
		
		&\rule{0pt}{3.2ex} \mathcal{T}_{\mathrm{equi}} & \vp(z_1,-z_2)=-\vp(z_1,z_2) \\
		&\rule{0pt}{3.2ex} \mathcal{T}_{\mathrm{hemi}} &\vp(z_1,-z_2)=\vp(-z_1,z_2)=-\vp(z_1,z_2)	\vspace{0.1cm} \\
		\hline
		\multirow{4}{*}{Neumann} &\rule{0pt}{3.2ex} \mathcal{Q}_l & \vp(z_1,\ldots,-z_j,\ldots,z_d)=\vp(z)\text{ for all }1\leq j\leq d \\
		&\rule{0pt}{3.2ex} \mathcal{T}_{\mathrm{iso}} & \vp(-z_1,z_2)=\vp(z_1,-z_2)=\vp(z_2,z_1)=\vp(z_1,z_2)   \\
		& \rule{0pt}{3.2ex}\mathcal{T}_{\mathrm{equi}} & \vp(z_1,-z_2)=\vp(z_1,z_2) \\
		& \rule{0pt}{3.2ex}\mathcal{T}_{\mathrm{hemi}} & \vp(z_1,-z_2)=\vp(-z_1,z_2)=\vp(z_1,z_2) 	\vspace{0.1cm}\\
		\hline
	\end{array} 
\end{equation}
\end{center}
\begin{remark}\label{conditions}
	In all these cases, the conditions correspond to symmetry or anti symmetry when reflecting with respect to one of the axes or the line $z_2=z_1$.  In the $d=2$ case, this can be generalized by defining \begin{equation}\label{reflex}
		\mathbb{S}_{m}\left(z_1,z_2\right)=\left(\frac{z_1\left(1-m^2\right)-2mz_2}{1+m^2},\frac{2mz_1+z_2\left(m^2-1\right)}{1+m^2}\right)
	\end{equation} as the reflection with respect to the line $z_2=mz_1$ (write $\mathbb{S}_\infty$ and consider the limit $m\rightarrow+\infty$ for the line $z_1=0$). Therefore, setting $\sigma_{\mathrm{D}}=1$ and $\sigma_{\mathrm{N}}=2$, we can rewrite this as follows:
	\begin{itemize}
		\item For $\mathcal{Q}_l$, we have for any $1\leq j\leq d$, \begin{equation}\label{sim}\varphi\left(\mathbb{S}_j^{\mathcal{Q}_l}\left(z\right)\right)=\varphi\left(z_1,\ldots,-z_j,\ldots,z_d\right)=\left(-1\right)^{\sigma_{\mathrm{BC}}}\varphi(z).\end{equation} 
		\item For $\mathcal{T}_{\mathrm{iso}}$, we have \begin{enumerate}
			\item \begin{equation}\vp\left(\mathbb{S}_\infty\left(z\right)\right)=\vp(-z_1,z_2)=\left(-1\right)^{\sigma_{\mathrm{BC}}}\vp(z_1,z_2)\end{equation}
			\item \begin{equation}
			    \vp\left(\mathbb{S}_0\left(z\right)\right)=\vp(z_1,-z_2)=\left(-1\right)^{\sigma_{\mathrm{BC}}}\vp(z_1,z_2)
			\end{equation}
			\item \begin{equation}\vp\left(\mathbb{S}_1\left(z\right)\right)=\vp(z_2,z_1)=\left(-1\right)^{\sigma_{\mathrm{BC}}}\vp(z_1,z_2).\end{equation}
		\end{enumerate}
		\item For $\mathcal{T}_{\mathrm{equi}}$, we have  $\vp\left(\mathbb{S}_0\left(z\right)\right)=\vp(z_1,-z_2)=\left(-1\right)^{\sigma_{\mathrm{BC}}}\vp(z_1,z_2)$.
			\item For $\mathcal{T}_{\mathrm{hemi}}$, we have  $\vp\left(\mathbb{S}_0\left(z\right)\right)=\vp(z_1,-z_2)=\left(-1\right)^{\sigma_{\mathrm{BC}}}\vp(z_1,z_2)$ and also $\vp\left(\mathbb{S}_\infty\left(z\right)\right)=\vp(-z_1,z_2)=\left(-1\right)^{\sigma_{\mathrm{BC}}}\vp(z_1,z_2)$.
	\end{itemize}
	\end{remark}

This new notation allows us to present in the next Table in a more simple way the conditions needed for Theorem \ref{T.LATICE}.
 \subsection{Symmetry conditions in Theorem \ref{T.LATICE}}
The conditions that we need to ask to $\varphi\in\MW$ are given by:

\begin{center}
\begin{equation}\label{TableB}
		\begin{array}{ccc} 
		\hline
	\rule{0pt}{3.2ex}\mathrm{BC}	& \text{Polygon} &\text{Symmetry condition} \\
		\hline
		\multirow{5}{*}{$\mathrm{D}$} &\rule{0pt}{3.2ex} \mathcal{Q}_l& \vp\left(\mathbb{S}_\infty\left(z\right)\right)=\vp\left(\mathbb{S}_0\left(z\right)\right)=-\vp(z_1,z_2)\\
		& \rule{0pt}{3.2ex}\mathcal{T}_{\mathrm{iso}} & \vp\left(\mathbb{S}_\infty\left(z\right)\right)=\vp\left(\mathbb{S}_0\left(z\right)\right)=\vp\left(\mathbb{S}_1\left(z\right)\right)=\vp\left(\mathbb{S}_{-1}\left(z\right)\right)=-\vp(z_1,z_2) \\
		
		&\rule{0pt}{3.2ex} \mathcal{T}_{\mathrm{equi}} & \vp\left(\mathbb{S}_0\left(z\right)\right)=\vp\left(\mathbb{S}_{\sqrt{3}}\left(z\right)\right)=\vp\left(\mathbb{S}_{-\sqrt{3}}\left(z\right)\right)=-\vp(z_1,z_2) \\
		&\rule{0pt}{3.2ex} \multirow{2}{*}{$\mathcal{T}_{\mathrm{hemi}}$} & \vp\left(\mathbb{S}_\infty\left(z\right)\right)=\vp\left(\mathbb{S}_0\left(z\right)\right)=\vp\left(\mathbb{S}_{\sqrt{3}}\left(z\right)\right)=\\&&\vp\left(\mathbb{S}_{-\sqrt{3}}\left(z\right)\right)=\vp\left(\mathbb{S}_{\sqrt{3}/3}\left(z\right)\right)=\vp\left(\mathbb{S}_{-\sqrt{3}/3}\left(z\right)\right)=-\vp(z_1,z_2)	\vspace{0.1cm} \\
		\hline
		\multirow{4}{*}{$\mathrm{N}$} &\rule{0pt}{3.2ex} \mathcal{Q}_l & \vp\left(\mathbb{S}_\infty\left(z\right)\right)=\vp\left(\mathbb{S}_0\left(z\right)\right)=\vp(z_1,z_2) \\
		& \rule{0pt}{3.2ex}\mathcal{T}_{\mathrm{iso}} & \vp\left(\mathbb{S}_\infty\left(z\right)\right)=\vp\left(\mathbb{S}_0\left(z\right)\right)=\vp\left(\mathbb{S}_1\left(z\right)\right)=\vp\left(\mathbb{S}_{-1}\left(z\right)\right)=\vp(z_1,z_2)   \\
		& \rule{0pt}{3.2ex}\mathcal{T}_{\mathrm{equi}}& \vp\left(\mathbb{S}_0\left(z\right)\right)=\vp\left(\mathbb{S}_{\sqrt{3}}\left(z\right)\right)=\vp\left(\mathbb{S}_{-\sqrt{3}}\left(z\right)\right)=\vp(z_1,z_2) \\
		&\rule{0pt}{3.2ex} \multirow{2}{*}{$\mathcal{T}_{\mathrm{hemi}}$} & \vp\left(\mathbb{S}_\infty\left(z\right)\right)=\vp\left(\mathbb{S}_0\left(z\right)\right)=\vp\left(\mathbb{S}_{\sqrt{3}}\left(z\right)\right)=\\&&\vp\left(\mathbb{S}_{-\sqrt{3}}\left(z\right)\right)=\vp\left(\mathbb{S}_{\sqrt{3}/3}\left(z\right)\right)=\vp\left(\mathbb{S}_{-\sqrt{3}/3}\left(z\right)\right)=\vp(z_1,z_2)	\vspace{0.1cm} \\
		\hline
	\end{array} \end{equation}
\end{center}\begin{remark}
Notice that conditions in Table \ref{TableA} are contained in those of Table \ref{TableB}. These extra restrictions appear as a consequence of the construction of the lattices $L_{\mathcal{P}}$. More precisely, what occurs is that the set of lines with respect to which we make the symmetry is not preserved when constructing the tiling (see Figure \ref{til}). This does not happen in the case of $\mathcal{Q}_l$, and consequently we do not need any new conditions.
\end{remark}
\begin{remark}\label{nod}
  The  reflections appearing in Table \ref{TableB} form in each case a finite group of isometries that divides the plane in a finite number of unbounded domains. We will choose one of those domains and call it the \emph{fundamental sector}, denoted by $\mathbb{F}_\mathcal{P}$. More precisely, we have the following:\begin{itemize}
      \item When $\mathcal{P}=\mathcal{Q}_l$, the group is $\mathbb{G}_{\mathcal{Q}_l}$ and generated by $\left\{\mathbb{S}_j^{\pi/2}\right\}_{j=1}^2=\left\{\mathbb{S}_0,\mathbb{S}_\infty\right\}$ and the fundamental sector is $\mathbb{F}_{\mathcal{Q}_l}=\left\{z\in\mathbb{R}^2,\ z_1>0,\ z_2>0\right\}$.
      \item If $\mathcal{P}=\mathcal{T}_{\mathrm{iso}}$, the group $\mathbb{G}_{\mathcal{T}_{\mathrm{iso}}}$ has as generators the isometries $\left\{\mathbb{S}_j^{\pi/4}\right\}_{j=1}^4=\left\{\mathbb{S}_0,\mathbb{S}_\infty,\mathbb{S}_1,\mathbb{S}_{-1}\right\}$ and the fundamental sector is given by the set $\mathbb{F}_{\mathcal{T}_{\mathrm{iso}}}=\left\{z\in\mathbb{R}^2,\ z_1>0,\ z_2>0,\ z_1>z_2\right\}$.
      \item When $\mathcal{P}=\mathcal{T}_{\mathrm{equi}}$, we have the group $\mathbb{G}_{\mathcal{T}_{\mathrm{equi}}}$  generated by $\left\{\mathbb{S}_j^{\pi/3}\right\}_{j=1}^3=\left\{\mathbb{S}_0,\mathbb{S}_{\sqrt{3}},\mathbb{S}_{-\sqrt{3}}\right\}$ and we have as fundamental sector the following set: $\mathbb{F}_{\mathcal{T}_{\mathrm{equi}}}=\left\{z\in\mathbb{R}^2,\ z_1>0,\ z_1>\sqrt{3}z_2\right\}$.
      \item Finally, when $\mathcal{P}=\mathcal{T}_{\mathrm{hemi}}$, the group of isometries $\mathbb{G}_{\mathcal{T}_{\mathrm{hemi}}}$  is generated by $\left\{\mathbb{S}_j^{\pi/6}\right\}_{j=1}^6=\left\{\mathbb{S}_0,\mathbb{S}_\infty,\mathbb{S}_{\sqrt{3}},\mathbb{S}_{-\sqrt{3}},\mathbb{S}_{\sqrt{3}/3},\mathbb{S}_{-\sqrt{3}/3}\right\}$ and the fundamental sector is defined as  $\mathbb{F}_{\mathcal{T}_{\mathrm{hemi}}}=\left\{z\in\mathbb{R}^2,\ z_1>0,\ z_1>\frac{\sqrt{3}}{3}z_2\right\}$.
      \end{itemize}
\end{remark}
	\begin{figure}\renewcommand\thefigure{B.1}
	    \centering
	 \begin{tikzpicture}[scale=1.1]
		
		\draw[thick] (0, -5) -- (0, 5); 
		\draw[thick] (-5, 0) -- (5, 0); 
		
		\begin{scope}
			\draw[blue, thick] (-4.5, 2.5) -- (-1.5, 2.5); 
			\draw[green, thick] (-3, 4) -- (-3, 1); 
				\node [scale=1] at (-1, 2.5) {$y=0$};
			\node [scale=1] at (-2.5, 4) {$x=0$};
			\node [scale=1] at (-3, 0.5) {$\mathcal{Q}_l$};
            \node [scale=1] at (-2.2, 3) {$\mathbb{F}_{\mathcal{Q}_l}$};
             \fill[fill=blue,fill opacity=0.1] (-3,2.5) -- (-1.5,2.5) -- (-1.5,3.8) -- (-3,3.8) -- cycle; 
		\end{scope}
		
		\begin{scope}[shift={(6, 0)}]
			\draw[blue, thick] (-4.5, 2.5) -- (-1.5, 2.5); 
			\draw[green, thick] (-3, 4) -- (-3, 1); 
			\draw[orange, thick] (-4.5, 1) -- (-1.5, 4); 
			\draw[purple, thick] (-4.5, 4) -- (-1.5, 1); 
				\node [scale=1] at (-1, 2.5) {$y=0$};
			\node [scale=1] at (-2.5, 3.8) {$x=0$};
				\node [scale=1] at (-1.2, 4.2) {$y=x$};
			\node [scale=1] at (-4.2, 4.2) {$y=-x$};
			\node [scale=1] at (-3, 0.5) {$\mathcal{T}_\mathrm{iso}$};
            \node [scale=0.9] at (-1.9, 3) {$\mathbb{F}_{\mathcal{T}_\mathrm{iso}}$};
            \fill[fill=blue,fill opacity=0.1] (-3,2.5) -- (-1.5,2.5) -- (-1.5,4)  -- cycle; 
		\end{scope}
		
		\begin{scope}[shift={(0, -5)}]
		\draw[blue, thick] (-4.5, 2.50) -- (-1.5, 2.50); 
			\draw[orange, thick] (-4, 0.768) -- (-1.9,4.41 ); 
			\draw[red, thick] (-4, 4.232) -- (-2, 0.768); 
				\node [scale=1] at (-1, 2.5) {$y=0$};
				\node [scale=1] at (-1.1, 4) {$y=\sqrt{3}x$};\node [scale=1] at (-4.7, 4) {$y=-\sqrt{3}x$};
			\node [scale=1] at (-3, 0.5) {$\mathcal{T}_\mathrm{equi}$};
             \node [scale=0.8] at (-2.3, 2.75) {$\mathbb{F}_{\mathcal{T}_\mathrm{equi}}$};
             \fill[fill=blue,fill opacity=0.1] (-3,2.5) -- (-1.9,2.5) -- (-1.9,4.41 )  -- cycle; 
		\end{scope}
		
		\begin{scope}[shift={(6, -5)}]
		\draw[blue, thick] (-4.5, 2.50) -- (-1.5, 2.50); 
	\draw[orange, thick] (-4, 0.768) -- (-2,4.232 ); 
	\draw[red, thick] (-4, 4.232) -- (-2, 0.768); 
			\draw[pink, thick] (-4.25, 1.779) -- (-1.75, 3.221); 
			\draw[gray, thick] (-4.25, 3.221) -- (-1.75, 1.779); 
				\draw[green, thick] (-3, 4) -- (-3, 1); 
					\node [scale=0.7] at (-5, 2.5) {$y=0$};
				\node [scale=0.7] at (-3, 4.2) {$x=0$};
					\node [scale=0.7] at (-1.3, 4) {$y=\sqrt{3}x$};\node [scale=0.7] at (-4.7, 4) {$y=-\sqrt{3}x$};
						\node [scale=0.7] at (-1.4, 1.5) {$y=\frac{\sqrt{3}}{3}x$};\node [scale=0.7] at (-4.6, 1.5) {$y=\frac{-\sqrt{3}}{3}x$};
				\node [scale=1] at (-3, 0.5) {$\mathcal{T}_\mathrm{hemi}$};
                \node [scale=0.7] at (-2.1, 2.7) {$\mathbb{F}_{\mathcal{T}_\mathrm{hemi}}$};
                 \fill[fill=blue,fill opacity=0.1] (-3,2.5) -- (-1.75,2.5) -- (-1.75, 3.221)  -- cycle; 
		\end{scope}

	\end{tikzpicture}
	  \caption{All the reflections considered in Table \ref{TableB}. The fundamental sector is also indicated }
	    \label{domain}
	\end{figure}
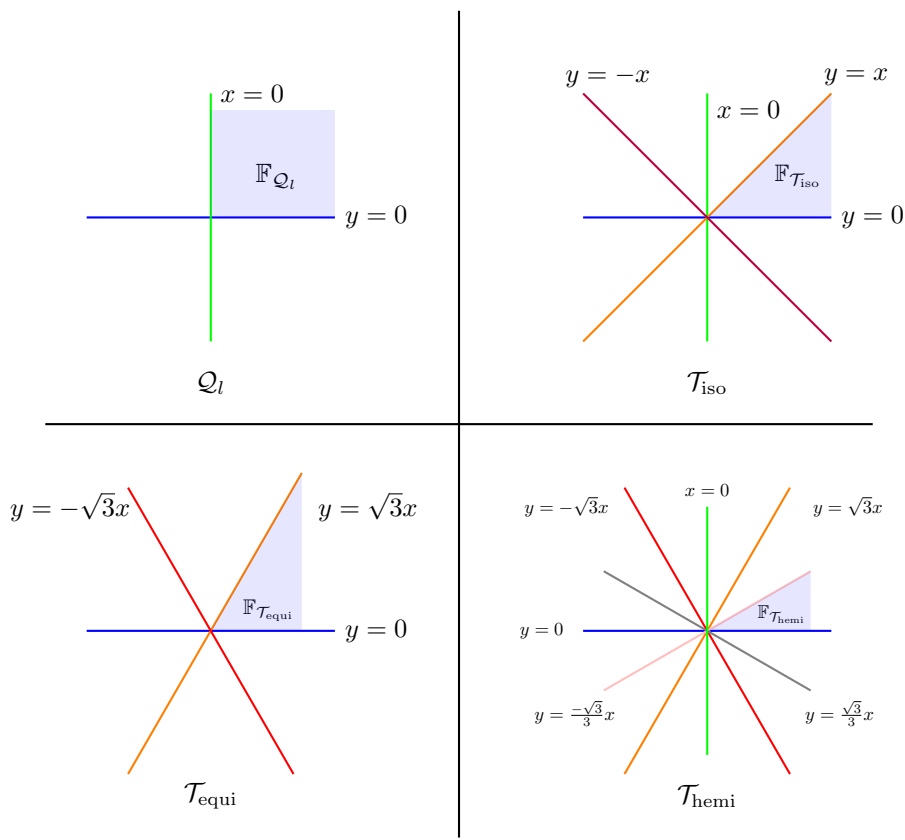

    \begin{remark}
        For every condition that we are considering in Table \ref{TableA} and Table \ref{TableB} there exist monochromatic waves in $\MW$ that satisfy that condition. More precisely, given any $\phi\in \MW$ we can construct a new symmetrized monochromatic wave $\varphi$ as follows: if  we denote by $g_i$ the elements of the group of isometries of Remark \ref{nod}, where $1\leq i\leq \left|\mathbb{G}_\mathcal{P}\right|$ and $g_1=\I$ is the identity, we can define the function $\varphi:\mathbb{R}^2\rightarrow\mathbb{R}$ as follows: \begin{equation}
        \varphi(z):=\frac{1}{\left|\mathbb{G}_\mathcal{P}\right|}\sum_{i=1}^{\left|\mathbb{G}_\mathcal{P}\right|}\left(-1\right)^{s_i^{\mathrm{BC}}}\phi\circ g_i,
    \end{equation}where $s_i^{\mathrm{N}}=0$ always and $s_i^{\mathrm{D}}$ is $0$ when $g_i$ is a composition of an even number of generators and $1$ when it is composed by an odd number. It is immediate to check that $\vp$ satisfies the conditions of Table \ref{TableB} and, consequently, of Table \ref{TableA}.
    \end{remark}

    \footnotesize{\bibliographystyle{alpha}\bibliography{main}}
    
\end{document}